\documentclass[12pt, twoside]{article}
\usepackage[T1, OT1]{fontenc}
\usepackage{hyperref}
\usepackage{authblk, lipsum}
\usepackage{amssymb}
\usepackage{amsmath, enumerate, mathabx, tikz}
\usepackage{amsthm}
\usepackage{color}
\usepackage{mathrsfs}
\usepackage{txfonts}
\usepackage{tikz}
\usetikzlibrary{positioning, patterns, decorations.pathreplacing}
\usetikzlibrary{positioning, patterns}
\usetikzlibrary{shapes.geometric,arrows}
\tikzstyle{box} = [rectangle,text centered, draw=black]
\tikzstyle{arrow} = [thick, ->, >=stealth]
\usepackage{float} 

\allowdisplaybreaks

\def\eps{\varepsilon}

\def\supp{\mathop\mathrm{\,supp\,}}

\newtheorem{theorem}{Theorem}[section]
\newtheorem{lemma}[theorem]{Lemma}
\newtheorem{proposition}[theorem]{Proposition}
\newtheorem{corollary}[theorem]{Corollary}
\theoremstyle{definition}
\newtheorem{definition}[theorem]{Definition}

\numberwithin{equation}{section}

\begin{document}
\UseRawInputEncoding
\arraycolsep=1pt

\title{\bf Measures supported on partly normal numbers 
\footnotetext{\hspace{-0.35cm} 2020 {\it
Mathematics Subject Classification}.
Primary 42A63, 28A78, 11K16; Secondary 42A38, 28A80, 11K36.
\endgraf {\it Key words and phrases}.
Fourier analysis, Fourier series, trigonometric series,  sets of uniqueness and of multiplicity, normal and non-normal numbers, Rajchman measures, metric number theory 
\endgraf
}}
\author{Malabika Pramanik and Junqiang Zhang}  
\date{ }
\maketitle
\newcommand{\Addresses}{{
  \bigskip
  \footnotesize

  M.~Pramanik, \textsc{Department of Mathematics, 1984 Mathematics Road, University of British Columbia, Vancouver, Canada V6T 1Z2} \par\nopagebreak
  \textit{E-mail address}: \texttt{malabika@math.ubc.ca}
  
  \medskip

  J.~Zhang, \textsc{School of Science, China University of Mining and Technology, Beijing 100083, P.~R.~China}\par\nopagebreak
  \textit{E-mail address}: \texttt{jqzhang@cumtb.edu.cn}
  
%

}}

\vspace{-0.8cm}

\begin{center}
\begin{minipage}{13cm}
{\small {\bf Abstract.} A real number $x$ is {\em{normal}} with respect to an integer base $b \geq 2$ if its digit expansion in this base is ``equitable'', in the sense that for $k \geq 1$, every ordered sequence of $k$ digits from $\{0, 1, \ldots, b-1\}$ occurs in the digit expansion of $x$ with the same limiting frequency. Borel's classical result \cite{b09} asserts that Lebesgue-almost every number $x$ is normal in every base $b \geq 2$. This three-part article considers sets of partial normality. Given any choice of integer bases $\mathscr{B}, \mathscr{B}' \subseteq \{2, 3, \ldots\}$, we investigate measure-theoretic properties of the set $\mathscr{N}(\mathscr{B}, \mathscr{B}')$, whose members are, by definition, normal in the bases of $\mathscr{B}$ and non-normal in the bases of $\mathscr{B}'$.
\vskip0.1in 
\noindent A pair of sets $(\mathscr{B}, \mathscr{B}')$ is {\em{compatible}} if 
any $(b, b') \in \mathscr{B} \times \mathscr{B}'$ is multiplicatively independent. For compatible $(\mathscr{B}, \mathscr{B}')$ with $\mathscr{B}' \ne \emptyset$, we construct singular probability measures supported on $\mathscr N(\mathscr{B}, \mathscr{B}')$  that are both Frostman and Rajchman, extending prior work of Pollington \cite{p81} and Lyons \cite{l86}. The Rajchman property completely answers a question of Kahane and Salem \cite{Kahane-Salem-64}, identifying $\mathscr N(\mathscr{B}, \mathscr{B}')$ as a set of multiplicity (in the Fourier-analytic sense) if and only if $(\mathscr{B}, \mathscr{B}')$ is compatible.  
\vskip0.1in
\noindent The methodological contribution of the article is the construction of a class of probability measures called {\em{skewed measures}}. These measures depend on a number of parameters that can be independently adjusted to ensure (subsets of) properties such as almost everywhere normality, non-normality, ball conditions and Fourier decay. As a consequence, certain skewed measures can be crafted  to enjoy the maximum possible normality subject to their parametric constraints. The first part of the article describes the construction of skewed measures and lists their properties needed for the main results. The next two parts establish these properties. The second part focuses on non-normality, Frostman and Rajchman properties of skewed measures. The third part is devoted exclusively to the study of normality on their support. Distinctive features of the last part include exponential sum estimates influenced by number-theoretic lemmas of Schmidt \cite{s60}.   
}    
\end{minipage}
\end{center}
\newpage
\tableofcontents
\newpage

\section{Introduction} \label{intro-section} 
Given a number $x \in [0,1)$ and an integer $b \geq 2$, let 
\[ x = \sum_{j=1}^{\infty} \frac{x_j}{b^{j}}, \qquad x_j \in \mathbb Z_b :=  \{0, 1, \ldots, b-1\} \]  denote the digit expansion of $x$ with respect to base $b$. We say that $x$ is {\em{normal}} in this base \cite{{b09}, {nz51}}, or {\em{$b$-normal}} for short, if for any $k \geq 1$, every $k$-long ordered sequence of integers with entries in $\mathbb Z_b$ occurs equally often in the digit expansion of $x$. Specifically, $x$ is $b$-normal if for every $k \geq 1$ and every block of digits $(d_1, \ldots, d_k) \in \mathbb Z_b^k$,  
  \begin{equation}  \lim_{N \rightarrow \infty} \frac{1}{N} \#\Bigl\{1 \leq n \leq N : (x_n, x_{n+1}, \ldots, x_{n+k-1}) = (d_1, \ldots, d_k) \Bigr\} = \frac{1}{b^k}. \label{b-normal} \end{equation}  
A real number $x$, not necessarily in $[0,1)$, is said to be $b$-normal if its fractional part $\{x\} := x - \lfloor x \rfloor \in [0, 1)$ is $b$-normal in the sense of \eqref{b-normal}. Here
$\lfloor x \rfloor$ denotes the largest integer less than or equal to $x$. It is said to be {\em{absolutely normal}} if it is normal in all integer bases $b \geq 2$. Otherwise, it is said to be {\em{non-normal}}. To paraphrase, a number $x$ is non-normal if it fails to be normal in some base $b \geq 2$. The research area surrounding normal numbers is vast and diverse, with connections to harmonic analysis, geometric measure theory, ergodic theory, metrical number theory and computer science. Different perspectives are reflected in Koksma \cite{Koksma-book}, Niven \cite{Niven-book}, Salem \cite{Salem-book}, Kuipers and Niederreiter \cite{KN-book}, Harman \cite{{Harman-book}, {Harman-2}}, Bugeaud \cite{b12}, and in the bibliography therein. This article focuses on sets of {\em{partly normal numbers}} or {\em{numbers of limited normality}}. Such sets are characterized by the property that  their members are normal with respect to a given choice of bases and not in others. In particular, this will include the set of {\em{absolutely non-normal numbers}}, namely those that are non-normal in every integer base at least 2.   

\subsection{On a question of Kahane and Salem} A classical theorem of \'E.~Borel \cite{b09} says that Lebesgue-almost every $x \in [0,1)$ is absolutely normal. This implies that $\lambda$-almost every point is absolutely normal as well, for any probability measure $\lambda$ on $[0, 1)$ that is absolutely continuous with respect to Lebesgue. It is well-known, by the Riemann-Lebesgue lemma, that 
\[ |\widehat{\lambda}(n)| \rightarrow 0 \quad \text{ as } \quad |n| \rightarrow \infty, \] 
where $\widehat{\lambda}(n)$ denotes the $n^{\text{th}}$ Fourier coefficient of $\lambda$:  
\begin{equation} \widehat{\lambda}(n) := \int_{0}^{1} e(nx) \, d\lambda(x), \qquad e(x) := e^{-2 \pi i x}, \quad n \in \mathbb Z := \{0, \pm 1, \pm 2, \ldots \}. \label{Fourier-coefficient-def} \end{equation}  
Further, the decay rate of $\widehat{\lambda}(n)$ can be arbitrarily slow as $|n| \rightarrow \infty$. It is therefore natural to ask if a statement like Borel's theorem holds for {\em{any}} Borel probability measure on $[0, 1)$ whose Fourier coefficients vanish at infinity. Indeed, the following question was posed by Kahane and Salem in \cite[p 261]{Kahane-Salem-64} in the context of a problem on dyadic expansions: 
\begin{equation} \label{KS-q1}  
\begin{aligned} 
&\text{``{\em{Is the set of non-normal numbers a set $U^{\ast}$? }}} \\  
&\text{\em{ In other words, is it a set of measure zero with respect to every }} \\
&\text{\em{ positive measure whose Fourier coefficients tend to zero at infinity?''}}
\end{aligned}
\end{equation}  
\footnote{The notation $U^{\ast}$ was used in \cite{Kahane-Salem-64} as shorthand for ``sets of uniqueness in the wide sense'', the description of which appears in the second part of their question \eqref{KS-q1}. The connection of this question to sets of uniqueness is explained in Section \ref{multiplicity-section}} 
Measures of the latter type are known in the literature as {\em{Rajchman measures}}, and a set supporting a Rajchman measure is said to possess the {\em{Rajchman property}}. So the question of Kahane and Salem can be reformulated as: 
\begin{equation} \label{KS-q} 
{\text{\em{Can a set of non-normal numbers support a Rajchman measure?}}}
\end{equation} 
A negative answer to Question \eqref{KS-q1}, which means an affirmative answer to Question \eqref{KS-q}, was obtained by Lyons \cite{l86}, for the special set of numbers that are not 2-normal. The proof in \cite{l86} generalizes easily to sets of numbers that are non-normal in a fixed single base, not necessarily base 2. But the problem remained open for general non-normal sets, whose members are non-normal in (potentially infinitely) many bases, and normal in others. The sharper version of this question appears in \eqref{KS-q2} below. One of the objectives of this article is to answer the question \eqref{KS-q} of Kahane and Salem in full generality.  See Theorem \ref{mainthm-1} in this section. 
\vskip0.1in
\noindent Let us introduce the following notation. For $\mathbb N := \{1, 2, 3, \ldots \}$ and non-empty collections $\mathscr{B}, \mathscr{B}' \subseteq \mathbb N \setminus \{1\}$, we set
\begin{align}  
\mathscr N(\mathscr B, \cdot) &:=  \Bigl\{x \in \mathbb R: x \text{ is $b$-normal for all } b \in \mathscr B \Bigr\}, \label{NB} \\  
\mathscr N(\cdot, \mathscr B') &:=  \Bigl\{x \in \mathbb R: x \text{ is not $b'$- normal for any } b' \in \mathscr B' \Bigr\}, \label{NB'}  \\
\mathscr N(\mathscr B, \mathscr B') &:= \left\{x \in \mathbb R \; \Biggl| \;  \begin{aligned} &x \text{ is normal in base } b \text{ for all } b \in \mathscr B, \\  &x \text{ is non-normal in base } b' \text{ for all } b' \in \mathscr B' \end{aligned}  \right\}. \label{NBB'} 
\end{align}  
Clearly $\mathscr N(\mathscr B, \mathscr B') = \mathscr N(\mathscr B, \cdot) \cap \mathscr N(\cdot, \mathscr B')$. If exactly one of the collections $\mathscr{B}, \mathscr{B}'$ is empty, we follow the conventions
\[ \mathscr N(\cdot, \emptyset) := \mathscr N(\mathbb N \setminus \{1\}, \cdot) \quad \text{ and } \quad \mathscr N(\emptyset, \cdot) := \mathscr N(\cdot, \mathbb N \setminus \{1\}). \] 
Thus $\mathscr N(\cdot, \emptyset)$ and $\mathscr N(\emptyset, \cdot)$ represent respectively the sets of numbers that are normal to every base (absolutely normal) and no base (absolutely non-normal). We always assume $\mathscr{B} \cup \mathscr{B}' \ne \emptyset$.
\vskip0.1in 
\noindent The notion of multiplicative independence is important in determining normality with respect to different bases. We say that two bases $r, s \in \mathbb N \setminus \{1\}$ are {\em{multiplicatively dependent}}, and write 
\begin{equation} \label{def-mult-dep} r \sim s \quad \text{ if } \quad \frac{\log r}{\log s} \text{ is rational }, \; \text{ that is, if there exist $m, n \in \mathbb N$ such that } r^m = s^n.  \end{equation}    
Otherwise, we write $r \not\sim s$ and call $r, s$ {\em{multiplicatively independent}}. Schmidt \cite[Theorem 1]{s60} has shown that if $r \sim s$, then any number that is $r$-normal must also be $s$-normal. In other words, the property of normality of a number, or lack thereof, remains invariant for multiplicatively dependent bases. As a result,  
\begin{equation} \label{N4} 
\mathscr N(\mathscr{B}, \mathscr{B}') = \mathscr N(\overline{\mathscr{B}}, \overline{\mathscr{B}'}), \text{ where } 
\overline{\mathscr{B}} := \left\{b \in \mathbb N \setminus \{1\} \; \Bigl| \; b \sim b_0 \text{ for some } b_0 \in \mathscr{B} \right\}
\end{equation}
denotes the closure of $\mathscr{B}$ under the relation $\sim$. The result of Schmidt also implies that 
\begin{equation} 
\text{ if } \overline{\mathscr{B}} \cap \overline{\mathscr{B'}} \neq \emptyset,
 \text{ then } 
 \mathscr{N}(\mathscr{B}, \mathscr{B}') = \emptyset. 
 \label{empty-normal} \end{equation} 
To avoid trivialities, such situations need to be eliminated from consideration. 
\vskip0.12in 
\begin{definition} \label{compatibility-definition}
Suppose that $\mathscr{B}, \mathscr{B}' \subseteq \mathbb N \setminus \{1\}$ are any two non-empty collections of integer bases. We call the pair $(\mathscr{B}, \mathscr{B}')$ {\em{compatible}} if 
\begin{equation} b \not\sim b' \text{ for all $(b,b') \in \mathscr{B} \times \mathscr{B}'$,} \; \; \text{and therefore for all $(b, b') \in \overline{\mathscr{B}} \times \overline{\mathscr{B}'}$.} \label{mult-indep-def} \end{equation}
\end{definition} 
\noindent A compatible pair $(\mathscr{B}, \mathscr{B}')$ will be called {\em{maximally compatible}} if $\overline{\mathscr{B}} \sqcup \overline{\mathscr{B}'} = \mathbb N \setminus \{1\}$. If exactly one of $\mathscr{B}$ or $\mathscr{B}'$ is empty, then \eqref{mult-indep-def} is taken to be vacuously true, so  $(\mathbb N \setminus \{1\}, \emptyset)$ and $(\emptyset, \mathbb N \setminus \{1\})$ are considered maximally compatible pairs. 
\vskip0.12in
\noindent For compatible pairs of bases $(\mathscr{B}, \mathscr{B}')$, points in $\mathscr N(\mathscr{B}, \mathscr{B}')$ are plentiful. Pollington \cite{p81} has shown that  the set $\mathscr N(\mathscr{B}, \mathscr{B}')$ has full Hausdorff dimension for any maximally compatible pair $(\mathscr{B}, \mathscr{B}')$, even though it is Lebesgue-null for $\mathscr{B}' \neq \emptyset$.  In this terminology, Question \eqref{KS-q} of Kahane and Salem \cite{Kahane-Salem-64} can be stated as: \begin{equation} \label{KS-q2}
{\text{\em{Does $\mathscr N(\mathscr{B}, \mathscr{B}')$ support a Rajchman measure for every compatible pair $(\mathscr{B}, \mathscr{B}')$?}}} 
\end{equation} 
Lyons' result \cite{l86} proves that $\mathscr{N}(\cdot, \{2\})$ supports such a measure $\nu_{\text{L}}$. But it does not specify any base with respect to which points in the support of $\nu_{\text{L}}$ are normal. As a result, this does not fully address the question posed in \eqref{KS-q2} even in the special case where $2 \in \mathscr{B}'$. In an earlier article \cite{PZ-1}, we used the measure $\nu_{\mathtt L}$ constructed by Lyons as a case study, and showed that $\nu_{\mathtt L}$-almost every point is normal in all odd bases, and non-normal in all even ones. Motivated by the analysis in \cite{PZ-1}, this article provides an affirmative answer to Question \eqref{KS-q2} for all compatible pairs $(\mathscr{B}, \mathscr{B}')$, including the edge case $(\mathscr{B}, \mathscr{B}') = (\emptyset, \mathbb N \setminus \{1\})$. 
\begin{theorem} \label{mainthm-1} 
For every compatible pair $(\mathscr{B}, \mathscr{B}')$, there exists a Rajchman probability measure $\mu = \mu[\mathscr{B}, \mathscr{B}']$ such that $\mu$-almost every point lies in $\mathscr N(\mathscr{B}, \mathscr{B}')$. In particular, this is true for $\mathscr N(\emptyset, \mathbb N \setminus \{1\})$, the set of absolutely non-normal numbers.  
\end{theorem} 
\noindent{\em{Remarks:}}
\begin{enumerate}[1.]
\item Theorem 1 is proved in Section \ref{mainthm-1-proof-section}. See Section \ref{proof-overview-section} for details.  
\item One naturally wonders about the Fourier decay rate of the Rajchman measure $\mu[\mathscr{B}, \mathscr{B}']$ provided by Theorem \ref{mainthm-1}. Our proof technique is constructive and effective; given any compatible pair $(\mathscr{B}, \mathscr{B}')$, it will produce a Rajchman measure whose decay rate is explicitly computable. While we have not attempted to quantify the rate of decay of  $\widehat{\mu}(\cdot)$ for an arbitrary compatible pair $(\mathscr{B}, \mathscr{B}')$, in earlier work \cite{PZ-1} we have provided an example where the rate is optimal, in the following sense: the set $\mathscr{N}(\mathscr{O}, \mathscr{E})$ consisting of odd-normal but not even-normal numbers supports a probability measure $\mu$ with the property
\[ \bigl| \widehat{\mu}(n) \bigr| \leq \bigl(\log \log |n| \bigr)^{-1 + \kappa} \text{ for all } \kappa > 0 \text{ and all sufficiently large } |n|, \; n \in \mathbb Z. \]    
This decay rate is sharp up to $\kappa$-loss; it is known \cite[Remark 1, page 8573]{PVZZ} that if a set $E \subseteq \mathbb R$ supports a measure $\nu$ such that 
\[ \bigl| \widehat{\nu}(n) \bigr| \leq  \bigl(\log \log |n| \bigr)^{-1 - \kappa} \text{ for some } \kappa > 0 \text{ and all sufficiently large } |n|,  \]
then $\nu$-almost every number in $E$ is absolutely normal.   
\vskip0.1in 
\noindent  On the other hand, for the set of absolutely non-normal numbers $\mathscr{N}(\emptyset, \cdot)$ we show in Section \ref{Appendix} that the Rajchman measure $\mu$ given by Theorem \ref{mainthm-1} decays like $1/(\log^{(3)}|\xi|)$, where $\log^{(3)}$ denotes the thrice-iterated logarithm. We do not know whether this is optimal. More generally, the relation between sets of partial normality and the optimal Fourier decay of probability measures supported therein remains an interesting open problem. 
\end{enumerate} 
\subsection{Partly normal numbers as sets of multiplicity} \label{multiplicity-section} Apart from its obvious connection with the theory of measures and numbers, Question \eqref{KS-q} is important in Fourier analysis. A set $E \subseteq [0,1)$ is a {\em{set of uniqueness}} if any complex-valued trigonometric series of the form 
\begin{equation*} 
\sum_{n \in \mathbb Z} a_n e(nx), \qquad e(y) := e^{-2 \pi i y}
\end{equation*} 
that converges to zero on $[0,1) \setminus E$ must be identically zero, i.e. $a_n = 0$ for all $n \in \mathbb Z$. If $E$ is not a set of uniqueness, it is called a {\em{set of multiplicity}}. With a rich history rooted in the early works of Riemann \cite{Riemann} and Cantor \cite{Cantor} and spanning more than a century, these sets have been studied in a variety of settings \cite{{Young}, {Menshov}, {Salem}, {Zygmund}, {Salem-Zygmund}, {Kahane-1}, {Bary}}. A comprehensive account may be found in the book of Kechris and Louveau  \cite{KL-book}. Recent advances in ergodic theory, fractal geometry and metrical number theory have generated renewed interest in properties of uniqueness, which have been examined for self-similar fractal sets \cite{{Li-Sahlsten}, {Varju-Yu}, {Bremont}, {Algom-Hertz-Wang}, {Gao-Ma-Song-Zhang}, {Rapaport}} and for number-theoretic sets occurring in Diophantine approximation \cite{{Lyons-thesis}, {Lyons-85}, {l86}, {b12}, {Bluhm}}. This article is a contribution to the literature on uniqueness and multiplicity of sets comprising numbers with limited normality. It is a natural follow-up of \cite{l86}, which established $\mathscr{N}(\cdot, \{2\})$ as a set of multiplicity, but did not specify whether the property of multiplicity is retained in any of its subsets that enjoy normality in other bases. Building on the case study of \cite{l86} conducted in \cite{PZ-1}, we are able to identify all non-trivial sets of partial (including zero) normality as sets of multiplicity.   
\vskip0.1in 
\noindent The connection between sets of uniqueness and the Rajchman property is classical, and described in \cite[Chapter 2, \S4-5]{KL-book}.  The following theorem \cite[Chapter 2, Theorem 4.1]{KL-book}, attributed to the combined work of Piatetski-Shapiro \cite{{P-S1}, {P-S2}} and Kahane and Salem \cite{Kahane-Salem-63}, is particularly relevant for this article: {\em{if $E$ supports a Rajchman measure, then $E$ is a set of multiplicity}}. Thus, an affirmative answer to Question \eqref{KS-q} for a set of non-normal numbers automatically identifies it as a set of multiplicity. 
\begin{corollary} 
Let $(\mathscr{B}, \mathscr{B}')$ be a compatible pair of base sets in the sense of Definition \ref{compatibility-definition}. Then the set of partly normal numbers $\mathscr N(\mathscr{B}, \mathscr{B}')$ is a set of multiplicity.  This includes $\mathscr N(\emptyset, \mathbb N \setminus \{1\})$, the set of real numbers that are non-normal in every base.  
\end{corollary} 
\noindent Combining the corollary with previously known results (for instance \cite{{s60},{p81}} and \eqref{empty-normal}) yields the following dichotomy for partly normal numbers:
\begin{itemize} 
\item Either $(\mathscr{B}, \mathscr{B}')$ is compatible, in which case $\mathscr N(\mathscr{B}, \mathscr{B}')$ is a set of multiplicity of full Hausdorff dimension;
\item Otherwise $(\mathscr{B}, \mathscr{B}')$ is incompatible, in which case $\mathscr N(\mathscr{B}, \mathscr{B}') = \emptyset$ and therefore a set of uniqueness (of zero Hausdorff dimension).   
\end{itemize}  
\subsection{Measures seeking maximal normality} Another motivation for this paper originates in a paper of Cassels \cite{c59}, in response to a question posed by Steinhaus. The intent in that original question was to explore the inter-dependencies of normality in different bases \footnote{The record of the original question seems to be lost in reprint. It appears in \cite{c59} and has been stated in the above form by de Bruijn \cite{deBruijn} as part of the review of \cite{c59}. Nagasaka has noted the omission \cite[p 91]{Nagasaka} in later editions of the reference: ``H. Steinhaus once raised a question in the ”New Scottish Book” as to how far the property of being normal with respect to different bases is independent. This problem was cited as Problem 144 by J.~W.~S. Cassels [4], but we cannot find any trace of this problem in the ”Scottish Book” newly edited by R. Daniel Mauldin [71].'' }: {\em{``... does normality with respect to infinitely many $b$-s imply normality with respect to all other $b$-s''?}} \cite[Problem 144, p14]{Steinhaus-Q},  \cite{deBruijn}. Cassels \cite{c59} answered this question in the negative by showing that almost every point in the middle-third Cantor set is normal in every base $b \notin \{3^{n}: n \in \mathbb N\}$, with respect to the Cantor-Lebesgue measure. Of course, every point in the middle-third Cantor set is non-normal in bases $\{3^n : n \in \mathbb N\}$, by construction. Since then, stronger negative answers to certain aspects of Steinhaus' question have been furnished \cite{{p81}, {s60}, {s61-62}, {s66}, {BBFKW-2010}}. For example, Pollington \cite{p81} has shown that all partly normal sets $\mathscr{N}(\mathscr{B}, \mathscr{B}')$ have full Hausdorff dimension, provided $(\mathscr{B}, \mathscr{B}')$ is a compatible pair in the sense of Definition \ref{compatibility-definition}. Nonetheless, Cassels' result remains striking in the following sense: it shows that the natural measure on the Cantor middle-third set, which is non-normal by construction, embodies the maximum allowable normality subject to its definition. Subsequently, the works of Schmidt \cite{{s60}, {s61-62}, {s66}} and many others \cite{{Furstenberg-1967}, {Furstenberg-1970}, {Furstenberg-2008}, {Hochman-JEMS-2012}, {Hochman-Lectures-2014}, {Hochman-Shmerkin-2012}, {Hochman-Shmerkin-2015}, {Peres-Shmerkin}, {Shmerkin-2019}} have substantiated the phenomenon of natural randomness inherent in multiplicatively independent bases, in a variety of ways. For instance, \cite[Lemma 5]{s60} quantifies a lack of shared structure in the digit sequence of integers in bases $r$ and $s$, provided $r \not\sim s$. Integers whose digit sequences are highly structured in base $r$ behave essentially randomly when expressed in base $s$.  
\vskip0.1in
\noindent Drawing upon the insights provided in \cite{{s60}, {p81}, {l86}} and for a given collection of bases $\mathscr{B}' \subseteq \mathbb N \setminus \{1\}$, we describe in this paper a class of probability measures called {\em{skewed measures}}. The mass distribution underpinning skewed measures favours, by design, certain digits in the bases of $\mathscr{B}'$. It is therefore no surprise that under mild assumptions, points in the support of such measures are almost surely non-normal in every base of $\mathscr{B}'$. The surprising observation is that depending on the choice of favoured digits, a large but strict subclass of these measures enjoys the maximum possible normality subject to this constraint, i.e., almost every point in their support is in $\mathscr N(\mathscr{B}, \mathscr{B}')$, where $\mathscr{B} = \{2, 3, \ldots\} \setminus \overline{\mathscr{B}'}$. We call this behaviour {\em{normality-seeking}}. 
\vskip0.1in
\noindent The class of normality-seeking skewed measures is rich in other ways.  
Although the middle-third Cantor set cannot support a Rajchman measure, our work shows that some normality-seeking measures can accommodate the Rajchman property. Further, one can construct a sequence of normality-seeking measures on $\mathscr{N}(\mathscr{B}, \mathscr{B}')$ with the Frostman property. This means that this sequence of measures obeys near-optimal ball conditions \cite[Chapter 8, p 112]{Mattila-1} ensuring the full Hausdorff dimension of their support \cite{p81}. This offers a finer response to Steinhaus' question, generalizing \cite{{c59}, {p81}, {l86}}: for every choice of $\mathscr{B}' \subseteq \mathbb N \setminus \{1\}$, there exist (normality-seeking) natural measures that are simultaneously Frostman and Rajchman, that target non-normality in bases of $\mathscr{B}'$, and in the process, encompass normality in all other allowable bases. 
\begin{theorem} \label{mainthm-2} 
For any non-empty choice of bases $\mathscr{B}' \subseteq \{2, 3, \ldots\}$, there are two collections of probability measures $\mathscr{M}(\mathscr{B}') \subseteq \mathscr{M}^{\ast}(\mathscr{B}')$ on $[0, 1)$ with the following properties. 
\begin{enumerate}[(a)] 
\item For any $\mu \in \mathscr{M}^{\ast}(\mathscr{B}')$, $\mu$-almost every point is non-normal in every base of $\mathscr{B}'$. \label{mainthm-2-parta}
\item For every $\mu \in \mathscr{M}(\mathscr{B}')$ and in the notation of \eqref{NBB'} and \eqref{N4}, 
\begin{equation} \text{$\mu$-almost every point $x$ lies in } \mathscr N(\mathscr{B}, \mathscr{B}'), \quad \text{ where } \mathscr{B} := \{2, 3, \ldots \} \setminus \overline{\mathscr{B}'} \label{normality-seeking-def} \end{equation}   
is the largest collection of bases for which $(\mathscr{B}, \mathscr{B}')$ is a compatible pair. \label{mainthm-2-partb} 
\item For every $\eta > 0$, there exists a measure $\mu_{\eta} \in \mathscr{M}(\mathscr{B}')$ obeying 
 \begin{equation} \label{Frostman-condition} 
 \sup \left\{\frac{\mu_{\eta}([x-r, x+r])}{r^{1 -\eta}} \; \Biggl| \; x \in \mathbb R, r > 0 \right\} < \infty. 
 \end{equation} 
The measure $\mu_{\eta}$ can be chosen to be Rajchman.  \label{mainthm-2-partc} 
 \end{enumerate} 
\end{theorem}  
\noindent {\em{Remarks: }} \begin{enumerate}[1.]
\item Theorem \ref{mainthm-2} is proved in Section \ref{mainthm-2-proof-section}; more details of the proof structure are in Section \ref{proof-overview-section}. 
\item Lyons' measure $\nu_{\mathtt L}$ \cite{l86} is not skewed, nor does it not seek normality, in the sense described above. It is however created to attain non-normality in base 2, and is Rajchman. 
\item The construction of $\nu_{\mathtt L}$ can be lightly modified to obtain a skewed measure $\tilde{\nu}_{\mathtt L} \in \mathscr{M}^{\ast}(\mathscr{B}')$ with $\mathscr{B}' = \{2\}$. As we point out in Section \ref{nonnormality-discussion-section}, the analysis in \cite[Section 3.3]{PZ-1} can be adpated to show that (like $\nu_{\mathtt L}$) $\tilde{\nu}_{\mathtt L}$-almost every point is non-normal in all even bases $\mathtt e$ and therefore for certain bases $\mathtt e  \nsim 2$. Thus the inclusion $\mathscr{M}(\mathscr{B}') \subseteq \mathscr{M}^{\ast}(\mathscr{B}')$ is in general strict. 
\end{enumerate}  
\subsection{Proof overview and layout of the article} \label{proof-overview-section}
The proofs of Theorems \ref{mainthm-1} and \ref{mainthm-2} are presented in three distinct parts, as seen from Figure \ref{layout-fig}. 
\begin{enumerate}[1.]
\item {\em{Part \ref{intro-part}: Construction and properties of skewed measures (Sections \ref{elem-op-section}--\ref{mainproof-section}):}} Section \ref{elem-op-section} describes an elementary operation that forms the building block of an iterative construction.Though inspired by earlier constructions of Lyons \cite{l86} and Pollington \cite{p81} in the sense that it favours certain digit sequences to encourage non-normality, the operation is not identical to either. A comparison and contrast of all three constructions is included in Section \ref{comparison-section}. A skewed measure is defined in Section \ref{iteration-section} as the final outcome of this iterative process. Section \ref{skewed-measure-properties-section} lists (in Propositions \ref{non-normality-prop}--\ref{normal-prop}), but does not prove, the key properties of skewed measures necessary for Theorems \ref{mainthm-1} and \ref{mainthm-2}. Assuming these, Theorems \ref{mainthm-1} and \ref{mainthm-2} are proved in Section \ref{mainproof-section}.
\item {\em{Part \ref{Part-nonnorm-F-R}: Three properties of skewed measures (Sections \ref{non-normality-section}--\ref{Appendix}):}} The rest of the article proves the four propositions of Section \ref{skewed-measure-properties-section} supporting the main results. Part \ref{Part-nonnorm-F-R} is devoted to
\begin{itemize} 
\item Non-normality (stated in Proposition \ref{non-normality-prop}, proved in Section \ref{non-normality-section})
\item Frostman property (stated in Proposition \ref{Frostman-prop}, proved in Section \ref{Frostman-section}), and 
\item Rajchman property (stated in Proposition \ref{Rajchman-prop}, proved in Section \ref{Rajchman-section}).  
\end{itemize}    
Establishing the Rajchman property is substantially more involved compared to the others, and relies on exponential sum estimates derived in Sections \ref{exp-sum-section} and \ref{mu-Rajchman-section}. 
These estimates are not reliant on number-theoretic properties of $\mathscr{B}'$, and are also useful 
in later Sections \ref{u-pointwise-estimate-section}, \ref{v-pointwise-section} and \ref{v-pointwise-section-take-2} of Part \ref{part-normality}. This section highlights the role of the elementary operation in attaining the Rajchman property. As an application, we find a probability measure supported on absolutely non-normal numbers with quantifiable Fourier decay; see Section \ref{Appendix}.  
\item {\em{Part \ref{part-normality}: Normality on the support of skewed measures (Section \ref{normality-overview-section}--\ref{estimating-vm-section-Part2}):}} Part \ref{part-normality} is given over to the proof of Proposition \ref{normal-prop}, which describes the normality-seeking behaviour of skewed measures. This is the most technically demanding part of the article. A detailed proof structure for this part is given on page \pageref{normality-proof}. In addition to the content of Section \ref{exp-sum-section}, the proof of normality requires finer exponential sum estimates. These are given in Sections \ref{u-pointwise-estimate-section}, \ref{v-pointwise-section} and \ref{v-pointwise-section-take-2}. 
Key ingredients of the proof involve number-theoretic consequences of multiplicative independence, noted by Schmidt \cite{s60}, and are recalled in Section \ref{number-theoretic-tools-section}, along with their impact on the problem at hand. This part of the article offers more granular information about the features of normality than is required for Theorems \ref{mainthm-1} and \ref{mainthm-2}; e.g. see Proposition \ref{normality-special-prop}. Such results may be of independent interest.    
\end{enumerate} 
\begin{center}
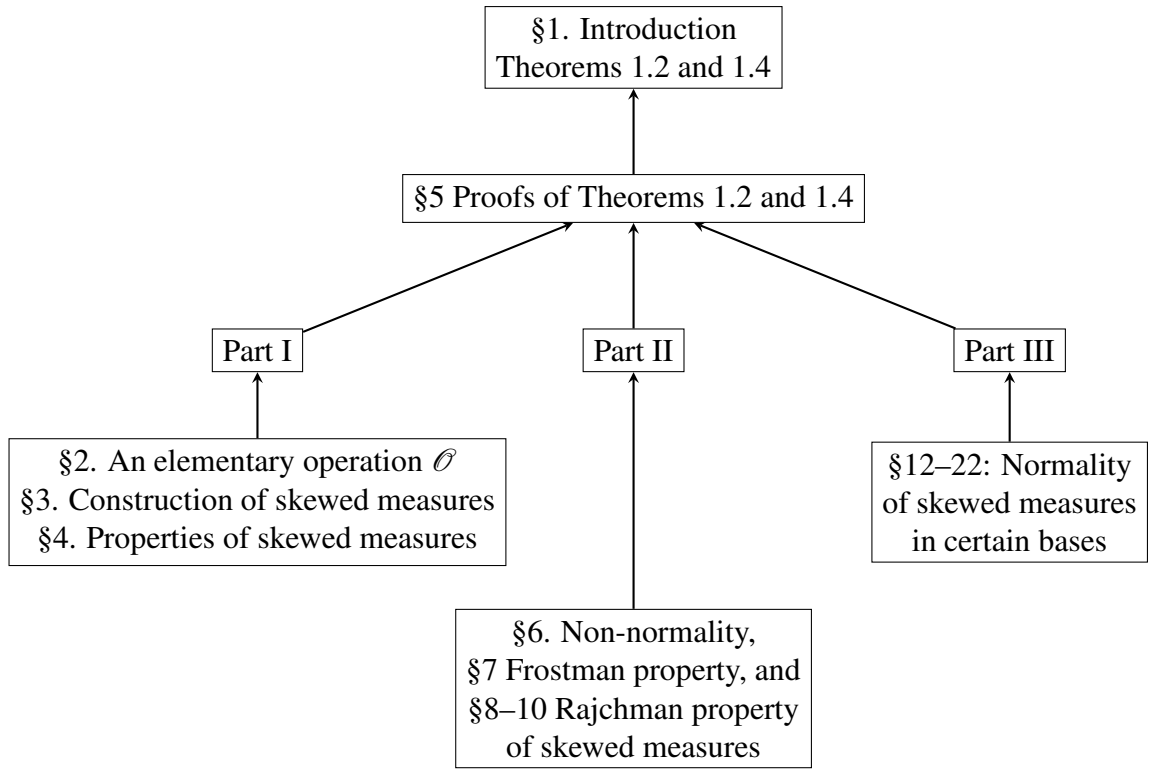
\begin{figure}
\begin{tikzpicture}[node distance=2cm]
\node (thm) [box, align=center] {\S \ref{intro-section}. Introduction \\ Theorems \ref{mainthm-1} and \ref{mainthm-2}};
\node (proof) [box, below of= thm, xshift=0cm, yshift =-0.01cm, align=center] {\S \ref{mainproof-section} Proofs of Theorems \ref{mainthm-1} and \ref{mainthm-2}};   
\node (S5) [box, below of= proof, xshift=-5cm, yshift=-0.02cm, align=center] {Part \ref{intro-part}};
\node (S4) [box, below of=proof, xshift=0cm, yshift=-0.02cm, align=center]{Part \ref{Part-nonnorm-F-R}};
\node (S6) [box, below of=proof, xshift=5cm, yshift=-0.02cm,align=center] {Part \ref{part-normality}};
\node (S7) [box, below of= S5, xshift=0cm, yshift=-0.02cm,align=center] {\S \ref{elem-op-section}. An elementary operation $\mathscr{O}$ \\ \S \ref{iteration-section}. Construction of skewed measures \\ \S \ref{skewed-measure-properties-section}. Properties of skewed measures};
\node (S8) [box, below of= S6, xshift=0cm, yshift=-0.02cm,align=center] {\S \ref{normality-overview-section}--\ref{estimating-vm-section-Part2}: Normality \\ of skewed measures \\ in certain bases};
\node (Sn) [box, below of= S4, xshift=0cm, yshift=-2.5cm,align=center] {\S \ref{non-normality-section}. Non-normality, \\ \S \ref{Frostman-section} Frostman property, and \\ \S \ref{exp-sum-section}--\ref{Rajchman-section} Rajchman property \\ of skewed measures};
\draw [arrow] (proof) -- (thm);
\draw [arrow] (S5) -- (proof) ; 
\draw [arrow] (S4) -- (proof) ;
\draw [arrow] (S6) -- (proof) ;
\draw [arrow] (S7) -- (S5) ;
\draw [arrow] (S8) -- (S6) ;
\draw [arrow] (Sn) -- (S4);
\end{tikzpicture}
\caption{Layout of the article} \label{layout-fig}
\end{figure} 
\end{center} 
\subsection{Acknowledgements} 
The authors thank Dr.~Xiang Gao for valuable discussions and references in metrical number theory. This work was initiated in 2022, when JZ was visiting University of British Columbia on a study leave from China University of Mining and Technology-Beijing, funded by China Scholarship Council. He would like to thank all three organizations for their support that enabled his visit. JZ was also supported by National Natural Science Foundation of China  (Grant nos.~11801555, 11971058 and 12071431). MP was partially supported by a Discovery grant from Natural Sciences and Engineering Research Council of Canada (NSERC). 
\newpage 
\part{Skewed measures} \label{intro-part}
\section{A building block construction} \label{elem-op-section} 
As mentioned in the introduction, all the main results of this paper rely on a class of measures called skewed measures, so named because they tend to favour numbers that are not normal in certain bases. This section and the next are given over to the construction of such measures. A skewed measure $\mu$ is built through the repeated application of an elementary operation $\mathscr{O}$. We describe $\mathscr{O}$ in this section and use it to define  $\mu$ in the next.  
\subsection{The construction parameters} \label{param-elem-op-section}
The operation $\mathscr{O}$ requires the following parameters as input:
\begin{itemize}  
\item {\em{Scale parameters}} $\mathtt t \in \mathbb N$, $\mathtt s \in \mathbb N \setminus \{1, 2\}$ and exponents $\mathtt a, \mathtt b \in \mathbb N$ with 
\begin{equation} 
\mathtt a < \mathtt b, \qquad \mathtt s^{\mathtt a} > 2\mathtt t.
\end{equation}  
Define $\mathtt N \in \mathbb N$ to be the unique integer, depending only on $\mathtt t, \mathtt s, \mathtt a$ such that 
\begin{equation} \label{N}
\mathtt N + 1 \leq \frac{\mathtt s^{\mathtt a}}{\mathtt t} < \mathtt N+2, \quad \text{ i.e., } \quad \mathtt N + 1 = \lfloor \frac{\mathtt s^{\mathtt a}}{\mathtt t} \rfloor \geq 2,
\end{equation}  
where the floor function $\lfloor x \rfloor$ denotes the largest integer less than or equal to $x$.
\item A {\em{restricted collection of (favoured) digits}}  $\mathscr{D} = \mathscr{D}(\mathtt s)$ that will play a special role:
\begin{equation} \label{special-digits} 
0 \in \mathscr{D}(\mathtt s) \subsetneq \mathbb Z_{\mathtt s} := \{0, 1, \ldots, \mathtt s-1\}, \; \; (\mathtt s-1) \not\in \mathscr{D}(\mathtt s), \; \; 1 \leq \mathtt r := \# \bigl(\mathscr{D}(\mathtt s) \bigr) < \mathtt s. 
\end{equation}  
The assumptions $0 \in \mathscr{D}(\mathtt s)$, $(\mathtt s-1) \not\in \mathscr{D}(\mathtt s)$ are for concreteness only. Other choices of members or non-members of $\mathscr{D}(\mathtt s)$ would also work, but we will henceforth use \eqref{special-digits}.  
\item A {\em{``bias'' parameter}} $\varepsilon \in (0,1)$.  
\item A collection $\mathscr{I}_0$ of closed intervals in $[0, 1]$ of the form 
\begin{align} 
\mathtt I &= n \mathtt t^{-1} + \bigl[0, \mathtt t^{-1} \bigr], \quad n \in \{0, 1, \ldots, \mathtt t-1\}, \text{ so that } \label{I0-1} \\   
|\mathtt I| &= \mathtt t^{-1} \leq 1 \; \text{ and } \;    
\text{int}(\mathtt I) \cap \text{int}(\mathtt I') = \emptyset \; \text{ for all } \mathtt I, \mathtt I' \in \mathscr{I}_0, \; \mathtt I \ne \mathtt I'.   \label{I0-2}
\end{align} 
The number of intervals in $\mathscr{I}_0$ could be strictly less than $\mathtt t$. Throughout this paper and for a measurable set $\mathtt U \subseteq \mathbb R$, the symbol $|\mathtt U|$ will refer to its Lebesgue measure. 
\item A collection $\mathscr{W}_0 := \bigl\{\mathtt w(\mathtt I) : \mathtt I \in \mathscr{I}_0  \bigr\}$ of numbers associated with the intervals in $\mathscr{I}_0$: 
\begin{equation} \label{weights}
0 < \mathtt w(\mathtt I) \leq 1, \qquad \sum_{\mathtt I \in \mathscr{I}_0} \mathtt w(\mathtt I) = 1. 
\end{equation} 
We say that $\mathtt I \in \mathscr{I}_0$ has ``mass'' $\mathtt w(\mathtt I)$. 
\end{itemize} 
The starting point of the building block construction is the set  
\begin{equation} \label{E0}
\mathtt E_0 := \bigcup \bigl\{\mathtt I : \mathtt I \in \mathscr{I}_0 \bigr\} \subseteq [0, 1],
\end{equation} 
equipped with a mass distribution $\mathscr{W}_0$ among its constituent basic intervals. Depending on the choice of $\mathscr{I}_0$, the inclusion in \eqref{E0} can be strict. The probability measure on $\mathtt E_0$ is the density function $\Phi_0$ associated with $\mathscr{W}_0$, which assigns mass $\mathtt w(\mathtt I)$ uniformly to each interval $\mathtt I \in \mathscr{I}_0$: 
\begin{equation} \label{Phi0}
  \Phi_0 (x)  := \sum_{\mathtt I \in \mathscr{I}_0} \frac{\mathtt w(\mathtt I)}{|\mathtt I|} \mathtt 1_{\mathtt I}(x), \quad \text{ so that } \quad \Phi_0(\mathtt I) = \int_{\mathtt I} \Phi_0(x) \, dx = \mathtt w(\mathtt I).  
\end{equation} 
Here and below, $\mathtt 1_{\mathtt I}$ denotes the indicator function of the interval $\mathtt I$, which takes the value 1 on $\mathtt I$ and 0 outside it. 
\vskip0.1in 
\noindent We will execute a two-part elementary operation on the pair $(\mathtt E_0, \Phi_0)$, decomposing each basic interval $\mathtt I \in \mathscr{I}_0$ into successive pieces of length $\mathtt s^{-\mathtt a}$ and $\mathtt s^{-\mathtt b}$ respectively. The final outcome of this operation will be another set-function pair $(\mathtt E_1, \Phi_1)$, with an accompanying collection of basic intervals $\mathscr{I}_1$ and a mass distribution $\mathscr{W}_1$. Each interval in $\mathscr{I}_1$ will be of length $\mathtt s^{-\mathtt b}$, and will lie in some interval $\mathtt I \in \mathscr{I}_0$. The mass distribution $\mathscr{W}_1$ will spread the mass $\mathtt w(\mathtt I)$ among the subintervals of $\mathtt I$ that belong to $\mathscr{I}_1$, in a deliberately inequitable manner. The basic intervals in the output set $\mathtt E_1 \subseteq \mathtt E_0$ are the intervals in $\mathscr{I}_1$. The probability measure $\Phi_1$ supported on $\mathtt E_1$ will reflect the mass distribution $\mathscr{W}_1$.  We will denote this two-part operation by $\mathscr{O}$ and write
\begin{equation} \label{O-def} 
(\mathtt E_1, \Phi_1) := \mathscr{O}(\mathtt E_0, \Phi_0) = \mathscr{O}(\pmb{\gamma}), \quad \text{ where } \quad \pmb{\gamma} := \bigl(\mathtt E_0, \Phi_0; \mathtt t, \mathtt s, \mathtt a, \mathtt b, \mathscr{D}, \varepsilon \bigr)
\end{equation} 
encodes the input vector of the operation and the underlying set of parameters. Let us describe each step in more detail. 
\subsection{The elementary operation $\mathscr{O}$} \label{O-description-section}
\subsubsection{The first part}
The length $\mathtt t^{-1}$ of an interval $\mathtt I \in \mathscr{I}_0$ need not be an integer multiple of $\mathtt s^{-\mathtt a}$. Our first task is to isolate a single subinterval $\mathtt J(\mathtt I) \subseteq \mathtt I$ with $|\mathtt J(\mathtt I)| = \mathtt N \mathtt s^{- \mathtt a} < \mathtt t^{-1}$,  where $\mathtt N$ is the positive integer defined in \eqref{N}. The left endpoint of $\mathtt J(\mathtt I)$ lies in $\mathtt I \cap \{0, \mathtt s^{-\mathtt a}, 2 \mathtt s^{-\mathtt a}, \ldots \}$ and is the smallest such point with this property. Let $\alpha(\mathtt I)$ denote the left endpoint of $\mathtt I$. Then for every $\mathtt I \in \mathscr{I}_0$, there exist a unique integer $\mathtt m = \mathtt m(\mathtt I)$ and a real number  $\eta(\mathtt I) \in [0, \mathtt s^{-\mathtt a})$ such that 
\begin{equation} \label{def-m}
\frac{\mathtt m-1}{\mathtt s^{\mathtt a}} < \alpha(\mathtt I) \leq \frac{\mathtt m}{\mathtt s^{\mathtt a}}, \qquad \alpha(\mathtt I) + \eta(\mathtt I) := \frac{\mathtt m}{\mathtt s^{\mathtt a}}.
\end{equation} 
Let us observe that 
\begin{equation} 
\label{step-1-inclusion} 
\mathtt J(\mathtt I) := \alpha(\mathtt I) + \eta(\mathtt I) + [0, \mathtt N \mathtt s^{-\mathtt a}] = \bigl[ \mathtt m \mathtt s^{-\mathtt a}, (\mathtt m+\mathtt N) \mathtt s^{-\mathtt a} \bigr] \subsetneq \mathtt I = \alpha(\mathtt I) + [0, \mathtt t^{-1}]
\end{equation} 
Indeed, the definitions \eqref{N} and \eqref{def-m} of $\mathtt N$ and $\mathtt m$ respectively dictate that 
\[ \alpha(\mathtt I) \leq \mathtt m \mathtt s^{-\mathtt a} < \frac{\mathtt m+ \mathtt N}{\mathtt s^{\mathtt a}} < \alpha(\mathtt I) + \frac{\mathtt N + 1}{\mathtt s^{\mathtt a}} \leq \alpha(\mathtt I) + \mathtt t^{-1}, \] 
proving the inclusion \eqref{step-1-inclusion}. Even though the inclusion is strict, the interval $\mathtt J(\mathtt I)$ covers most of $\mathtt I$; specifically, the relation \eqref{N} implies 
\begin{equation}  \label{right-hand-spill} 
\alpha(\mathtt I) + \mathtt t^{-1} - \frac{\mathtt m + \mathtt N}{\mathtt s^{\mathtt a}} = \mathtt t^{-1} - \eta(\mathtt I) - \frac{\mathtt N}{\mathtt s^{\mathtt a}} < \mathtt t^{-1} + (2\mathtt s^{-\mathtt a} - \mathtt t^{-1}) = 2 \mathtt s^{-\mathtt a}. \end{equation} 
Combining \eqref{def-m} and \eqref{right-hand-spill}, we find that the portion of $\mathtt I$ not covered by $\mathtt J(\mathtt I)$ is a union of at most two disjoint intervals at the two edges of $\mathtt I$:
\begin{equation}  \label{spill}
\mathtt I \setminus \mathtt J(\mathtt I)  = \bigl[\alpha(\mathtt I),  {\mathtt m}{\mathtt s^{-\mathtt a}} \bigr) \bigcup \bigl((\mathtt m + \mathtt N){\mathtt s^{- \mathtt a}}, \alpha(\mathtt I) + \mathtt t^{-1} \bigr], \; \text{ so that } \; 
\bigl|\mathtt I \setminus \mathtt J(\mathtt I) \bigr| < 3 \mathtt s^{-\mathtt a}. 
\end{equation} 
\vskip0.1in 
\noindent We are now ready to describe the basic intervals of the first step of the elementary operation. These are the subintervals of $\mathtt J(\mathtt I)$ with length $\mathtt s^{-\mathtt a}$ (the intermediate scale) and left endpoints in $\{\mathtt m \mathtt s^{-\mathtt a}, (\mathtt m+1) \mathtt s^{-\mathtt a}, \ldots, (\mathtt m + \mathtt N-1) \mathtt s^{-\mathtt a} \}$. Set 
\begin{align} \mathtt J_k(\mathtt I) &:= \alpha(\mathtt I) + \eta(\mathtt I) + \bigl[ k \mathtt s^{-\mathtt a}, (k+1) \mathtt s^{-\mathtt a} \bigr] \subseteq \mathtt J(\mathtt I) \subseteq \mathtt I, \; \text{ where } \label{JkI} \\  k \in \mathbb K &:= \mathbb Z_{\mathtt N} = \{0, 1, \ldots, \mathtt N-1\}, \; \text{ so that } \; \mathtt J(\mathtt I) = \bigcup_{k \in \mathbb K} \mathtt J_k(\mathtt I).   \label{def-K}
\end{align}  
For each $\mathtt I \in \mathscr{I}_0$, its mass $\mathtt w(\mathtt I)$ is divided equally among the descendants $\{\mathtt J_k(\mathtt I) : k \in \mathbb K \}$ at this intermediate scale. Thus each interval $\mathtt J_k(\mathtt I)$ is endowed with a weight
\begin{equation} \label{mass-wk} \mathtt w(\mathtt J_k(\mathtt I)) := \frac{1}{\mathtt N} \mathtt w(\mathtt I), \qquad \text{ so that } \qquad \sum_{k=0}^{\mathtt N-1} \mathtt w(\mathtt J_k(\mathtt I)) = \mathtt w(\mathtt I). \end{equation}  
The ``spill-overs'', i.e. the edge intervals of $\mathtt I \setminus \mathtt J(\mathtt I)$ given in \eqref{spill}, get zero mass. The density function encoding this mass distribution is given by $\Psi_1 = \Psi_1[\pmb{\gamma}]$, with $\pmb{\gamma}$ as in \eqref{O-def}:
\begin{align}
 \Psi_1(x) &:= \sum_{\mathtt I \in \mathscr{I}_0} \mathtt w(\mathtt I)  \mathtt 1_{\mathtt J(\mathtt I)}(x) \label{intermediate-density-2} \\ &= \sum_{\mathtt I \in \mathscr{I}_0} \sum_{k=0}^{\mathtt N-1} \frac{\mathtt w(\mathtt J_k(\mathtt I))}{|\mathtt J_k(\mathtt I)|} \mathtt 1_{\mathtt J_k(\mathtt I)}(x), \; \text{ so that } \; \Psi_1(\mathtt J_k(\mathtt I)) = \mathtt w(\mathtt J_k(\mathtt I)).  \label{intermediate-density}
\end{align} 
This concludes the first step of the construction. 
\subsubsection{The second part} \label{second-step-section}  
Let us continue to the next stage. For each $\mathtt I \in \mathscr{I}_0$ and $k \in \mathbb K = \mathbb Z_{\mathtt N}$, we now decompose the interval $\mathtt J_k(\mathtt I)$ in \eqref{JkI} into $\mathtt s^{\mathtt b - \mathtt a}$ subintervals of equal length $\mathtt s^{-\mathtt b}$. Thus for $k \in \mathbb K$ and $\mathbf j = (k, \ell)$, 
\begin{align} \label{step-2-decomp}
\mathtt J_k(\mathtt I) &= \bigcup_{\ell \in \mathbb L} \mathtt I (\mathbf j)  \; \text{ where } \; \mathbb L := \mathbb Z_{\mathtt s^{\mathtt b-\mathtt a}} = \bigl\{0, 1, \ldots, \mathtt s^{\mathtt b - \mathtt a} - 1\bigr\},  \\ 
\mathtt I (\mathbf j)  &:= \alpha(\mathtt I) + \eta(\mathtt I) + k\mathtt s^{-\mathtt a} + \ell \mathtt s^{-\mathtt b} + [0, \mathtt s^{-\mathtt b}] \subseteq \mathtt J_k(\mathtt I) \subseteq \mathtt I, \label{step-2-basic}
\end{align} 
The intervals $\mathtt I(\mathbf j)$ are the basic intervals of the final set $\mathtt E_1$ resulting from this two-step elementary operation. We define
\begin{align} 
\mathbb J &:= \mathbb K\times \mathbb L, \qquad \mathscr{I}_1 := \bigl\{\mathtt I(\mathbf j) : \mathtt I \in \mathscr{I}_0, \; \mathbf j \in \mathbb J  \bigr\},  \label{JKL} \\ \mathtt E_1 &:= \bigcup \bigl\{\mathtt J : \mathtt J \in \mathscr{I}_1 \bigr\} = \bigcup \bigl\{ \mathtt I(\mathbf j) : \mathtt I \in \mathscr{I}_0, \; \mathbf j \in \mathbb J\bigr\}. \label{E1}
\end{align} 
Unlike the first step, the mass distribution among the intervals $\mathtt I(\mathbf j)$ is non-uniform and depends on the nature of the index $\ell$, with a bias towards indices $\ell$ whose digits lie in $\mathscr{D}$. Specifically, suppose that $\mathbf d(\ell, \mathtt s) := \bigl(\mathtt d_0(\ell, \mathtt s), \ldots, \mathtt d_{\mathtt b - \mathtt a-1}(\ell, \mathtt s) \bigr) $ denotes the $(\mathtt b - \mathtt a)$-long sequence of digits of the integer $\ell$ written in base $\mathtt s$: 
\begin{align} \ell &= \sum_{j=0}^{\mathtt b - \mathtt a-1} \mathtt d_j(\ell, \mathtt s) \mathtt s^j, \quad \mathtt d_j(\ell, \mathtt s) \in \mathbb Z_{\mathtt s}, \quad  \text{ and set } \label{rep-l-s} \\ \mathbb L^{\ast} &= \mathbb L^{\ast}(\mathtt s, \mathtt a, \mathtt b, \mathscr{D}) := \bigl\{\ell \in \mathbb L : \mathbf d(\ell, \mathtt s) \in \mathscr{D}^{\mathtt b- \mathtt a} \bigr\},  \label{L*}
\end{align} 
where $\mathscr{D} = \mathscr{D}(\mathtt s)$ is the special set of digits \eqref{special-digits} fixed at the outset in Section \ref{param-elem-op-section} with $\# \mathscr{D} (\mathtt s) = \mathtt r < \mathtt s$. Therefore $\#\bigl(\mathbb L^{\ast}(\mathtt s, \mathtt a, \mathtt b, \mathscr{D}) \bigr) = {\mathtt r}^{\mathtt b - \mathtt a}$. 
\vskip0.1in
\noindent For each $\mathtt I \in \mathscr{I}_0$, the mass on $\mathtt J_k(\mathtt I)$ given by \eqref{mass-wk} is distributed among its descendants $\{\mathtt I(\mathbf j) : \ell \in \mathbb L\}$; the mass $\mathtt w(\mathtt I(\mathbf j))$ allocated to $\mathtt I(\mathbf j)$ is defined via the formula 
\begin{equation} 
\mathtt w\bigl(\mathtt I(\mathbf j) \bigr) := \mathtt w(\mathtt J_k(\mathtt I)) \times \vartheta_{\ell}  =  \frac{\mathtt w(\mathtt I)}{\mathtt N} \times \vartheta_{\ell}, \label{second-step-weight} 
\end{equation} 
where $\bigl\{\vartheta_{\ell} : \ell \in \mathbb L \bigr\} \subseteq (0,1)$ is a collection of positive numbers summing up to to 1. While any such choice of $\vartheta_{\ell}$ serves as a mass distribution in principle, the following choice will be important both for ensuring non-normality as well as the Rajchman property:
\begin{equation} \label{choice-of-theta}  
\vartheta_{\ell} = \vartheta_{\ell}(\mathtt s, \mathtt a, \mathtt b, \mathscr{D}, \varepsilon) :=  (1 - \varepsilon) \mathtt s^{-(\mathtt b - \mathtt a)} + \begin{cases} \varepsilon  \mathtt r^{-(\mathtt b - \mathtt a)} &\text{ if } \; \ell \in \mathbb L^{\ast}, \\ 0 &\text{ if } \; \ell \in \mathbb L \setminus \mathbb L^{\ast}.  \end{cases} 
\end{equation} 
The density function on $\mathtt E_1$ associated with this mass distribution is given by 
\begin{equation} \label{Phi1} 
\Phi_1(x) := \sum_{\mathtt J \in \mathscr{I}_1} \frac{\mathtt w(\mathtt J)}{|\mathtt J|} 1_{\mathtt J}(x) = \sum_{\mathtt I \in \mathscr{I}_0} \sum_{\mathbf j \in \mathbb J} \mathtt w(\mathtt I(\mathbf j)) |\mathtt I(\mathbf j)|^{-1} 1_{\mathtt I(\mathbf j)}(x). 
\end{equation} 
The pair $(\mathtt E_1, \Phi_1)$ given by \eqref{E1} and \eqref{Phi1} completes the description of $\mathscr{O}$ in \eqref{O-def}. 
\subsection{The random variables in $\mathscr{O}$} \label{random-variables-section} 
 In \cite{PZ-1}, we gave an alternative description of the elementary operation $\mathscr{L}$ in \cite{l86} in terms of random variables on a probability space $(\mathbb R, \mathbb P)$. We do the same for $\mathscr{O}$ here. 
\vskip0.1in 
\noindent Starting with a collection of intervals $\mathscr{I}_0$ satisfying \eqref{I0-1} and \eqref{I0-2}, let $\mathfrak X_0$ be a discrete random variable on $(\mathbb R, \mathbb P)$ with probability distribution  
\[ \mathbb P \left(\mathfrak X_0 = \alpha(\mathtt I) \right) = \mathtt w(\mathtt I) \quad \text{ for } \mathtt I \in \mathscr{I}_0. \]   
Then $\Phi_0$ given by \eqref{Phi0} is the probability density function of the continuous random variable $X_0 = \mathfrak X_0 + \mathfrak U_0$, where the random variable $\mathfrak U_0$ is uniform on $[0, \mathtt t^{-1}]$ and independent of $\mathfrak X_0$. The intermediate and final densities $\Psi_1$ and $\Phi_1$, given by \eqref{intermediate-density} and \eqref{Phi1} respectively, admit similar interpretations. For example, $\Phi_1$
is the probability density of the random variable
\[ X_1 = \mathfrak X_1 + \mathfrak U_1 \text{ where } \; \mathfrak X_1 = \mathfrak X_0' + \mathfrak Y_1 \mathtt s^{-\mathtt a} + \mathfrak Z_1 \mathtt s^{-\mathtt b}.  \] 
Here  $\mathfrak U_1$ is a uniform random variable on $[0, \mathtt s^{-\mathtt b}]$, while $\mathfrak X_0'$, $\mathfrak Y_1$ and $ \mathfrak Z_1$ are discrete random variables with the following mass distributions:
\begin{align*} 
&\mathbb P\bigl(\mathfrak X_0' = \alpha(\mathtt I) + \eta(\mathtt I) \bigr) = \mathtt w(\mathtt I) \text{ for } \mathtt I \in \mathscr{I}_0,  \\ 
&\mathbb P \left(\mathfrak Y_1 = k \right) = {\mathtt N}^{-1} \text{ for } k \in \mathbb K, \quad \; \mathbb P \left(\mathfrak Z_1 = \ell \right) = \vartheta_{\ell} \text{ for } \ell \in \mathbb L. 
\end{align*} 
The set of random variables $\{\mathfrak X_0', \mathfrak Y_1, \mathfrak Z_1, \mathfrak U_1\}$ is independent. In addition, $\{\mathfrak Y_1, \mathfrak Z_1, \mathfrak U_1\}$ is independent of $\{\mathfrak X_0, \mathfrak U_0\}$. The variables $\mathfrak X_0$ and $\mathfrak X_0'$ are {\em{not}} independent; they depend on each other via a deterministic relation.

\begin{figure}
\begin{center}
\begin{tikzpicture}[scale=0.7]

\def\ycoord{0} 
\def\eps{0.05}
\def\tick{0.15}
\def\endpt{15}
\def\midpt{\endpt * 0.5}

\draw[fill=lightgray] (0,\ycoord) rectangle ++(\endpt,4*\tick);

\draw (0,\ycoord) -- (\endpt,\ycoord);
\draw (0,\ycoord-\tick) node [below] {$\alpha(\mathtt I)$};
\draw (\endpt,\ycoord-\tick) node [below] {$\alpha(\mathtt I) + \mathtt t^{-1}$};
\draw (\midpt,\ycoord+5*\tick) node [above] {$\mathtt I$};
\draw (0,\ycoord-\tick) -- (0, \ycoord+\tick);
\draw (\endpt,\ycoord-\tick) -- (\endpt,\ycoord+\tick);

\def\ycoord{-3} 

\draw[dotted] (0,\ycoord) -- (\endpt,\ycoord);

\def\subint{\endpt / 12}
\foreach \x in {0, 0.5, 1.5, 2.5, 3.5, 4.5, 5.5, 6.5, 7.5, 8.5, 9.5, 10.5,
11.5,  12}
\draw (\x*\subint,\ycoord-\tick) -- (\x*\subint,\ycoord+\tick);

\foreach \x in {0.5, 1.5, 2.5, 3.5, 4.5, 10.5}
\draw[fill=lightgray] (\x*\subint,\ycoord-\eps) rectangle ++(\subint,5*\tick);
\draw (0.5*\subint,\ycoord-\tick) node [below] {$\frac{\mathtt m}{\mathtt s^{\mathtt a}}$};
\draw (\subint,\ycoord+4*\tick) node [above] {$\mathtt J_0(\mathtt I)$};
\draw (2*\subint,\ycoord-\tick) node [below] {$\mathtt J_1(\mathtt I)$};
\draw (5*\subint,\ycoord+4*\tick) node [above] {$\mathtt J_k(\mathtt I)$};
\draw (11*\subint ,\ycoord+4*\tick) node [above] {$\mathtt J_{\mathtt N-1}(\mathtt I)$};
\draw (11.5*\subint,\ycoord-\tick) node [below] {$\frac{\mathtt m+\mathtt N-1}{\mathtt s^{\mathtt a}}$};

\draw[dotted] (5.5*\subint,\ycoord - \eps + 5*\tick) -- (10.5*\subint,\ycoord - \eps + 5*\tick);

\draw[->](5*\subint, \ycoord - 5*\eps) -- (5.5*\subint, \ycoord - 8*\tick);

\def\ycoord{-6}
\draw (0,\ycoord) -- (\endpt,\ycoord);
\draw (0,\ycoord-\tick) -- (0, \ycoord+\tick);
\draw (\endpt,\ycoord-\tick) -- (\endpt,\ycoord+\tick);
\draw (\midpt,\ycoord+6*\tick) node [above] {$\mathtt J_k(\mathtt I)$};

\draw[fill=lightgray] (0,\ycoord) rectangle ++(12*\subint,5*\tick);
\draw (0,\ycoord-\tick) node [below] {$\frac{\mathtt m+k}{\mathtt s^{\mathtt a}}$};
\draw (\endpt,\ycoord-\tick) node [below] {$\frac{\mathtt m+(k+1)}{\mathtt s^{\mathtt a}}$};

\def\ycoord{-9}
\draw (0,\ycoord) -- (3*\subint,\ycoord);
\draw[dotted](3*\subint, \ycoord) -- (4*\subint, \ycoord); 
\draw[dotted](5*\subint, \ycoord) -- (6*\subint, \ycoord); 
\draw[dotted](6*\subint, \ycoord) -- (7*\subint, \ycoord);
\draw[dotted](7*\subint, \ycoord) -- (9*\subint, \ycoord);
\draw[dotted](10*\subint, \ycoord) -- (11*\subint, \ycoord);
\draw[dotted](3*\subint + 2*\tick, \ycoord + 3*\tick) -- (3*\subint + 6*\tick, \ycoord + 3*\tick);
\draw[dotted](10*\subint + 2*\tick, \ycoord + 3*\tick) -- (10*\subint + 6*\tick, \ycoord + 3*\tick);
\draw[dotted](5*\subint + 2*\tick, \ycoord + 3*\tick) -- (6*\subint + 6*\tick, \ycoord + 3*\tick);
\foreach \x in {0,1, 2, 3, 4, 5, 6, 7, 8, 9, 10, 11, 12}
\draw (\x*\subint,\ycoord-\tick) -- (\x*\subint, \ycoord+\tick);
\draw (0.5*\subint,\ycoord-\tick) node [below] {{\scriptsize{$\mathtt I(k,0)$}}};
\draw (1.5*\subint,\ycoord+8*\tick) node [above] {{\scriptsize{$\mathtt I(k,1)$}}};
\draw (0.5*\subint,\ycoord-\tick) node [below] {{\scriptsize{$\mathtt I(k,0)$}}};
\draw (4.5*\subint,\ycoord+5*\tick) node [above] {{\scriptsize{$\mathtt I(k,\ell)$}}};
\draw (7.5*\subint,\ycoord-\tick) node [below] {{\scriptsize{$\mathtt I(k,\mathtt s)$}}};
\draw (8.75*\subint,\ycoord+8*\tick) node [above] {{\scriptsize{$\mathtt I(k,\mathtt s+1)$}}};
\draw (11.75*\subint,\ycoord-\tick) node [below] {{\scriptsize{$\mathtt I(k, \mathtt s^2-1)$}}};

\draw[fill=lightgray] (0*\subint,\ycoord) rectangle ++(\subint,8*\tick);
\draw[fill=lightgray] (1*\subint,\ycoord) rectangle ++(\subint,8*\tick);
\draw[fill=lightgray] (2*\subint,\ycoord) rectangle ++(\subint,5*\tick);
\draw[fill=lightgray] (4*\subint,\ycoord) rectangle ++(\subint,5*\tick);
\draw[fill=lightgray] (7*\subint,\ycoord) rectangle ++(\subint,8*\tick);
\draw[fill=lightgray] (8*\subint,\ycoord) rectangle ++(\subint,8*\tick);
\draw[fill=lightgray] (9*\subint,\ycoord) rectangle ++(\subint,5*\tick);
\draw[fill=lightgray] (11*\subint,\ycoord) rectangle ++(\subint,5*\tick);

\draw [decorate,decoration={brace,amplitude=5pt,mirror,raise=2ex}]
(5,\ycoord) -- (6.2,\ycoord) node [midway,yshift=-2em] {{\scriptsize{$\mathtt s^{-\mathtt b}$}}};

\draw [decorate,decoration={brace,amplitude=5pt,mirror,raise=2ex}] (-0.5,0.5) -- (-0.5,-3);

\draw (-1.3,-1.3) node[left] {Step 1};

\draw [decorate,decoration={brace,amplitude=5pt,mirror,raise=2ex}] (-0.5,-5.2) -- (-0.5,-9.2);

\draw (-1.3,-7.3) node[left] {Step 2};

\end{tikzpicture}
\caption{The elementary operation $\mathscr{O}$ with $\mathtt b = \mathtt a +2$, $\mathscr{D}(\mathtt s) = \{0, 1\}$} \label{Interval}
\end{center} 
\end{figure}
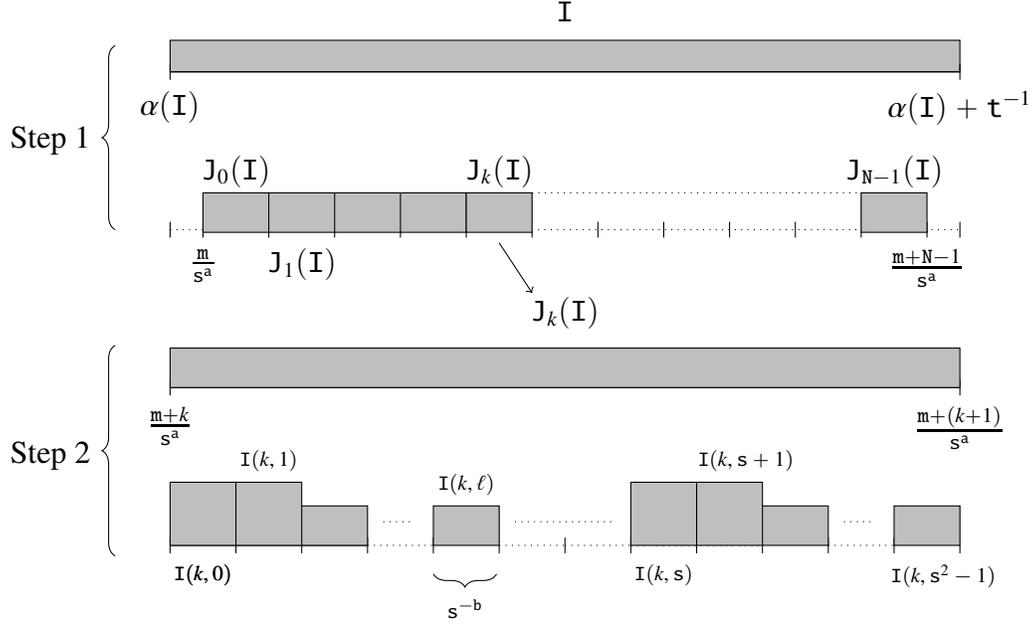
\subsection{Elementary operations for (non)-normality: a comparative study} \label{comparison-section} 
It is instructive to compare the elementary operation $\mathscr{O}$ in Section \ref{O-description-section} with the corresponding units of construction in \cite{l86} and \cite{p81}, which we denote by $\mathscr{L}$ and $\mathscr{P}$ respectively. Despite certain features of similarity among the three, they are all distinct. 
\subsubsection{$\mathscr{O}$ and $\mathscr{L}$} In $\mathscr{L}$, one starts with two integer scale parameters $\mathtt K_0 < \mathtt K_1$ and a bias parameter $\varepsilon \in (0,1)$. The initial set $\mathtt E_0$ consists of basic intervals $\mathtt I \in \mathscr{I}_0$ of length $\mathtt t^{-1} = 2^{-\mathtt K_0}$. The operation $\mathscr{L}$ divides each $\mathtt I$ into subintervals $\{\mathtt I(\ell) : \ell = 0, 1, \ldots, 2^{\mathtt K_1-\mathtt K_0}-1\}$ of length $\mathtt s^{-\mathtt b} = 2^{- \mathtt K_1}$. Each integer $\ell$ is uniquely identified with a $(\mathtt K_1-\mathtt K_0)$-long binary sequence $\mathbf d(\ell, 2)$ representing the digits of $\ell$ in base 2. Of these descendants of $\mathtt I$ only the leftmost subinterval $\mathtt I( 0)$, which corresponds to $\ell = 0$ and  $\mathbf d(\ell, 2) = \mathbf 0$, receives an extra mass of $\varepsilon$ compared to the others. 
\vskip0.1in 
\noindent The points of distinction between $\mathscr{O}$ and $\mathscr{L}$ are the following.
\begin{itemize} 
\item The successive scales in $\mathscr{L}$ are powers of 2, so the decomposition of each interval $\mathtt I$ into the descendants $\mathtt I(\ell)$ is exact. Since there are no spill-overs, the initial set $\mathtt E_0$ and the final set $\mathtt E_1$ are the same, though their mass distributions are different. In contrast, the elementary operation $\mathscr{O}$ does not require the scale parameter $\mathtt t$ to divide either $\mathtt s^{\mathtt a}$ or $\mathtt s^{\mathtt b}$. As a result, $\mathtt E_0 \setminus \mathtt E_1$ is in general non-empty, as we have seen in \eqref{spill}. 

\item The construction in $\mathscr{L}$ favours only a single digit sequence $\mathbf 0$, assigning it the entirety of the extra mass $\varepsilon$. The operation $\mathscr{O}$ spreads the bias $\varepsilon$ over a certain subset of digit sequences. The counterpart of $\mathbb L^{\ast}$ in $\mathscr{L}$ is a singleton. Having the bias concentrated only on $\mathtt I(0)$ allows the resulting measure to be pointwise non-normal in more bases than $\{2^n : n \in \mathbb N\}$ (as we have shown in \cite{PZ-1}), preventing the normality-seeking behaviour.  The size of the favoured digit set $\mathscr{D}$ affects both the Frostman behaviour as well as the features of normality of the resulting measure, as we will see in Propositions \ref{Frostman-prop} and \ref{normal-prop}. 

\item The first part of $\mathscr{O}$ has no analogue in $\mathscr{L}$, which is a single step operation. Unlike $\mathscr{L}$, the first step of $\mathscr{O}$ distributes the mass $\mathtt w(\mathtt I)$ evenly among the descendants of an intermediate scale. As a consequence, one obtains better decay estimates than merely Rajchman for ``most'' Fourier coefficients. It is still possible to achieve the Rajchman property without this step, as shown in \cite{l86}. However, establishing normality for a specified collection of bases requires Fourier decay estimates that are finer than merely Rajchman. Distributing the mass uniformly over the first generation intervals achieves the desired effect. 
\end{itemize} 
\subsubsection{$\mathscr{O}$ and $\mathscr{P}$} \label{O-vs-P-section} 
The elementary operation $\mathscr{P}$ in \cite{p81} has three scale parameters $\mathtt t < \mathtt s^{\mathtt a} < \mathtt s^{\mathtt b}$ and two parts, similar to $\mathscr{O}$. At the first step of $\mathscr{P}$, a basic interval $\mathtt I \in \mathscr{I}_0$ of length $\mathtt t^{-1}$ is decomposed into $\mathtt N$ sub-intervals $\mathtt J_k$ of equal length $\mathtt s^{-\mathtt a}$, modulo small spill-overs. Depending on a a pre-fixed choice of $(\mathtt b - \mathtt a)$-long favoured digit sequences, certain sub-intervals $\mathtt I(\mathbf j) \subseteq \mathtt J_k(\mathtt I)$ of length $\mathtt s^{-\mathtt b}$ are retained, others are removed. This part of the construction ensures non-normality in base $\mathtt s$, and is identical to its counterpart in $\mathscr{O}$.  The second step of $\mathscr{P}$ involves another round of elimination of certain intervals $\mathtt I(\mathbf j)$ based on a different criterion \cite[Lemma 3]{p81}. This step is non-constructive and depends on the existence of certain subintervals $\mathtt I(\mathbf j)$, in general a strict subset of those obtained in the previous step, whose left endpoints give rise to small exponential sums in bases that are multiplicatively independent of $\mathtt s$. This establishes normality.     
\vskip0.1in 
\noindent The points of distinction between $\mathscr{O}$ and $\mathscr{P}$ are the following:
\begin{itemize} 
\item The elementary operation $\mathscr{P}$ is the building block for a Cantor-type construction of a set, not a measure. It is natural to ask whether the resulting set, which is of specified (non)-normality, supports a Rajchman measure. The answer turns out to be no, in general, as we prove in Section \ref{Pollington-not-Rajchman-section}. In contrast, $\mathscr{O}$ is crafted to achieve the Rajchman property.  
\vskip0.1in 
\item The second step of $\mathscr{P}$ has no analogue in $\mathscr{O}$. Whereas $\mathscr{P}$ uses an exponential sum estimate as a criterion for selection of basic intervals, the construction $\mathscr{O}$ does not require any such constraint. While this makes $\mathscr{O}$ simpler to describe, the normality that it imparts is not as simple to verify as for a set built using $\mathscr{P}$. Indeed the bulk of the technical complexity in Part \ref{part-normality} of this article lies in {\em{establishing}} the smallness of certain exponential sums, a property that was easily available as a defining feature for $\mathscr{P}$.  
\end{itemize} 

\section{Construction of a skewed measure} \label{iteration-section}
In Section \ref{elem-op-section}, we described an elementary operation $\mathscr{O}$ that transforms one set-density pair $(\mathtt E_0, \Phi_0)$ into another $(\mathtt E_1, \Phi_1)$.  In this section, we iterate this elementary operation on many scales to craft a singular measure $\mu$ on $[0,1]$. Its properties, as claimed in Theorems \ref{mainthm-1} and \ref{mainthm-2}, will be verified in later sections. 
\subsection{Choosing the bases for non-normality} \label{iteration-parameters-section} 
Let us fix a choice of parameters:
\begin{align}  \label{param-seq}
&\pmb{\Pi} := (\mathcal S, \mathcal A, \mathcal B, \mathcal D, \mathcal E), \; \text{ where each of the collections } \; \mathcal S := \bigl\{\mathtt s_m \bigr\} \subseteq \mathbb N \setminus \{1, 2\}, \\ & \mathcal A := \bigl\{\mathtt a_m \bigl\} \subseteq \mathbb N, \quad \mathcal B := \bigl\{\mathtt b_m \bigr\} \subseteq \mathbb N, \quad \mathcal D := \bigl\{\mathscr{D}_m \bigr\}, \quad  \mathcal E := \bigl\{\varepsilon_m\bigr\} \subseteq (0, 1), \nonumber
\end{align} 
is a sequence (either of numbers or sets) indexed by $m \in \mathbb N$. The entries of $\mathcal S, \mathcal A, \mathcal B$ are positive integers with the property that 
\begin{equation} \label{abs}
\mathtt a_m < \mathtt b_m \quad \text{ and } \quad 2 \mathtt s_{m}^{\mathtt b_m} < \mathtt s_{m+1}^{\mathtt a_{m+1}} \quad \text{ for all } \quad m \geq 1. 
\end{equation} 
The sequence $\mathcal S$, whose members need not be distinct, identifies the bases in which points in supp$(\mu)$ are most likely to be non-normal. The defining feature of non-normality is a bias towards certain digits or sequences of digits. For $m \geq 1$, the non-empty set $\mathscr{D}_m$, which satisfies
\begin{equation} \label{rmsm-assumption}
 0 \in \mathscr{D}_m \subsetneq \mathbb Z_{\mathtt s_m} = \bigl\{0, 1, \ldots, \mathtt s_m-1 \bigr\}, \quad 1 \leq \mathtt r_m := \# \bigl(\mathscr{D}_m \bigr) < \mathtt s_{m}, 
\end{equation}
represents the choice of preferred digits in base $\mathtt s_m$. The indices $\mathtt a_m \in \mathcal A$ and $\mathtt b_m \in \mathcal B$ specify the endpoints of a block of a digit sequence in base $\mathtt s_m$ where the elements of $\mathscr{D}_m$ receive preferential treatment. The collective bias towards the digit sequences in $\mathscr{D}_m^{\mathtt b_m-\mathtt a_m}$ is quantified by $\varepsilon_m$ in $\mathcal E$.  
\subsection{The iteration} \label{iteration-subsection}
Starting with $E_0 := [0,1]$, $\mathtt s_0 := \mathtt b_0 := 1$, $\varphi_0(x) := 1_{[0,1]}(x)$, let us define 
\begin{equation} \label{iteration} 
(E_{m}, \varphi_m) := \mathscr{O}(\pmb{\gamma}_m) \; \text{ where } \; \pmb{\gamma}_m := \bigl(E_{m-1}, \varphi_{m-1}; \mathtt s_{m-1}^{\mathtt b_{m-1}}, \mathtt s_{m}, \mathtt a_m, \mathtt b_m, \mathscr{D}_m, \varepsilon_m \bigr), \quad m \geq 1, 
\end{equation} 
where $\mathscr{O}$ is the elementary operation \eqref{O-def} described in Section \ref{elem-op-section}. It follows from \eqref{E1} in that section that $E_m$ is a union of basic intervals of length $\mathtt s_m^{-\mathtt b_m}$ with disjoint interiors:  
\begin{equation} \label{def-Em}
E_m = \bigcup \bigl\{\mathtt I(\mathbf i_m) : \mathbf i_m \in \mathbb I_m \bigr\} \quad \text{ where } \quad \mathtt I(\mathbf i_m) := \alpha(\mathbf i_m) + \bigl[0, \mathtt s_m^{- \mathtt b_m} \bigr].
\end{equation} 
The multi-index $\mathbf i_m$ runs over a set $\mathbb I_m$ described in Section \ref{indexing-section-1} below. 
The operation $\mathscr O$ dictates that the sets $\{E_m\}$ are closed and nested with $E_{m+1} \subseteq E_m \subseteq [0,1]$. The limiting non-empty compact set $E := \bigcap_m E_m$ is the support of the singular measure $\mu$, also defined in that section. 

\subsection{Indexing the basic intervals of $E_m$} \label{indexing-section-1} 
The index set $\mathbb I_m$ in \eqref{def-Em} representing the basic intervals of $E_m$ is constructed recursively. Set
\begin{equation} 
 \mathbb J_m := \mathbb K_m \times \mathbb L_m, \; \text{ with } \; \mathbb K_m  := \mathbb Z_{\mathtt N_m}, \; \; \mathtt N_m + 1 := \lfloor \mathtt s_m^{\mathtt a_m}/\mathtt s_{m-1}^{\mathtt b_{m-1}}\rfloor, \; \; \mathbb L_m := \mathbb Z_{\mathtt s_m^{\mathtt b_m - \mathtt a_m}}, \; \; m\geq 1 \label{def-J} 
 \end{equation} 
 so that $\mathbb J_m, \mathbb K_m, \mathbb L_m$ are the counterparts of $\mathbb J, \mathbb K, \mathbb L$ in \eqref{JKL}, \eqref{def-K}, \eqref{step-2-decomp} respectively. Then 
 \begin{equation} 
 \mathbb I_1 := \mathbb J_1, \quad \mathbb I_m := \mathbb I_{m-1} \times \mathbb J_m = \bigl\{(\mathbf i_{m-1}, \mathbf j_m): \mathbf i_{m-1} \in \mathbb I_{m-1}, \; \mathbf j_m \in \mathbb J_m \bigr\} \subseteq \mathbb Z_{\geq 0}^{2m}. \label{def-I}
 \end{equation} 
The relation \eqref{step-2-basic} connects $\alpha(\mathbf i_m)$ and $\alpha(\mathbf i_{m-1})$ given by \eqref{def-Em}.
For $\mathbf i_m = (\mathbf i_{m-1}, \mathbf j_m)$,  
 \begin{equation} \label{relation-between-endpoints}
\alpha(\mathbf i_m) = \alpha(\mathbf i_{m-1}) + \eta(\mathbf i_{m-1}) + k_m \mathtt s_m^{-\mathtt a_m} + \ell_m \mathtt s_m^{-\mathtt b_m},  \text{ where } \; \mathbf j_m = (k_m, \ell_m) \in \mathbb J_m.        
\end{equation} 
In \eqref{def-J} and \eqref{relation-between-endpoints} above, the quantities $\mathtt N_m$ and $\eta(\mathbf i_{m-1})$ are the corresponding counterparts of $\mathtt N$ and $\eta(\mathtt I)$, defined by the relations \eqref{N} and \eqref{def-m} respectively, with 
\begin{equation} 
\mathtt I = \mathtt I (\mathbf i_{m-1}), \; \mathtt s = \mathtt s_m, \; \mathtt a = \mathtt a_m \; \text{ and } \; \mathtt t = \mathtt s_{m-1}^{\mathtt b_{m-1}}, \; \text{ so that } \; \eta(\mathbf i_{m-1}) \in \bigl[0, \mathtt s_m^{-\mathtt a_m} \bigr), \; m \geq 1.  \label{etaim}
\end{equation} 
The initial left endpoint is set to $\alpha(\mathbf i_0) = 0$ by convention. 
\subsection{The densities $\varphi_m$ and associated random variables} \label{indexing-section-2}
For $m \geq 1$, the iteration process \eqref{iteration} identifies a probability density $\varphi_m$ supported on the set $E_m$ in \eqref{def-Em}. The mass assigned by $\varphi_m$ to the basic interval $\mathtt I(\mathbf i_m)$ is denoted by $\mathtt w(\mathbf i_m)$, an abbreviation of $\mathtt w \bigl(\mathtt I(\mathbf i_m)\bigr)$ from \eqref{second-step-weight}. In this notation, \eqref{weights}, \eqref{mass-wk} and \eqref{second-step-weight} take the form 
\begin{equation} \label{wt-sum}
\sum_{\mathbf i_m \in \mathbb I_m} \mathtt w(\mathbf i_m) = 1, \quad \sum_{\mathbf j_m \in \mathbb J_m} \mathtt w(\mathbf i_m) = \mathtt w(\mathbf i_{m-1}), \quad \mathtt w(\mathbf i_m) = \frac{\mathtt w(\mathbf i_{m-1})}{\mathtt N_m} \times \vartheta_{\ell_m}
(\mathtt s_m, \mathtt a_m, \mathtt b_m, \mathscr{D}_m, \varepsilon_m). 
\end{equation} 
with $\vartheta_{\ell}$ as in \eqref{choice-of-theta}. The collection $\mathbb L^{\ast}$ in \eqref{L*} used to define $\vartheta_{\ell}$ in \eqref{choice-of-theta} is replaced by
\begin{equation} \label{Lm-star} 
\mathbb L_m^{\ast} := \mathbb L^{\ast} (\mathtt s_m, \mathtt a_m, \mathtt b_m, \mathscr{D}_m) = \bigl\{\ell \in \mathbb L_m : \mathbf d(\ell, \mathtt s_m) \in \mathscr{D}_m^{\mathtt b_m- \mathtt a_m}\bigr\}. 
\end{equation} 
The formula for $\varphi_m$ follows from \eqref{Phi1}: 
\begin{equation}  \label{density} 
\varphi_m(x) := \mathtt s_m^{\mathtt b_m} \sum_{\mathbf i_m \in \mathbb I_m} \mathtt w(\mathbf i_m) 1_{\mathtt I(\mathbf i_m)}(x). 
\end{equation}   
\noindent The discussion in Section \ref{random-variables-section} permits $\varphi_m$ to be realized as the probability density of an appropriately defined random variable. Such a description is often useful in determining dependencies between events, as seen in \cite{{l86}, {PZ-1}}. Given a probability space $(\mathbb R, \mathbb P)$ with 
\begin{align}  
&{\text{an infinite sequence of independent random variables $\{\mathfrak Y_m, \mathfrak Z_m, \mathfrak U_m : m \geq 1\}$}}, \label{independence-rvs} \\
&\mathbb P(\mathfrak Y_m = k) = \mathtt N_m^{-1} \text{ for } k \in \mathbb K_m, \; \; \mathbb P(\mathfrak Z_m = \ell) = \vartheta_{\ell} \text{ for } \ell \in \mathbb L_m, \;\;  \mathfrak U_m \text{ uniform on $[0, \mathtt s_m^{-\mathtt b_m}]$},  \nonumber
\end{align} 
the function $\varphi_m$ in \eqref{density} is the density of the random variable $X_m$, where
\begin{equation}  X_m = \mathfrak X_m + \mathfrak U_m, \qquad \mathfrak X_m = \mathfrak X_{m-1}' + \mathfrak Y_m \mathtt s_m^{-\mathtt a_m}+ \mathfrak Z_m \mathtt s_m^{-\mathtt b_m}.  \label{XmYmZmUm} 
\end{equation}  
The variable $\mathfrak X_m$, which depends only on $\{\mathfrak X_{m-1}', \mathfrak Y_m, \mathfrak Z_m\}$, represents the left endpoints of the basic intervals $\mathtt I(\mathbf i_m)$ at the end of the $m^{\text{th}}$ step of the iteration. 
\[ \mathbb P(\mathfrak X_m = \alpha(\mathbf i_m)) = \mathtt w(\mathbf i_m), \qquad \mathbf i_m \in \mathbb I_m. \]  
The variable $\mathfrak X_{m}'$ is completely determined by $\mathfrak X_m$; it takes the value $\alpha(\mathbf i_{m}) + \eta(\mathbf i_{m})$ whenever $\mathfrak X_{m} = \alpha(\mathbf i_{m})$. This description of $\varphi_m$ will play an important role in Section \ref{non-normality-section}.  
\subsection{Existence of a limiting measure} \label{measure-definition-section} 
The sequence \eqref{density} of probability densities $\{\varphi_{m}\}$ will lead to a measure $\mu$. We define it now. 
\begin{proposition} \label{def-mu-prop}  
Let $\pmb{\Pi} = (\mathcal S, \mathcal A, \mathcal B, \mathcal D, \mathcal E)$ be a choice of parameters given by \eqref{param-seq}, obeying the properties  \eqref{abs} and \eqref{rmsm-assumption}. Let  $\{(E_m, \varphi_m): m \geq 1 \}$ be the sequence of set-density pairs given by \eqref{iteration}, \eqref{def-Em} and \eqref{density}. Then there exists a Borel probability measure $\mu = \mu(\pmb{\Pi})$ supported on $E \subseteq [0,1]$ given by 
\begin{equation}
\mu = \mu(\pmb{\Pi}) := \text{weak}^{\ast} \lim_{m \rightarrow \infty} \varphi_m \quad \text{ with } \quad  E := \bigcap_{m=1}^{\infty} E_m.  \label{mu-weak-limit}
\end{equation} 
\end{proposition}
\vskip0.1in     
\begin{proof} 
The conclusions of the proposition will follow from the assertions below. For any $\mathtt f \in \mathcal C[0,1]=$ the space of 1-periodic complex-valued continuous functions on $[0,1]$,
\begin{align} 
\langle \mu, \mathtt f \rangle :=  &\lim_{m \rightarrow \infty} \int_{0}^{1} \mathtt f(x) \varphi_m(x) \, dx \text{ exists, and is non-negative if } \mathtt f \geq 0, \label{limit} \\ 
&\bigl| \langle \mu, \mathtt f \rangle \bigr| \leq || \mathtt f || _{\infty}, \quad  \langle \mu, \mathtt 1_{[0,1]} \rangle = 1. \label{limit-2} \\
&  \langle \mu, \mathtt f \rangle = 0 \; \text{ if } \; \text{supp}(\mathtt f) \cap E = \emptyset. \label{limit-3}
\end{align} 
Indeed, the existence of the limit in \eqref{limit} combined with the first inequality in \eqref{limit-2} establishes $\mu$ as a non-negative bounded linear functional on $\mathcal C[0,1]$. The Riesz representation theorem identifes $\mu$ as a unique non-negative Borel measure on $[0,1]$; it is a probability measure by the second relation in \eqref{limit-2}. The final claim \eqref{limit-3} shows that $\mu$ is supported on $E$. 
\vskip0.1in 
\noindent We set about proving the assertions \eqref{limit}-\eqref{limit-3}.  Recalling from \eqref{density} the expression for $\varphi_m$, we obtain for any $\mathtt f \in \mathcal C[0,1]$ and every $m \geq 1$,   
\begin{equation} \label{cauchy-0}
\int_{0}^{1} \mathtt f(x) \varphi_m(x) \, dx = \mathtt s_m^{\mathtt b_m} \sum_{\mathbf i \in \mathbb I_m} \mathtt w(\mathbf i) \int_{\mathtt I(\mathbf i)} \mathtt f(x) \, dx = \sum_{\mathbf i \in \mathbb I_m} \mathtt w(\mathbf i) \mathtt f\bigl(x(\mathbf i)\bigr).
\end{equation} 
At the last step of \eqref{cauchy-0}, we have used the mean value theorem for integrals, so that $x(\mathbf i) \in \mathtt I(\mathbf i)$. 
For $m' > m$, let us note from \eqref{def-I} that 
\[ \mathbb I_{m'} = \mathbb I_m \times \overline{\mathbb J} \; \text{ where } \; \overline{\mathbb J} = \mathbb J_{m+1} \times \cdots \times \mathbb J_{m'}, \; \text{ so that } \sum_{\mathbf j \in \overline{\mathbb J}} \mathtt w(\mathbf i, \mathbf j) = \mathtt w(\mathbf i) \text{ for all } \mathbf i \in \mathbb I_m. \] 
The last identity follows from the second relation in \eqref{wt-sum}, applied $(m'-m)$ times. Inserting these representations of $\mathbb I_{m'}$ and $\mathtt w(\mathbf i)$ into \eqref{cauchy-0} leads to  
\begin{align*} 
\left| \int_{0}^{1} \mathtt f(x) \varphi_{m'}(x) \, dx - \int_{0}^{1} \mathtt f(x) \varphi_m(x) \, dx  \right| &= \Bigl| \sum_{\mathbf i \in \mathbb I_m} \sum_{\mathbf j \in \overline{\mathbb J}}\mathtt w(\mathbf i, \mathbf j) \mathtt f\bigl(x(\mathbf i, \mathbf j)\bigr) - \sum_{\mathbf i \in \mathbb I} \mathtt w(\mathbf i) \mathtt f\bigl(x(\mathbf i)\bigr) \Bigr| \\ 
&= \Bigl| \sum_{\mathbf i \in \mathbb I_m} \sum_{\mathbf j \in \overline{\mathbb J}}\mathtt w(\mathbf i, \mathbf j) \Bigl[ \mathtt f\bigl(x(\mathbf i, \mathbf j)\bigr) - \mathtt f\bigl(x(\mathbf i)\bigr) \Bigr] \Bigr| \\
&\leq \omega_m(\mathtt f) \sum_{\mathbf i \in\mathbb I_m} \sum_{\mathbf j \in \overline{\mathbb J}} \mathtt w(\mathbf i, \mathbf j) = \omega_m(\mathtt f).
\end{align*}
The final step in the display above uses the first relation in \eqref{wt-sum}, namely that the weights $\{\mathtt w(\mathbf i, \mathbf j) : (\mathbf i, \mathbf j) \in \mathbb I_{m'}\}$ sum up to 1. The penultimate step follows from the observation that $x(\mathbf i) \in \mathtt I(\mathbf i)$ and $x(\mathbf i, \mathbf j) \in \mathtt I(\mathbf i, \mathbf j) \subseteq \mathtt I(\mathbf i)$; this means $|x(\mathbf i) - x(\mathbf i, \mathbf j)| \leq |\mathtt I(\mathbf i)| = \mathtt s_m^{-\mathtt b_m}$, and therefore  
 \[ \bigl| \mathtt f\bigl(x(\mathbf i, \mathbf j)\bigr) - \mathtt f\bigl(x(\mathbf i)\bigr) \bigr| \leq  \omega_m(\mathtt f) := \sup \bigl\{|\mathtt f(x) - \mathtt f(x')| : x, x' \in [0,1], \; |x - x'| \leq \mathtt s_m^{-\mathtt b_m}  \bigr\}. \] 
The uniform continuity of $\mathtt f$ on $[0,1]$ dictates that its modulus of continuity $\omega_m(\mathtt f) \rightarrow 0$ as $m \rightarrow \infty$. Thus the complex-valued sequence $\bigl\{\langle \varphi_m, \mathtt f \rangle : m \geq 1\}$ is Cauchy, confirming the existence of the limit in \eqref{limit}. Since $\varphi_m \geq 0$, the statement concerning non-negativity of $\mu$ follows. Since the function $\varphi_m$ is a probability density, we have $||\varphi_m||_1=1$ for each $m$. Thus  
\[ |\langle \mu, \mathtt f \rangle | = \lim_{m \rightarrow \infty} |\langle \varphi_m, \mathtt f \rangle | = \lim_{m \rightarrow \infty} \Bigl| \int_{0}^{1} \mathtt f(x) \varphi_m(x) \, dx \Bigr| \leq \lim_{m \rightarrow \infty}||\mathtt f||_{\infty} ||\varphi_m||_1 = ||\mathtt f||_{\infty}, \] 
which leads to \eqref{limit-2}, with equality when $\mathtt f$ is the constant function 1.   
\vskip0.1in 
\noindent Finally, let us choose $\mathtt f \in \mathcal C[0,1]$ with 
\[ \text{supp}(\mathtt f) \subseteq [0,1] \setminus E = \bigcup_{m=1}^{\infty} [0,1] \setminus E_m. \]
Since supp$(\mathtt f)$ is compact and its (relatively) open cover $\{[0,1] \setminus E_m: m \geq 1\}$ is a monotone increasing sequence of sets, there exists $m_0 \geq 1$ such that 
\[ \text{supp}(\mathtt f) \subseteq [0,1] \setminus E_{m_0}, \quad \text{ which in turn implies }\quad \supp(\mathtt f) \subseteq [0,1] \setminus E_{m} \text{ for all } m \geq m_0. \] We know from the construction in Section \ref{iteration-section} that $\varphi_m$ is supported on $E_m$. This means that $\langle \varphi_m, \mathtt f \rangle = 0$ for all $m \geq m_0$, from which \eqref{limit-3} follows.    
\end{proof} 
\begin{definition} \label{skewed-measure-def} 
Let  $\pmb{\Pi}= (\mathcal S, \mathcal A, \mathcal B, \mathcal D, \mathcal E)$ be a choice of parameters \eqref{param-seq} obeying \eqref{abs} and \eqref{rmsm-assumption}. The probability measure $\mu = \mu(\pmb{\Pi})$ constructed by the iterative application of $\mathscr{O}$, as described in Sections \ref{iteration-subsection}--\ref{indexing-section-2},  will henceforth referred to as a {\em{skewed measure}}. Proposition \ref{def-mu-prop} ensures its existence.
\end{definition}  

\section{Properties of skewed measures} \label{skewed-measure-properties-section} 
Having defined a skewed measure $\mu = \mu(\pmb{\Pi})$ in Section \ref{measure-definition-section}, the next step is to investigate its properties. This is the focus of this section.  
Finer assumptions on the parameter $\pmb{\Pi}$ are needed to establish these properties, beyond the necessary ones \eqref{param-seq}--\eqref{rmsm-assumption} for creating a skewed measure.
Section \ref{preliminary-Pi-choice-section} summarizes these assumptions. Each of Sections \ref{nonnormality-statement-section}--\ref{normality-statement-section} states without proof a property of $\mu$, some under further conditions on $\pmb{\Pi}$. These properties, namely normality and non-normality, Frostman and Rajchman properties, are the four pillars of the proofs of Theorems \ref{mainthm-1} and \ref{mainthm-2}. Assuming the results recorded in this section, we prove the theorems in Section \ref{mainproof-section}, by furnishing choices of $\pmb{\Pi}$ that ensure simultaneous occurrence of these properties. Parts \ref{Part-nonnorm-F-R} and \ref{part-normality} of this article are devoted to proving the four main results, Propositions \ref{non-normality-prop}-\ref{normal-prop}, stated here.    
\vskip0.1in 
\noindent A few words about the notation, which will be henceforth used without further reference.  
\begin{itemize}
\item The sum (respectively union) of a collection of numbers (respectively elements) $\sigma(n)$ over an index set $\mathbb S$ will be written as 
\[ \sum_n \bigl\{\sigma(n) : n \in \mathbb S \bigr\}, \quad \text{respectively} \quad \bigcup_n  \bigl\{\sigma(n) : n \in \mathbb S \bigr\}. \] 
\item Inequalities from this section onward will often involve large (or small) positive constants, denoted generically by $C, C_0$ (or $c, c_0$ respectively). As a rule of thumb, $C_0$, $c_0$ refer to absolute constants, whereas $C$, $c$  indicate fixed constants for a given $\mu=\mu(\pmb{\Pi})$. Exact values of these constants may change from one occurrence to the next.
\end{itemize}
\subsection{Basic assumptions on $\pmb{\Pi}$} \label{preliminary-Pi-choice-section}
Let us henceforth fix a choice of bases $\mathscr{B}' \subsetneq \mathbb N \setminus \{1\}$, as dictated by Theorems \ref{mainthm-1} and \ref{mainthm-2}. Let $\mathscr{C} \subseteq \mathbb N \setminus \{1, 2\}$ be a collection, minimal with respect to inclusion, such that \begin{equation} \label{what-is-C} 
\overline{\mathscr{C}} = \overline{\mathscr{B}'}.
\end{equation} 
While the choice of $\mathscr{C}$ is non-unique, the definition decrees that any two elements of $\mathscr{C}$ are multiplicatively independent, and every base in $\mathscr{B}'$ depends multiplicatively on a unique element of $\mathscr{C}$. Certain fundamental assumptions on $\pmb{\Pi}= (\mathcal S, \mathcal A, \mathcal B, \mathcal D, \mathcal E)$ involving $\mathscr{B}'$ and $\mathscr{C}$ underpin all the results in this paper. We state them here, defining en route auxiliary parameters $\mathfrak M, \mathfrak N$ and $\mathscr{D}$. Additional hypotheses beyond these baseline assumptions will be introduced as needed in the statements of the relevant results.  The order of selection of the parameters is as follows: 
\[\mathscr{B}' \to \mathscr{C} \to \mathcal{S} \to (\mathfrak M, \mathfrak N) \to \mathcal{E} \to (\mathcal{A},\,\mathcal{B}),\qquad \mathscr{C} \to \mathscr{D}(t) \to \mathcal{D}. \]
\vskip0.1in 
\begin{itemize}
	\item Let $\mathcal S = \{\mathtt s_m : m \geq 1\}$ be an indexed sequence of bases such that for each $\mathtt s \in \mathcal S$, 
		\begin{align}\label{s-condition-plus}
			\mathtt s = \mathtt{t}^n \; \text{ for some } \; n \in \mathbb{N} \; \text{ and } \; \mathtt t \in \mathscr{C}.
		\end{align} 
	Unlike $\mathscr{C}$, repetitions and dependencies are allowed in $\mathcal S$; in other words, there may exist $\mathtt t \in \mathscr{C}$ and $m, m', n, n' \in \mathbb N$, $m \ne m'$, such that $\mathtt s_m = \mathtt t^n \sim \mathtt t^{n'} = \mathtt s_{m'}$. This results in two integer-valued infinite sequences 
	\[ \mathfrak M(\mathtt t) = \{m_j = m_j(\mathtt t) : j \geq 1\} \subseteq \mathbb N \quad \text{ and } \quad \mathfrak N(\mathtt t) = \{n_j = n_j(\mathtt t) : j \geq 1 \} \subseteq \mathbb N \] with the following properties:
	\item The sequence $\mathfrak M(\mathtt t) \subseteq \mathbb N$ is strictly increasing, with
		\begin{equation} 
			\mathfrak M(\mathtt t) \cap \mathfrak M(\mathtt t') = \emptyset \; \text{ for } \; \mathtt t \ne \mathtt t', \quad \bigsqcup \bigl\{\mathfrak M(\mathtt t) : \mathtt t \in \mathscr{C} \bigr\} = \mathbb N.  \label{s-condition-0}
		\end{equation} 
	\item The sequence $\mathfrak N(\mathtt t)$ need not be monotone or distinct. It is related to $\mathfrak M(\mathtt t)$ by
		\begin{equation} 
			\mathtt s_{m_j} = \mathtt t^{n_j} \in \mathcal S \quad \text{ for all } m_j \in \mathfrak M(\mathtt t). \label{M-and-N} 
		\end{equation}  The first-time reader can choose $\mathfrak N(\mathtt t)$ to be the constant sequence 1 for simplicity. In this case, $\mathcal S$ and $\mathscr{C}$ coincide as sets, and $\mathfrak M(\mathtt t)$ identifies the indices of the sequence $\mathcal S$ that correspond to the occurrence of $\mathtt t$.     
	
 \item In Propositions \ref{non-normality-prop} and \ref{Frostman-prop}, we assume that  for every $\mathtt t \in \mathscr{C}$, there exist a collection of restricted digits 
\begin{equation} \label{0-and-t-1}
\mathscr{D}(\mathtt t) \subsetneq \mathbb Z_{\mathtt t}, \quad 0 \in \mathscr{D}(\mathtt t), \quad (\mathtt t-1) \not\in \mathscr{D}(\mathtt t). \end{equation} 
The collection $\mathcal D = \bigl\{ \mathscr{D}_m : m \geq 1\bigr\}$ is defined as follows. For $m \in \mathbb N$, the relation \eqref{s-condition-0} identifies unique elements $\mathtt t \in \mathscr{C}$ and $j$ such that $m = m_j(\mathtt t) \in \mathfrak M(\mathtt t)$. We set  
\begin{equation} \label{s-condition-1}
\mathscr{D}_{m} := \Biggl\{\sum_{\ell=0}^{n-1} \mathtt d_{\ell} \mathtt t^{\ell} \, \bigl| \, \mathtt d_{\ell} \in \mathscr{D}(\mathtt t) \text{ for } 0 \leq \ell < n \Biggr\} \subsetneq \mathbb Z_{\mathtt s_{m}}, \text{ with } \mathtt s_m = \mathtt t^{n}, \;  n = n_j(\mathtt t).
\end{equation} 
Then \eqref{0-and-t-1} implies that $0 \in \mathscr{D}_m $ and $(\mathtt s_m - 1) \notin \mathscr{D}_m$ as required in \eqref{rmsm-assumption}. This special construction of $\mathscr{D}_{m}$ is not needed in Proposition \ref{Rajchman-prop} or \ref{normal-prop}.
\item The two sequences $\mathfrak M(\mathtt t), \mathfrak N(\mathtt t)$ are related to each other and to the parameters $\mathcal E$  via
\begin{equation}
	\sum_m \bigl\{ \varepsilon_{m} : m \in \mathfrak M(\mathtt t) \bigr\} = \infty.  \label{s-condition}
\end{equation} 
Put together, the assumptions \eqref{s-condition-plus}-\eqref{s-condition} say that for every $\mathtt t \in \mathscr{C}$, powers of $\mathtt t$ occur in the sequence $\mathcal S$ infinitely often. The restricted digit set $\mathscr{D}_m\subsetneq \mathbb Z_{\mathtt s_m}$ corresponding to $\mathtt s_m = \mathtt t^n$ consists of integers whose $n$-long digit sequence in base $\mathtt t$ lies in $\mathscr{D}(\mathtt t)^n$. The bias parameters $\varepsilon_m$ associated with these infinitely many indices $m$ form a divergent sum.
\item In Proposition \ref{non-normality-prop}, the sequences $\mathcal A, \mathcal B$ are required to obey  
\begin{equation} \label{am/bmzero}
\mathtt a_m < \mathtt b_m, \; 2 \mathtt s_{m-1}^{\mathtt b_{m-1}} < \mathtt s_{m}^{\mathtt a_{m}},\; \text{for all } m\geq 2, \quad\text{and} \quad \lim_{m \rightarrow \infty} \frac{\mathtt a_m}{\mathtt b_m} = 0.
\end{equation} 
\item Propositions \ref{Frostman-prop} and \ref{Rajchman-prop} need a different growth condition on $\mathcal A, \mathcal B$ than \eqref{abs}; namely,  
\begin{equation} \label{ambm-take-2}
\mathtt a_m < \mathtt b_m \quad \text{and}\quad  \mathtt s_{m-1}^{3\mathtt b_{m-1}} < \mathtt s_m^{\mathtt a_m} \quad \text{ for all } m. 
\end{equation}  
Proposition \ref{normal-prop} requires an even stronger growth property than \eqref{ambm-take-2}.  
\end{itemize}     
With these assumptions in place, we are ready to state the main results.
\subsection{Pointwise non-normality} \label{nonnormality-statement-section}
The iterative construction of a skewed measure $\mu = \mu(\pmb{\Pi})$ embeds a bias towards certain digits in some base of $\mathcal S$ at every step. One therefore expects points in the support of $\mu$ to develop a tendency for non-normality in these bases. Indeed, normality fails in a rather strong sense (even in its simplest form) under fairly mild hypotheses. Proposition \ref{non-normality-prop} below makes this precise. 
\vskip0.1in 
\noindent  
Recall that a number $x \in [0, 1)$ is {\em{simply $b$-normal}} if the defining condition \eqref{b-normal} holds for $k=1$. Clearly every $b$-normal number is simply $b$-normal, but not conversely. We describe a set of conditions under which $\mu$-almost every point fails to be simply $c$-normal, for every $c \in \mathscr{C}$.  

\begin{proposition}
\label{non-normality-prop}
For any choice of bases $\mathscr{B}' \subset \mathbb{N} \setminus \{1\}$, 
assume that $\mathscr{C}$ is a minimal representation set of $\mathscr{B}'$ as  in \eqref{what-is-C}. Suppose 
$\pmb{\Pi} = (\mathcal S,  \mathcal A, \mathcal B, \mathcal D, \mathcal E)$ obeys the conditions \eqref{s-condition-plus}-\eqref{am/bmzero}.  
Let $\mu = \mu(\pmb{\Pi})$ be the resulting skewed measure, as in Definition \ref{skewed-measure-def}. Then 
$\mu$-almost every $x$ belongs to $\mathscr N(\cdot, \mathscr{B}')$, i.e., $\mu$-almost every point is not $b'$-normal for any $b' \in \mathscr{B}'$.
More precisely, $\mu$-almost every point is not even simply $c$-normal for any $c \in \mathscr{C}$. 
\end{proposition}  
\noindent {\em{Remarks: }}
\begin{enumerate}[1.]
\item Proposition \ref{non-normality-prop} is proved in Section \ref{non-normality-section}.  
\item As we will see in Section \ref{mainproof-section}, the criteria \eqref{s-condition-plus}-\eqref{M-and-N} and \eqref{s-condition} required by Proposition \ref{non-normality-prop} are relatively easy to meet. In fact, given any non-empty $\mathscr{B}' \subseteq \{2, 3, \ldots\}$, one can always find sequences $\mathcal S$ and $\mathcal E$ that obey \eqref{s-condition-plus}-\eqref{M-and-N} and \eqref{s-condition}.   
\item \label{non-normality-prop-variant} While Proposition \ref{non-normality-prop} has been stated in the form needed for Theorems \ref{mainthm-1} and \ref{mainthm-2}, its proof goes through for an arbitrary $\mathscr{C}$, not necessarily minimal, whose members may be multiplicatively dependent. For any $\pmb{\Pi}$ obeying \eqref{s-condition-plus}-\eqref{am/bmzero}, the proof yields that $\mu$-almost every $x$ is in $\mathscr{N}(\cdot, \overline{\mathscr{C}})$. An application of this statement may be found in Section \ref{abs-non-normal-section}.
\item It is natural to ask whether $\overline{\mathscr{C}} = \overline{\mathscr{B}'}$ is the largest set of bases, with respect to inclusion, for which the measure $\mu(\pmb{\Pi})$ is non-normal. The answer turns out to depend crucially on $\mathcal D$, and perhaps to a lesser extent on $\mathcal A, \mathcal B$. In Proposition \ref{normal-prop}, we provide a set of conditions on $\mathcal D, \mathcal A, \mathcal B$ for which $\overline{\mathscr{C}} = \overline{\mathscr{B}'}$ is indeed the maximal base set for non-normality. In Proposition \ref{less-than-maximal-nonnormaility-prop} of Section \ref{nonnormality-discussion-section}, we offer a criterion where it is not.   
\end{enumerate} 

\subsection{Frostman condition} 
Frostman's lemma \cite[Chapter 8, p112]{Mattila-1} is a ubiquitous tool in geometric measure theory, characterizing the size of a Borel set in terms of properties of probability measures supported on it. It establishes the following identity for all Borel sets $E$ in Euclidean space: 
\begin{equation}  \text{dim}_{\mathbb H}(E) := \sup \left\{\alpha > 0 \; \Biggl| \; \begin{aligned} &\exists \text{ a probability measure $\nu$ supported on $E$ such } \\ &{\text{ that }} \sup \Bigl\{ \frac{\nu(B(x, r))}{r^{\alpha}} : x \in \mathbb R^d, r > 0 \Bigr\} < \infty. \end{aligned}  \right\}. \label{Hausdorff-dim-def}  \end{equation}    
A probability measure $\nu$ obeying the finite supremum condition in \eqref{Hausdorff-dim-def} is said to satisfy a Frostman-like ball condition of order $\alpha$. A family of $E$-supported measures $\{\nu_{\alpha} : \alpha < \text{dim}_{\mathbb H}(E) \}$ that obeys this condition for a sequence of exponents $\alpha$ leading up to the extremal value $\dim_{\mathbb H}(E)$ is termed a sequence of Frostman measures.  It is known from \cite{p81} that $\text{dim}_{\mathbb H}\bigl(\mathscr N(\mathscr{B}, \mathscr{B}')\bigr) = 1$ for any pair $(\mathscr{B}, \mathscr{B}')$ that is compatible in the sense of Definition \ref{compatibility-definition}. In this sub-section, we extract a sequence of Frostman measures from the class of skewed measures. 
\vskip0.1in 
\noindent Let $\mathtt C_0 > 1$ be a fixed integer; for instance $\mathtt C_0 = 100$ will suffice. For any $\eta > 0$,  let $\mathtt K_0 \geq 2$  be an integer depending only on $\mathtt C_0$ and $\eta$, satisfying
\begin{equation} \label{C0N0}
16\mathtt C_0^2 < \mathtt K_0^{\eta}. 
\end{equation} 
For a collection $\mathscr{B}' \subseteq \mathbb N \setminus \{1\}$ as in Section \ref{preliminary-Pi-choice-section}, we choose, following \eqref{what-is-C}, a minimal set $\mathscr{C}_0$,     
\begin{equation} \label{C0-large}
 \overline{\mathscr{C}}_0 = \overline{\mathscr{B}'}  \quad \text{ with the additional restriction } \quad \mathscr{C}_0 \subseteq \bigl\{\mathtt K_0, \mathtt K_0+1, \ldots \bigr\}. 
\end{equation}
For each $\mathtt t \in \mathscr{C}_0$, the collection $\mathscr{D}(\mathtt t)$ in \eqref{0-and-t-1} has to be large, in the sense that   
\begin{equation} 
 \frac{\mathtt t}{\mathtt C_0}  \leq \underline{\mathtt t} := \# \bigl[\mathscr D(\mathtt t)\bigr] < \mathtt t. \label{size-Dt} 
\end{equation} 
\begin{proposition} \label{Frostman-prop} 
Assume $\mathscr{B}' \subseteq \mathbb N \setminus \{1\}$ is any non-empty collection of bases. Given any $\eta > 0$, let $\mathscr{C}_0 = \mathscr{C}_0 (\eta)$ be a minimal representation set of $\mathscr{B}'$ satisfying \eqref{C0N0}--\eqref{C0-large}. For any choice of parameters $\pmb{\Pi}_{\eta}$
satisfying the conditions \eqref{s-condition-plus}--\eqref{s-condition-1}, \eqref{ambm-take-2} and \eqref{size-Dt}, 
there is a constant $C_{\eta} = C(\pmb{\Pi}_\eta) > 0$ for which the measure $ \mu_{\eta} = \mu(\pmb{\Pi}_{\eta})$ satisfies
\begin{equation} \label{ball-condition-mu}
\mu_{\eta}(I) \leq C_{\eta} |I|^{1 - \eta} \quad \text{ for all intervals } I \subseteq \mathbb R. 
\end{equation} 
\end{proposition} 
\noindent {\em{Remarks: }} 
\begin{enumerate}[1.]
\item Proposition \ref{Frostman-prop} is proved in Section \ref{Frostman-section}.  
\item On a related note, the probability measure $\nu_{\text{L}}$ on $\mathscr{N}(\cdot, \{2\})$ constructed by Lyons \cite{l86} does not obey a Frostman ball condition of any positive order. In that construction, for every $\ell \geq 1$, there exists a basic interval $I_{\ell} = [0, 2^{-K_{\ell}}]$ that receives the mass 
\[ \nu_{\text{L}}(I_{\ell}) = \prod_{j=1}^{\ell} \bigl[\varepsilon_j + (1 - \varepsilon_j) 2^{-K_j} \bigr]. \]
Choosing $\varepsilon_{\ell} = \ell^{-1}$ and $K_{\ell} = K^{\ell \lfloor \sqrt{\log \ell} \rfloor}$ for some large integer $K$ ensures that supp$(\nu_{\text{L}})$ is odd-normal but not even-normal, as has been shown in \cite{PZ-1}. However, 
\begin{align*} 
&\nu_{\text{L}}(I_{\ell}) |I_{\ell}|^{-\eta} \geq 2^{\eta K_{\ell}}\prod_{j=1}^{\ell} \varepsilon_j = 2^{\eta K_{\ell}}/ \ell! \rightarrow \infty \text{ as } \ell \rightarrow \infty, \text{ for every $\eta > 0$}, \\ 
&\text{as a result of which } \not\exists \; \text{ any } \eta > 0 \text{ satisfying } \sup\Bigl\{ {\nu_{\text{L}}(I_{\ell})}{|I_{\ell}|^{-\eta}} : \ell \geq 1 \Bigr\} < \infty.   
 \end{align*}   
\item The Frostman property does not need further assumptions on parameters $\mathcal E$ in $\pmb{\Pi}$ except for $\varepsilon_m \in (0,1)$ as in \eqref{param-seq}. In particular, the divergence condition \eqref{s-condition} is not used. 
\end{enumerate} 

\subsection{Rajchman property}
Next we investigate the behaviour of the Fourier coefficients of $\mu$ at infinity. The formula for the $n^{\text{th}}$ Fourier coefficient of a measure has been given in \eqref{Fourier-coefficient-def}. Let us recall from \eqref{def-J} the definition of the integer $\mathtt N_m$. We will also need two quantities $\mathtt Q_m > 1$ and $\mathtt \tau_m < 1$ based on $\pmb{\Pi}$:  
\begin{equation} \label{Qtau}
\mathtt N_m + 1 \leq \mathtt s_m^{\mathtt a_m} \mathtt s_{m-1}^{-\mathtt b_{m-1}} < \mathtt N_m + 2, \qquad \mathtt Q_m := \mathtt s_m^{{(\mathtt a_m + \mathtt b_m)}/{2}}, \qquad \tau_m := \mathtt s_m^{{(\mathtt a_m - \mathtt b_m)}/{2}}.
\end{equation} 
The pointwise estimate on $\widehat{\mu}(\xi)$ will be based on these quantities.
\begin{proposition}
\label{Rajchman-prop}
Let $\pmb{\Pi}$ be a choice of parameters as in \eqref{param-seq} obeying \eqref{ambm-take-2} and \eqref{rmsm-assumption}. 
Then there exists a constant $C = C(\pmb{\Pi}) > 0$ for which the skewed measure $\mu = \mu(\pmb{\Pi})$ enjoys  the following property. 
For all $m \in \mathbb N$,
\begin{equation} \label{pre-Rajchman}
|\widehat{\mu}(\xi)| \leq C(\pmb{\Pi}) \Bigl[ \varepsilon_{m} + \varepsilon_{m+1} + \tau_{m} + \tau_{m+1} + \mathtt N_{m+1}^{-\frac{1}{2}}\Bigr]  \quad
\text{ provided } \quad |\xi| \in (\mathtt Q_{m}, \mathtt Q_{m+1} ] \cap \mathbb N.  
\end{equation} 
Here $\mathtt N_m, \mathtt Q_m$ and $\tau_m$ are the parameters depending on $\pmb{\Pi}$ defined in \eqref{Qtau}. 
\end{proposition} 
\noindent The growth condition \eqref{ambm-take-2} implies that $\mathtt s_{m-1}^{\mathtt a_{m-1}} < \mathtt s_{m-1}^{\mathtt b_{m-1}} < \mathtt s_{m}^{\frac{\mathtt a_{m}}{3}} < \mathtt s_m^{\mathtt a_m}$. In other words, the sequence $\{\mathtt s_m^{\mathtt a_m}: m \geq 1\}$ is strictly increasing, and hence $\mathtt Q_m, \mathtt N_m \rightarrow \infty$. If $\mathcal A, \mathcal B$ also satisfy \eqref{am/bmzero}, then $\tau_m \rightarrow 0$ as well. Substituting these into \eqref{pre-Rajchman} yields the following.  
\begin{corollary} \label{Rajchman-corollary}
If the parameter $\pmb{\Pi} = (\mathcal S, \mathcal A, \mathcal B, \mathcal D, \mathcal E)$ obeys the conditions \eqref{am/bmzero}, \eqref{ambm-take-2} and $\varepsilon_m \rightarrow 0$, then $\mu = \mu(\pmb{\Pi})$ is Rajchman.  
\end{corollary} 
\noindent {\em{Remarks: }} 
\begin{enumerate}[1.] 
\item Proposition \ref{Rajchman-prop} is proved in Section \ref{Rajchman-section}, based on estimates gathered in Sections \ref{exp-sum-section} and \ref{mu-Rajchman-section}.
\item Like non-normality (Proposition \ref{non-normality-prop}), the Rajchman property (Proposition \ref{Rajchman-prop}) is also relatively easy to meet, by imposing the additional mild restriction that $\mathcal E$ converges to zero.  
\item \label{remark-Rajchman} Proposition \ref{Rajchman-prop} does not bear an explicit relation with $\mathscr{B}'$ or $\mathscr{C}$ as mentioned in Section \ref{preliminary-Pi-choice-section}. That means the parameters $\mathcal{S}$ can be an arbitrary indexed sequence of bases, which is only required to satisfy \eqref{ambm-take-2}, and need not obey \eqref{s-condition-plus}--\eqref{M-and-N}.  Neither does the proposition require any assumptions on the restricted digit sets $\mathscr{D}_m$ beyond \eqref{rmsm-assumption}.
\end{enumerate}

\subsection{Pointwise normality} \label{normality-statement-section} 
Finally we address the question: is there any base with respect to which $\mu$-almost every point is normal? In view of Proposition \ref{non-normality-prop}, certain bases must be eliminated from this consideration. Our goal is to show that, under appropriate hypotheses on $\mathcal S, \mathcal A, \mathcal B, \mathcal D$ and regardless of $\mathcal E$, the answer is affirmative for all other allowable integer bases. 
\vskip0.1in 
\noindent The proof of pointwise normality requires the exponent sequences $\mathcal A$ and $\mathcal B$ to grow far more rapidly than they did in previous propositions, for instance in \eqref{ambm-take-2}. The choice of such rapidly increasing sequences $\mathcal A$ and $\mathcal B$ depends on the sequence of bases $\mathcal{S}$.
For any given $\mathcal{S}=\{\mathtt s_m:\; m\geq 1\} \subset \mathbb N \setminus \{1\}$, the proof determines a numerical sequence $\{\kappa_m : m \geq 1\} \subseteq (0, 1)$; see Lemma \ref{Schmidt-lemma-extremal-digits} for the definition of $\kappa_m$. 
Fixing a large absolute constant $C_0 \geq 10$, we then choose a fast-increasing sequence of integers $\mathcal K = \{\mathtt K_m : m \geq 1\}$ such that 
\begin{align} 
& \mathtt K_1 \geq C_0, \; \log(\mathtt K_m) \geq (m \mathtt K_{m-1} \mathtt s_{m-1})^{C_0} \; \text{ and } \; \kappa_m \mathtt K_m/\mathtt s_m^{C_0} \rightarrow \infty \text{ as } m \rightarrow \infty;\label{Km-choice} \\  
&\text{for instance, choosing } C_0 = 100 \text{ and } \log(\mathtt K_m) \geq \bigl( \kappa_m^{-1} \mathtt s_m + m \mathtt K_{m-1} \mathtt s_{m-1} \bigr)^{100} \text{ will suffice.} \nonumber 
\end{align}  
The sequences $\mathcal A = \{\mathtt a_m\}$, $\mathcal B = \{\mathtt b_m\}$ will be of the following form: 
\begin{equation} \label{normality-hypotheses}
\mathtt a_m = \mathtt s_m^{\mathtt K_m^2}, \quad \mathtt b_m = 2m\mathtt a_m, \quad
\text{which implies} \quad 1 \leq \mathtt a_m < \mathtt a_{m+1}. 
\end{equation}   
We ask the reader to verify the criterion \eqref{ambm-take-2} for this choice of $\mathcal A, \mathcal B$, in light of \eqref{Km-choice}. 
\begin{proposition} \label{normal-prop}
Given any non-empty collection of bases $\mathscr{B}' \subseteq \mathbb N \setminus \{1\}$ and one of its minimal representations $\mathscr{C}$ obeying \eqref{what-is-C}, let $\mathcal S = \{\mathtt s_m : m \geq 1\} \subseteq \mathbb N \setminus \{1,2\}$ be an ordered sequence of (not necessarily distinct) bases obeying \eqref{s-condition-plus}.
Then there exists a numerical sequence $\{\kappa_m = \kappa (\mathtt s_m): m \geq 1\} \subseteq (0,1)$ as follows.  
\vskip0.1in 
\noindent Suppose that $\mu = \mu(\pmb{\Pi})$ is the skewed measure associated with $\pmb{\Pi} = (\mathcal S, \mathcal A, \mathcal B, \mathcal D, \mathcal E)$, where $\mathcal A, \mathcal B$ obey \eqref{Km-choice} and \eqref{normality-hypotheses} with this choice of $\{\kappa_m\}$, $\mathcal E = \{\varepsilon_m \} \subseteq (0, 1)$, and  
\begin{equation} \label{normality-digit-01} 
\{0, 1\} \subseteq \mathscr{D}_m \subsetneq \mathbb Z_{\mathtt s_m} \quad \text{ for all } m \geq 1.
\end{equation}
Then  $\mu$-almost every point $x$ is $b$-normal for every
\begin{equation} \label{bases-BC-bar}
b \in \mathscr{B} := \{2, 3, \ldots\} \setminus \overline{\mathscr B'}.
\end{equation} 
\end{proposition} 
\noindent {\em{Remarks: }} 
\begin{enumerate}[1.] 
\item In Part \ref{part-normality} of this paper, we reduce the proof of Proposition \ref{normal-prop} to the verification of a summability criterion; see Proposition \ref{DEL-difference-prop} in Section \ref{normality-overview-section}. A schematic structure of the proof is given in Figure \ref{normality-proof} on page \pageref{normality-proof}. The details are carried out in Sections \ref{normality-overview-section}--\ref{estimating-vm-section-Part2}. 
\item The constants $\kappa_m = \kappa(\mathtt s_m)$ appearing in Proposition \ref{normal-prop} depend only on $\mathtt s_m$, and are related to properties of integers expressed in base $\mathtt s_m$. The defining property of $\kappa_m$ appears in Lemma \ref{Schmidt-lemma-extremal-digits}. See also page \pageref{kappa-crossref}. 
\item Unlike Proposition \ref{Rajchman-prop}, the inclusion of at least two digits in $\mathscr{D}_m$ is essential for Proposition \ref{normal-prop}. Choosing $\mathscr{D}_m = \{0\}$ for all $m$ is in general insufficient for ensuring maximal normality; see Proposition \ref{less-than-maximal-nonnormaility-prop}. 
It is noteworthy that the simple expedient of adding an extra digit destroys any extraneous non-normality in $\mu(\pmb{\Pi})$ permitted by $\mathscr{D}_m =\{0\}$.   
\item Propositions \ref{Rajchman-prop} and \ref{normal-prop} point to an interesting connection between the rate of decay of $\widehat{\mu}$ and normality on supp$(\mu)$. The faster the growth of $\mathcal A = \{\mathtt a_m\}, \mathcal B = \{\mathtt b_m\}$ (and hence $\mathtt Q_m$ in \eqref{Qtau}), the slower the decay of the upper bound in \eqref{pre-Rajchman}. At the same time, $\mu$-everywhere normality is directly linked to the phenomenon that {\em{most}} Fourier coefficients are small, in a sense quantified by a criterion of Davenport, Erd\H{o}s and LeVeque (Lemma \ref{DEL-lemma}). The growth conditions on $\mathcal A, \mathcal B$ given by \eqref{Km-choice}, ensure ``smallness of the majority'' of Fourier coefficients. Large values of $|\widehat{\mu}|$, even if infrequent, could be far greater than the average.   
\item The hypotheses of Proposition \ref{normal-prop} does not place any constraint on $\mathcal E$. In particular, $\pmb{\Pi}$ need not satisfy $\varepsilon_m \rightarrow 0$, as required by Corollary \ref{Rajchman-corollary} for the Rajchman property. 
\end{enumerate}

\section{Conditional proofs of Theorems \ref{mainthm-1} and \ref{mainthm-2}} \label{mainproof-section}
\subsection{Propositions \ref{non-normality-prop}, \ref{Rajchman-prop} and \ref{normal-prop} imply Theorem \ref{mainthm-1}} \label{mainthm-1-proof-section} 
\begin{proof} 
Let $(\mathscr{B}, \mathscr{B}')$ be a compatible pair of sets consisting of integer bases, in the sense of Definition \ref{compatibility-definition}. If $\mathscr{B}' = \emptyset$, then $\mathscr{B} = \mathbb N \setminus \{1\}$, so that $\mathscr{N}(\mathscr{B}, \mathscr{B}')$ consists of absolutely normal numbers in $[0,1)$. Borel's theorem \cite{b09} furnishes Lebesgue measure on $[0,1)$ as the necessary Rajchman measure in this case. Henceforth, $\mathscr{B}'$ will be assumed to be non-empty. Set 
\[  \mathscr{A}' := \mathscr{B}', \quad \mathscr{A} := \{2, 3, \ldots\} \setminus \overline{\mathscr{B}'} \supseteq \mathscr{B}, \] 
so that $(\mathscr{A}, \mathscr{A}')$ a maximally compatible pair, as defined on page \pageref{compatibility-definition}. We choose $\pmb{\Pi}$ so that the corresponding skewed measure $\mu = \mu(\pmb{\Pi})$ is Rajchman, and $\mu$-almost every $x$ obeys the criterion
\begin{equation} \label{revised-mainthm-1} x \in \mathscr{N}(\mathscr{A}, \mathscr{A}') \subseteq \mathscr{N}(\mathscr{B}, \mathscr{B}'), \text{ which is the conclusion of Theorem \ref{mainthm-1}}. \end{equation} 
\vskip0.1in
\noindent In order to specify $\pmb{\Pi}$, let us first fix a minimal representation $\mathscr{C} \subseteq \mathbb N \setminus \{1, 2\}$ of $\mathscr{A}'= \mathscr{B}'$ obeying \eqref{what-is-C}, and a sequence $\mathcal E_0 = \{\mathtt e_j : j \geq 1\} \subseteq (0, 1)$ such that 
\begin{equation} \label{kappaj}
\mathtt e_j \searrow 0, \qquad \sum_{j=1}^{\infty} \mathtt e_j = \infty; \quad \text{e.g. setting } \mathtt e_j = \frac{1}{j} \text{ will suffice.} 
\end{equation}
Next, we specify two sequences $\mathcal S = \{\mathtt s_m: m \geq 1\}$ and $\mathcal E = \{\varepsilon_m : m \geq 1\}$ satisfying \eqref{s-condition-plus}--\eqref{M-and-N} and \eqref{s-condition}.  
This can be achieved in many ways. For instance given an enumeration of $\mathscr{C} = \{\mathtt t_1, \mathtt t_2, \ldots \}$, 
a possible choice of $(\mathcal S, \mathcal E)$ is given in Table \ref{table-1}, with
\begin{equation} \mathtt s_m = \mathtt t_{m - 2^j + 2} \; \text{ and } \; \varepsilon_m = \mathtt e_j \; \text{ for all } \; m \in \bigl\{2^j-1, 2^j, \ldots, 2^{j+1}-2 \bigr\}. \label{sm-arrangement} \end{equation}  
We observe that $\mathcal S, \mathcal E$ coincide with $\mathscr{C}, \mathcal{E}_0$ respectively as sets, but their elements appear with multiplicity in a prescribed order. With $\mathcal{S}$ as in Table \ref{table-1}, the sequences $\mathfrak M(\mathtt t) = \{m_j(\mathtt t) \}$ and $\mathfrak N(\mathtt t) = \{n_j(\mathtt t) \}$ given by \eqref{s-condition-0} and \eqref{M-and-N} are automatically determined; namely
for $\mathtt t = \mathtt t_k \in \mathscr{C}$, 
\begin{equation} m_j(\mathtt t) = k + 2^j - 2 \quad \text{ and } \quad n_j(\mathtt t) = 1 \quad \text{ for } j \geq 1. \label{M-and-N-choice} \end{equation}  
Then we set $\mathscr{D}(\mathtt t) :=  \{0, 1\}$ so that \eqref{0-and-t-1} holds, and
\begin{equation}
	\mathscr{D}_{m_j(\mathtt t)} := \mathscr{D}(\mathtt t) = \{0, 1\} \subsetneq \mathbb Z_{\mathtt s_{m_j(\mathtt t)}} = \mathbb Z_{\mathtt t}.
	\label{not-Frostman-D} 
\end{equation} 
Thus the conditions \eqref{s-condition-1} and \eqref{normality-digit-01} are satisfied.

\begin{table}[htb]
		\begin{center}
			\begin{tabular}{|l|ll|llll|l|lll|l}
				\hline
				$j$ & 1 &   &  2 &   &   &   & $\cdots$ & $\ell$ &  & & $\cdots$\\
				\hline
				$m$ & 1 & 2 & 3 & 4 & 5 & 6 & $\cdots$ & $2^{\ell}-1$ & $\cdots$ & $2^{\ell +1}-2$ & $\cdots$\\
				\hline
				$\mathtt s_m$ & $\mathtt t_1$ & $\mathtt t_2$ & $\mathtt t_1$ & $\mathtt t_2$ & $\mathtt t_3$ & $\mathtt t_4$ & $\cdots$ & $\mathtt t_1$ & $\cdots$ & $\mathtt t_{2^{\ell}}$ & $\cdots$\\
				\hline
				$\varepsilon_m$ & $\mathtt e_1$ & $\mathtt e_1$ & $\mathtt e_2$ & $\mathtt e_2$ & $\mathtt e_2$ & $\mathtt e_2$ & $\cdots$ & $\mathtt e_{\ell}$ & $\cdots$ & $\mathtt e_{\ell}$ & $\cdots$\\
				\hline
			\end{tabular}
		\end{center}	
		\caption{A possible choice of $m$, $\mathtt s_m$ and $\varepsilon_m$, given $\mathcal{E}_0$ and $\mathscr{C}$} \label{table-1}
\end{table}
\vskip0.1in
\noindent Once $\mathcal S$ is fixed, Proposition \ref{normal-prop} identifies a sequence $\mathfrak{K} = \{ \kappa_m \} \subseteq (0, \infty)$, which in turn gives rise to integer-valued sequences $\mathcal K = \{\mathtt K_m : m\geq 1\}$ obeying \eqref{Km-choice}, and $\mathcal A = \{\mathtt a_m\}$, $\mathcal B = \{\mathtt b_m\}$ as in \eqref{normality-hypotheses}. Choosing $\mathtt K_m \nearrow \infty$ sufficiently rapidly, for instance obeying \eqref{Km-choice}, one can ensure that \eqref{am/bmzero} and \eqref{ambm-take-2} hold. Thus $\pmb{\Pi}$ meets all the hypotheses of Propositions \ref{non-normality-prop}, \ref{Rajchman-prop} and \ref{normal-prop}.    
\vskip0.1in
\noindent It follows from Propositions \ref{non-normality-prop} and \ref{normal-prop} that $\mu$-almost every point is in $\mathscr{N}(\mathscr{A}, \mathscr{A}')$, as claimed in \eqref{revised-mainthm-1}. 
Since $\varepsilon_m \rightarrow 0$, Corollary \ref{Rajchman-corollary} implies that $|\widehat{\mu}(\xi)| \rightarrow 0$ as $|\xi| \rightarrow \infty$, i.e., $\mu$ is Rajchman.  This completes the proof of Theorem \ref{mainthm-1}, modulo Propositions \ref{non-normality-prop}, \ref{Rajchman-prop} and \ref{normal-prop}.
 \end{proof} 
 \subsection{Propositions \ref{non-normality-prop}--\ref{normal-prop} imply Theorem \ref{mainthm-2}}  \label{mainthm-2-proof-section}
\begin{proof}
Let $\mathscr{B}' \subseteq \{2, 3, \ldots\}$ be any non-empty collection of bases. The class of measures $\mathscr{M}^{\ast}(\mathscr{B}')$, whose existence has been claimed  in Theorem \ref{mainthm-2}, is defined as 
\begin{equation} 
\mathscr{M}^{\ast}(\mathscr{B}') : =\left\{ \mu = \mu(\pmb{\Pi}) \; \Biggl| \; \begin{aligned} &\exists \text{ a minimal representation } \mathscr{C} \text{ of $\mathscr{B}'$ } \text{ as in } \eqref{what-is-C}, \\
&\pmb{\Pi} = (\mathcal S, \mathcal A, \mathcal B, \mathcal D, \mathcal E) \text{ obeys } \eqref{s-condition-plus}-\eqref{am/bmzero}. 
\end{aligned}\right\}.
\end{equation} 
The construction of $\pmb{\Pi}$ in Section \ref{mainthm-1-proof-section} shows that $\mathscr{M}^{\ast}(\mathscr{B}')$ is non-empty. 

\subsubsection{Proof of Theorem \ref{mainthm-2} \eqref{mainthm-2-parta}} 
The class of measures $\mathscr{M}^{\ast}(\mathscr{B}')$ has been crafted so that Proposition \ref{non-normality-prop} applies to its constituents $\mu = \mu(\pmb{\Pi})$. If $\mathscr{C}$ is the minimal representation of $\mathscr{B}'$ underlying $\pmb{\Pi}$, 
Proposition \ref{non-normality-prop} ensures that $\mu(\pmb{\Pi})$-almost every point fails to be simply normal in every base of $\mathscr{C}$, and therefore is non-normal in every base of $\mathscr{B}'$.  

\subsubsection{Proof of Theorem \ref{mainthm-2} \eqref{mainthm-2-partb}}
 Let $\mathscr{M}(\mathscr{B}')$ denote the sub-collection of $\mathscr{M}^{\ast}(\mathscr{B}')$ for which $\pmb{\Pi}$ additionally obeys the hypotheses of Proposition \ref{normal-prop}, namely \eqref{normality-hypotheses} and \eqref{normality-digit-01}. This collection is also non-empty, as shown in Section \ref{mainthm-1-proof-section}. One concludes from Proposition \ref{normal-prop} that for $\mu \in \mathscr{M}(\mathscr{B}')$, almost every point lies in $\mathscr{N}(\mathscr{B}, \mathscr{B}')$, with $\mathscr{B}$ as in \eqref{normality-seeking-def}. This is the conclusion of part \eqref{mainthm-2-partb} of Theorem \ref{mainthm-2}.

\subsubsection{Proof of Theorem \ref{mainthm-2} \eqref{mainthm-2-partc}}
It remains to show that one can extract a sequence of Rajchman measures from $\mathscr{M}(\mathscr{B}')$ that is also Frostman. Given $\eta \in (0,1)$, we first fix $\mathtt C_0, \mathtt K_0$ satisfying \eqref{C0N0}, then select a minimal representation $\mathscr{C}_0$ of $\mathscr{B}'$ obeying \eqref{C0-large}, and finally choose $\pmb{\Pi}_{\eta}$ that meets all the conditions for applying Propositions \ref{non-normality-prop}, \ref{Frostman-prop}, \ref{normal-prop} and Corollary \ref{Rajchman-corollary}. Indeed, once an enumeration of $\mathscr{C}_0 = \{\mathtt t_1, \mathtt t_2, \ldots \}$ is determined, choosing $\mathcal S, \mathcal E$ as in Table \ref{table-1} will suffice, following \eqref{kappaj}, \eqref{sm-arrangement}, \eqref{M-and-N-choice}. The main distinction from Section \ref{mainthm-1-proof-section} is that this time $\mathscr{D}(\mathtt t)$, and hence $\mathscr{D}_m$, is chosen differently; instead of \eqref{not-Frostman-D}, we set
\[ \mathscr{D}(\mathtt t) = \bigl\{0, 1, \ldots, \underline{\mathtt t}-1 \bigr\} \; \; \text{ with } \; \; \underline{\mathtt t} = \lfloor \mathtt t/\mathtt C_0 \rfloor + 1 \; \text{ so that \eqref{size-Dt} and \eqref{0-and-t-1} hold,}  \]
and use \eqref{s-condition-1} to define $\mathscr{D}_m$ with this choice of $\mathscr{D}(\mathtt t)$.
The resulting measure $\mu_{\eta} = \mu(\pmb{\Pi}_{\eta})$ lies in $\mathscr{M}(\mathscr{B}')$ by definition. Proposition \ref{Frostman-prop} then asserts that $\mu_{\eta}$ satisfies the Frostman condition \eqref{ball-condition-mu}, which is the same as \eqref{Frostman-condition} in Theorem \ref{mainthm-2}. Since $\varepsilon_m \rightarrow 0$,  each measure $\mu_{\eta}$ is also Rajchman by virtue of Corollary \ref{Rajchman-corollary}. This completes the proof of Theorem \ref{mainthm-2}.
\end{proof}

\newpage 
\part{Measure-theoretic properties} \label{Part-nonnorm-F-R}
This segment of the paper is a direct continuation of the preceding one. Properties of skewed measures that were stated in Propositions \ref{non-normality-prop}, \ref{Frostman-prop} and \ref{Rajchman-prop} will be proved here.  

\section{Non-normality on the support of skewed measures} \label{non-normality-section} 
This section contains the proof of Proposition \ref{non-normality-prop}. 
\subsection{A full-measure set of non-normal numbers} 
\noindent The statement of Proposition \ref{non-normality-prop} posits a set $\mathscr{E}^{\ast} \subseteq E = \text{supp}(\mu)$ of full $\mu$-measure, whose elements fail to be simply normal in every base of $\mathscr{C}$. In this section, we identify the set $\mathscr{E}^{\ast}$ in terms of sets $\mathscr{E}_m$ arising at finitary stages of the construction of $\mu$.
\vskip0.1in 
\noindent For $b \in \{2, 3, \ldots \}$ and $x \in [0,1)$, let us express $x$ in base $b$ as  
\begin{equation} 
x = \sum_{\ell=1}^{\infty} d_{\ell}(x, b) b^{-\ell} \; \text{ with } \; d_{\ell}(x, b) \in \mathbb Z_{b} := \{0, 1, \ldots, b-1\} \text{ for all } \ell \geq 1.
\label{expansion-base-t}   
\end{equation}
Such a digit expansion is unique for most $x$, in the following sense. If $x$ admits two distinct digit expansions of the form \eqref{expansion-base-t}, then both digit sequences must eventually become constant, consisting either of 0 (hence a finite expansion), or $(b-1)$. Such a point $x$ must of course be rational. Points $x$ that are of interest in this section have the property that $d_{\ell} \neq b-1$ for infinitely many indices $\ell$, hence their expansion is unique. 
\vskip0.1in 
\noindent For $\mathtt s_m \in \mathcal S$ and $x$ expanded in base $\mathtt s_m$ as in \eqref{expansion-base-t}, we denote  
\[ \Delta(x, m) := \bigl(d_{\mathtt a_m+1}(x,\mathtt s_m), \ldots, d_{\mathtt b_m}(x, \mathtt s_m) \bigr) \in \mathbb Z_{\mathtt s_m}^{\mathtt b_m - \mathtt a_m}, \]  
The set of interest $\mathscr{E}_m$ consists of points $x \in E$ for which $\Delta(x, m)$ falls in the restricted class: 
\begin{align}
   \label{Njt}
\mathscr{E}_m &:= \bigl\{x \in E: \Delta(x, m) \in \mathscr{D}_m^{\mathtt b_m - \mathtt a_m} \bigr\} \\ 
&\;= \Bigl\{x \in E : d_{\ell}(x, \mathtt s_m) \in \mathscr{D}_m  \text{ for all } \ell \in \bigl\{\mathtt a_{m}+1, \ldots, \mathtt b_{m} \bigr\} \Bigr\},  \nonumber
\end{align} 
where $\mathscr{D}_m$ is the favoured digit set \eqref{rmsm-assumption} used in the $m^{\text{th}}$ step of the construction. The following lemma concerning $\mathscr{E}_m$ will be important in the construction of $\mathscr{E}^{\ast}$.
\begin{lemma} \label{lemma-prob} 
The events $\{\mathscr{E}_m: m \geq 1\}$ are (stochastically) independent in the probability space $(\mathbb R, \mathbb P)$ described in Section \ref{indexing-section-2}. Further,  
\begin{equation} \label{prob}
\mu\bigl(\mathscr{E}_m \bigr) = \varepsilon_m + (1 - \varepsilon_m) \left(\frac{\mathtt r_m}{\mathtt s_m} \right)^{\mathtt b_m - \mathtt a_m} \; \text{ for every } \;  m \geq 1,
\end{equation} 
where $\mathtt r_m := \#(\mathscr{D}_m)$ as in \eqref{rmsm-assumption} and $\mu$ is the limiting probability measure as in Proposition \ref{def-mu-prop}. 
\end{lemma} 
\begin{proof} 
We begin with the proof of \eqref{prob}. In order to compute $\mu(\mathscr{E}_m)$, let us recall from \eqref{iteration} and \eqref{def-Em} the description of the set $E_m$ occurring at the $m^{\text{th}}$ stage of the iteration leading to $\mu$. The relation \eqref{relation-between-endpoints} emerging from this iteration decrees that for any vector of digits $\mathbf d \in \mathbb Z_{\mathtt s_m}^{\mathtt b_m-\mathtt a_m}$, every basic interval  $\mathtt I(\mathbf i_{m-1})$ of $E_{m-1}$ generates the same number of basic subintervals $\mathtt I(\mathbf i_m)$ in $E_m$ whose left endpoints $\alpha(\mathbf i_m)$ obey
\[ \Delta(\alpha(\mathbf i_m), m) = \mathbf d; \quad \text{ as a result, } \quad \Delta(x, m) = \mathbf d \text{ for every } x \in \text{int}\bigl[ \mathtt I(\mathbf i_m) \bigr]. \]
The right endpoint of $\mathtt I(\mathbf i_m)$ may or may not lie in $\mathscr{E}_m$, but this does not affect the evaluation of the measure $\mu$, whose mass on any finite number of points is zero. In particular, specializing to $\mathbf d \in \mathscr{D}_m^{\mathtt b_m-\mathtt a_m}$ as in \eqref{Lm-star}, we find that  
\begin{equation}  
\mathscr{E}_m \cap E_m \subseteq \bigcup \bigl\{\mathtt I(\mathbf i_m): \mathbf i_m  \in \mathbb I_m^{\ast} \bigr\} \; \text{ where }  \; \mathbb I_m^{\ast} := \mathbb I_{m-1} \times \mathbb K_m \times \mathbb L_m^{\ast} \subsetneq \mathbb I_m. \label{special-collection}
\end{equation} 
The index sets $\mathbb I_m, \mathbb K_m$ and $\mathbb L_m^{\ast}$ appearing in \eqref{special-collection} have been defined in Sections \ref{indexing-section-1} and \ref{indexing-section-2}. Moreover, the difference of the two sets in \eqref{special-collection}, namely 
\[ \bigcup \bigl\{\mathtt I(\mathbf i_m): \mathbf i_m  \in \mathbb I_m^{\ast} \bigr\} \setminus \bigl(\mathscr{E}_m \cap E_m \bigr) \]  consists of at most finitely many points, and is thus of measure zero according to any of the densities $\varphi_{\mathtt M}$. For $\mathtt M \geq m$, the mass assigned to $\mathscr{E}_m$ by $\varphi_{\mathtt M}$ is therefore
\begin{equation*} 
\varphi_{\mathtt M} \bigl( \mathscr{E}_m \cap E_m \bigr) = \int_{\mathscr{E}_m \cap E_m} \varphi_{\mathtt M}(x) \, dx = \sum \bigl\{ \varphi_{\mathtt M} \bigl(\mathtt I(\mathbf i_m) \bigr) : \mathbf i_m \in \mathbb I_m^{\ast} \bigr\}.  
\end{equation*} 
But the mass distribution of $\mathscr{O}$ preserves mass on every basic interval (see \eqref{wt-sum}), as a result of which $\varphi_{\mathtt M}(\mathtt I) = \varphi_{m}(\mathtt I)$ for every $m^{\text{th}}$ level basic interval $\mathtt I$ and $\mathtt M \geq m$. This allows us to write  
\begin{align*} 
\mu\bigl(\mathscr{E}_m \bigr) &= \lim_{\mathtt M \rightarrow \infty} \varphi_{\mathtt M} \bigl( \mathscr{E}_m \cap E_m \bigr) = \sum \bigl\{ \varphi_{m} \bigl(\mathtt I(\mathbf i_m) \bigr) : \mathbf i_m \in \mathbb I_m^{\ast} \bigr\} =  \sum \bigl\{ \mathtt w(\mathbf i_m): \mathbf i_m \in \mathbb I_m^{\ast} \bigr\} \\ 
&= \sum_{\mathbf i \in \mathbb I_{m-1}} \mathtt w(\mathbf i) \sum_{k \in \mathbb K_m} \sum_{\ell \in \mathbb L^{\ast}_m} \mathtt N_m^{-1} \vartheta_{\ell}(\mathtt s_m, \mathtt a_m, \mathtt b_m, \mathscr{D}_m, \varepsilon_m) \quad \text{ using \eqref{special-collection}} \\
&= \sum_{\ell \in \mathbb L_m^{\ast}} \vartheta_{\ell}(\mathtt s_m, \mathtt a_m, \mathtt b_m, \mathscr{D}_m, \varepsilon_m) = \mathtt r_m^{\mathtt b_m - \mathtt a_m} \left[ \frac{\varepsilon_m}{\mathtt r_m^{\mathtt b_m-\mathtt a_m}} + \frac{1 - \varepsilon_m}{\mathtt s_m^{\mathtt b_m-\mathtt a_m}}\right],  
 \end{align*}  
which yields the expression in \eqref{prob}. The penultimate equality uses the fact from \eqref{wt-sum} that the weights $\{\mathtt w(\mathbf i) : \mathbf i \in \mathbb I_{m-1}\}$ add up to 1 and that $\#(\mathbb K_m) = \#(\mathbb Z_{\mathtt N_m}) = \mathtt N_m$.  The last step of the display above invokes $\#(\mathbb L_m^{\ast}) = \mathtt r_m^{\mathtt b_m-\mathtt a_m}$, along with the expression of $\vartheta_{\ell}$ from \eqref{choice-of-theta}. 
\vskip0.1in
\noindent Let us continue to the proof of independence of the events $\mathscr{E}_m$. For this, we will use the description of $(E_m, \varphi_m)$  in terms of the random variable $X_m$, as explained in Section \ref{indexing-section-2}. The random variable $X_m$ in turn is specified in terms of an independent collection of random variables $\{\mathfrak Y_k, \mathfrak Z_k, \mathfrak U_k : k \leq m \}$. In this notation, specifically \eqref{independence-rvs} and \eqref{XmYmZmUm}, $\Delta(\cdot, \mathtt s_m)$ is identified with the digits of the random variables $\mathfrak Z_m$ expressed in base $\mathtt s_m$. Thus the event $\mathscr{E}_m$ can be expressed as $\mathscr{E}_m = \{\mathfrak Z_m \in \mathbb L_m^{\ast} \}$, with $\mathbb L_m^{\ast}$ as in \eqref{Lm-star}. Since the random variables $\{\mathfrak Z_m : m \geq 1\}$ are independent by assumption \eqref{independence-rvs}, so are the events $\mathscr{E}_m$, as claimed in Lemma \ref{lemma-prob}. 
\end{proof}
\vskip0.1in  
\noindent We can now define the set $\mathscr{E}^{\ast}$. Recall the collections $\mathscr{B}', \mathscr{C}$ from Section \ref{preliminary-Pi-choice-section}.  Fix $\mathtt t \in \mathscr{C}$, and also recall the infinite sequence $\mathfrak M(\mathtt t) = \{m_j = m_j(\mathtt t) : j \geq 1\} \subseteq \mathbb N$ from \eqref{s-condition-0} and \eqref{M-and-N}. Set 
\[ \bar{\mathscr{E}}_j = \bar{\mathscr{E}}_j(\mathtt t) := \mathscr E_{m_j} \text{ where $ \mathscr E_{m_j}$ is defined as in \eqref{Njt} with $m$ replaced by $m_j(\mathtt t)$}. \] By Lemma \ref{lemma-prob}, the events $\bar{\mathscr{E}}_j$ are independent; moreover, 
\[ \sum_{j=1}^{\infty} \mu\bigl(\bar{\mathscr{E}}_j \bigr) \geq \sum_{j=1}^{\infty} \varepsilon_{m_j(\mathtt t)} = \infty \quad \text{by \eqref{s-condition}}.  \]  
Therefore by the Borel-Cantelli lemma \cite[Theorem 2.3.6]{d10}, 
\begin{equation} \label{full-measure-Nt} 
\mu \bigl(\bar{\mathscr E}(\mathtt t) \bigr)= 1 \; \; \text{ where } \; \; \bar{\mathscr{E}}(\mathtt t) := \limsup_{j \rightarrow \infty} \bar{\mathscr{E}}_j(\mathtt t) = \bigcap_{\ell=1}^{\infty} \bigcup_{j=\ell}^{\infty} \bar{\mathscr{E}}_j(\mathtt t).   
\end{equation} 
The set of interest $\mathscr{E}^{\ast}$, whose existence is claimed in Proposition \ref{non-normality-prop}, is now defined as follows: 
\begin{equation} \label{full-measure-N-star}
\mathscr{E}^{\ast} := \bigcap_{\mathtt t \in \mathscr{C}} \bar{\mathscr{E}}(\mathtt t); \quad \text{ then } \quad \mu(\mathscr{E}^{\ast}) = 1 
\end{equation} 
by virtue of \eqref{full-measure-Nt}, the countability of $\mathscr{C}$ and the trivial inclusion $[0,1) \setminus \mathscr{E}^{\ast} \subseteq \bigcup_{\mathtt t \in \mathscr{C}} [0,1) \setminus \bar{\mathscr{E}}(\mathtt t)$.     
\subsection{Proof of Proposition \ref{non-normality-prop}} 
Since \eqref{full-measure-N-star} already shows $\mathscr{E}^{\ast}$ to be of full $\mu$-measure in \eqref{full-measure-Nt}, it remains to prove
\begin{equation} \label{N-non-normal} 
\mathscr{E}^{\ast} \subseteq \mathscr N(\cdot, \mathscr{B}'), \; \; \text{ where $\mathscr N(\cdot, \mathscr{B}')$ is the non-normal set defined in \eqref{NB'}}. 
\end{equation} 
For $x \in \mathscr{E}^{\ast}$, it follows from the definition \eqref{full-measure-N-star} that $x \in \bar{\mathscr{E}}(\mathtt t)$ for every $\mathtt t \in \mathscr{C}$. Since $\bar{\mathscr{E}}(\mathtt t)$ is the lim sup of sets $\bar{\mathscr{E}}_j$ by its definition \eqref{full-measure-Nt}, one can find an  infinite collection of strictly increasing indices $\mathscr{J} \subseteq \mathbb N$ depending on $x$ and $\mathtt t$ such that $x \in \bar{\mathscr E}_j = \mathscr{E}_{m_j}$ for every $j \in \mathscr{J}$. In view of the definition \eqref{Njt},  a point $x \in \mathscr{E}_{m_j}$ can be written in the form 
\begin{equation} x = \sum_{\ell=1}^{\infty} d_{\ell}(x, \mathtt s_{m_j}) \mathtt s_{m_j}^{-\ell} \quad \text{ with } d_{\ell} \in \mathscr{D}_{m_j} \subsetneq \mathbb Z_{\mathtt s_{m_j}}\text{ for } \ell \in \bigl\{\mathtt a_{m_j}+1, \ldots, \mathtt b_{m_j} \bigr\}. \label{xrep-1} 
\end{equation} 
Here $\mathscr{D}_{m_j}$ is the collection of favoured digits defined in \eqref{s-condition-1}. Substituting  $\mathtt s_{m_j} = \mathtt t^{n_j}$ from \eqref{M-and-N} into \eqref{xrep-1} and expanding the integers $d_{\ell}$ in base $\mathtt t$, we arrive at 
\begin{equation}   
x =\sum_{\ell=1}^{\infty} \Bigl[ \sum_{k=0}^{n_j-1} \overline{d}_{\ell, k}(x, \mathtt t) \mathtt t^{k} \Bigr] \mathtt t^{-\ell n_j} = \sum_{\ell=1}^{\infty} \sum_{k=0}^{n_j-1} \overline{d}_{\ell, k}(x, \mathtt t) \mathtt t^{-(\ell n_j-k)}. \label{xrep-2}
\end{equation}  
The expansions \eqref{xrep-1} and \eqref{xrep-2} of $x$, in bases $\mathtt s_{m_j}$ and $\mathtt t$ respectively, are unique. Indeed, the assumptions \eqref{0-and-t-1} and \eqref{s-condition-1} dictate that 
\begin{equation}  \label{uniqueness-argument}
d_{\ell} \neq \mathtt t^{n_j}-1 \text{ for all } \ell \in \{\mathtt a_{m_j}+1, \ldots, \mathtt b_{m_j}\} \text{ and } j \in \mathscr{J}, \; {\mathtt s_{m_j}}^{\mathtt a_{m_j}} = {\mathtt t}^{n_j \mathtt a_{m_j}}\to \infty \text{ as } j \to \infty. 
\end{equation}  
Since $d_{\ell}$ assumes the value $\mathtt t^{n_j}-1$ if and only if $\overline{d}_{\ell, k} = \mathtt t-1$ for all $k \in \mathbb Z_{n_j}$, \eqref{uniqueness-argument} establishes the existence infinitely many indices $\ell$ for which $\overline{d}_{\ell, k} \neq \mathtt t - 1$ for some $k \in \mathbb Z_{n_j}$. 
Comparing \eqref{xrep-2} and \eqref{expansion-base-t} therefore gives $\overline{d}_{\ell, k}(x, \mathtt t) = d_{\ell n_j-k}(x, \mathtt t)$ for all $\ell \in \mathbb N$ and $k \in \mathbb Z_{n_j}$. The definition \eqref{s-condition-1} of $\mathscr{D}_{m_j}$ implies that  
\[ \overline{d}_{\ell, k}(x, \mathtt t) = d_{\ell n_j-k}(x, \mathtt t) \in \mathscr{D}(\mathtt t) \text{ for }  \ell \in \bigl\{\mathtt a_{m_j}+1, \ldots, \mathtt b_{m_j} \bigr\}, \; k \in \mathbb Z_{n_j}.\] 
Equivalently stated, in the notation of \eqref{expansion-base-t},
\begin{equation} 
 \label{in-D} 
d_{\ell}(x, \mathtt t) \in \mathscr{D}(\mathtt t) \text{ for } \ell \in \bigl\{n_j \mathtt a_{m_j} +1, \ldots, n_j \mathtt b_{m_j}\bigr\}.
\end{equation} 
The inclusion \eqref{in-D} implies that for any $\mathtt d \in \mathbb Z_{\mathtt t}\setminus \mathscr{D}(\mathtt t)$ and all $j \in \mathscr{J}$,  
\begin{equation*} \label{no-d-in-an-bn}  
 \# \Bigl\{1 \leq \ell \leq  n_j \mathtt b_{m_j} : d_\ell(x, \mathtt t) =  \mathtt d \Bigr\} = \# \Bigl\{1 \leq \ell \leq n_j \mathtt a_{m_j} : d_\ell(x, \mathtt t) = \mathtt d \Bigr\}  \leq n_j \mathtt a_{m_j}.
\end{equation*}  
Setting $N = n_j \mathtt b_{m_j}$ and in view of the assumption \eqref{am/bmzero}, this means that 
\begin{align*} 
0 &\leq \liminf_{N \rightarrow \infty} \frac{1}{N} \# \Bigl\{1 \leq \ell \leq N : \; d_\ell(x, \mathtt t) = \mathtt d \Bigr\} \\ &\leq \lim_{j \rightarrow \infty} 
 \Bigl\{\frac{n_j \mathtt a_{m_j}}{n_j \mathtt b_{m_j}} : j \in \mathscr{J} \Bigr\} = 0  \; \text{ for all } \mathtt d \in \mathbb Z_{\mathtt t} \setminus \mathscr{D}(\mathtt t).  \end{align*}  But this violates the requirement \eqref{b-normal} of simple normality in base $\mathtt t$, according to which this limit should be $\mathtt t^{-1}$ for every $\mathtt d \in \mathbb Z_{\mathtt t}$. Thus $x$ is not simply normal in base $\mathtt t$ for any $\mathtt t \in \mathscr{C}$. In particular, this proves \eqref{N-non-normal}.  \qed

\subsection{Remarks on the proof} \label{nonnormality-discussion-section} 
While the hypotheses of Proposition \ref{non-normality-prop} are sufficient for non-normality in the bases of $\overline{\mathscr{B}'}$, stronger restrictions on $\pmb{\Pi}$ may impart non-normality in more bases. In \cite[Section 3]{PZ-1}, we established non-normality for Lyons' measure $\nu_{\mathtt L}$ using Wall's \cite{w49} equivalence for normality in terms of uniform distribution mod one. The proof of \cite[Proposition 3.2]{PZ-1} can be transferred essentially verbatim to prove the following result, stated here for comparison purposes: 
\begin{proposition} \label{less-than-maximal-nonnormaility-prop}
Suppose $\pmb{\Pi}$ obeys the hypotheses of Proposition \ref{non-normality-prop}. If additionally there is a base $\mathtt t \in \mathscr{C}$ for which 
\[ \mathscr{D}_m = \{0\} \quad \text{ for all } m \in \mathfrak M(\mathtt t),  \]
then $\mu$-almost every $x$ is not normal in any base $b$ that is divisible by $\mathtt t$.   
\end{proposition}

\section{Ball conditions for skewed measures} \label{Frostman-section} 
The goal of this section is to prove Proposition \ref{Frostman-prop}. This is the claim that a skewed measure $\mu = \mu(\pmb{\Pi})$ obeys a Frostman-type ball condition \eqref{ball-condition-mu}, provided the elements of $\mathcal S$ are large in value and the restricted digit sets $\mathscr{D}_m$ are large in size, in the sense described in assumptions \eqref{C0N0}--\eqref{size-Dt}. We will verify \eqref{ball-condition-mu} first for all basic intervals of $E_m$ (Section \ref{Frostman-basic-section}) and certain subintervals thereof (Section \ref{Frostman-special-scales-section}), from which the general case will follow (Section \ref{Frostman-prop-proof-section}).    
\subsection{Frostman property on basic intervals} \label{Frostman-basic-section}  
\begin{lemma} \label{Frostman-lemma} 
	For $\eta > 0$, let $\mathtt C_0, \mathtt K_0, \pmb{\Pi}_{\eta}$ be as in the statement of Proposition \ref{Frostman-prop}. Then \eqref{ball-condition-mu} holds for the basic intervals generated in the iterative construction of $\mu$. More precisely, one can find a constant $C_{\eta} = C(\mathtt C_0, \mathtt K_0, \pmb{\Pi}_{\eta}) > 0$ such that  
	\begin{align} \label{ball-condition} 
		\varphi_m(\mathtt I(\mathbf i_m)) = \mathtt w(\mathbf i_m) &\leq C_{\eta} \mathtt s_m^{-\mathtt b_m(1 - \eta)} \quad \text{ for all  } \; \; \mathbf i_m \in \mathbb I_m, \; m \geq 1.   
	\end{align} 
	Here $\mathbb I_m$ is the index set for basic intervals as defined in \eqref{def-I}, whereas $\mathtt w(\mathbf i_m)$ is the weight associated with the  basic interval $\mathtt I(\mathbf i_m)$ as mentioned in \eqref{wt-sum}. 
\end{lemma} 
\begin{proof} 
	Let us choose the constant $C_{\eta} > 0$ in \eqref{ball-condition} according to the relation   
	\begin{equation} \label{define-Ceta}
		\frac{1}{\mathtt N_1} \mathtt r_1^{-(\mathtt b_1-\mathtt a_1)} \leq C_{\eta} \mathtt s_1^{-\mathtt b_1(1 - \eta)}
		\quad \text{ with }\; \mathtt r_1 \; \text{ as in }\; \eqref{rmsm-assumption}.
	\end{equation}   
	We will prove \eqref{ball-condition} using induction on $m$. Let us start with the base case $m=1$. For any $\mathbf{i}_1 \in \mathbb I_1$, it follows from the mass distribution formula \eqref{second-step-weight}, \eqref{choice-of-theta} and \eqref{define-Ceta} that 
	\begin{align*} 
		\mathtt w(\mathbf{i}_1) &\leq \frac{1}{\mathtt N_1} \bigl[\varepsilon_1 \mathtt r_1^{-(\mathtt b_1 - \mathtt a_1)} +  (1 - \varepsilon_1) \mathtt s_1^{-(\mathtt b_1 - \mathtt a_1)} \bigr]  \leq \frac{1}{\mathtt N_1} \times \mathtt r_1^{-(\mathtt b_1 - \mathtt a_1)}  \leq C_{\eta} \mathtt s_1^{-(1 - \eta)}.
	\end{align*} 
	This concludes the proof of the base case. Let us proceed to the induction step. Suppose that the estimate \eqref{ball-condition} holds for all $1 \leq m \leq k-1$. For any $k\in\mathbb{N}$, the assumption \eqref{s-condition-plus} ensures the existence of some $\mathtt t_k \in \mathscr{C}_0$ and $n_k \in \mathbb{N}$ such that $\mathtt s_k = \mathtt t_k^{n_k}$. Then \eqref{0-and-t-1}, \eqref{s-condition-1} and \eqref{size-Dt} yield 
	\begin{equation} \label{skrk-induction} 
	\#\bigl[\mathscr{D}_k \bigr] = \mathtt r_k = \bigl( \underline{\mathtt t_k} \bigr)^{n_k}\ \ \text{with}\ \ \underline{\mathtt t_k} := \#\bigl[\mathscr{D} (\mathtt t_k) \bigr], \quad 
	\mathtt s_k = \mathtt t_k^{n_k} \leq \mathtt C_0^{n_k} \bigl( \underline{\mathtt t_k} \bigr)^{n_k} = \mathtt C_0^{n_k} \mathtt r_k.  
	\end{equation}  
	For $\mathbf i_{ k} \in \mathbb I_k$, estimates similar to the base case combined with \eqref{skrk-induction} lead to
	\begin{align} 
		\mathtt w(\mathbf i_{ k}) &\leq \Bigl[{\mathtt w(\mathbf i_{k-1})} \times \mathtt N_{k}^{-1} \Bigr] \times \Bigl[\varepsilon_{ k} \mathtt r_{k}^{-(\mathtt b_{ k} - \mathtt a_{ k})} + (1 - \varepsilon_{k}) \mathtt s_{k}^{-(\mathtt b_{k} - \mathtt a_{k})}\Bigr] \nonumber \\ 
		&\leq \mathtt w(\mathbf i_{ k-1}) \times \Bigl[2 {\mathtt s_{k-1}^{\mathtt b_{k-1}}}{\mathtt s_{k}^{-\mathtt a_{k}}} \Bigr] \times \mathtt r_{k}^{-(\mathtt b_{\mathtt k} - \mathtt a_{k})} \nonumber \\ 
		&\leq \Bigl[ C_{\eta} \mathtt s_{k-1}^{-\mathtt b_{k-1}(1 - \eta)} \Bigr] \times \Bigl[ 2 {\mathtt s_{k-1}^{\mathtt b_{k-1}}}{\mathtt s_{k}^{-\mathtt a_{k}}} \Bigr] \times \Bigl[ \mathtt C_0^{n_k(\mathtt b_{k} - \mathtt a_{k})} \mathtt s_{k}^{-(\mathtt b_{k} - \mathtt a_{k})}  \Bigr]  \nonumber = 2C_{\eta} \mathtt C_0^{n_k(\mathtt b_{k} - \mathtt a_{k})} \mathtt s_{k-1}^{\eta \mathtt b_{k-1}} \mathtt s_{k}^{- \mathtt b_{k}} \nonumber \\  
		&\leq C_{\eta} \Bigl[2 \mathtt C_0^{n_k \mathtt b_{k}} \mathtt s_{k}^{-\eta \mathtt b_{k}/2}\Bigr] \times \Bigl[\mathtt s_{k-1}^{ \mathtt b_{k-1}} \mathtt s_{k}^{-\mathtt b_{k}/2} \Bigr]^{\eta} \times  \mathtt s_{k}^{-\mathtt b_{k}(1 - \eta)} \nonumber \\ 
		&\leq C_{\eta} \Bigl[2 \mathtt C_0 {\mathtt t_k^{-\frac{\eta}{2}} }\Bigr]^{\mathtt b_{k}n_k} \times  \mathtt s_{k}^{-\mathtt b_{k}(1 - \eta)}
		\leq C_{\eta} \Bigl[2 \mathtt C_0 \mathtt K_0^{-\frac{\eta}{2}} \Bigr]^{\mathtt b_{k}n_k}   \mathtt s_{k}^{-\mathtt b_{k}(1 - \eta)}  \leq C_{\eta} \mathtt s_{k}^{-\mathtt b_k(1 - \eta)}.  \label{ball-condition-step2}
	\end{align}   
	The second inequality in the display above uses the relation $\mathtt r_k < \mathtt s_k$ and also 
	\begin{equation}  \mathtt N_k+ 1 = \lfloor \mathtt s_k^{\mathtt a_k}/\mathtt s_{k-1}^{\mathtt b_{k-1}}\rfloor \geq {\mathtt s_k^{\mathtt a_k}/(2\mathtt s_{k-1}^{\mathtt b_{k-1}})+1,} \; \text{ a consequence of the definition \eqref{def-J} of $\mathtt N_k$.} \label{Nm-lower-bound}  
	\end{equation}
	The third inequality uses the induction hypothesis \eqref{ball-condition} with $m = k-1$ and the relation $\mathtt s_k \leq \mathtt C_0^{n_k} \mathtt r_k$ from \eqref{skrk-induction}. The first inequality in \eqref{ball-condition-step2} follows from the relation {$\mathtt s_k = \mathtt t_k^{n_k}$} and the assumption \eqref{ambm-take-2}, which ensures that $\mathtt s_{k-1}^{\mathtt b_{k}-1} \leq \mathtt s_{k}^{\mathtt b_{k}/2}$. The next two use the bound {$\mathtt t_k \geq \mathtt K_0$} from \eqref{C0-large} and the choice \eqref{C0N0} of $\mathtt C_0, \mathtt K_0$. This completes the induction and the proof of \eqref{ball-condition}.      
\end{proof} 
\subsection{Frostman property on non-basic special scales} \label{Frostman-special-scales-section} 
The inequality \eqref{ball-condition} establishes the ball condition \eqref{ball-condition-mu} for basic intervals of length $\mathtt s_m^{-\mathtt b_m}$  that appear in $\mathscr{O}_m$. Here we prove \eqref{ball-condition-mu} for intervals at certain intermediate scales of length $\mathtt s_m^{-\mathtt a_m-j}$, $0 \leq j \leq \mathtt b_m - \mathtt a_m$.   
\vskip0.1in 
\noindent Let $m \geq 1$. 
For every $\mathbf i_{m-1} \in \mathbb I_{m-1}$ and following \eqref{relation-between-endpoints}, let us recall the interval $\mathtt J$ as in \eqref{JkI} and define its subinterval ${I}_{\ell}(j)$ as follows. For $k \in \mathbb Z_{\mathtt N_m}$, $0 \leq \ell < \mathtt s_m^j$ and $0\leq j \leq \mathtt b_m-\mathtt a_m$,
\begin{align} 
	&\mathtt J = \mathtt J_k = \alpha(\mathbf i_{m-1})+ \eta(\mathbf i_{m-1}) + {k}{\mathtt s_m^{-\mathtt a_m}} + \bigl[0, \mathtt s_m^{-\mathtt a_m}\bigr], \text{ and }  \nonumber\\  
	&I_{\ell}(j) := \alpha (\mathbf i_{m-1}) + \eta (\mathbf i_{m-1}) + k \mathtt s_{m}^{-\mathtt a_m } + \ell \mathtt s_m^{-\mathtt a_m - j} + \left[0, \mathtt s_m^{-\mathtt a_m-j} \right] \subseteq \mathtt J \subseteq \mathtt I(\mathbf i_{m-1}). \label{def-Ilj}
\end{align}
Thus $\mathtt J$ is a basic interval obtained after the first step of $\mathscr{O}_m$. For each fixed $j$, the intervals $\{{I}_{\ell}(j) : 0 \leq \ell < \mathtt s_m^j \}$ are (essentially) disjoint constituents of $\mathtt J$, each of length $\mathtt s_m^{-\mathtt a_m - j}$. Both $\mathtt J$ and $I_{\ell}(j)$ depend on $k \in \mathbb Z_{\mathtt N_m}$; however the remainder of the argument in this section works for any fixed $k$, and we choose to omit its reference to de-clutter notation.
\begin{lemma} 
	Under the assumptions of Lemma \ref{Frostman-lemma}, the intervals ${I}_{\ell}(j)$ as in \eqref{def-Ilj} obey 
	\begin{equation} \label{Frostman-nonspecial-scales}
		\varphi_{m} \bigl({I}_{\ell}(j)\bigr) \leq C_{\eta} \mathtt s_m^{-(1 - \eta)(\mathtt a_m + j)} \quad \text{ for all } m \geq 1, \; \ell \in \mathbb Z_{\mathtt s_m^j}, \; 0 \leq j < \mathtt b_m-\mathtt a_m.  
	\end{equation} 
\end{lemma} 
\begin{proof} 
	The mass distribution principle  \eqref{choice-of-theta} dictates that $\varphi_m({I}_{\ell}(j))$ depends on the number of basic intervals of the form $\mathtt I(\mathbf i_m)$ contained within ${I}_{\ell}(j)$ that correspond to the favoured digit set $\mathscr{D}_m \subseteq \mathbb Z_{\mathtt s_m}$. In order to make this precise, let us denote by $\mathscr{J}$ (respectively $\mathscr{J}^{\ast}$) the collection of all (respectively favoured) basic intervals $\mathtt I(\mathbf i_m)$ contained in $\mathtt J$. 
	Then for a fixed index $j$:
	\begin{itemize} 
		\item Any $\mathtt I \in \mathscr{J}$ is either contained in ${I}_{\ell}(j)$, or its interior is disjoint from $I_{\ell}(j)$. Each interval $I_{\ell}(j)$ contains the same number of intervals of $\mathscr{J}$,
		\begin{equation} 
			\# \bigl\{\mathtt I \in \mathscr{J} : \mathtt I \subset {I}_{\ell}(j) \bigr\} = {\mathtt s_m}^{\mathtt b_m - \mathtt a_m-j}.
		\end{equation}
		\item Similarly, 
		\begin{equation}  \label{favoured-numbers} 
			\begin{aligned} 
				\text{ if } \ell &= \sum_{r=0}^{j-1} \tilde{\ell}_r \mathtt s_m^{r} \in \mathbb Z_{\mathtt s_m^j}, \; \; \text{ then } \\  \# \bigl\{\mathtt I \in \mathscr{J}^{\ast} : \mathtt I \subset {I}_{\ell}(j) \bigr\} &= \begin{cases}  {\mathtt r_m}^{\mathtt b_m - \mathtt a_m - j} &\text{ if $(\tilde{\ell}_0, \ldots, \tilde{\ell}_{j-1}) \in \mathscr{D}_m^j$}, \\ 0 &\text{ otherwise.}  \end{cases} \end{aligned} \end{equation} 
	\end{itemize} 
	Motivated by \eqref{favoured-numbers}, the inequality \eqref{Frostman-nonspecial-scales} is proved in two cases. If $\tilde{\ell}_r \notin \mathscr{D}_m$ for some $r \in \{0, 1, \ldots, j-1\}$, then \eqref{favoured-numbers} says that all the basic intervals $\mathtt I(\mathbf i_m) \subseteq I_{\ell}(j)$ belong to $\mathscr{J} \setminus \mathscr{J}^{\ast}$ and therefore receive the same mass $(1 - \varepsilon_m) \mathtt s_m^{-(\mathtt b_m - \mathtt a_m)}$. Applying this along with the already established ball condition \eqref{ball-condition} for $\mathtt I(\mathbf i_{m-1})$, \eqref{Nm-lower-bound}, and \eqref{ambm-take-2}, we arrive at
	\begin{align*} 
		\varphi_m(I_{\ell}(j)) &= \Bigl[{\mathtt w(\mathbf i_{m-1})} \times {\mathtt N_m}^{-1} \Bigr] \times \mathtt s_m^{\mathtt b_m - \mathtt a_m - j} \times \Bigl[(1 - \varepsilon_m) \mathtt s_m^{-(\mathtt b_m - \mathtt a_m)} \Bigr] \\ &\leq \mathtt w(\mathbf i_{m-1}) \times \Bigl[2 \mathtt s_m^{-\mathtt a_m} \mathtt s_{m-1}^{\mathtt b_{m-1}} \Bigr] \times \mathtt s_m^{-j} \leq 2C_{\eta} \mathtt s_{m-1}^{\eta \mathtt b_{m-1}} \mathtt s_m^{-\mathtt a_m - j} \\ &\leq 2C_{\eta} \mathtt s_{m-1}^{\eta \mathtt b_{m-1}} \mathtt s_m^{-\eta(\mathtt a_m + j)} |I_{\ell}(j)|^{1 - \eta} 
		\leq 2C_{\eta} \Bigl[\mathtt s_{m-1}^{\mathtt b_{m-1}} \mathtt s_{m}^{-\mathtt a_m} \Bigr]^{\eta} |I_{\ell}(j)|^{1 - \eta} \leq C_{\eta} |I_{\ell}(j)|^{1 - \eta}, 
	\end{align*} 
	which is the desired inequality \eqref{Frostman-nonspecial-scales}. If $(\tilde{\ell}_0, \ldots, \tilde{\ell}_{j-1}) \in \mathscr{D}_m^j$, then \eqref{choice-of-theta} and \eqref{favoured-numbers} give    
	\begin{align*} 
		\frac{\varphi_m(I_{\ell}(j))}{\varphi_m(\mathtt J)} &={\mathtt r_m}^{\mathtt b_m - \mathtt a_m - j} \Bigl[ \varepsilon_m \mathtt r_m^{-(\mathtt b_m - \mathtt a_m)} + (1 - \varepsilon_m) \mathtt s_m^{-(\mathtt b_m- \mathtt a_m)}\Bigr] \\ &\hskip1.5in + \bigl(\mathtt s_m^{\mathtt b_m - \mathtt a_m - j} - {\mathtt r_m}^{\mathtt b_m - \mathtt a_m - j} \bigr) (1 - \varepsilon_m) \mathtt s_m^{-(\mathtt b_m - \mathtt a_m)}   \\
		&= \varepsilon_m {\mathtt r_m}^{-j} + (1 - \varepsilon_m) \mathtt s_m^{-j}. 
	\end{align*} 
	The same reasoning as above involving \eqref{ball-condition}, \eqref{skrk-induction}, \eqref{Nm-lower-bound} and \eqref{C0N0}--\eqref{size-Dt} leads to 
	\begin{align*} 
		\varphi_m(I_{\ell}(j)) &= \Bigl[{\mathtt w(\mathbf i_{m-1})}{\mathtt N_m}^{-1} \Bigr] \times \Bigl[ \varepsilon_m {\mathtt r_m}^{-j} + (1 - \varepsilon_m) \mathtt s_m^{-j} \Bigr] \leq \Bigl[{\mathtt w(\mathbf i_{m-1})}{\mathtt N_m}^{-1} \Bigr] \times {\mathtt r_m}^{-j}\\ &\leq {\mathtt w(\mathbf i_{m-1})}\Bigl[2{\mathtt s_m^{-\mathtt a_m}} \mathtt s_{m-1}^{\mathtt b_{m-1}} \Bigr] {\mathtt r_m}^{-j} \leq 2C_{\eta} \mathtt C_0^{n_m j} \mathtt s_{m-1}^{\eta \mathtt b_{m-1}} \mathtt s_m^{-\mathtt a_m-j} \\ &= 2C_{\eta} |I_{\ell}(j)|^{1 - \eta} \mathtt C_0^{n_m j} \mathtt s_m^{-\eta(\mathtt a_m +j)} \mathtt s_{m-1}^{\eta \mathtt b_{m-1}} = 2C_{\eta} |I_{\ell}(j)|^{1 - \eta} \Bigl[\mathtt s_{m-1}^{\mathtt b_{m-1}} \mathtt s_m^{-\mathtt a_m}\Bigr]^{\eta} \times \Bigl[ \mathtt C_0^{n_m} \mathtt s_m^{-\eta}\Bigr]^j  \\ 
		&\leq C_{\eta}  |I_{\ell}(j)|^{1 - \eta}  \times \Bigl[ \mathtt C_0 \mathtt t_m^{-\eta}\Bigr]^{n_m j} \leq C_{\eta}  |I_{\ell}(j)|^{1 - \eta}  \times \Bigl[ \mathtt C_0 \mathtt K_0^{-\eta}\Bigr]^{n_m j}
		\leq C_{\eta} |I_{\ell}(j)|^{1 - \eta}.
	\end{align*} 
	The last line uses \eqref{skrk-induction}, \eqref{C0-large} and \eqref{C0N0}, which ensure that $\mathtt s_m = \mathtt t_m^{n_m}$, $\mathtt K_0 \le \mathtt t_m \in \mathscr{C}_0$ and $\mathtt C_0 \mathtt K_0^{-\eta} \le 1$. This completes the proof of \eqref{Frostman-nonspecial-scales}. 
\end{proof}

\subsection{Proof of Proposition \ref{Frostman-prop}} \label{Frostman-prop-proof-section} 
\begin{proof} 
	Since $E = \text{supp}(\mu) \subseteq [0,1]$, we only need to verify \eqref{ball-condition-mu} for intervals $I$ with $I \subseteq  [0,1]$. Given such an interval $I$, let $m \geq 1$ be the unique integer such that 
	\begin{equation} \mathtt s_m^{-\mathtt b_m} \leq |I| < \mathtt s_{m-1}^{-\mathtt b_{m-1}}. \label{I-size} \end{equation}
	We consider two cases, depending on the length of $I$ relative to the intermediate scale $\mathtt s_m^{-\mathtt a_m}$. 
	\vskip0.1in 
	\noindent {\em{Case 1:}} First suppose that 
	\begin{equation} 
		\mathtt s_m^{-\mathtt a_m} \leq |I| < \mathtt s_{m-1}^{-\mathtt b_{m-1}}. \label{case-1-assumption} 
	\end{equation} 
	Then there are at most two adjacent intervals 
	$\mathtt I_1$ and $\mathtt I_2$ of the form  
	\[ \mathtt I_1 = \bigl[n_1 \mathtt s_{m-1}^{-\mathtt b_{m-1}}, (n_1+1) \mathtt s_{m-1}^{-\mathtt b_{m-1}} \bigr], \; 
	\mathtt I_2 = \bigl[(n_1+1) \mathtt s_{m-1}^{-\mathtt b_{m-1}}, (n_1+2) \mathtt s_{m-1}^{-\mathtt b_{m-1}} \bigr], \text{ with } n_1 \in \mathbb N \cup \{0\}\] such that $I \subseteq \mathtt I_1 \cup \mathtt I_2$. If $\mathtt I_j$ is not a basic interval of $E_{m-1}$ for some $j \in \{1, 2\}$, then $\varphi_{m-1}(I \cap \mathtt I_j) = 0$, and hence $\varphi_{\mathtt M}(I \cap \mathtt I_j) = 0$ for all $\mathtt M \geq m$. As a result $\mu(I \cap \mathtt I_j) =0$, or $\mu(I) = \mu(I \setminus \mathtt I_j)$ for such $j$. Thus we may assume without loss of generality that both $\mathtt I_1$ and $\mathtt I_2$ are basic intervals of $E_{m-1}$.  The portion of $I$ that falls outside the intermediate basic intervals $\mathtt J_k(\mathtt I_i)$ (given by\eqref{JkI} and of length $\mathtt s_m^{-\mathtt a_m}$) also receives zero mass at stage $m$ and beyond. As a result,  
	\begin{align*} 
		&\varphi_{\mathtt M}\Bigl[I \setminus \bigcup_{i, k} \big\{\mathtt J_k(\mathtt I_i) \cap I :  i = 1, 2 ,\; {k \in \mathbb K_m } \bigr\}\Bigr] = 0 \; \text{ for all } \mathtt M \geq m \; \text{ and } \\
		&\max_{i=1,2}\#\bigl\{k \in \mathbb K_m = \mathbb Z_{\mathtt N_m}: \mathtt J_k(\mathtt I_i) \cap I \ne \emptyset \bigr\} \leq \lfloor |I| \mathtt s_m^{\mathtt a_m} \rfloor + 1 \leq 2 |I| \mathtt s_m^{\mathtt a_m}. 
	\end{align*}  It follows from the mass distribution principle \eqref{mass-wk} that for $\mathtt M \geq m$, 
	\begin{align*} \varphi_{\mathtt M}(I) &\leq \sum_{i, k} \bigl\{\varphi_{\mathtt M}\bigl(\mathtt J_k(\mathtt I_i) \bigr): \mathtt J_k(\mathtt I_i) \cap I \neq \emptyset, i = 1, 2  \bigr\} \\ &= \sum_{i, k} \bigl\{\varphi_{m}\bigl(\mathtt J_k(\mathtt I_i) \bigr): \mathtt J_k(\mathtt I_i) \cap I \neq \emptyset, i=1, 2  \bigr\}  \\ &\leq 2 |I| \mathtt s_m^{\mathtt a_m} \mathtt N_m^{-1} \max \bigl\{ \mathtt w(\mathbf i_{m-1}) : \mathbf i_{m-1} \in \mathbb I_{m-1} \bigr\}.
	\end{align*} 
	Since the right side is independent of $\mathtt M$, letting $\mathtt M \rightarrow \infty$ and applying \eqref{Nm-lower-bound}, \eqref{ball-condition} shows
	\[
	\mu(I) \leq 4|I| \mathtt s_{m-1}^{\mathtt b_{m-1}} \max \bigl\{ \mathtt w(\mathbf i_{m-1}) : \mathbf i_{m-1} \in \mathbb I_{m-1} \bigr\} \leq 4C_{\eta}|I| \mathtt s_{m-1}^{\mathtt b_{m-1}} \mathtt s_{m-1}^{- \mathtt b_{m-1}(1- \eta)} \\ 
	\leq 4 C_{\eta} |I|^{1 - \eta}.  
	\]
	The last inequality uses the assumption \eqref{case-1-assumption}. This proves the inequality \eqref{ball-condition-mu} in case 1. 
	\vskip0.1in
	\noindent {\em{Case 2:}} In the complementary scenario where 
	\begin{equation} 
		\mathtt s_m^{-\mathtt b_m} \leq |I| < \mathtt s_{m}^{-\mathtt a_{m}},
		\label{case-2-assumption} 
	\end{equation} 
	one can find an index $j \in \{0, 1, \ldots, \mathtt b_m - \mathtt a_m - 1\}$ such that
	\begin{equation} \label{case-2-assumption-plus}
		\mathtt s_m^{-\mathtt a_m - j - 1} \leq |I| < \mathtt s_m^{-\mathtt a_m - j}.
	\end{equation}  
	Then $I$ is contained in at most {$2|I| \mathtt s_m^{\mathtt a_m + j + 1}$} (adjacent) intervals of the form $I_{\ell}(j+1)$, given by \eqref{def-Ilj}. Invoking the ball condition \eqref{Frostman-nonspecial-scales} for $I_{\ell}(j+1)$ with $\eta$ replaced by $\eta/2$, we find that 
	\begin{align*}  
		\varphi_m(I) \leq C_{\eta/2} \bigl(2|I| \mathtt s_m^{\mathtt a_m +j + 1}\bigr) \times \bigl( \mathtt s_m^{-\mathtt a_m - j - 1} \bigr)^{1 - \eta/2} 
		{\leq 2C_{\eta/2} |I| \mathtt s_m^{{\eta}(\mathtt a_m  + j + 1)/2} \leq 2C_{\eta/2} |I|^{1 - \eta}. }
	\end{align*}  
	The last inequality is a consequence of {$|I|^{\eta} \mathtt s_m^{\frac{\eta}{2}(\mathtt a_m + j + 1)} \leq 1$,} {which follows from $|I| \mathtt s_m^{(\mathtt a_m  + j + 1)/2} \leq 1$.} In the {range \eqref{case-2-assumption-plus},} this is equivalent to checking that 
	\[ \mathtt s_m^{-(\mathtt a_m + j)} \mathtt s_m^{(\mathtt a_m + j + 1)/2} = \mathtt s_m^{-\frac{1}{2}(\mathtt a_m + j) + \frac{1}{2}} \leq 1, \] 
	which is of course true for {$\mathtt s_m = \mathtt t_m^{n_m} \geq \mathtt K_0^{n_m}$} and $\mathtt a_m \geq 2$. This proves Proposition \ref{Frostman-prop}. 
\end{proof}

\section{Exponential sum estimates related to $\mathscr{O}$} \label{exp-sum-section} 
In this section and the next we focus on $\widehat{\mu}$, with the intent to establish the Rajchman property of $\mu$ in Section \ref{Rajchman-section}. 
Given the fundamental role of the operation $\mathscr{O}$ in the construction of $\mu$ (Sections \ref{elem-op-section} and \ref{iteration-section}), it is natural that the Fourier coefficients of $\mu$ build on those arising from $\mathscr{O}$. Part of this section identifies the exponential sums that occur in the Fourier coefficients of  the output function  $\Phi_1$ of the operation $\mathscr{O}$. 
The rest of it serves as a repository of estimates for these exponential sums that we will draw upon at various subsequent stages. 
\subsection{Preliminary estimates involving $\widehat{1}$} 
As noted in \eqref{density}, the density function $\varphi_{m}$ is a weighted linear combination of indicator functions of intervals. We start by recording the Fourier estimates of these building blocks. 
Let $\mathtt 1$ denote the constant function on $[0,1]$ with unit value. Then 
\begin{equation} \label{1-hat} 
\widehat{\mathtt 1}(\xi) = \int_{0}^{1} e(x \xi) \, dx = \frac{1 - e(\xi)}{2 \pi i \xi} \; \; \text{ where } \; \; e(y) := e^{-2 \pi i y}, \qquad |\widehat{\mathtt 1}(\cdot)| \leq 1.  
\end{equation} 
\begin{lemma} \label{1-hat-lemma}
The function $\widehat{\mathtt 1}(\cdot)$ is analytic on $\mathbb R$ and obeys the estimates
\begin{equation} 
\bigl|\widehat{\mathtt 1}(\xi) \bigr| \leq \min \left[1, \frac{1}{\pi |\xi|} \right]  \; \; \text{ and } \; \; \bigl|\widehat{\mathtt 1}'(\xi) \bigr| \leq 2 \pi \min \left[ 1, \frac{1}{|\xi|}\right]  \quad \text{ for all } \xi \in \mathbb R. \label{1-hat-bound} 
\end{equation} 
\end{lemma} 
\begin{proof} 
The trivial bound $|\widehat{\mathtt 1}(\xi)| \leq 1$ has already been noted in \eqref{1-hat}, and follows from the triangle inequality on the defining integral of $\widehat{\mathtt 1}$.  The second bound on $\widehat{\mathtt 1}$ follows from the explicit formula for $\widehat{\mathtt 1}$: 
\[ \bigl| \widehat{\mathtt 1}(\xi) \bigr| =  \Bigl| \frac{1 - e(\xi)}{2 \pi i \xi}\Bigr| \leq \frac{2}{2 \pi |\xi|} = \frac{1}{\pi |\xi|}. \] 
Combining the two estimates leads to the first inequality \eqref{1-hat-bound}. 
\vskip0.1in 
\noindent The estimation for $\widehat{\mathtt 1}'$ is similar. The relation \eqref{1-hat} can be used in two ways:
\begin{align} 
\widehat{\mathtt 1}'(\xi) &= \int_{0}^{1} (- 2\pi i x) e(x \xi) \, dx, \; \; \text{ which gives } \; \; \bigl| \widehat{\mathtt 1}'(\xi) \bigr| \leq 2 \pi, \quad \text{ and } \label{1hat'-bound1} \\ 
\widehat{\mathtt 1}'(\xi) &= \frac{d}{d \xi} \left[\frac{1 - e(\xi)}{2 \pi i \xi} \right] = \frac{2 \pi i \xi e(\xi) - (1 - e(\xi))}{2 \pi i \xi^2}, \; \; \text{ which yields }  \; \; \bigl|\widehat{\mathtt 1}'(\xi) \bigr| \leq \frac{2}{|\xi|}. \label{1-hat'-bound2}
\end{align}  
The second inequality in \eqref{1-hat-bound} is a consequence of \eqref{1hat'-bound1} and \eqref{1-hat'-bound2}.  
\end{proof} 
\begin{lemma} \label{1-hat-difference-lemma}
There exists an absolute constant $C_0 > 0$ such that for all $\theta_1, \theta_2 \in \mathbb R$, with $|\theta_2| > |\theta_1|$, the following relation holds:  
\begin{equation} 
\bigl| \widehat{\mathtt 1}(\theta_2) - \widehat{\mathtt 1}(\theta_1) \bigr| \leq C_0 \min(1, |\theta_1 - \theta_2|) \times \min \bigl(1, {|\theta_1|}^{-1}\bigr).  \label{1-hat-difference-estimate}
\end{equation}   
\end{lemma} 
\begin{proof} 
The triangle inequality, combined with the estimate \eqref{1-hat-bound} for $\widehat{\mathtt 1}$ from Lemma \ref{1-hat-lemma}, gives 
\begin{equation} 
\bigl| \widehat{\mathtt 1}(\theta_2) - \widehat{\mathtt 1}(\theta_1) \bigr| \leq \bigl| \widehat{1}(\theta_1) \bigr| + \bigl| \widehat{1}(\theta_2)\bigr| \leq \min \Bigl[1, \frac{1}{\pi |\theta_1|}\Bigr] + \min \Bigl[1, \frac{1}{\pi |\theta_2|}\Bigr] \leq C_0 \min\bigl(1, {|\theta_1|}^{-1} \bigr). \label{difference-est-1}
\end{equation}  
The mean value theorem, coupled with the inequality \eqref{1-hat-bound} for $\widehat{\mathtt 1}'$, yields a different estimate: 
\begin{align} 
\bigl| \widehat{\mathtt 1}(\theta_2) - \widehat{\mathtt 1}(\theta_1) \bigr| &\leq |\theta_2 - \theta_1| \sup \Bigl\{\bigl| \widehat{\mathtt 1}'(\theta)\bigr| : \theta \in [\theta_1, \theta_2] \Bigr\} \nonumber \\ 
&\leq C_0 |\theta_2-\theta_1| \min\bigl(1, {|\theta_1|}^{-1}\bigr). 
\label{difference-est-2}
\end{align}  
The two inequalities \eqref{difference-est-1} and \eqref{difference-est-2} have been jointly summarized in \eqref{1-hat-difference-estimate}. 
\end{proof} 

\subsection{Factorization of densities}  \label{density-factorization-section} 
Let us recall the expressions \eqref{Phi0} and \eqref{Phi1} for the densities $\Phi_0$ and $\Phi_1$ in the construction $\mathscr{O}$. Then  for $j=0,1$ and $\xi \in \mathbb Z$, the $\xi^{\text{th}}$ Fourier coefficient of $\Phi_j$ is given by 
\begin{equation}  \label{Phi1-hat} 
\widehat{\Phi}_j(\xi) = \int_{0}^{1} \Phi_j(x) e(x \xi) \, dx  =:  \Upsilon_j(\xi) \times \Biggl\{\begin{aligned} \widehat{1}(\xi \mathtt t^{-1}) &\text{ if } j = 0, \\  \widehat{\mathtt 1}(\xi \mathtt s^{-\mathtt b}) &\text{ if } j =1,  \end{aligned}  \; \text{ where } \; e(y) := e^{-2 \pi i y}. 
\end{equation}  
The factors $\Upsilon_j$ in \eqref{Phi1-hat} are exponential sums of the form 
\begin{equation} \label{Yj}  
\Upsilon_j(\xi) = \Upsilon_j(\xi; \pmb{\gamma}) = \sum_{\mathtt I \in \mathscr{I}_j} \mathtt w(\mathtt I) e \bigl(\xi \alpha(\mathtt I) \bigr), \quad j = 0, 1, 
\end{equation} 
where $\alpha(\mathtt I)$ denotes the left endpoint of $\mathtt I$, $\mathscr{I}_j$ is the collection of basic intervals of $\mathtt E_j$, $j = 0, 1$,
and $\pmb{\gamma}$ is the input vector of $\mathscr{O}$ defined in \eqref{O-def}. The reader will recognize $\Upsilon_j$ as the Fourier coefficient of the discrete probability measure that is supported on the points $\{\alpha(\mathtt I) : \mathtt I \in \mathscr{I}_j\}$, with a mass $\mathtt w(\mathtt I)$ assigned to $\alpha(\mathtt I)$. We are especially interested in the behaviour of the exponential sum $\Upsilon_1$ associated with $(\mathtt E_1, \Phi_1) = \mathscr{O}(\mathtt E_0, \Phi_0)$. In fact, $\Upsilon_1$ is uniquely specified by $\pmb{\gamma}$.
Inserting the formula \eqref{JKL} of $\mathscr I_1$ into \eqref{Yj} with $j = 1$, we arrive at       
\begin{align} 
 \Upsilon_1(\xi) &= \Upsilon_1(\xi; \pmb{\gamma}) :=  \sum_{\mathtt I \in \mathscr{I}_0} \sum_{\mathbf j \in \mathbb J} \mathtt w(\mathtt I(\mathbf j)) e\bigl[\xi \alpha(\mathtt I(\mathbf j)) \bigr], \text{ which simplifies to }  \label{Y-0} \\
 &=  \sum_{\mathtt I \in \mathscr{I}_0} \mathtt w(\mathtt I) \sum_{\mathbf j \in \mathbb J} \frac{\vartheta_{\ell}}{\mathtt N} e\Bigl(\xi \Bigl[ \alpha(\mathtt I) + \eta(\mathtt I) + \frac{k}{\mathtt s^{\mathtt a}} + \frac{\ell}{\mathtt s^{\mathtt b}} \Bigr]\Bigr) \nonumber \\
 &=  \Bigl[\sum_{\mathtt I \in \mathscr{I}_0} \mathtt w(\mathtt I) e \bigl(\xi \bigl[ \alpha(\mathtt I) + \eta(\mathtt I) \bigr] \bigr) \Bigr] \times \Bigl[\sum_{k \in \mathbb K} \frac{1}{\mathtt N} e \bigl(\xi k \mathtt s^{-\mathtt a} \bigr) \Bigr] \times \Bigl[ \sum_{\ell \in \mathbb L} \vartheta_{\ell} e(\xi \ell \mathtt s^{-\mathtt b}) \Bigr]
  \label{Y} \\
  &=:  \Omega_1(\xi) \times \mathfrak A(\xi) \times \mathfrak B(\xi). \label{factorization} 
\end{align}
The computations above rely on the description of $\mathtt I(\mathbf j)$, the product structure of the index set $\mathbb J$ and the weights $\mathtt w(\mathtt I(\mathbf j))$ given by \eqref{step-2-basic}, \eqref{JKL} and \eqref{second-step-weight}. Of the three factors of $\Upsilon_1$ in \eqref{factorization}, 
the first factor $\Omega_1$ given by
\begin{equation} \label{Omega_1} 
\Omega_1(\xi) = \Omega_1(\xi; \pmb{\gamma}) := \sum_{\mathtt I \in \mathscr{I}_0} \mathtt w(\mathtt I)  e \Bigl(\xi \bigl[\alpha(\mathtt I) + \eta(\mathtt I) \bigr]\Bigr) 
\end{equation} 
is the counterpart of $\Upsilon_j$ for the intermediate density $\Psi_1$ given by \eqref{intermediate-density-2} and \eqref{intermediate-density}. Indeed, computing the Fourier coefficient of $\Psi_1$ from the formulae \eqref{intermediate-density-2} and \eqref{intermediate-density} yield two equivalent expressions for $\widehat{\Psi}_1$:   
\begin{align} \label{intermediate-Fourier-coefficient} 
\widehat{\Psi}_1(\xi) &= \widehat{1}(\xi \mathtt N \mathtt s^{-\mathtt a}) \sum_{\mathtt I \in \mathscr{I}_0} \mathtt w(\mathtt I)  e \Bigl(\xi \bigl[\alpha(\mathtt I) + \eta(\mathtt I) \bigr]\Bigr)  =  \widehat{1}(\xi \mathtt N \mathtt s^{-\mathtt a}) \times \Omega_1(\xi) \\ 
 &=  \widehat{1}(\xi \mathtt s^{-\mathtt a}) \times \mathfrak A(\xi) \times \Omega_1(\xi). \label{intermediate-Fourier-coefficient-2} 
\end{align}   
The functions $\mathfrak A$ and $\mathfrak B$, which are the second and third factors in the product \eqref{Y}, will be discussed in Section \ref{A-and-B-section} below momentarily. As with the exponential sum $\Upsilon_j$, the function $\widehat{\Psi}_1$ can also be identified as the $\xi^{\text{th}}$ Fourier coefficient of a discrete measure that assigns mass $\mathtt w(\mathtt I)$ to the point $\alpha(\mathtt I) + \eta(\mathtt I)$ for every $\mathtt I \in \mathscr{I}_0$.    
The function $\Omega_1$ is very similar in structure to $\Upsilon_0$, except for the presence of the spill-over terms $\eta(\mathtt I)$ in the exponent. Each can therefore be estimated by the other, at the expense of a quantifiable error. We record this error in Lemma \ref{error-lemma} below. The factors $\mathfrak A$ and $\mathfrak B$ will be examined in Sections \ref{A-and-B-section}, \ref{AB-single-section} and \ref{AB-double-section}.  
\begin{lemma} \label{error-lemma} 
The functions $\Upsilon_0$ and $\Omega_1$ defined in \eqref{Yj} and \eqref{factorization} satisfy the estimate:  
\begin{equation} \label{difference} 
|\Omega_1(\xi) - \Upsilon_0(\xi)| \leq \min \left[2, 2 \pi |\xi| \mathtt s^{-\mathtt a}\right] \quad \text{ for all } \xi \in \mathbb Z. 
\end{equation}
\end{lemma} 
\begin{proof} 
Since the weights $\{\mathtt w(\mathtt I): \mathtt I \in \mathscr{I}_0 \}$ sum to 1 by \eqref{weights}, both $\Omega_1$ and $\Upsilon_0$ are uniformly bounded above by 1. The triangle inequality therefore gives the trivial bound of 2 in \eqref{difference}. Now suppose that $\pi |\xi| \leq \mathtt s^{\mathtt a}$. The map $y \mapsto e(y)$ is Lipschitz with constant $2 \pi$, which means 
\begin{equation} \label{Lipschitz} 
|e(y) -1 | \leq 2\pi |y| \quad \text{ for all } \quad |y| \leq 1.
\end{equation} 
Using this inequality and comparing \eqref{Yj} and \eqref{Omega_1} leads to the desired conclusion: 
\begin{align*}
|\Omega_1(\xi) - \Upsilon_0(\xi)| &= \Biggl| \sum_{\mathtt I \in \mathscr{I}_0} \mathtt w(\mathtt I) e \bigl(\xi \alpha(\mathtt I)\bigr) \bigl[ e \bigl(\xi\eta(\mathtt I)\bigr) - 1\bigr] \Biggr| \\
& \leq   \sum_{\mathtt I \in \mathscr{I}_0} \mathtt w(\mathtt I) \bigl|  e \bigl(\xi \eta(\mathtt I)\bigr) - 1 \bigr| \leq 2 \pi \sum_{\mathtt I \in \mathscr{I}_0} \mathtt w(\mathtt I) |\xi|  \eta(\mathtt I) \\
&\leq 2 \pi |\xi| \mathtt s^{-\mathtt a} \sum_{\mathtt I \in \mathscr{I}_0} \mathtt w(\mathtt I)  = 2 \pi |\xi| \mathtt s^{-\mathtt a}.
\end{align*} 
Here we have used the fact that $\eta(\mathtt I) \in [0, \mathtt s^{-\mathtt a})$, a consequence of its defining property \eqref{def-m}.
\end{proof}
\subsection{The factors $\mathfrak A$ and $\mathfrak B$} \label{A-and-B-section} 
The factorization \eqref{factorization} of $\Upsilon_1$ yields two other factors $\mathfrak A$ and $\mathfrak B$. These are exponential sums in their own right, each bounded above by 1 in absolute value, and can be further simplified. For instance, $\mathfrak A$ is a geometric sum of complex exponentials:  
\begin{equation} \label{def-A}
\mathfrak A(\xi; \mathtt N, \mathtt s, \mathtt a) = \mathfrak A(\xi)  :=  \frac{1}{\mathtt N} \sum_{k=0}^{\mathtt N -1}  e \left(\xi k \mathtt s^{-\mathtt a} \right)  = \left\{ \begin{aligned}  &1  &\text{ if } \mathtt s^{\mathtt a} \mid \xi, \\ &\frac{1}{\mathtt N}\frac{1 - e(\xi \mathtt N \mathtt s^{-\mathtt a})}{1 - e(\xi \mathtt s^{-\mathtt a})}  &  \text{ if } \mathtt s^{\mathtt a} \nmid \xi. \end{aligned} \right\}
\end{equation}
Here and throughout, the notation $\mathtt p \mid \mathtt q$ means that $\mathtt p$ is an integer that divides $\mathtt q$. Similarly, $\mathtt p \nmid \mathtt q$ implies that $\mathtt q$ is not divisible by $\mathtt p$. 
\vskip0.1in
\noindent The factor $\mathfrak B$ in \eqref{factorization}, on the other hand, admits the following representation, in view of \eqref{choice-of-theta}: 
 \begin{align}
&\mathfrak B(\xi; \mathtt s, \mathtt a, \mathtt b, \mathscr{D}, \varepsilon) = \mathfrak B (\xi) :=   \sum_{\ell \in \mathbb L} \vartheta_{\ell} e \left(\xi \ell \mathtt s^{-\mathtt b} \right)  = \varepsilon \mathfrak B^{\ast}(\xi) + (1 - \varepsilon) \mathfrak C(\xi), \text{ where } \label{def-B} \\ 
&\qquad\mathfrak B^{\ast}(\xi) := {\mathtt r}^{\mathtt a - \mathtt b} \sum_{\ell \in \mathbb L^{\ast}} e(\xi \ell \mathtt s^{-\mathtt b}) \; \; \text{ and } \; \; \mathfrak C(\xi) :=  \mathtt s^{\mathtt a - \mathtt b}  \sum_{\ell \in \mathbb L} e\bigl( \xi \ell \mathtt s^{-\mathtt b} \bigr).  \label{B*C}
\end{align}
 The definitions \eqref{def-A} and \eqref{B*C} express the three exponential sums $\mathfrak A$, $\mathfrak B^{\ast}$ and $\mathfrak C$ as averages of uni-modular terms, which imply the following trivial bound for all of them:
 \begin{equation} \label{B*C-trivial-bound}
\bigl|\mathfrak A(\xi) \bigr| \leq 1, \quad  \bigl|\mathfrak B^{\ast}(\xi) \bigr| \leq 1, \quad \bigl|\mathfrak C(\xi) \bigr| \leq 1 \; \; \text{ and therefore }  \; \;  \bigl|\mathfrak B(\xi) \bigr| \leq 1. 
\end{equation} 
Let us pause for a moment to observe that among the three functions $\mathfrak A$, $\mathfrak B^{\ast}$ and $\mathfrak C$, the set $\mathscr{D}$ of restricted digits involved in $\mathscr{O}$ appears only in $\mathfrak B^{\ast}$. This dependence is through the index set $\mathbb L^{\ast}$, which is identified with $\mathscr{D}^{\mathtt b - \mathtt a}$ according to \eqref{L*}. We ask the reader to note that the estimates derived in this section and the next use only the trivial bound \eqref{B*C-trivial-bound} on $\mathfrak B^{\ast}$ and is therefore independent of the choice of $\mathscr{D}$.  This observation substantiates Remark \ref{remark-Rajchman} following Proposition \ref{Rajchman-prop}, concerning the presence of the Rajchman property of $\mu$ for any choice of $\mathcal D$. 
\subsection{The function $\mathfrak C$} \label{C-section} 
\begin{lemma} \label{BC-lemma}
The function 
$\mathfrak C$ defined in \eqref{B*C} is $\mathtt s^{\mathtt b}$-periodic, and can be expressed as follows: 
\begin{equation}  
\label{C} 
\mathfrak C(\xi) = \left\{ 
\begin{aligned} &\mathtt s^{\mathtt a - \mathtt b}\frac{1 - e(\xi \mathtt s^{-\mathtt a})}{1 - e(\xi \mathtt s^{-\mathtt b})} \quad \text{ if } \mathtt s^{\mathtt a} \nmid \xi, \\
&0 \hskip1.2in \text{ if  } \mathtt s^{\mathtt a} \mid \xi, \; \mathtt s^{\mathtt b} \nmid \xi, \\  
&1 \hskip1.2in  \text{ if } \mathtt s^{\mathtt b} \mid \xi.  
\end{aligned} \right\}
\end{equation} 
\end{lemma} 
\begin{proof} 
Since $\mathfrak C$ is an exponential sum over the index set $\mathbb L$ defined in \eqref{step-2-decomp}, let us start with an explicit representation of its summands. Every $\ell \in \mathbb L$ is characterized by a digit vector $\mathtt d = (\mathtt d_0, \ldots, \mathtt d_{\mathtt b-\mathtt a-1}) \in \mathbb Z_{\mathtt s}^{\mathtt b - \mathtt a}$ obeying the relation in \eqref{rep-l-s}. Therefore 
\begin{align*}
e(\xi \ell \mathtt s^{-\mathtt b}) = e \Bigl( \xi \sum_{j=0}^{\mathtt b - \mathtt a-1} \mathtt d_j \mathtt s^{j - \mathtt b}\Bigr) = \prod_{j=0}^{\mathtt b-\mathtt a-1} e \bigl(\xi \mathtt d_j \mathtt s^{j - \mathtt b} \bigr) 
\end{align*} 
for some $\mathbf d = (\mathtt d_0, \ldots, \mathtt d_{\mathtt b-\mathtt a-1}) \in \mathbb Z_{\mathtt s}^{\mathtt b-\mathtt a}$.  As a result, the sum  $\mathfrak C$  decomposes into a product: 
\begin{align} 
\mathtt s^{\mathtt b-\mathtt a} \mathfrak C(\xi) &= \sum_{\mathbf d \in \mathbb Z_{\mathtt s}^{\mathtt b-\mathtt a}} \prod_{j=0}^{\mathtt b-\mathtt a-1} e \bigl(\xi \mathtt d_j \mathtt s^{j - \mathtt b} \bigr) = \prod_{j=0}^{\mathtt b - \mathtt a-1} \Bigl[ \sum_{\mathtt d_j =0}^{\mathtt s-1} e\bigl( \xi \mathtt d_j \mathtt s^{j-\mathtt b} \bigr) \Bigr] =  \prod_{\mathtt k=\mathtt a+1}^{\mathtt b} \Bigl[ \sum_{\mathtt d =0}^{\mathtt s-1}  e\bigl( \xi \mathtt d \mathtt s^{-\mathtt k} \bigr) \Bigr]  \nonumber \\
&= \prod_{\mathtt k=\mathtt a+1}^{\mathtt b} \left[ 1 + e(\xi \mathtt s^{-\mathtt k}) + \ldots + e \bigl(\xi(\mathtt s-1) \mathtt s^{- \mathtt k} \bigr) \right].  \label{C-factor} 
\end{align} 
Each factor occurring in \eqref{C-factor} is a geometric sum, with the $\mathtt k^{\text{th}}$ factor given by 
\begin{equation}   1 + e(\xi \mathtt s^{-\mathtt k}) + \ldots + e \bigl(\xi(\mathtt s-1) \mathtt s^{- \mathtt k} \bigr) = \left\{ \begin{aligned} &\mathtt s &\text{ if } \mathtt s^{\mathtt k} | \xi, \\ &\frac{1 - e(\xi \mathtt s^{-\mathtt k + 1})}{1 - e(\xi \mathtt s^{-\mathtt k})} &\text{ if } \mathtt s^{\mathtt k} \nmid \xi.  \end{aligned} \right\} \label{geometric-sum}  \end{equation} 
Three cases arise. If $\mathtt s^{\mathtt b} \mid \xi$, then each factor of \eqref{C-factor} is $\mathtt s$, which yields $\mathfrak C(\xi) = 1$. If $\mathtt s^{\mathtt a} \nmid \xi$ then inserting \eqref{geometric-sum} into \eqref{C-factor} leads to a telescoping product: 
\[\mathfrak C(\xi) = \mathtt s^{\mathtt a - \mathtt b} \prod_{\mathtt k = \mathtt a+1}^{\mathtt b} \left[ \frac{1 - e(\xi \mathtt s^{-\mathtt k + 1})}{1 - e(\xi \mathtt s^{-\mathtt k})} \right] = \mathtt s^{\mathtt a- \mathtt b} \frac{1 - e(\xi \mathtt s^{-\mathtt a})}{1 - e(\xi \mathtt s^{-\mathtt b})}. \] 
Finally, if $\mathtt s^{\mathtt a} \mid \xi$, $\mathtt s^{\mathtt b} \nmid \xi$ and $\mathtt k_0 \in \{\mathtt a+1, \ldots, \mathtt b\}$ is the smallest integer such that $\mathtt s^{\mathtt k_0} \nmid \xi$, then each factor of \eqref{C-factor} corresponding to $\mathtt k < \mathtt k_0$ equals $\mathtt s$, whereas the factor corresponding to $\mathtt k = \mathtt k_0$ equals zero. Combining all these observations establishes \eqref{C}.  
\end{proof}

\subsection{Product $\mathfrak A \mathfrak B$ at a single scale} \label{AB-single-section} 
In the factorization formula \eqref{factorization} for $\Upsilon_1$ (and hence $\widehat{\Phi}_1$), the exponential sums $\mathfrak A$ and $\mathfrak B$ given by \eqref{def-A} and \eqref{def-B} occur in the form of a product. Accordingly, let us define 
\begin{equation} \label{AB} 
\mathfrak F(\xi) = \mathfrak F(\xi; \mathtt N, \mathtt s, \mathtt a, \mathtt b, \mathscr{D}, \varepsilon) := \mathfrak A(\xi; \mathtt N, \mathtt s, \mathtt a) \times \mathfrak B(\xi; \mathtt s, \mathtt a, \mathtt b, \mathscr{D}, \varepsilon). 
\end{equation}  
We will estimate $\mathfrak F$ in this sub-section. 
\begin{lemma} \label{single-scale-product-lemma}
For any choice of parameters $\pmb{\gamma}$ as in \eqref{O-def}, let $\mathfrak F$ be the product given in \eqref{AB} and $\mathtt N$ be as in \eqref{N}. Then for all $\xi \in \mathbb Z$ with $1 \leq |\xi| < \mathtt s^{\mathtt b}$, one has the estimate:
\begin{equation} \label{F-est}
|\mathfrak F(\xi)| \leq \Bigl[ \varepsilon + \frac{\mathtt s^{\mathtt a - \mathtt b}}{\mathtt N} \bigl| \sin \bigl(\pi \xi \mathtt s^{-\mathtt b} \bigr)\bigr|^{-1}\Bigr].  
\end{equation} 
\end{lemma} 
\begin{proof} 
The assumption $|\xi| < \mathtt s^{\mathtt b}$ ensures that $\mathtt s^{\mathtt b} \nmid \xi$. In view of the expressions for $\mathfrak A$, $\mathfrak B$ and $\mathfrak C$ from \eqref{def-A}, \eqref{def-B} and Lemma \ref{BC-lemma} respectively, two cases arise. 
\vskip0.1in
\noindent {\em{Case 1: }} Suppose $\mathtt s^{\mathtt a} \mid \xi$. Then $\mathfrak A(\xi) = 1$ from \eqref{def-A} and $\mathfrak C(\xi) = 0$ from Lemma \ref{BC-lemma}. Therefore, 
\begin{equation}  
|\mathfrak F(\xi)| = | \mathfrak B(\xi) | = \varepsilon |\mathfrak B^{\ast}(\xi)| \leq \varepsilon,     \label{F1} 
\end{equation} 
where the last step uses the trivial bound \eqref{B*C-trivial-bound} of $\mathfrak B^{\ast}$.
\vskip0.1in 
\noindent {\em{Case 2: }} If $\mathtt s^{\mathtt a} \nmid \xi$, then \eqref{def-A}-\eqref{B*C} and Lemma \ref{BC-lemma} yield different expressions for $\mathfrak A$ and $\mathfrak C$:  
\begin{align} 
|\mathfrak F(\xi)| &= |\varepsilon \mathfrak A(\xi) \mathfrak B^{\ast}(\xi) +  (1 - \varepsilon) \mathfrak A(\xi) \mathfrak C(\xi)| \leq \varepsilon + |\mathfrak A(\xi) \mathfrak C(\xi)| \nonumber \\ 
&\leq \varepsilon + \left|\frac{1}{\mathtt N}\frac{1 - e(\xi \mathtt N \mathtt s^{-\mathtt a})}{1 - e(\xi \mathtt s^{-\mathtt a})} \right| \times \left|\mathtt s^{\mathtt a - \mathtt b}\frac{1 - e(\xi \mathtt s^{-\mathtt a})}{1 - e(\xi \mathtt s^{-\mathtt b})} \right| \nonumber \\
&\leq \varepsilon + \frac{2\mathtt s^{\mathtt a - \mathtt b}}{\mathtt N} \bigl| 1 - e(\xi \mathtt s^{-\mathtt b}) \bigr|^{-1} =  \varepsilon + \frac{\mathtt s^{\mathtt a - \mathtt b}}{\mathtt N} |\sin(\pi \xi \mathtt s^{-\mathtt b})|^{-1}.  \label{F2} 
\end{align} 
The display above incorporates several trivial bounds: $|\mathfrak A|, |\mathfrak B^{\ast}| \leq 1$ from \eqref{B*C-trivial-bound} in the first inequality, and $|1 - e(\cdot)| \leq 2$ in the third.  
\vskip0.1in 
\noindent Comparing \eqref{F1} and \eqref{F2}, we reach \eqref{F-est} in both cases. This completes the proof. 
\end{proof} 
\subsection{Product $\mathfrak A \mathfrak B$ at consecutive scales} \label{AB-double-section} 
In the definition \eqref{AB} of $\mathfrak F$, the factors $\mathfrak A$ and $\mathfrak B$ correspond to the two parts of a single application of the elementary operation $\mathscr{O}$. However, as we have seen in Section \ref{iteration-subsection}, the construction of the measure $\mu$ involves repeated applications of $\mathscr{O}$, where the collection of basic intervals $\mathscr{I}_1$ from one step serves as $\mathscr{I}_0$ for the next, in the notation of Section \ref{elem-op-section}. The Fourier coefficient of $\mu$ therefore involves terms of the form $\mathfrak A$ and $\mathfrak B$ at many different scales, in a way that has been made precise in Section \ref{mu-Rajchman-section}. In anticipation of this application, it is useful to study the behaviour of the product $\mathfrak A \mathfrak B$ where each factor originates in a different scale. 
\vskip0.1in 
\noindent Towards that end, suppose that the operation $\mathscr{O}$ is applied twice in a row, with the following choice of parameters: $\mathbf s \in (\mathbb{N} \setminus \{1,2\})^3$, $\mathbf a, \mathbf b \in \mathbb N^{3}$, $\pmb{\varepsilon} \in (0,1)^3$, 
\begin{equation}  
\mathtt a_i < \mathtt b_i, \quad \text{ and } \quad 16 \mathtt s_{i}^{\mathtt b_i} < \mathtt s_{i+1}^{\mathtt a_{i+1}}, \quad i = 0,1. \label{sabe}   
\end{equation} 
Here $\mathbf s = (\mathtt s_0, \mathtt s_1, \mathtt s_2)$, and similarly for $\mathbf a, \mathbf b, \pmb{\varepsilon}$.  As can be seen by comparing with the criterion \eqref{ambm-take-2} of admissibility, the condition \eqref{sabe} holds for corresponding triples $(\mathbf s, \mathbf a, \mathbf b)$ of $\mathcal S, \mathcal A, \mathcal B$ with consecutive entries, for an admissible choice of parameters $\pmb{\Pi}$. 
\vskip0.1in 
\noindent The two-fold application $\mathscr{O}$ proceeds as follows: for $i = 1, 2$, we set 
\begin{equation*} 
(\mathtt E_{i}, \Phi_{i}) = \mathscr{O}_{i}(\mathtt E_{i-1}, \Phi_{i-1}) := \mathscr{O}(\mathtt E_{i-1}, \Phi_{i-1}; \mathtt t_{i}, \mathtt s_{i}, \mathtt a_{i}, \mathtt b_{i},  \mathscr{D}_i, \varepsilon_{i}) \; 
\text{ where } \mathtt t_{i} = \mathtt s_{i-1}^{\mathtt b_{i-1}}.
\end{equation*}
We will explore the behaviour of the product $\mathfrak A \mathfrak B$ if the factor $\mathfrak B$ comes from the second step of $\mathscr{O}_1$, and
$\mathfrak A$ comes from the first step of $\mathscr{O}_2$. Accordingly, let us define 
\begin{equation} \label{def-G} 
\mathfrak G(\xi) = \mathfrak G(\xi; \mathbf s, \mathbf a, \mathbf b, \pmb{\varepsilon}) := \mathfrak B(\xi; \mathtt s_1, \mathtt a_1, \mathtt b_1, \mathscr{D}_1, \varepsilon_1)  \times \mathfrak A(\xi; \mathtt N_2, \mathtt s_2, \mathtt a_2), 
\end{equation}   
where $\mathfrak A$ and $\mathfrak B$ are as in \eqref{def-A} and \eqref{def-B} respectively. The integer $\mathtt N_2$ is the counterpart of $\mathtt N$ for $\mathscr{O}_2$, and is defined by the relation \eqref{N}:  
\begin{equation} \label{N2} 
\mathtt N_2 + 1 \leq \mathtt s_2^{\mathtt a_2} \mathtt s_1^{-\mathtt b_1} < \mathtt N_2 + 2, \quad \text{i.e.} \quad \mathtt N_2 + 1 = \lfloor {\mathtt s_2^{\mathtt a_2}}{\mathtt s_1^{-\mathtt b_1}} \rfloor.  
\end{equation} 
The main result of this subsection, Proposition \ref{prop-G} below, gives an estimate for $\mathfrak G$. The intermediate lemmas will provide the key ingredients of its proof.  
\begin{lemma} \label{A2-lemma}
There exists an absolute constant $C_0 > 0$ such that for $\mathbf s, \mathbf a, \mathbf b$ as in \eqref{sabe}, and  
\begin{equation} \label{xi-range-0} 
\text{ for } \xi \in \mathbb Z \; \; \text{ with } \; \; \mathtt s_1^{\mathtt b_1} \mid \xi \; \;  \text{ and } \; \; |\xi|^2 \leq \mathtt s_1^{\mathtt b_1} \mathtt s_2^{\mathtt a_2}, 
\end{equation} the following estimate holds:  
\begin{equation} \label{A2} \bigl|\mathfrak A(\xi; \mathtt N_2, \mathtt s_2, \mathtt a_2) \bigr| \leq C_0 \mathtt N_2^{-1} \leq C_0 \mathtt s_1^{\mathtt b_1} \mathtt s_2^{-\mathtt a_2}. \end{equation}  
\end{lemma} 
\begin{proof} 
The second inequality in \eqref{A2} is a consequence of the size bound on $\mathtt N_2$ given in \eqref{N2}. We therefore focus on proving the first inequality in \eqref{A2}.
\vskip0.1in
\noindent The range of $\xi$ given by the last inequality in \eqref{xi-range-0} ensures that 
\begin{equation} |\xi| \leq \mathtt s_1^{\frac{\mathtt b_1}{2}} \mathtt s_2^{\frac{\mathtt a_2}{2}} < \frac{1}{4}\mathtt s_2^{\mathtt a_2} \text{ by \eqref{sabe}}, \quad \text{ and hence } \quad \mathtt s_2^{\mathtt a_2} \nmid \xi. \label{range-of-xi-reduced} \end{equation} 
As a result, \eqref{def-A} provides the following expression for $\mathfrak A$: 
\[ \mathfrak A(\xi; \mathtt N_2, \mathtt s_2, \mathtt a_2) = \frac{1}{\mathtt N_2} \frac{\mathfrak n}{\mathfrak d}, \quad \text{ where } \quad \mathfrak n:= 1 - e(\xi \mathtt N_2 \mathtt s_2^{-\mathtt a_2}), \quad \mathfrak d:= 1 - e(\xi \mathtt s_2^{-\mathtt a_2}). \] 
The desired conclusion \eqref{A2} will follow from \eqref{N2} and the two inequalities 
\begin{equation} \label{nd} 
|\mathfrak n| = 2\bigl| \sin(\pi \xi \mathtt N_2 \mathtt s_2^{-\mathtt a_2} )\bigr| \leq C_0 |\xi| \mathtt s_2^{-\mathtt a_2}, \qquad |\mathfrak d| =2 \bigl| \sin(\pi \xi \mathtt s_2^{-\mathtt a_2})\bigr| \geq C_0^{-1} |\xi| \mathtt s_2^{-\mathtt a_2}.  
\end{equation} 
Both inequalities in \eqref{nd} will be proved using the well-known property of the sine function:
\begin{equation} \label{sine}  
C_0^{-1} |t| \leq |\sin(\pi t)| \leq C_0 |t| \quad \text{  for }  |t| \in \bigl[0, {1}/{2}\bigr].
\end{equation} 
The second estimate in \eqref{nd} is a direct consequence of the left inequality in \eqref{sine}, applied with $t = \xi \mathtt s_2^{-\mathtt a_2}$. The condition \eqref{range-of-xi-reduced} verifies that $t \in [0, 1/2]$, as required for \eqref{sine}. 
\vskip0.1in
\noindent We therefore focus on proving the first inequality in \eqref{nd}, which requires us to understand the fractional part of $\xi \mathtt N_2 \mathtt s_2^{-\mathtt a_2}$. Let us note from \eqref{N2} that 
\begin{equation}  \label{N2+1}
\mathtt N_2 + 1 = \lfloor \mathtt s_2^{\mathtt a_2} \mathtt s_1^{-\mathtt b_1} \rfloor = \mathtt s_2^{\mathtt a_2} \mathtt s_1^{-\mathtt b_1} - \delta, \; \text{ and hence } \; \mathtt N_2 + 1 = \mathtt s_2^{\mathtt a_2} \mathtt s_1^{-\mathtt b_1} - \delta, 
\end{equation} 
where $\delta \in [0, 1)$ is the fractional part of $\mathtt s_2^{\mathtt a_2} \mathtt s_1^{-\mathtt b_1}$.  It follows from \eqref{N2+1} that 
\[ \xi \mathtt N_2 \mathtt s_2^{-\mathtt a_2} = \xi \mathtt s_1^{- \mathtt b_1} - (1 + \delta) \xi \mathtt s_2^{-\mathtt a_2} \in \mathbb Z - (1 + \delta) \xi \mathtt s_2^{-\mathtt a_2} , \text{ since } \mathtt s_1^{\mathtt b_1} \mid \xi \text{ by assumption \eqref{xi-range-0}}.\] 
The  periodicity of the sine function, the range of $\xi$ given by \eqref{range-of-xi-reduced} and the right inequality in \eqref{sine} now imply the first estimate in \eqref{nd}, namely
\[ |\mathfrak n| = 2 |\sin \bigl(\pi (1 + \delta) \xi \mathtt s_2^{-\mathtt a_2}\bigr)| \leq C_0 |\xi| \mathtt s_2^{-\mathtt a_2}, \quad
\text{ since } \quad (1 + \delta) |\xi| \mathtt s_2^{-\mathtt a_2} \leq  2|\xi| \mathtt s_2^{-\mathtt a_2} < \frac{1}{2}. \] 
This completes the proof of \eqref{nd} and hence also the proof of the lemma.  
\end{proof} 
\begin{lemma} \label{H-lemma}
There exists an absolute constant $C_0 > 0$ with the following property. Let $\mathbf s, \mathbf a, \mathbf b$ be as in \eqref{sabe} and $\mathtt N_2$ as in \eqref{N2}. Then the function $\mathfrak H(\xi)$ given by 
\begin{align} 
&\mathfrak H(\xi) := \frac{1 - e(\xi \mathtt N_2 \mathtt s_2^{-\mathtt a_2})}{1 - e(\xi \mathtt s_1^{-\mathtt b_1})} \quad \text{ obeys } \quad |\mathfrak H (\xi) - 1| \leq C_0 |\xi| \mathtt s_1^{\mathtt b_1} \mathtt s_2^{-\mathtt a_2},  \label{H}  \\ &\text{ provided } \quad \xi \in \mathbb Z, \quad  \mathtt s_1^{\mathtt b_1} \nmid \xi, \quad \text{ and } \quad |\xi|^2 \leq \mathtt s_1^{\mathtt b_1} \mathtt s_2^{\mathtt a_2}. \label{xi-range} 
\end{align} 
\end{lemma}  
\begin{proof} 
For $\mathfrak H$ defined in \eqref{H}, the expression $\mathfrak H - 1$ can  be written as 
\begin{equation} \label{HND}
\mathfrak H(\xi) -  1= \frac{\mathfrak N}{\mathfrak D}, \; \text{ where } \; \mathfrak N := e(\xi \mathtt s_1^{-\mathtt b_1}) - e(\xi \mathtt N_2 \mathtt s_2^{-\mathtt a_2}) \; \text{ and } \; \mathfrak D := 1 - e(\xi \mathtt s_1^{-\mathtt b_1}). 
\end{equation}   
The numerator $\mathfrak N$ is estimated from above, as follows: 
\begin{equation} \label{N-est}  
|\mathfrak N| = \bigl| e \bigl(\xi( \mathtt s_1^{-\mathtt b_1} - \mathtt N_2 \mathtt s_2^{-\mathtt a_2}) \bigr) -1 \bigr| \leq C_0 |\xi| \bigl| \mathtt s_1^{-\mathtt b_1}  - \mathtt N_2 \mathtt s_2^{-\mathtt a_2} \bigr| \leq C_0 |\xi| \mathtt s_2^{-\mathtt a_2}. 
\end{equation}   
The first inequality in \eqref{N-est} follows from the Lipschitz property \eqref{Lipschitz} of the map $t \mapsto e(t)$ on $[-1,1]$. The second inequality is a consequence of the definition \eqref{N2} of $\mathtt N_2$, which says that 
\[ \mathtt s_2^{-\mathtt a_2} \leq \mathtt s_1^{-\mathtt b_1} - \mathtt N_2 \mathtt s_2^{-\mathtt a_2} \leq 2 \mathtt s_2^{-\mathtt a_2}. \]  The range of $\xi$ assumed in \eqref{xi-range} and the size conditions on $\mathtt s_1, \mathtt s_2$ specified in \eqref{sabe} ensure that 
\[ |\xi| \mathtt s_2^{-\mathtt a_2} \leq  \mathtt s_1^{\frac{\mathtt b_1}{2}} \mathtt s_2^{-\frac{\mathtt a_2}{2}} < 1, \]
which justifies this step. 
\vskip0.1in 
\noindent Let us turn now to the denominator $\mathfrak D$ in \eqref{HND}, whose absolute value 
\[ |\mathfrak D| = 2 |\sin(\pi \xi \mathtt s_1^{-\mathtt b_1})| = 2 |\sin \bigl(\pi \{\xi \mathtt s_1^{-\mathtt b_1} \} \bigr)| \]  we bound from below. Here $\{x\} \in [0, 1)$ denotes the fractional part of $x$. Since $\mathtt s_1^{\mathtt b_1} \nmid \xi$, we have $\xi \mathtt s_1^{-\mathtt b_1} \in \mathbb Z + \mathtt s_1^{-\mathtt b_1} \{1, 2, \ldots, \mathtt s_1^{\mathtt b_1}-1\}$. 
In other words, $\{\xi \mathtt s_1^{-\mathtt b_1}\} \in (0, 1)$ is a nonzero integer multiple of $\mathtt s_1^{-\mathtt b_1}$. In view of the monotone increasing nature of the function $t \mapsto |\sin(t)|$ on $[0, \frac{\pi}{2}]$ and its symmetry about $t = \frac{\pi}{2}$, we deduce   
\begin{equation} \label{D-est} 
|\mathfrak D| = 2|\sin(\pi \xi \mathtt s_1^{-\mathtt b_1})| =  2|\sin \bigl(\pi \{\xi \mathtt s_1^{-\mathtt b_1} \} \bigr)| \geq 2|\sin(\pi \mathtt s_1^{-\mathtt b_1})| \geq C_0^{-1} \mathtt s_1^{-\mathtt b_1}.
\end{equation} 
The last step above follows from the relation \eqref{sine} applied with $t = \mathtt s_1^{-\mathtt b_1}$. Combining \eqref{N-est}, \eqref{D-est} with \eqref{HND} leads to the conclusion \eqref{H}. 
\end{proof} 
\begin{corollary} \label{H-cor} 
There is an absolute constant $C_0 > 0$ for which the function $\mathfrak H$ in \eqref{H} obeys
\begin{equation} \label{s-new}
 |\mathfrak{H}(\xi)| \leq C_0, \; \text{ provided } \; \mathtt s_2^{\mathtt a_2} > \mathtt s_1^{3\mathtt b_1} \; \text{ and  $\xi$ obeys the requirements of \eqref{xi-range}}. 
\end{equation} 
\end{corollary} 
\begin{proof} 
Under the growth assumption \eqref{s-new} on $\mathbf s$, any $\xi$ in the range \eqref{xi-range} obeys 
\[ |\xi| \mathtt s_1^{\mathtt b_1} \mathtt s_2^{-\mathtt a_2} \leq \mathtt s_1^{\frac{3 \mathtt b_1}{2}} \mathtt s_2^{-\frac{\mathtt a_2}{2}} < 1. \]
Substituting this into the bound \eqref{H} of $|\mathfrak H - 1|$ establishes \eqref{s-new}, by the triangle inequality.    
\end{proof}
\begin{proposition} \label{prop-G} 
There exists an absolute constant $C_0 > 0$ as follows. For any choice of parameters $\mathbf s, \mathbf a, \mathbf b, \pmb{\varepsilon}$ obeying \eqref{sabe} and the stronger growth condition $\mathtt s_2^{\mathtt a_2} > \mathtt s_1^{3\mathtt b_1}$ as in \eqref{s-new}, the function $\mathfrak G$ defined in \eqref{def-G} satsfies the estimate: 
\begin{equation} \label{G-est}
|\mathfrak G(\xi)| \leq C_0 \left[ \varepsilon_1 + \mathtt s_1^{\mathtt a_1} |\xi|^{-1} + \mathtt s_1^{\mathtt b_1} \mathtt s_2^{-\mathtt a_2} \right]
\quad \text{ for all } \xi \in \mathbb Z, \; \; 2 \mathtt s_1^{\mathtt a_1} \leq |\xi| < \mathtt s_1^{\frac{\mathtt b_1}{2}}\mathtt s_2^{\frac{\mathtt a_2}{2}}.   
\end{equation} 
\end{proposition} 
\begin{proof} 
For brevity of notation, we will write $\mathfrak G = \mathfrak B_1 \times \mathfrak A_2$, where $\mathfrak B_1$ and $\mathfrak A_2$ denote the first and second factors of \eqref{def-G} respectively. Using the expressions of $\mathfrak A$ and $\mathfrak B$ from \eqref{def-A} and \eqref{def-B}, and the trivial bound $|\mathfrak B_1^{\ast}|, |\mathfrak A_2| \leq 1$ from \eqref{B*C-trivial-bound},  we find that
\begin{align}
|\mathfrak G(\xi)| &= \bigl| \varepsilon_1 \mathfrak A_2(\xi) \mathfrak B_1^{\ast}(\xi) + (1 - \varepsilon_1) \mathfrak A_2(\xi) \mathfrak C_1(\xi) \bigr| \nonumber \\ 
&\leq \varepsilon_1 + |\mathfrak A_2(\xi)| \times |\mathfrak C_1(\xi)|. \label{G2}
\end{align}  
As in the proof of Lemma \ref{single-scale-product-lemma}, the argument henceforth splits into several cases depending on the divisibility of $\xi$ by powers of $\mathtt s_1$. 
\vskip0.1in 
\noindent {\em{Case 1: }} Suppose first $\mathtt s_1^{\mathtt a_1} \mid \xi$ but $\mathtt s_1^{\mathtt b_1} \nmid \xi$. It follows from Lemma \ref{BC-lemma} that in this case $\mathfrak C_1(\xi) = 0$. As a result, \eqref{G2} yields $|\mathfrak G(\xi)| \leq \varepsilon_1$, which is in fact a stronger estimate than the one claimed in \eqref{G-est}. 
\vskip0.1in
\noindent {\em{Case 2: }} Next assume $\mathtt s_1^{\mathtt b_1} \mid \xi$. Then Lemma \ref{A2-lemma} applies to $\mathfrak A_2$, with \eqref{A2} dictating that 
\[ |\mathfrak G| = |\mathfrak B_1| \times |\mathfrak A_2| \leq |\mathfrak A_2| \leq C_0 \mathtt s_1^{\mathtt b_1} \mathtt s_2^{-\mathtt a_2}. \]
This too is stronger than the estimate claimed in \eqref{G-est}.   
\vskip0.1in
\noindent {\em{Case 3: }} Finally suppose that $\mathtt s_1^{\mathtt a_1} \nmid \xi$. Let us consider the product $\mathfrak A_2 \mathfrak C_1$ appearing in \eqref{G2}. It follows from \eqref{def-A} and Lemma \ref{BC-lemma} that 
\begin{align}  |\mathfrak A_2(\xi) \mathfrak C_1(\xi)| &= \frac{\mathtt s_1^{\mathtt a_1-\mathtt b_1}}{\mathtt N_2} \left| \frac{1 - e(\xi \mathtt N_2 \mathtt s_2^{-\mathtt a_2})}{1 - e(\xi \mathtt s_2^{-\mathtt a_2})} \times \frac{1 - e(\xi \mathtt s_1^{-\mathtt a_1})}{1 - e(\xi \mathtt s_1^{-\mathtt b_1})} \right| \nonumber \\ 
&\leq C_0 \frac{\mathtt s_1^{\mathtt a_1}}{\mathtt s_2^{\mathtt a_2}} |\mathfrak H(\xi)| \times |\sin(\pi \xi \mathtt s_2^{-\mathtt a_2})|^{-1} 
\leq C_0 \mathtt s_1^{\mathtt a_1} |\xi|^{-1}. \label{case2} 
\end{align} 
Here $\mathfrak H(\xi)$ is the function defined in \eqref{H}. The first inequality of the display above follows from the trivial bound 
\[ |1 - e(\xi \mathtt s_1^{-\mathtt a_1})| \leq 2 \quad \text{ and the estimate } \quad \mathtt N_2 \geq \frac{1}{2} \mathtt s_2^{\mathtt a_2} \mathtt s_1^{-\mathtt b_1}, \] which in turn is a consequence of \eqref{N2}. The last inequality in \eqref{case2} involves two separate estimates: the first factor $\mathfrak H$ is estimated using the bound \eqref{s-new} from Corollary \ref{H-cor}. The assumption $\mathtt s_1^{\mathtt a_1} \nmid \xi$ ensures $\mathtt s_1^{\mathtt b_1} \nmid \xi$ as well; further, the choice of $\xi$ in \eqref{G-est} meets the requirements in \eqref{xi-range}, permitting the application of the corollary. The second factor $\sin(\pi \xi \mathtt s_2^{-\mathtt a_2})$ is estimated using the lower bound \eqref{sine} on the sine function, based on the observation that any $\xi$ in the range given in \eqref{G-est} obeys $|\xi| < \frac{1}{2} \mathtt s_2^{\mathtt a_2}$. Substituting the estimate in \eqref{case2} into \eqref{G2} yields an upper bound for $|\mathfrak G|$ that is consistent with \eqref{G-est}. This completes the analysis of all three cases, and concludes the proof.   
\end{proof}

\section{Fourier coefficients of $\mu$} \label{mu-Rajchman-section}
Backed by the oscillatory estimates from Section \ref{exp-sum-section}, we are ready to delve into the study of Fourier coefficients of $\mu$.  The weak convergence property \eqref{mu-weak-limit} implies that for all $\xi \in \mathbb Z$, 
\begin{equation} \label{Fourier} 
\widehat{\mu}(\xi) := \int_{0}^1 e(x\xi) d\mu(x) = \lim_{\mathtt M \rightarrow \infty} \widehat{\varphi}_{\mathtt M}(\xi) = \lim_{\mathtt M \rightarrow \infty} \int_{0}^{1} \varphi_{\mathtt M}(x) e(x\xi) \, dx,\end{equation}   
where $\varphi_{\mathtt M}$ is the probability density associated with the $\mathtt M^{\text{th}}$ step of the construction described in Section \ref{iteration-section}. We therefore seek to understand the behaviour of $\widehat{\varphi}_{\mathtt M}(\xi)$ for all sufficiently large $\mathtt M$. 
\subsection{Reduction to products of $\mathfrak A$ and $\mathfrak B$}  
Our first task is to derive an explicit formula for $\widehat{\varphi}_{\mathtt M}$ for $\mathtt M \geq 1$ and reduce it to a form amenable to the analysis presented in Section \ref{exp-sum-section}. Let us recall from \eqref{density} that 
\begin{align}  
&\varphi_{\mathtt M}(x) = \mathtt s_{\mathtt M}^{\mathtt b_{\mathtt M}} \sum_{\mathbf i_{\mathtt M}\in \mathbb I_{\mathtt M}}  \mathtt w(\mathbf i_{\mathtt M}) 1_{\mathtt I(\mathbf i_{\mathtt M})}, \; \text{ and therefore } \nonumber \\ \widehat{\varphi}_{\mathtt M}(\xi) =  &\widehat{\mathtt 1} \bigl(\xi \mathtt s_{\mathtt M}^{-\mathtt b_{\mathtt M}} \bigr) \Gamma_{\mathtt M}(\xi) \quad \text{ where } \quad
\Gamma_{\mathtt M}(\xi) := \sum_{\mathbf i_{\mathtt M} \in \mathbb I_{\mathtt M}} \mathtt w(\mathbf i_{\mathtt M})  e \bigl(\xi \alpha(\mathbf i_{\mathtt M})\bigr), \quad \xi \in \mathbb Z. \label{Fourier2} 
\end{align}  
The function $\Gamma_{\mathtt M}$ is not new. Since $(E_{\mathtt M}, \varphi_{\mathtt M})$ is the output of an elementary operation $\mathscr{O}$, the exponential sum $\Gamma_{\mathtt M}$ is in fact the same as $\Upsilon_1$ appearing in \eqref{Yj}, for an appropriate choice of parameters. Specifically, in the notation of \eqref{O-def}, \eqref{Y-0} and \eqref{Y}, 
\begin{equation} \label{big-and-little-gamma}
\Gamma_{\mathtt M}(\xi) = \Upsilon_1(\xi; \pmb{\gamma}_{\mathtt M}) \quad \text{ where } \quad \pmb{\gamma}_{\mathtt M} := \bigl(E_{\mathtt M-1}, \varphi_{\mathtt M-1}; \mathtt s_{\mathtt M-1}^{\mathtt b_{\mathtt M-1}}, \mathtt s_{\mathtt M}, \mathtt a_{\mathtt M}, \mathtt b_{\mathtt M}, \mathscr{D}_{\mathtt M}, \varepsilon_{\mathtt M}\bigr).
\end{equation}  
In \eqref{Fourier2}, $\alpha(\mathbf i_{\mathtt M})$ and $\mathtt w(\mathbf i_{\mathtt M})$ denote respectively the left endpoint and the mass of the basic interval $\mathtt I(\mathbf i_{\mathtt M})$, as has been described in the indexing scheme in Sections \ref{indexing-section-1} and \ref{indexing-section-2}. Both these quantities are described via certain recursion formulae, given by \eqref{relation-between-endpoints} and \eqref{wt-sum}. Applying these relations repeatedly leads to the following expressions: for every $m < \mathtt M$, 
\begin{align}  
&\alpha(\mathbf i_{\mathtt M}) = \alpha(\mathbf i_m) + \sum_{r=m+1}^{\mathtt M} \Bigl[ \eta(\mathbf i_{r-1}) + \frac{k_r}{\mathtt s_r^{\mathtt a_r}} + \frac{\ell_r}{\mathtt s_r^{\mathtt b_r}} \Bigr], \label{alpha} \\
&\mathtt w(\mathbf i_{\mathtt M}) = \mathtt w(\mathbf i_m) \times  \prod_{r=m+1}^{\mathtt M} \frac{\vartheta_{\ell_r}(\varepsilon_r, \mathtt s_r, \mathtt a_r, \mathtt b_r)}{\mathtt N_r}, \quad \text{ where } \label{w} \\ 
&\mathbf i_{\mathtt M} = (\mathbf i_{m}, k_{m+1}, \ell_{m+1}, \ldots, k_{\mathtt M}, \ell_{\mathtt M}) \in \mathbb I_{\mathtt M} = \mathbb I_{m} \times \prod_{r=m+1}^{\mathtt M} \mathbb J_r. \label{index-product}
\end{align}  
The quantities $\mathtt N_r$ and $\vartheta_{\ell_r}$ in \eqref{w} are defined by \eqref{N} and \eqref{choice-of-theta} respectively. 
Substituting \eqref{alpha} and \eqref{w} into \eqref{Fourier2} generates the following representation of $\Gamma_{\mathtt M}$ for every $m < \mathtt M$: 
\begin{align}
\Gamma_{\mathtt M}(\xi) &= \sum_{\mathbf i_{\mathtt M} \in \mathbb I_{\mathtt M}} \mathtt w(\mathbf i_{\mathtt M})  e\Bigl(\xi \alpha(\mathbf i_{m}) + \xi  \sum_{r=m+1}^{\mathtt M}  \Bigl[ \eta(\mathbf i_{r-1}) + \frac{k_r}{\mathtt s_r^{\mathtt a_r}} + \frac{\ell_r}{\mathtt s_r^{\mathtt b_r}} \Bigr] \Bigr) \label{YM0} \\ 
&= \sum_{\mathbf i_{\mathtt M} \in \mathbb I_{\mathtt M}} \mathtt w(\mathbf i_{m}) \Bigl[ \prod_{r=m+1}^{\mathtt M} \frac{\vartheta_{\ell_r}}{\mathtt N_r}  \Bigr] \times e\Bigl(\xi \alpha(\mathbf i_{m}) + \xi  \sum_{r=m+1}^{\mathtt M}  \Bigl[ \eta(\mathbf i_{r-1}) + \frac{k_r}{\mathtt s_r^{\mathtt a_r}} + \frac{\ell_r}{\mathtt s_r^{\mathtt b_r}} \Bigr] \Bigr). \label{not-a-product} 
\end{align} 
Both the index set $\mathbb I_{\mathtt M}$ and the weight $\mathtt w(\mathbf i_{\mathtt M})$ appearing in $\Gamma_{\mathtt M}$ enjoy a product structure, as can be seen from \eqref{index-product} and \eqref{w} respectively. In view of this, one would ideally like $\Gamma_{\mathtt M}(\xi)$ to factorize into a product of $(\mathtt M - m)$ factors, each corresponding to an elementary operation in the construction of $\mu$. However, a preliminary examination of \eqref{not-a-product} reveals that the presence of the spill-overs $\eta(\mathbf i_{r-1})$ prevent such a factorization. In order to simplify the estimation, we define an auxiliary function $\Delta_{\mathtt M}$ that omits these terms and therefore allows the necessary factorization: 
\begin{align} 
\Delta_{\mathtt M} \bigl(\xi;m\bigr)  &:=  \sum_{\mathbf i_{\mathtt M} \in \mathbb I_{\mathtt M}} \mathtt w(\mathbf i_{\mathtt M})  e\Bigl(\xi \alpha(\mathbf i_{m}) + \xi  \sum_{r=m+1}^{\mathtt M}  \Bigl[\frac{k_r}{\mathtt s_r^{\mathtt a_r}} + \frac{\ell_r}{\mathtt s_r^{\mathtt b_r}} \Bigr] \Bigr) \label{PsiM0} \\ 
&=\sum_{\mathbf i_{\mathtt M} \in \mathbb I_{\mathtt M}} \mathtt w(\mathbf i_{m})  e \bigl(\xi \alpha(\mathbf i_{m}) \bigr) \prod_{r=m+1}^{\mathtt M} \frac{\vartheta_{\ell_r}}{\mathtt N_r} e \Bigl( \xi \Bigl[  \frac{k_r}{\mathtt s_r^{\mathtt a_r}} + \frac{\ell_r}{\mathtt s_r^{\mathtt b_r}} \Bigr] \Bigr) \nonumber \\ 
 &= \sum_{\mathbf i_{m} \in \mathbb I_{m}} \mathtt w(\mathbf i_{m})  e \bigl(\xi \alpha(\mathbf i_{m}) \bigr) \prod_{r=m+1}^{\mathtt M} \mathfrak A_r \times \mathfrak B_r = \Gamma_m(\xi) \prod_{r=m+1}^{\mathtt M} \mathfrak A_r \times \mathfrak B_r \nonumber \\ &= \Lambda_m(\xi) \prod_{r=m}^{\mathtt M} \mathfrak A_r \times \mathfrak B_r.  \label{Psi-factors}
\end{align} 
Here $\mathfrak A_r$ and $\mathfrak B_r$ are the functions $\mathfrak A$ and $\mathfrak B$ defined in \eqref{def-A} and \eqref{def-B} respectively, with parameter choices depending on $r$:
\begin{equation} \label{ArBr}  
\mathfrak A_r(\xi) := \mathfrak A(\xi; \mathtt N_r, \mathtt s_r, \mathtt a_r), \qquad \mathfrak B_r(\xi) := \mathfrak B(\xi; \mathtt s_r, \mathtt a_r, \mathtt b_r, \mathscr{D}_r, \varepsilon_r).
\end{equation}
The last expression in \eqref{Psi-factors} follows from the factorization \eqref{factorization} of $\Gamma_m = \Upsilon_1(\cdot; \pmb{\gamma}_m)$, with $\Lambda_m$ being the analogue of $\Omega_1$ in that decomposition, and $\pmb{\gamma}_m$ the choice of parameters as in \eqref{big-and-little-gamma}: 
\begin{equation} \Lambda_m(\xi) := \Omega_1(\xi; \pmb{\gamma}_m) := \sum_{\mathbf i \in \mathbb I_{m-1}} \mathtt w(\mathbf i) e \bigl(\xi (\alpha(\mathbf i) + \eta(\mathbf i)) \bigr). \label{def-Lambda} \end{equation}  
In Section \ref{exp-sum-section}, we have already derived estimates for $\mathfrak A \mathfrak B$ at single and multiple scales. The goal is to
put them together to obtain estimates for $\Gamma_{\mathtt M}$ and from there to $\widehat{\varphi}_{\mathtt M}$ and hence $\widehat{\mu}$. 
\subsection{Estimating $\Gamma_{\mathtt M}$ and $\Delta_{\mathtt M}$} 
We first ascertain that $\Delta_{\mathtt M}$ is a good approximation of $\Gamma_{\mathtt M}$ for certain ranges of $\xi$. 
\begin{lemma} \label{approximation-lemma} 
For any $\xi \in \mathbb Z$ and any $m < \mathtt M$, $m, \mathtt M \in \mathbb N$, 
\begin{equation} \label{Y-Psi-error}
|\Gamma_{\mathtt M}(\xi) - \Delta_{\mathtt M}(\xi; m)| \leq C_0 \min \bigl[1, |\xi| \mathtt s_{m+1}^{-\mathtt a_{m+1}} \bigr]. 
\end{equation} 
\end{lemma} 
\begin{proof}
The argument is similar to the one used in Lemma \ref{error-lemma}. By definition, both $\Gamma_{\mathtt M}$ and $\Delta_{\mathtt M}(\cdot; m)$ are sums over the same index set $\mathbb I_{\mathtt M}$; moreover, their summands are almost identical, except for the presence (or omission) of the multiplicative factor involving $\eta(\mathbf i_{r-1})$. Thus comparing \eqref{YM0} and \eqref{PsiM0}, we find that their difference can be estimated as follows,    
\begin{align*} 
|\Gamma_{\mathtt M}(\xi) - \Delta_{\mathtt M}(\xi; m)| 
&= \Bigl| \sum_{\mathbf i_{\mathtt M} \in \mathbb I_{\mathtt M}} 
\mathtt w(\mathbf i_{\mathtt M}) e\Bigl(\xi \Bigl[\alpha(\mathbf i_{m}) +  \sum_{r = m+1}^{\mathtt M} \Bigl( \frac{k_r}{\mathtt s_r^{\mathtt a_r}} + \frac{\ell_r}{\mathtt s_r^{\mathtt b_r}} \Bigr) \Bigr] \Bigr) \\ &\hskip1.8in \times \Bigl[ e\Bigl( \xi \sum_{r=m+1}^{\mathtt M} \eta(\mathbf i_{r-1}) \Bigr) - 1\Bigr] \Bigr| \\ 
&\leq C_0 \sum_{\mathbf i_{\mathtt M} \in \mathbb I_{\mathtt M}} \mathtt w(\mathbf i_{\mathtt M}) \Bigl|  \xi \sum_{r=m+1}^{\mathtt M} \eta(\mathbf i_{r-1}) \Bigr| \leq C_0 \sum_{\mathbf i_{\mathtt M} \in \mathbb I_{\mathtt M}} \mathtt w(\mathbf i_{\mathtt M}) \Bigl[ |\xi| \sum_{r=m+1}^{\mathtt M} \mathtt s_{r}^{-\mathtt a_r} \Bigr] \\
&\leq C_0 |\xi| \sum_{r=m+1}^{\mathtt M} \mathtt s_{r}^{-\mathtt a_r} \leq C_0 |\xi| \mathtt s_{m+1}^{-\mathtt a_{m+1}} \sum_{r=0}^{\mathtt M-m-1} 16^{-r} \leq  C_0 |\xi| \mathtt s_{m+1}^{-\mathtt a_{m+1}}.  
\end{align*} 
The first inequality above uses the Lipschitz property \eqref{Lipschitz} of the map $y \mapsto e(y)$. The second one uses the fact $\eta(\mathbf i_{r-1}) \in [0, \mathtt s_{r}^{-\mathtt a_r})$, stated in \eqref{etaim}. The sum of $\{\mathtt w(\mathbf i_{\mathtt M}) : \mathbf i_{\mathtt M} \in \mathbb I_{\mathtt M} \}$ is 1, which leads to the third inequality. The geometric nature of the sum established in the fourth inequality follows from the growth assumption in \eqref{sabe}, which provides 
\[ \mathtt s_{r+1}^{\mathtt a_{r+1}} > 16 \mathtt s_r^{\mathtt b_r} > 16 \mathtt s_r^{\mathtt a_r} \text{ for all } r \geq 1. \] 
This completes the proof.  
\end{proof}
\noindent In view of \eqref{Y-Psi-error},  the onus of estimating $\Gamma_{\mathtt M}$ shifts to $\Delta_{\mathtt M}(\cdot, m)$. The product form of $\Delta_{\mathtt M}$ offers a variety of possible upper bounds. We list a few of them here.  
\begin{lemma} \label{Psi-bound-lemma} 
For $\Delta_{\mathtt M}(\xi; m)$ as in \eqref{PsiM0}, one has the estimate 
\begin{equation}  \label{Psi-bound}
|\Delta_{\mathtt M}(\xi; m)| \leq \min \bigl\{|\mathfrak F_r(\xi)|, |\mathfrak G_r(\xi)| : r = m, \ldots, \mathtt M \bigr\}, 
\end{equation} 
where 
\begin{equation} 
\mathfrak F_r := \mathfrak A_r \mathfrak B_r \quad \text{ and } \quad \mathfrak G_r := \mathfrak B_r \mathfrak A_{r+1}
\end{equation} are the analogues of the functions $\mathfrak F$ and $\mathfrak G$ defined in \eqref{AB} and \eqref{def-G} respectively, with $\mathfrak A_r, \mathfrak B_r$ as in \eqref{ArBr}. 
\end{lemma} 
\begin{proof} 
The estimate \eqref{Psi-bound} is a direct consequence of the factorization \eqref{Psi-factors}. For each $r \in \{m, \ldots, \mathtt M\}$, this factorization expresses $\Delta_{\mathtt M}(\xi; m)$ as a product: 
\[ \Delta_{\mathtt M}(\xi; m) = \mathfrak F_r \mathfrak F^{\ast}_r = \mathfrak G_r \mathfrak G_r^{\ast}, \qquad \mathfrak F^{\ast}_r := \Lambda_m \prod_{\begin{subarray}{c}r' = m \\ r' \ne r \end{subarray}}^{\mathtt M} \mathfrak A_{r'} \mathfrak B_{r'}, \quad \mathfrak G^{\ast}_r := \Lambda_m \prod_{\begin{subarray}{c}r' = m \\ r' \ne r+1 \end{subarray}}^{\mathtt M} \mathfrak A_{r'} \prod_{\begin{subarray}{c}r' = m \\ r' \ne r\end{subarray}}^{\mathtt M}\mathfrak B_{r'}. \] 
Each one of the functions $\mathfrak F_r^{\ast}$, $\mathfrak G_r^{\ast}$ is a product of factors in its own right, 
where each factor has absolute value at most 1, by virtue of the trivial bounds recorded in \eqref{B*C-trivial-bound}. Therefore $|\mathfrak F^{\ast}_r|, |\mathfrak G^{\ast}_r| \leq 1$, resulting in \eqref{Psi-bound}.   
\end{proof} 
\section{Rajchman property of skewed measures} \label{Rajchman-section} 
With the exponential sum estimates gathered in Section \ref{exp-sum-section} and \ref{mu-Rajchman-section}, we are now ready to show that the measure $\mu = \mu(\pmb{\Pi})$ given by Proposition \ref{def-mu-prop} obeys the Fourier estimate \eqref{pre-Rajchman}, for an admissible choice of parameters $\pmb{\Pi}$. Proving this estimate is the main objective of this section.
\vskip0.1in 
\noindent Let us recall the definition $\mathtt Q_m = \mathtt s_m^{(\mathtt a_m + \mathtt b_m)/2}$ from \eqref{Qtau} and fix a frequency $\xi \in \mathbb Z$ with $|\xi| \in (\mathtt Q_m, \mathtt Q_{m+1}]$. Not surprisingly, in view of the weak convergence of $\varphi_{\mathtt M}$ to $\mu$, the estimate \eqref{pre-Rajchman} will follow from a similar estimate on $\widehat{\varphi}_{\mathtt M}$ that depends on $m$ but is independent of $\mathtt M$. The formula \eqref{Fourier2} of $\widehat{\varphi}_{\mathtt M}$, combined with the bound \eqref{1-hat} on $\widehat{\mathtt 1}$, gives 
 \begin{equation} \label{phiM-hat-bound} 
 |\widehat{\varphi}_{\mathtt M}(\xi)| \leq |\Gamma_{\mathtt M}(\xi)| \leq   \bigl|\Gamma_{\mathtt M}(\xi) - \Delta_{\mathtt M}(\xi; n)\bigr| + \bigl|\Delta_{\mathtt M}(\xi;n) \bigr| \quad \text{ for any index } n \in \mathbb N. 
 \end{equation} 
 The two summands above will be controlled using the functions obtained in Lemma \ref{approximation-lemma} and \ref{Psi-bound-lemma} for a suitable choice of $n$ depending on $\xi$. In particular, functions of the form $\mathfrak F, \mathfrak G$ dominate $\Delta_{\mathtt M}$. These in turn will be estimated using corresponding bounds from Section \ref{exp-sum-section} -  Lemma \ref{single-scale-product-lemma} and Proposition \ref{prop-G} respectively.
 \vskip0.1in 
 \noindent Let us make this precise. In Lemmas \ref{Xi1-lemma} and \ref{Xi2-lemma} below, let $\pmb{\Pi}$ be a choice of parameters as in Proposition \ref{Rajchman-prop}. We decompose the range of $|\xi|$ into two parts: 
\begin{align}
&\bigl(\mathtt Q_{m}, \mathtt Q_{m+1} \bigr] = \Xi_1(m) \sqcup \Xi_2(m), \quad\text{ where } \nonumber \\ \Xi_1(m) := &\Bigl(\mathtt Q_{m}, \mathtt s_{m}^{{\mathtt b_{m}}/{2}} \mathtt s_{m+1}^{{\mathtt a_{m+1}}/{2}} \Bigr], \quad \Xi_2(m) := \Bigl(\mathtt s_{m}^{{\mathtt b_{m}}/{2}} \mathtt s_{m+1}^{{\mathtt a_{m+1}}/{2}}, \mathtt Q_{m+1} \Bigr].  \label{Xi-12}
\end{align}
We will estimate $\widehat{\varphi}_{\mathtt M}$ separately in each range. 
\subsection{Frequencies in $\Xi_1$} 
\begin{lemma} \label{Xi1-lemma}
There exists a constant $C = C(\pmb{\Pi}) > 0$ such that
\begin{equation} 
|\widehat{\varphi}_{\mathtt M}(\xi)| \leq   C \Bigl[ \varepsilon_{m} + \tau_{m} + \mathtt N_{m+1}^{-{1}/{2}}\Bigr] \quad  \text{ for all } \;  |\xi| \in \Xi_1(m) \cap \mathbb N, \; \mathtt M \geq m \geq 1,  \label{mu-hat-case1}
\end{equation} 
where $\tau_m = \mathtt s_m^{(\mathtt a_m-\mathtt b_m)/2}$ is defined in \eqref{Qtau}. 
\end{lemma}
\begin{proof} 
We apply the inequality \eqref{phiM-hat-bound} with the choice of index $n=m$, and use $|\mathfrak G_m|$ as an upper bound of $|\Delta_{\mathtt M}(\xi;m)|$ for $1 \leq m \leq \mathtt M$. This is permissible in view of \eqref{Psi-bound} from Lemma \ref{Psi-bound-lemma}. The error bound \eqref{Y-Psi-error} from Lemma \ref{approximation-lemma} provides an estimate for the difference $\Gamma_{\mathtt M} - \Delta_{\mathtt M}$ in \eqref{phiM-hat-bound}. Putting all of this together, we arrive at  
\begin{align} 
|\widehat{\varphi}_{\mathtt M}(\xi)| 
&\leq  \bigl|\Gamma_{\mathtt M}(\xi) - \Delta_{\mathtt M}(\xi; m)\bigr| + |\mathfrak G_{m}(\xi)| \nonumber \\ 
&\leq C \Bigl[ |\xi| \mathtt s_{m+1}^{-\mathtt a_{m+1}} + \varepsilon_{m} + \mathtt s_{m}^{\mathtt a_{m}} |\xi|^{-1} + \mathtt s_{m}^{\mathtt b_{m}} \mathtt s_{m+1}^{-\mathtt a_{m+1}} \Bigr]. \label{mu-hat-case-0} 
\end{align} 
The last step \eqref{mu-hat-case-0} is a consequence of the estimate \eqref{G-est} of $\mathfrak G$, since
 \[ \mathfrak G_m(\cdot):= \mathfrak G(\cdot; \mathbf s, \mathbf a, \mathbf b, \pmb{\varepsilon})  \; \; \text{ where } \; \; \mathbf s = (\mathtt s_{m -1}, \mathtt s_{m}, \mathtt s_{m+1}), \;\; \mathbf a = (\mathtt a_{m -1}, \mathtt a_{m}, \mathtt a_{m+1}), 
 \] 
and similarly for $\mathbf b, \pmb{\varepsilon}$. 
\vskip0.1in
\noindent Let us estimate each summand in \eqref{mu-hat-case-0} separately. The last term is bounded as 
 \begin{equation} \mathtt s_{m}^{\mathtt b_{m}} \mathtt s_{m+1}^{-\mathtt a_{m+1}} \leq C_0 \mathtt N_{m+1}^{-1}, \; \text{ a consequence of the defining property of $\mathtt N_{m}$ from \eqref{Qtau}}. \label{ineq-1} \end{equation} 
Estimates specific to  $|\xi| \in \Xi_1(m)$ lead to upper bounds for two of the other summands in \eqref{mu-hat-case-0}:  
\begin{equation} |\xi| \mathtt s_{m+1}^{-\mathtt a_{m+1}} \leq \mathtt s_{m}^{\mathtt b_{m}/2}  \mathtt s_{m+1}^{-\mathtt a_{m+1}/2} \leq C_0 \mathtt N_{m+1}^{-\frac{1}{2}} \quad \text{ and } \quad \mathtt s_{m}^{\mathtt a_{m}} |\xi|^{-1} \leq \mathtt s_{m}^{\mathtt a_{m}} \mathtt Q_{m}^{-1} = \tau_{m}. \label{ineq-23}  \end{equation} 
The first inequality above uses the upper bound on $|\xi|$ offered by $\Xi_1(m)$ in \eqref{Xi-12}; the second inequality utilizes the lower one.  Inserting the estimates from \eqref{ineq-1} and \eqref{ineq-23} into \eqref{mu-hat-case-0} leads to \eqref{mu-hat-case1}.
\end{proof} 
\subsection{Frequencies in $\Xi_2$} 
\begin{lemma} \label{Xi2-lemma} 
There exists a constant $C = C(\pmb{\Pi}) > 0$ such that
\begin{equation} 
|\widehat{\varphi}_{\mathtt M}(\xi)| \leq C \Bigl[ \tau_{m+1} + \varepsilon_{m+1} + \mathtt N_{m+1}^{-\frac{1}{2}} \Bigr]  \quad  \text{ for all } \;  |\xi| \in \Xi_2(m) \cap \mathbb N, \; \mathtt M \geq m \geq 1,  \label{mu-hat-case-2}
\end{equation} 
where $\tau_m = \mathtt s_m^{(\mathtt a_m-\mathtt b_m)/2}$ is defined in \eqref{Qtau}. 
\end{lemma} 
\begin{proof} 
 The argument follows the same broad strokes as Lemma \ref{Xi1-lemma}. The main distinctions are that this time we use the index $n = m+1$ in  the inequality \eqref{phiM-hat-bound}, and then use $\mathfrak F_{m+1}$ to bound $\Delta_{\mathtt M}(\cdot; m+1)$ (instead of $\mathfrak G_m$ in Lemma \ref{Xi1-lemma}).  
Since \[ \mathfrak F_m(\xi) := \mathfrak F(\xi; \mathtt N_m, \mathtt s_m, \mathtt a_m, \mathtt b_m) \quad \text{ in the notation of \eqref{AB}}, \] 
the estimate \eqref{F-est} derived for $\mathfrak F$ in Lemma \ref{single-scale-product-lemma} is used to bound $\mathfrak F_{m+1}$, with 
\[\mathtt s = \mathtt s_{m+1}, \quad \mathtt a = \mathtt a_{m+1}, \quad \mathtt b = \mathtt b_{m+1} \quad \text{ and } \quad \varepsilon = \varepsilon_{m+1}. \] 
The corresponding steps are therefore 
\begin{align} 
|\widehat{\varphi}_{\mathtt M}(\xi)| 
&\leq   \bigl|\Gamma_{\mathtt M}(\xi) - \Delta_{\mathtt M}(\xi; m+1)\bigr| + \bigl|\Delta_{\mathtt M}(\xi; m+1) \bigr| \nonumber 
\\ &\leq  \bigl|\Gamma_{\mathtt M}(\xi) - \Delta_{\mathtt M}(\xi; m+1)\bigr| + |\mathfrak F_{m+1}(\xi)| \nonumber \\ 
&\leq  C \Bigl[ |\xi| \mathtt s_{m+2}^{-\mathtt a_{m+2}} + \varepsilon_{m+1} + {\mathtt s_{m+1}^{\mathtt a_{m+1} - \mathtt b_{m+1}}}{\mathtt N_{m+1}^{-1}} \bigl|\sin(\pi \xi \mathtt s_{m + 1}^{-\mathtt b_{m+1}}) \bigr|^{-1} \Bigr] \nonumber \\ 
&\leq C \Bigl[ |\xi| \mathtt s_{m+2}^{-\mathtt a_{m+2}} + \varepsilon_{m+1} + \mathtt s_{m}^{\mathtt b_{m}} |\xi|^{-1} \Bigr]. \label{last-step-X2} 
\end{align} 
The last term in \eqref{last-step-X2} is obtained from its counterpart in the preceding step using the definition \eqref{Qtau} of $\mathtt N_{m+1}$, and the lower bound \eqref{sine} on the sine function.  The latter is applicable since 
\[ |\xi| \leq \mathtt Q_{m+1} \leq \frac{1}{2}\mathtt s_{m+1}^{\mathtt b_{m+1}} \quad \text{ for }  |\xi| \in \Xi_2(m). \]   
\vskip0.1in
\noindent As in Lemma \ref{Xi1-lemma}, estimates specific to $|\xi| \in \Xi_2(m)$ are needed for further simplification of \eqref{last-step-X2}. For instance, the bounds on $|\xi| \in  \Xi_2(m)$ given by \eqref{Xi-12} and the growth condition \eqref{ambm-take-2} result in the inequalities 
\begin{equation} \label{Xi2-estimates} |\xi| \mathtt s_{m+2}^{-\mathtt a_{m+2}} \leq \mathtt Q_{m+1} \mathtt s_{m+2}^{-\mathtt a_{m+2}} \leq \tau_{m+1}, \quad \text{ whereas } \quad \mathtt s_{m}^{\mathtt b_{m}} |\xi|^{-1} \leq \mathtt s_{m}^{\mathtt b_{m}/2} \mathtt s_{m+1}^{-\mathtt a_{m+1}/2} \leq C \mathtt N_{m+1}^{-\frac{1}{2}}. 
\end{equation}    
Inserting \eqref{Xi2-estimates} into \eqref{last-step-X2} results in the desired estimate \eqref{mu-hat-case-2}, completing the proof.   
\end{proof} 
\subsection{Proof of Proposition \ref{Rajchman-prop}}    
The bulk of the work has already been carried out in Lemmas \ref{Xi1-lemma} and \ref{Xi2-lemma}; we simply put them together. Since every $|\xi| \in (\mathtt Q_m, \mathtt Q_{m+1}]$ must fall either in $\Xi_1(m)$ or $\Xi_2(m)$, one or the other of the two inequalities \eqref{mu-hat-case1} and \eqref{mu-hat-case-2} must hold for $\widehat{\varphi}_{\mathtt M}(\xi)$ for every $\mathtt M > m$. The right hand sides of both these inequalities are dominated by  the upper bound in \eqref{pre-Rajchman}. Lemmas \ref{Xi1-lemma} and \ref{Xi2-lemma} therefore jointly imply that 
\[ |\widehat{\varphi}_{\mathtt M}(\xi)| \leq C \bigl[ \varepsilon_m + \varepsilon_{m+1} + \tau_m + \tau_{m+1} + \mathtt N_{m+1}^{-\frac{1}{2}}\bigr] \text{ for every $\mathtt M > m$.} \]  Since this estimate is uniform in $\mathtt M$, letting $\mathtt M \rightarrow \infty$ recovers the same bound for $\widehat{\mu}(\xi)$, in view of \eqref{Fourier}. The proof of Proposition \ref{Rajchman-prop} is complete.  
\qed

\section{Appendix} \label{Appendix}
\subsection{Absolutely non-normal numbers} \label{abs-non-normal-section}
The Fourier estimate \eqref{pre-Rajchman}  for a skewed measure $\mu$ depends on quantities like $\mathtt a_m$ and $\mathtt b_m$, which in turn depend on the ordered sequence $\mathcal S$ and are open to choice subject to certain specifications. The decay of $\widehat{\mu}$ thus depends on the ordering of bases. However, once the collections $\mathscr{B}', \mathscr{C}, \mathcal S$ are fixed, one can use Proposition \ref{Rajchman-prop} to derive effective bounds for $\widehat{\mu}$. One such example is presented in this section. Using the results proved in Part \ref{Part-nonnorm-F-R}, we construct a Rajchman measure supported on the set of absolutely non-normal numbers, for which a Fourier decay rate can be specified.  We do not know if this decay rate is optimal.
\begin{proposition} \label{abs-non-normal-prop} 
There exist a constant $C > 0$ and a probability measure $\mu$ supported on the set $\mathscr{N}(\emptyset, \mathbb N \setminus \{1\})$ of absolutely non-normal numbers such that 
\begin{equation} \label{abs-non-normal-decay}
\bigl|\widehat{\mu}(n) \bigr| \leq C \bigl( \log^{(3)} |n| \bigr)^{-1} \quad \text{ for all } n \in \mathbb Z, \; |n| \geq C. 
\end{equation} 
Here $\log^{(n)}$ denotes the $n$-fold composition of the logarithm. 
\end{proposition}
\subsubsection{Choice of parameter $\pmb{\Pi}$} 
For every $m \geq 1$, let $j = j(m) \in \mathbb N$ be the unique index such that $m \in \mathbb H_j$ where 
\begin{equation} \label{m-and-j}
\mathbb H_j := \bigl\{2^j-1, 2^j, \ldots, 2^{j+1}-2 \bigr\}; \quad \text{ this means } \quad C_0^{-1} \log m \leq j \leq C_0 \log m,   
\end{equation} 
for some absolute constant $C_0 > 0$.  We will apply Propositions  \ref{non-normality-prop} and \ref{Rajchman-prop}, with the following choice of parameters $\pmb{\Pi}$: 
\begin{itemize}
	\item $\mathscr{B}' = \mathscr{C} := \{\mathtt t_m = m+2 : m \geq 1 \} = \{3, 4, 5, \ldots \}$. Even though $\mathscr{C}$ is not a minimal representation of $\mathscr{B}'$, Proposition \ref{non-normality-prop} still applies; see  Remark \ref{non-normality-prop-variant} on page \pageref{non-normality-prop-variant}.
	\item 
	For $j = j(m)$ as in \eqref{m-and-j}, 
	\begin{align}
		&\mathtt s_m := \left\{\begin{aligned}
			&3,\quad 1 \le m \le {\mathtt A_0}^2;\\
			&\mathtt t_{m - 2^j + 2},\quad  m > {\mathtt A_0}^2,
		\end{aligned}\right. \hskip0.8in 
		\varepsilon_m := \frac{1}{j} \; \text{ for } \; m \in \mathbb H_j, \label{pi-1}\\
		&\mathtt a_m := \left\{\begin{aligned}
			&3^m - 2,\quad 1 \le  m  \le {\mathtt A_0}^2;\\
			&\mathtt A_0 m! \sqrt{m},\quad  m > {\mathtt A_0}^2,
		\end{aligned}\right.\hskip0.55in 
		\mathtt b_m := \left\{\begin{aligned}
			&3^m - 1,\quad 1 \le  m  \le {\mathtt A_0}^2;\\
			& (m + 1)!,\quad  m > {\mathtt A_0}^2,
		\end{aligned}\right.\label{pi-2} \\
		&\hskip1.5in \mathscr{D}_m := \{0\} \text{ for all } m \geq 1.\label{pi-3}
	\end{align}
\end{itemize}

\noindent Here $\mathtt A_0 > 0$ is a large absolute integer whose value will be specified in the sequel. 
\begin{lemma} \label{verification-lemma} 
The choice of parameters $\pmb{\Pi}$ given by \eqref{pi-1}--\eqref{pi-3} obeys the conditions \eqref{s-condition-plus}--\eqref{ambm-take-2}, required for applying Propositions \ref{non-normality-prop} and \ref{Rajchman-prop}.  
\end{lemma} 
\begin{proof} 
The criteria \eqref{s-condition-plus}--\eqref{s-condition} are easily verifiable from the definitions and are left to the reader. 
Turning to \eqref{am/bmzero} and \eqref{ambm-take-2}, we see that $\mathtt a_m < \mathtt b_m$ holds for all $m$. Further, 
\[ \frac{\mathtt a_m}{\mathtt b_m} = \frac{\mathtt A_0\sqrt{m}}{m+1} \rightarrow 0 \text{ as }  m \rightarrow \infty, \; \text{ which confirms the last condition in \eqref{am/bmzero}.} \]
To complete the proof, we need to verify that the second relation in \eqref{ambm-take-2} holds, which would also imply the second inequality in \eqref{am/bmzero}. 
\vskip0.1in 
\noindent 
For $1\le m \le {\mathtt A_0}^2$, let us note that $\mathtt a_{m} = 3^{m} - 2 > 3(3^{m-1} - 1) = 3\mathtt b_{m-1}$, hence  
\[\mathtt s_{m-1}^{3\mathtt b_{m-1}} = 3^{3\mathtt b_{m-1}} < 3^{\mathtt a_{m}} = \mathtt s_{m} ^{\mathtt a_{m}}, \text{ i.e. the second condition of \eqref{ambm-take-2} holds in this case.} \] For $m = {\mathtt A_0}^2 + 1=: \mathtt B + 1$, a similar computation shows  
\begin{align*}
	\mathtt s_{m-1}^{3\mathtt b_{m-1}} = \mathtt s_{{\mathtt B}} ^{3 \mathtt b_{\mathtt B}} = 3 ^{3(3^{\mathtt B} - 1)}
	< 3 ^ { \mathtt A_0 (\mathtt B + 1)!\sqrt{\mathtt B + 1 } }
	\le \mathtt s_{\mathtt B +1} ^ { \mathtt A_0 (\mathtt B + 1)!\sqrt{\mathtt B + 1 } } = \mathtt s_{\mathtt B +1} ^ {\mathtt a_{\mathtt B +1}} = \mathtt s_m^{\mathtt a_m},
\end{align*} 
where the first inequality holds if $\mathtt A_0 \geq 4$. For $m > {\mathtt A_0}^2+1$, 
we verify the second inequality of \eqref{ambm-take-2} in two cases, depending on the block $\mathbb H_k$ where $(m-1)$ lies. Let $j=j(m)$ be as in \eqref{m-and-j}.  
\begin{itemize} 
\item {\em{Case 1: }} First suppose $m-1 \in \mathbb H_j$. Then it follows from \eqref{pi-1} and \eqref{pi-2} that 
\[ \mathtt s_{m-1} = \mathtt t_{m-2^j + 1} = m - 2^j + 3, \quad \mathtt s_m = \mathtt t_{m - 2^j + 2} = m - 2^j + 4.\] The inequality of interest in \eqref{ambm-take-2} therefore transforms to 
\begin{align*} 
&\mathtt s_{m-1}^{3 \mathtt b_{m-1}} = (m - 2^j + 3)^{3m!} < (m - 2^j +4)^{\mathtt A_0 m!\sqrt{m}} = \mathtt s_{m}^{\mathtt a_m}, \\  &\text{i.e.,} \; (m - 2^j + 3)^3 < (m - 2^j +4)^{\mathtt A_0 \sqrt{m}}.
\end{align*}
Setting $ n = m- 2^j + 3$, this amounts to proving that  
\begin{equation} \label{eqn-n} n^3 < (n + 1)^{\mathtt A_0 \sqrt{n + 2^j-3}} \; \text{ for all } \; n \in \bigl\{2, 3, \ldots, 2^j+1 \bigr\}. \end{equation}  
Since $n - 3 + 2^j = m \geq 2^j - 1\geq 1$ for $j \geq 1$, we see that \eqref{eqn-n} holds by choosing $\mathtt A_0 \geq 4$. This completes the proof of \eqref{ambm-take-2} in case 1. 
\item {\em{Case 2: }} Next suppose that $m-1 \notin \mathbb H_j$. This means that $m-1 = 2^{j} - 2 \in \mathbb H_{j-1}$, so
\[ \mathtt s_{m-1} = \mathtt t_{2^{j-1}} = 2^{j-1}+2 \quad \text{ and } \quad \mathtt s_m = \mathtt t_1 = 3. \]  
On one hand, $\mathtt s_{m-1}^{3\mathtt b_{m-1}} = (2^{j-1} + 2)^{3m!} < (2^j + 1)^{3m!} = (m+2)^{3 m!}$, whereas $\mathtt s_m^{\mathtt a_m} = 3^{\mathtt A_0 \sqrt{m} m!}$. The second inequality \eqref{ambm-take-2} therefore follows from 
\[ (m+2)^{3m!} < 3^{\mathtt A_0 \sqrt{m} m!}, \text{ which is equivalent to } (m+2)^3 < 3^{\mathtt A_0 \sqrt{m}}.\]
This holds for all $m > \mathtt A_0^2+1$, provided $\mathtt A_0 \geq 4$. 
\end{itemize}
This verifies the hypotheses of both Propositions \ref{non-normality-prop} and \ref{Rajchman-prop} for $\pmb{\Pi}$ as in \eqref{pi-1}--\eqref{pi-3}.  
\end{proof} 
\subsubsection{Proof of Proposition \ref{abs-non-normal-prop}} 
\begin{proof} 
Let $\pmb{\Pi}$ be the choice of parameters given by \eqref{pi-1}-\eqref{pi-3}. Lemma \ref{verification-lemma} justifies the application of Propositions  \ref{non-normality-prop} and \ref{Rajchman-prop} to $\mu = \mu(\pmb{\Pi})$. Since $\overline{\mathscr{B}'} = \mathbb N \setminus \{1, 2\}$, Proposition \ref{Rajchman-prop} implies that $\mu$-almost every point is absolutely non-normal.  On the other hand, applying Proposition \ref{Rajchman-prop} to $\mu$, we obtain from \eqref{pre-Rajchman} that 
\begin{align} 
\bigl| \widehat{\mu}(\xi) \bigr| &\leq C(\pmb{\Pi}) \bigl[\varepsilon_m + \varepsilon_{m+1} + \tau_m + \tau_{m+1} + \mathtt N_m^{-\frac{1}{2}}\bigr] \leq \frac{C(\pmb{\Pi})}{j} \label{mu-decay-step-1} \\ &\text{ for } |\xi| \in \bigl(\mathtt Q_m, \mathtt Q_{m+1} \bigr], \; m \in \mathbb N, \; j=j(m)\;\text{is as in \eqref{m-and-j} such that}\; \mathbb H_j \ni m. \nonumber 
\end{align} 
The quantities $\mathtt N_m, \mathtt Q_m$ and $\tau_m$ are defined  as in \eqref{Qtau}. The last inequality in \eqref{mu-decay-step-1} is derived from the following estimates: $\varepsilon_{m+1} \le \varepsilon_m = 1/j$ by \eqref{pi-1},   
\begin{align*} 
&\mathtt N_m \geq \frac{1}{2} \frac{\mathtt s_m^{\mathtt a_m}}{\mathtt s_{m-1}^{\mathtt b_{m-1}}} \geq \frac{1}{2} \mathtt s_{m}^{2\mathtt a_m/3} \geq \frac{2^{m!}}{3}, \; \text{ while for $m > \mathtt A_0^2$, } \; \tau_m = \mathtt s_m^{(\mathtt a_m-\mathtt b_m)/2} \leq 3^{-m!/2}, \\ 
&\text{as a result of which } \tau_m + \tau_{m+1} + \mathtt N_m^{-\frac{1}{2}} \leq \frac{C(\pmb{\Pi})}{j}.   
\end{align*}  
In order to estimate $j^{-1}$ in terms of $|\xi|$, we observe from \eqref{mu-decay-step-1} and \eqref{Qtau} that 
\begin{align} 
&\mathtt a_{m+1} + \mathtt b_{m+1} = \mathtt A_0 (m+1)! \sqrt{m+1} + (m+2)! \leq m^{\mathtt A_0 m}; \text{ this leads to }\\
&|\xi|^2 \leq \mathtt Q_{m+1}^2 = \mathtt s_{m+1}^{(\mathtt a_{m+1} + \mathtt b_{m+1})} \leq m^{m^{\mathtt A_0 m}}, \; \text{ or } \;  \log^{(3)}|\xi| \leq \mathtt A_0 \log m \leq \mathtt A_0 j.  \label{mu-decay-step-2} 
\end{align}
The second inequality in \eqref{mu-decay-step-2} follows from the definition of $\mathtt s_m$ in \eqref{pi-1} for $m > \mathtt A_0^2$: since $m \in \mathbb H_j$ implies $m+1 \in \mathbb H_j$ or $\mathbb H_{j+1}$, we deduce that 
\[ \mathtt s_{m+1} = \mathtt t_{m - 2^k + 2} = m+4 - 2^k \text{ with  $k = j$ or $j+1$}, \] 
In either event, one has the bound $\mathtt s_{m+1} \leq m$ for $j \geq 2$. The last step of \eqref{mu-decay-step-2} uses the relation \eqref{m-and-j} between $m$ and $j$. Substituting the lower bound on $j$ from \eqref{mu-decay-step-2} into the right end of the inequality in \eqref{mu-decay-step-1}, we reach \eqref{abs-non-normal-decay}. 
\vskip0.1in 
\noindent 
\end{proof}   
\subsection{Measures of prescribed normality without the Rajchman property} \label{Pollington-not-Rajchman-section}
In Section \ref{Pollington-not-Rajchman-section}, we stated that the Cantor-like set constructed in \cite{p81} using the elementary operation $\mathscr{P}$ cannot support a Rajchman measure. We prove this statement here. 
\vskip0.1in 
\noindent Let us recall from \cite[p.\,38]{KL-book} that a set $E \subset [0,1)$ is an {\em{$H$-set}} if there exists a strictly increasing sequence of positive integers $\{ n_k : k \geq 1 \}$ and a nonempty interval $I \subset [0,1]$ such that 
\begin{equation}  \label{Hset-def} 
(n_k E) \cap I = \emptyset \text{ for all $k \in \mathbb{N}$,} \; \; \text{where} \; \;  nE  := \big\{\{nx\}\; : x \in E \big\}. 
\end{equation} 
Here $\{x\}$ denotes the fractional part of $x$. In 1922, Rajchman (\cite[p.\,39]{KL-book}) showed that every $H$-set is a set of uniqueness, which further implies that an $H$-set cannot support a Rajchman measure. In what follows, we will show that the set $E$ constructed in \cite{p81} is an $H$-set. 
\vskip0.1in
\noindent  Let $(\mathscr{B}, \mathscr{B}')$ be a maximally compatible pair of bases of $\mathbb{N} \setminus \{1\}$.
The construction of $E$ proceeds in two stages.  First, a sequence of nested Cantor-type sets 
\begin{equation} [0,1]\supset J_1 \supset J_2 \supset \cdots  \text{ is defined; it is shown that }  
F := \bigcap_{i=1}^{\infty} J_i \subseteq \mathscr{N}\bigl( \cdot, \mathscr{B}' \bigr). \label{F-Hset} 
\end{equation}
Second, a new sequence of nested sets $K_0=[0,1] \supset K_1 \supset K_2 \supset \cdots$ is constructed, with $K_i \subset J_i$ for each $i\in\mathbb{N}$.
It is shown that 
\[ E := \bigcap_{i=1}^{\infty} K_i \subseteq \mathscr{N}(\mathscr{B}, \mathscr{B}'). \] 
To prove that the set $E$ thus constructed is an $H$-set, it is enough to show its superset $F$ defined in \eqref{F-Hset} is an $H$-set. 
\vskip0.1in
\noindent Let $\{s_i: i \in \mathbb N\}$ be an enumeration of $\mathscr{B}'$. The sets $J_i$  constructed in \cite{p81} depend on parameters 
\[ s(m):=s_{h(m)} \; \; \text{ with } \; \; h(m):=\inf\{h\in\mathbb{N}:\ m \not\equiv 0 \text{ mod } {2^h} \}, \] 
and strictly increasing positive integers $\{a_m\},\{b_m\}$ that are determined by $s(m)$. On one hand, this means that 
$s(m)=s_1$ for every odd integer $m \geq 1$. On the other, the construction imposes certain constraints on the digits of $\xi \in F$ in base $s_1$; namely,  
\begin{align*} &\text{ if } \xi=\sum_{j=1}^\infty {d_j}s_1^{-j} \in F \text{ with } d_j \in \mathbb Z_{s_1}, \text{ then } d_j \neq s_1-1 
 \text{ for all } j \in \{a_{m}+1, \ldots, b_{m} \}, \text{ $m$ odd; } \\
 &\text{as a result, }  \bigl\{s_1^{a_m}\xi \bigr\} =  \sum_{j=a_{m}+1}^\infty {d_j}{s_1^{-j+a_{m}}} \leq  (s_1-2) \sum_{j=1}^{\infty} s_1^{-j} =  \frac{s_1-2}{s_1-1}, \\ 
&\text{which means that } \; \; (s_1^{a_m} F)\cap \bigl(\frac{s_1-2}{s_1-1}, 1\bigr] =\emptyset \text{ for all odd $m$}. \end{align*} 
Comparing this with \eqref{Hset-def} proves that $F$ is an $H$-set. In the special case when $\mathscr{B}'=\{s\}$ for some $s \in \mathbb N \setminus \{1\}$, the set $F$ is just a classical Cantor set given by 
\begin{align*}
	F =\Big\{\xi\in[0,1]:\ \xi=\sum_{j=1}^\infty\frac{d_j}{s^j},\; d_j\in\{0,1,\cdots,s-2\}\Big\}.
\end{align*}

\newpage
\part{Measures seeking normality} \label{part-normality} 
This is the final part of a three-part series. In Part \ref{intro-part}, we introduced the definition of a skewed measure (Sections \ref{elem-op-section}-\ref{iteration-section}) and stated some of its properties (Section \ref{skewed-measure-properties-section}).  Here we investigate one of these properties, namely the feature of normality: what are the bases in which almost every point in the support of a skewed measure is normal? 
\vskip0.1in 
\noindent In terms of specifics, the goal in this part is to prove Proposition \ref{normal-prop}, the remaining unresolved piece in the statements of Theorems \ref{mainthm-1} and \ref{mainthm-2}. En route, we record another result on normality that holds under weaker assumptions, and in particular recovers a version of the result in \cite{PZ-1}. 
Together with Proposition \ref{normal-prop}, the other main result of this part of the article is the following.
\begin{proposition} \label{normality-special-prop}
For a non-empty collection of bases $\mathscr{B}' \subseteq \mathbb N \setminus \{1\}$, let $\mathscr{C}$ be as in \eqref{what-is-C} and let $\mathcal S = \{\mathtt s_m : m \geq 1\} \subseteq \mathbb N \setminus \{1,2\}$ be an ordered sequence of (not necessarily distinct) bases obeying \eqref{s-condition-plus}. Then there is a sequence $\{\kappa_m=\kappa(\mathtt s_m):\ m\geq 1\} \subseteq (0,1)$ as follows.   
\vskip0.1in 
\noindent Suppose that $\mu = \mu(\pmb{\Pi})$ is the skewed measure whose parameter $\pmb{\Pi} = (\mathcal S, \mathcal A, \mathcal B, \mathcal D, \mathcal E)$ satisfies the assumptions: $\mathcal A, \mathcal B$ obey \eqref{Km-choice} and \eqref{normality-hypotheses} with the above-mentioned choice of $\{\kappa_m\}$, 
\begin{equation*} 
\mathcal E = \{\varepsilon_m \} \subseteq (0, 1), \quad  0 \in \mathscr{D}_m \subsetneq \mathbb Z_{\mathtt s_m} \; \; \text{ for all } m \geq 1.
\end{equation*} 
Then $\mu$-almost every point $x$ is $b$-normal for every $b \in \mathscr{B}_1$, where 
\begin{equation} \label{bases-B1-def}
\mathscr{B}_1 := \bigl\{b \in \mathbb N \setminus \{1,2\} : \text{ for every } \mathtt s \in \mathscr{C}, \exists \,\text{ a prime } \, \mathtt p \ni \;  \mathtt p \nmid b \; \text{ but } \; \mathtt p \mid \mathtt s \bigr\}.
\end{equation} 
\end{proposition} 
\noindent {\em{Remarks: }} 
\begin{enumerate}[1.]
\item The only difference in Propositions \ref{normal-prop} and \ref{normality-special-prop} is in their hypotheses for $\mathcal D$.  The former requires $\{0,1\}\subseteq \mathscr{D}_m$, the latter imposes the weaker restriction $0\in \mathscr{D}_m$.
\item Choosing $\mathtt s_m = 2$ and $\mathscr{D}_m = \{0 \}$ for all $m$ recovers a version of \cite{PZ-1}. There the Rajchman measure $\nu_{\text{L}}$ constructed by Lyons was shown to be normal in every base of $\mathscr{B}_1$, namely the set of odd integer bases $\geq 3$.  
\item Similarly, given any non-trivial partition of the primes, say $\mathcal P_1, \mathcal P_2$, one can use Propositions \ref{non-normality-prop} and \ref{normality-special-prop} with $\mathscr{D}_m \equiv \{0\}$ to find a skewed Rajchman measure that is normal with respect to $\mathcal P_1$ and non-normal with respect to $\mathcal P_2$.  
\end{enumerate} 
\section{A sufficient condition for $b$-normality}  \label{normality-overview-section} 
\subsection{Normality and distribution modulo 1} \label{unif-dist-section}
There are several equivalent formulations of normality. One of the earliest such characterizations involves the notion of uniformly distributed sequences. We briefly recall the definition:  a sequence $(x_n)_{n \geq 0}$ is said to be {\em{uniformly distributed modulo 1}} if for every $u, v \in [0, 1)$, $u < v$, the following condition holds:
\[ \lim_{N \rightarrow \infty} \frac{1}{N} \# \Bigl\{0 \leq n \leq N-1: \{x_n \} \in [u, v) \Bigr\} = v - u, \qquad \{x\} := x - \lfloor x \rfloor. \]
A classical theorem known as Weyl's criterion \cite[Theorem 1.2]{b12} states that $(x_n)_{n \geq 1}$ is uniformly distributed modulo 1 if and only if 
\begin{equation} \label{Weyl}
\lim_{N \rightarrow \infty} \frac{1}{N} \sum_{n=0}^{N-1} e(h x_n)  = 0 \quad \text{ for all}\; h \in \mathbb{Z}\setminus\{0\}. 
\end{equation}   
The following result, proved by Wall \cite{w49} in his Ph.D. thesis, connects normality with uniform distribution mod one; a proof can also be found in \cite[Theorem 4.14]{b12}. 
\begin{lemma}[\cite{w49}, {\cite[Theorem 4.14]{b12}}] \label{Weyl-lemma} 
Let $b \geq 2$. Then a real number $\xi$ is normal to base $b$ if and only if $(\xi b^n)_{n \geq 1}$ is uniformly distributed modulo 1. In view of \eqref{Weyl}, this is equivalent to the condition:
\begin{equation} \label{Weyl-2} 
\lim_{N \rightarrow \infty} \frac{1}{N} \sum_{n=0}^{N-1} e(h \xi b^n)  = 0 \quad \text{ for all}\; h \in \mathbb{Z}\setminus\{0\}.
\end{equation}  
\end{lemma}  
\subsection{Summability criterion of Davenport, Erd\H{o}s and LeVeque} 
 A fundamental tool in the theory of uniform distribution is the theorem of Davenport, Erd\H{o}s and LeVeque \cite{del63}. For generic points $x$ in the support of a measure $\nu$ on $[0,1)$, it connects uniform distribution properties of a sequence of the form $\{s_n(x) : n \geq 1\}$ with decay properties of $\widehat{\nu}$. In view of Lemma \ref{Weyl-lemma} and specializing to $s_n(x) = b^n x$, this result provides a sufficient condition for establishing generic $b$-normality of numbers on supp$(\nu)$. While the criterion was originally stated in \cite{del63} for the Lebesgue measure on $[0,1)$, the same proof generalizes verbatim for any probability measure $\nu$, and is now ubiquitous in that generalized form in the literature. We state the latter version with appropriate references.       
\begin{lemma}[\cite{del63}, {\cite[Theorem 1.3]{Gao-Ma-Song-Zhang}}, {\cite[Theorem DEL]{PVZZ}}] \label{DEL-lemma}
Let $\nu$ be a probability measure on $[0,1)$ and $\{s_n(\cdot) : n \geq 1\}$ a sequence of measurable functions, also on $[0,1)$. If 
\begin{equation} \label{DEL-v1} 
\sum_{N=1}^{\infty} N^{-1} \int_{0}^{1} \Bigl| \frac{1}{N} \sum_{n=0}^{N-1} e \bigl(h s_n(x) \bigr)\Bigr|^2 \, d\nu(x) < \infty \quad \text{ for every } h \in \mathbb Z \setminus \{0\},  
\end{equation}  
then $\{s_n(x) : n \geq 1\}$ is uniformly distributed modulo 1 for $\nu$-almost every $x$. 
\vskip0.1in 
\noindent Suppose $b \in \mathbb N \setminus \{1\}$. Specializing \eqref{DEL-v1} to the case $s_n(x) = b^n x$ and applying Lemma \ref{Weyl} leads to the following statement: if the infinite series 
\begin{equation}
\begin{aligned} 
\mathscr{S}(\nu, h, b) &:= \sum_{N=1}^{\infty} N^{-3} \sum_{u = 0}^{N-1} \sum_{v = 0}^{N-1} \widehat{\nu}(h(b^v-b^u)) \\ &= \sum_{N=1}^{\infty} N^{-3} \sum_{u = 0}^{N-1} \sum_{v = 0}^{N-1-u} \widehat{\nu}(hb^u (b^v-1)) \end{aligned} \text{ converges } \label{DEL-sum}
\end{equation}  
 for every $h \in \mathbb Z \setminus \{0\}$, then $\nu$-almost every $x$ is $b$-normal.  
\end{lemma} 
\noindent Thus verifying the convergence of the infinite series \eqref{DEL-sum} offers a mechanism for establishing pointwise normality on the support of $\mu$. Let us state a version of Lemma \ref{DEL-lemma} in the form we will need.  
\begin{corollary} \label{DEL-corollary} 
Let $\{\varphi_m : m \geq 1\}$ be a sequence of absolutely continuous probability measures on $[0, 1)$, 
\begin{equation} \label{nu-weak-convergence} 
\varphi_m \rightarrow \mu \text{ weakly as } m \rightarrow \infty, \quad \text{ for some probability measure $\mu$. }  
\end{equation}  Suppose that $b \geq 2$ is an integer base with the following property: for every $h \in \mathbb Z \setminus \{0\}$, 
 there exists a sequence $\{\mathtt c_m(h, b) : m \geq 1\}$ of non-negative real numbers such that  
\begin{align}
\mathcal S_m(\mu, h, b) &:= \sum_{N=1}^{\infty} \frac{1}{N^{3}} \mathscr{S}_m(N;  h, b) \leq \mathtt c_m(h,b) \; \text{ with } \; \sum_{m=1}^{\infty} \mathtt c_m(h, b) < \infty,  \text{ where } \label{Sm-summable} \\
\mathscr{S}_m(N; h, b) &:= \sum_{v = 0}^{N-1} \sum_{u = 0}^{N-1} \bigl|(\widehat{\varphi}_{m} - \widehat{\varphi}_{m-1})\bigl(h b^u(b^v-1) \bigr) \bigr| \quad \text{with the convention}\; \varphi_0 \equiv 0. \label{def-Sm}
\end{align} 
Then $\mu$-almost every point is $b$-normal. 
\end{corollary} 
\begin{proof}
For $\mathtt M \geq 1$, we write $\varphi_{\mathtt M}$, and therefore $\widehat{\varphi}_{\mathtt M}$, as a telescoping sum: 
\[ \widehat{\varphi}_{\mathtt M} = \widehat{\varphi}_{0} + \sum_{m=1}^{\mathtt M} \bigl(\widehat{\varphi}_{m} - \widehat{\varphi}_{m-1}\bigr), \quad \text{ so that } \quad |\widehat{\varphi}_{\mathtt M}| \leq |\widehat{\varphi}_{0}| + \sum_{m=1}^{\mathtt M} \bigl|\widehat{\varphi}_{m} - \widehat{\varphi}_{m-1}\bigr|.  \] 
A base $b$ obeying the assumption \eqref{Sm-summable} therefore allows us to control sums of the following form $\mathscr{S}^\ast (\varphi_{\mathtt M}, h, b) $ by a bound uniform in $\mathtt M$: 
\begin{align} 
\mathscr{S}^{\ast}(\varphi_{\mathtt M}, h, b) &:= \sum_{N=1}^{\infty} \frac{1}{N^{3}} \sum_{u = 0}^{N-1} \sum_{v = 0}^{N-1} \bigl|\widehat{\varphi}_{\mathtt M}\bigl(hb^u (b^v-1) \bigr) \bigr| \nonumber \\ &\leq  \sum_{N=1}^{\infty}  \frac{1}{N^{3}} \Bigl[ \sum_{m=1}^{\mathtt M} \mathscr{S}_m(N; h, b)  \Bigr] = \sum_{m=1}^{\mathtt M} \mathcal S_m(\mu, h, b) \leq \sum_{m=1}^{\infty} \mathtt c_m(h, b) < \infty. \label{S-uniform-bound} 
\end{align}   
The weak convergence \eqref{nu-weak-convergence} of $\varphi_m$ to $\mu$ means that 
\[ \lim_{m \rightarrow \infty} \widehat{\varphi}_m(\xi) = \widehat{\mu}(\xi) \quad \text{ for all } \xi \in \mathbb Z.  \]  An application of Fatou's lemma, combined with \eqref{S-uniform-bound}, then leads to the relation 
\begin{align*} \mathscr{S}(\mu, h, b) &\leq \mathscr{S}^{\ast}(\mu, h, b) = \sum_{N=1}^{\infty} \frac{1}{N^{3}} \sum_{u = 0}^{N-1} \sum_{v = 0}^{N-1} \bigl|\lim_{\mathtt M \rightarrow \infty}\widehat{\varphi}_{\mathtt M} \bigl(hb^u (b^v-1) \bigr) \bigr|  \\
 &\leq \lim_{\mathtt M \rightarrow \infty} \mathscr{S}^{\ast}(\varphi_{\mathtt M}, h, b) \leq   \sum_{m=0}^{\infty} \mathtt c_m(h, b)  < \infty. \end{align*} 
The criterion of Davenport, Erd\H{o}s and LeVeque (Lemma \ref{DEL-lemma}) now confirms that $\mu$-almost every point is $b$-normal. This is the conclusion of Corollary \ref{DEL-corollary}. 
\end{proof}
\subsection{Conditional proof of Propositions \ref{normality-special-prop} and \ref{normal-prop}} 
\noindent We will use Corollary \ref{DEL-corollary} to prove that $\mu = \mu(\pmb{\Pi})$-almost every point is 
\begin{itemize} 
\item in $\mathscr{N}(\mathscr{B}_1, \cdot)$ with $\mathscr{B}_1$ as in \eqref{bases-B1-def}, provided $\pmb{\Pi}$ obeys the hypotheses of Proposition \ref{normality-special-prop};
\item in $\mathscr N(\mathscr{B}, \cdot)$ with $\mathscr{B}$ as in \eqref{bases-BC-bar}, provided $\pmb{\Pi}$ obeys the hypotheses of Proposition \ref{normal-prop}.
\end{itemize} 
Indeed by Corollary \ref{DEL-corollary}, both propositions follow immediately from the result below.
\begin{proposition} \label{DEL-difference-prop}
Let $\pmb{\Pi}$ be any choice of parameters as in \eqref{iteration-parameters-section}, and let $\varphi_m$ be the $m^{\text{th}}$ density function \eqref{density} that occurs in the construction of $\mu = \mu(\pmb{\Pi})$.  Then for every $h \in \mathbb Z \setminus \{0\}$ and $\mathscr{S}_m$ as in \eqref{def-Sm}, the summability condition \eqref{Sm-summable} holds 
\begin{enumerate}[(a)]
\item for every $b \in \mathscr{B}_1$, provided $\pmb{\Pi}$ obeys the hypotheses of Proposition \ref{normality-special-prop};   \label{DEL-difference-prop-parta}
\item for every $b \in \mathscr{B}$, provided $\pmb{\Pi}$ obeys the hypotheses of Proposition \ref{normal-prop}. \label{DEL-difference-prop-partb}
\end{enumerate} 
\end{proposition} 
\subsubsection{Proof layout of Proposition \ref{DEL-difference-prop}} In view of the above, the task of proving almost everywhere normality of supp$(\mu)$ in bases of $\mathscr{B}_1$ and $\mathscr{B}$ (as claimed in Propositions \ref{normality-special-prop} and \ref{normal-prop} respectively) shifts to verifying a summability criterion of the form \eqref{Sm-summable} (Proposition \ref{DEL-difference-prop}). We henceforth focus on proving the latter result. A few words are in order about the proof strategy from this step onwards. In order to establish \eqref{Sm-summable},    
we will perform a series of reductions on the sum $\mathcal S_m$, decomposing it into summands that are controllable due to different reasons. 
\begin{itemize} 
\item In Section \ref{subsum-reduction-section}, the sum $\mathcal S_m$ is decomposed into two sub-sums $\mathcal U_m$ and $\mathcal V_m$. The sum $\mathcal U_m$ reflects the first part of the elementary operation in step $m$ of the construction of $\mu$, and is relatively straightforward to handle. It is estimated in Section \ref{estimating-um-section}, using pointwise estimates for its summands derived in Section \ref{u-pointwise-estimate-section}. These in turn draw upon fundamental exponential sums introduced in Section \ref{exp-sum-section}.    
\item The sum $\mathcal V_m$ corresponds to the second part of the elementary operation. Estimation of this sum partly employs strategies similar to $\mathcal U_m$. Pointwise estimates for its summand, analogous to $\mathcal U_m$ and relying on Section \ref{exp-sum-section}, are obtained in Section \ref{v-pointwise-section}. The portion of $\mathcal V_m$ amenable to these estimates, termed $\mathcal V_{m1}$,  is isolated in Section \ref{estimating-vm-section} and its estimation is completed in Section \ref{estimating-Vm1-section}.  It is worth noting that the estimates used to bound $\mathcal U_m$ and $\mathcal V_{m1}$ hold for any choice of $b$ and $\mathtt s_m$; in particular, they do not rely on multiplicative independence between $b$ and $\mathtt s_m$. 
\item In contrast, the remaining portion of $\mathcal V_m$, termed $\mathcal V_{m2}$, needs more sophisticated tools customized for this problem. These are developed in two sections. Section \ref{v-pointwise-section-take-2} derives a second set of analytical estimates for the summands of $\mathcal V_m$ that depend on the number-theoretic nature of the frequencies. Section \ref{number-theoretic-tools-section} gathers number-theoretic tools and their consequences, some of these involving multiplicatively independent bases. The estimation of $\mathcal V_{m2}$ is heavily reliant on the material from these two sections. 
\item The estimation of $\mathcal V_{m2}$ itself is carried out in two steps (Sections \ref{estimating-vm-section-Part1} and \ref{estimating-vm-section-Part2}).  Whereas the treatment of $\mathcal U_m$ and $\mathcal V_{m1}$ does not distinguish between $\mathscr{B}_1$ and $\mathscr{B}$ (and is therefore the same for parts \eqref{DEL-difference-prop-parta} and \eqref{DEL-difference-prop-partb} of Proposition \ref{DEL-difference-prop}), estimation of $\mathcal V_{m2}$ requires separate analyses for the two cases. Section \ref{estimating-vm-section-Part1} addresses this for $\mathscr{B}_1$, completing the proof of part \eqref{DEL-difference-prop-parta}. The corresponding steps for the more general part \eqref{DEL-difference-prop-partb} are completed in Section \ref{estimating-vm-section-Part2}.  
\end{itemize} 
A sketch of the main steps of the proof of Proposition \ref{DEL-difference-prop} has been included in Figure \ref{normality-proof}.
\begin{center}
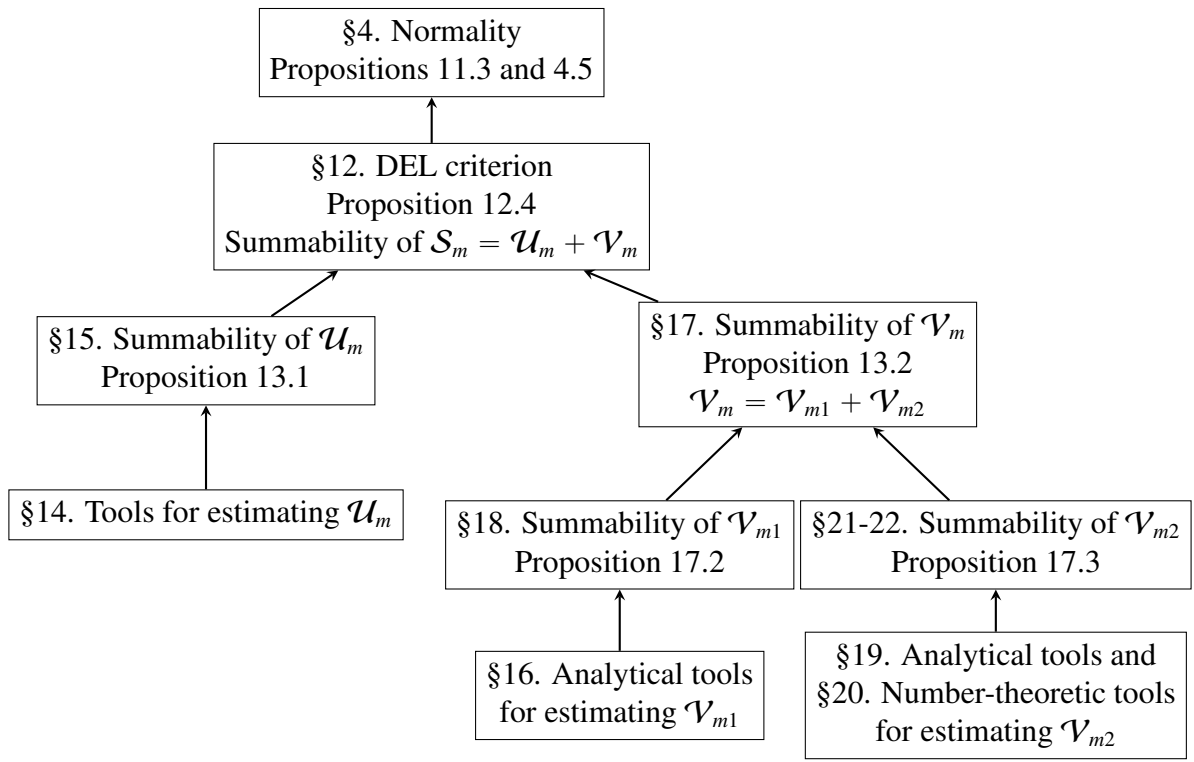
\begin{figure}
\begin{tikzpicture}[node distance=2cm]
\node (thm) [box, align=center] {\S \ref{skewed-measure-properties-section}. Normality \\ Propositions \ref{normality-special-prop} and \ref{normal-prop}};
\node (S5) [box, below of= thm, xshift=0cm, yshift=-0.05cm, align=center] {\S \ref{normality-overview-section}. DEL criterion \\ Proposition \ref{DEL-difference-prop} \\ Summability of $\mathcal{S}_m = \mathcal{U}_m + \mathcal{V}_m$};
\node (S4) [box, below of=S5, xshift=-3cm, yshift=-0.05cm, align=center]{\S \ref{estimating-um-section}. Summability of $\mathcal{U}_m$\\ Proposition \ref{U-prop}};
\node (S6) [box, below of=S5, xshift=5cm, yshift=-0.1cm,align=center] {\S \ref{estimating-vm-section}. Summability of $\mathcal V_m$ \\ Proposition \ref{V-prop} \\ $\mathcal V_m = \mathcal V_{m1} + \mathcal V_{m2}$};
\node (S7) [box, below of= S6, xshift=-2.5cm, yshift=-0.4cm,align=center] {\S \ref{estimating-Vm1-section}. Summability of $\mathcal V_{m1}$ \\ Proposition \ref{W-sum-prop}};
\node (S9) [box, below of= S6, xshift=2.5cm, yshift=-0.4cm,align=center] {\S \ref{estimating-vm-section-Part1}-\ref{estimating-vm-section-Part2}. Summability of $\mathcal V_{m2}$ \\ Proposition \ref{V2-sum-prop}};
\node (Sn) [box, below of= S4, xshift=0cm, yshift=-0.05cm,align=center] {\S \ref{u-pointwise-estimate-section}. Tools for estimating $\mathcal U_m$}; 
\node (S8)  [box, below of= S7, xshift=0cm, yshift=-0.003cm, align=center] {\S \ref{v-pointwise-section}. Analytical tools \\  for estimating $\mathcal V_{m1}$};
\node (S14)  [box, below of= S9, xshift=0cm, yshift=-0.003cm,align=center] {\S \ref{v-pointwise-section-take-2}. Analytical tools and \\ \S \ref{number-theoretic-tools-section}. Number-theoretic tools \\ for estimating $\mathcal V_{m2}$};
%
\draw [arrow] (S4) -- (S5) ; 
\draw [arrow] (S5) -- (thm);
\draw [arrow] (S6) -- (S5) ;
\draw [arrow] (S7) -- (S6) ;
\draw [arrow] (S9) -- (S6);
\draw [arrow] (Sn) -- (S4);
\draw [arrow] (S14) -- (S9);
\draw [arrow] (S8) -- (S7); 
\end{tikzpicture}
\caption{Proof sketch for Propositions \ref{normality-special-prop} and \ref{normal-prop}} \label{normality-proof}
\end{figure} 
\end{center}

\section{Reduction of $\mathscr{S}_m$ to two sub-sums} \label{subsum-reduction-section}
As noted in Section \ref{elem-op-section}, the elementary operation $\mathscr{O}$ taking $(E_{m-1}, \varphi_{m-1})$ to $(E_m, \varphi_m)$ involves two distinct parts. It is therefore natural that the estimation of $\mathscr{S}_m$ as defined in \eqref{def-Sm} will reflect the two components of the operation. Let us denote by $\psi_m$ the density of the set that appears in the intermediate step of the construction between $E_{m-1}$ and $E_m$. In other words, $\psi_m$ is the counterpart of $\Psi_1$ in \eqref{intermediate-density}, with $\pmb{\gamma} = \pmb{\gamma}_m$ in the notation of \eqref{O-def} and \eqref{iteration}. We consider the following decomposition:
\begin{align} 
\widehat{\varphi}_m  &- \widehat{\varphi}_{m-1} = \mathtt u_m + \mathtt v_m \; \text{ which leads to } \;   \mathscr{S}_m \leq \mathscr{U}_m + \mathscr{V}_m, \text{ with } \label{difference-decomp} \\
\mathscr{U}_m &= \mathscr{U}_m(N; h, b) := \sum_{v = 0}^{N-1} \sum_{u = 0}^{N-1} \bigl|\mathtt u_m \bigl(h b^u(b^v-1) \bigr) \bigr|, \quad \mathtt u_m := \widehat{\psi}_{m} - \widehat{\varphi}_{m-1}, \label{def-um} \\
 \mathscr{V}_m &= \mathscr{V}_m(N; h, b) := \sum_{v = 0}^{N-1} \sum_{u = 0}^{N-1} \bigl|\mathtt v_m \bigl(h b^u(b^v-1) \bigr) \bigr|, \quad  \mathtt {v}_m := \widehat{\varphi}_{m} - \widehat{\psi}_{m}. \label{def-vm}
\end{align}
Proposition \ref{DEL-difference-prop} then follows from the two estimates below. 
 \begin{proposition} \label{U-prop} 
Let $\pmb{\Pi}$ be a choice of parameters obeying \eqref{param-seq}, \eqref{abs} and \eqref{rmsm-assumption}. Then for every $h \in \mathbb Z \setminus \{0\}$ and $b \in \mathbb N \setminus \{1\}$, there is a constant $C(h, b, \pmb{\Pi}) > 0$ such that for all $m \geq 1$, 
\begin{equation} \label{Um-sum} 
\mathcal U_m(h, b) := \sum_{N=1}^{\infty} N^{-3} \mathscr{U}_m(N; h, b)   \leq C(h, b, \pmb{\Pi})  {\mathtt s_{m-1}^{\mathtt b_{m-1}}}{\mathtt s_m^{-\frac{\mathtt a_m}{2}}}.
\end{equation}   
\end{proposition} 
\noindent {\em{Remark:}} Proposition \ref{U-prop} allows $b$ to be multiplicatively dependent on the bases in $\mathcal S$. 

\begin{proposition} \label{V-prop} 
Suppose that one of the following assumptions holds:
\begin{enumerate}[(a)] 
\item The parameter $\pmb{\Pi}$ obeys the hypotheses of Proposition \ref{normality-special-prop}, and $b \in \mathscr{B}_1$ as in \eqref{bases-B1-def}. 
\item  The parameter $\pmb{\Pi}$ obeys the hypotheses of Proposition \ref{normal-prop}, and $b \in \mathscr{B}$ as in \eqref{bases-BC-bar}. 
\end{enumerate} 
Then there is a sequence of positive real numbers $\{\mathtt c_m(h, b, \pmb{\Pi}) : m \geq 1\}$ such that  
\begin{equation} \label{Vm-sum} 
\mathcal V_m(h,b) := \sum_{N=1}^{\infty} N^{-3} \mathscr{V}_m(N; h, b)  \leq \mathtt c_m(h, b, \pmb{\Pi}), \quad \sum_{m=1}^{\infty} \mathtt c_m(h, b, \pmb{\Pi}) < \infty. 
\end{equation} 
\end{proposition}  
\noindent Proposition \ref{U-prop} will be proved in Section \ref{estimating-um-section}. Proposition \ref{V-prop} will be proved in Section \ref{Vprop-proof-section-part1}, modulo two other statements that will be verified further along. Using these statements, the proof of Proposition \ref{DEL-difference-prop} is completed as follows.  
\subsection{Proof of Proposition \ref{DEL-difference-prop}, assuming Propositions \ref{U-prop} and \ref{V-prop}} 
In view of the decomposition \eqref{difference-decomp} of $\mathscr{S}_m$ into $\mathscr{U}_m$ and $\mathscr{V}_m$, we see that the sums $\mathcal S_m$, $\mathcal U_m$ and $\mathcal V_m$, given respectively by \eqref{Sm-summable}, \eqref{Um-sum}, \eqref{Vm-sum}, are related by the inequality 
\[ \mathcal S_m(h, b) \leq \mathcal U_m(h, b) + \mathcal V_m(h, b).  \] Thus, in order to prove the summability condition \eqref{Sm-summable} for $\mathcal S_m$,  it suffices to show that both the quantities $\mathcal U_m$ and $\mathcal V_m$ are summable in $m$. For $\mathcal V_m$, this is already ensured by Proposition \ref{V-prop}, since the two parts \eqref{DEL-difference-prop-parta} and \eqref{DEL-difference-prop-partb} of Proposition \ref{DEL-difference-prop} correspond exactly to the two situations covered by Proposition \ref{V-prop}.  On the other hand, the inequality \eqref{Um-sum} for $\mathcal U_m$ in Proposition \ref{U-prop} holds independently of the choice of $b$ and $\pmb{\Pi}$, and therefore applies to both parts of Proposition \ref{DEL-difference-prop} simultaneously.    
Thus we only need to verify that the quantity on the right hand side of \eqref{Um-sum}, which bounds $\mathcal{U}_m$, is summable. The assumption \eqref{ambm-take-2} on $\mathtt a_m, \mathtt b_m$, which are ensured by \eqref{Km-choice} and \eqref{normality-hypotheses}, dictates that 
\begin{equation} {\mathtt s_{m-1}^{\mathtt b_{m-1}}}{\mathtt s_m^{-\frac{\mathtt a_m}{2}}} \leq  \mathtt s_m^{-\frac{\mathtt a_m}{6}} \leq C(\pmb{\Pi}) 2^{-c_0 \mathtt a_m}. \label{Vm-sum-bound} \end{equation} 
The assumption \eqref{normality-hypotheses} requires the positive integers $\{\mathtt a_m\}$ to be distinct. As a result, the terms $2^{-c_0 \mathtt a_m}$ on the right hand side of \eqref{Vm-sum-bound} are distinct elements of the geometric sequence $\{2^{-c_0m} : m \geq 1\}$. Hence \eqref{Vm-sum-bound} implies that 
\begin{equation}  \label{V-prop-est}
\sum_{m=1}^{\infty} \mathcal{U}_m(h, b) \leq C(\pmb{\Pi}) \sum_{m=1}^{\infty} 2^{-c_0 \mathtt a_m} \leq C(\pmb{\Pi}) \sum_{m=1}^{\infty} 2^{-c_0 m} <  \infty. 
\end{equation}  
This completes the proof of Proposition \ref{DEL-difference-prop}. 
\qed
\vskip0.1in
\noindent To summarize the situation thus far: we have reduced the task of establishing normality to the verification of the summability criteria \eqref{Um-sum} and \eqref{Vm-sum} for $\mathcal U_m$ and $\mathcal V_m$ respectively. We deal with these tasks in the remainder of the article.

\section{A pointwise estimate for $\mathtt u$} \label{u-pointwise-estimate-section}
The decomposition \eqref{difference-decomp} of $\mathscr{S}_m$ and the definitions \eqref{def-um}, \eqref{def-vm} lead us to a study of
the exponential sums that occur in the Fourier coefficients of the density difference $\Phi_1 - \Phi_0$, where $\Phi_0$ and $\Phi_1$ are the densities in \eqref{Phi0} and \eqref{Phi1} respectively. As noted in Section \ref{elem-op-section}, the elementary operation $\mathscr{O}$ taking $(\mathtt E_0, \Phi_0)$ to $(\mathtt E_1, \Phi_1)$ involves two distinct parts. To reflect this, we write the difference as the sum of two terms: 
\begin{equation}  \widehat{\Phi}_1 - \widehat{\Phi}_0 = \mathtt u + \mathtt v, \quad \text{ where } \quad  \mathtt u := \widehat{\Psi}_1 - \widehat{\Phi}_0, \quad \mathtt v := \widehat{\Phi}_1 - \widehat{\Psi}_1.  \label{uv-def} \end{equation} 
Here $\Psi_1$ is the density at the intermediate step of the operation $\mathscr{O}$, given by \eqref{intermediate-density}. 
\vskip0.1in
\noindent The purpose of this section is to develop the analytical machinery needed to prove Proposition \ref{U-prop}. The estimates on $\mathtt u$ recorded below apply to $\mathtt u_m$ and therefore to $\mathcal U_m$, and will be used in Section \ref{estimating-um-section} for this purpose.  A similar exercise will be repeated for $\mathtt v$ and $\mathtt v_m$ later, but with certain key differences.   
\begin{lemma} \label{u-lemma}
There exists an absolute constant $C_0 > 0$ for which the function $\mathtt u$ defined in \eqref{uv-def} obeys the estimate
\begin{equation} \label{u-pointwise-estimate}
|\mathtt u(\xi)| \leq C_0 \min \left[1, \frac{\mathtt t}{|\xi|}, \frac{|\xi|}{\mathtt s^{\mathtt a}} \right], \quad \xi \in \mathbb Z. 
\end{equation} 
\end{lemma} 
\begin{proof} 
\noindent Let us recall from \eqref{Phi1-hat} and \eqref{intermediate-Fourier-coefficient} that
\[ \widehat{\Phi}_{0}(\xi) = \widehat{\mathtt 1} \bigl(\xi \mathtt t^{-1} \bigr) \Upsilon_0(\xi) \quad \text{ and } \quad 
\widehat{\Psi}_1(\xi) = \widehat{\mathtt 1}\bigl( \xi \mathtt N \mathtt s^{-\mathtt a}\bigr) \Omega_1(\xi), \]  
where $\Upsilon_0$ and $\Omega_1$ are the exponential sums defined in \eqref{Yj} and \eqref{Omega_1} respectively. We estimate the difference of the two as follows, 
\begin{align} \mathtt u(\xi) &= \widehat{\Psi}_1(\xi) - \widehat{\Phi}_{0}(\xi) 
= \mathtt z_{1}(\xi) + \mathtt z_{2}(\xi), \; \text{ where } \label{vm-decomp} \\ 
\mathtt z_{1}(\xi) & := \Bigl[ \widehat{\mathtt 1}(\xi \mathtt N \mathtt s^{-\mathtt a}) - \widehat{\mathtt 1}(\xi \mathtt t^{-1}) \Bigr] \Omega_{1}(\xi), \label{def-vm1} \\ 
\mathtt z_{2}(\xi) &:= \widehat{\mathtt 1}(\xi \mathtt t^{-1}) \bigl[ \Omega_1(\xi) - \Upsilon_{0}(\xi) \bigr]. \label{def-vm2} 
\end{align}  
With this decomposition of $\mathtt u$ behind us, we are ready to prove \eqref{u-pointwise-estimate}. 
\vskip0.1in
\noindent In view of \eqref{vm-decomp}, we will estimate $\mathtt z_{1}$ and $\mathtt z_{2}$ separately. Let us start with $\mathtt z_{1}$, given by \eqref{def-vm1} as a product of two factors. The first factor lends itself to the estimate obtained in Lemma \ref{1-hat-difference-lemma}, while $\Omega_1$ admits the trivial bound 1. Setting  
$\theta_1 = \xi \mathtt N \mathtt s^{-\mathtt a}$, $\theta_2 = \xi \mathtt t^{-1}$ in \eqref{1-hat-difference-estimate}, we obtain 
\begin{align}
|\mathtt z_{1}(\xi) | &\leq \bigl| \widehat{\mathtt 1}(\xi \mathtt N \mathtt s^{-\mathtt a}) - \widehat{\mathtt 1}(\xi \mathtt t^{-1}) \bigr| \leq C_0 \min \bigl(1, |\theta_1-\theta_2|\bigr) \times \min \bigl(1, |\theta_1|^{-1} \bigr) \nonumber \\ 
&\leq C_0 \min \left[1, \frac{1}{|\theta_1|}, |\theta_1 - \theta_2| \right] \leq C_0  \min \left[1, \frac{\mathtt t}{|\xi|}, \frac{|\xi|}{\mathtt s^{\mathtt a}}\right]. \label{vm1-est}
\end{align}
The display above relies on the defining property \eqref{N} of $\mathtt N$, which implies  
\[ |\theta_1| = |\xi| \mathtt N \mathtt s^{-\mathtt a} \leq |\xi| \mathtt t^{-1} = |\theta_2|, \; \; \frac{1}{|\theta_1|} = \frac{\mathtt s^{\mathtt a}}{\mathtt N |\xi|} \leq \frac{C_0 \mathtt t}{|\xi|} \; \; \text{ and } \; \; |\theta_1 - \theta_2| \leq 3 |\xi| \mathtt s^{-\mathtt a}. \]
\vskip0.1in  
\noindent Next we turn to the estimation of the quantity $\mathtt z_{2}$ defined in \eqref{def-vm2}. This too is a product of two factors; the first one is of the form $\widehat{\mathtt 1}(\cdot)$, and hence can be estimated using Lemma \ref{1-hat-lemma}. The second factor  $|\pmb{\Upsilon}_0 - \Omega_1|$ has already been estimated in Lemma \ref{error-lemma}.
Invoking the estimates \eqref{1-hat-bound} and \eqref{difference} from Lemmas \ref{1-hat-lemma} and \ref{error-lemma} respectively, we obtain 
\begin{align} 
|\mathtt z_{2}(\xi)| &= \bigl| \widehat{\mathtt 1}(\xi \mathtt t^{-1}) \bigr| \times \bigl| \Omega_1 (\xi) - \Upsilon_{0}(\xi)\bigr| \nonumber \\
&\leq C_0 \min \left[1, \frac{\mathtt t}{|\xi|}\right] \min \left[1, \frac{|\xi|}{\mathtt s^{\mathtt a}}\right] \leq C_0  \min \left[1, \frac{\mathtt t}{|\xi|}, \frac{|\xi|}{\mathtt s^{\mathtt a}}\right]. \label{vm2-est}
\end{align} 
Combining \eqref{vm1-est} and \eqref{vm2-est} with \eqref{vm-decomp} results in the desired conclusion \eqref{u-pointwise-estimate}. 
\end{proof} 

\section{Estimating $\mathcal{U}_m$} \label{estimating-um-section} 
The task of estimating the sum $\mathscr{U}_m$ defined in \eqref{def-um} involves a study of its summands $\mathtt u_m$. In Section \ref{u-pointwise-estimate-section}, we have recorded a pointwise estimate for the function $\mathtt u$ that can be lifted directly to obtain a similar estimate for $\mathtt u_m$. We state this in Lemma \ref{um-est-lemma} below in the form that we will apply. This estimate is the only ingredient needed to produce the bound for $\mathscr{U}_m$ stated in Proposition \ref{U-prop}. We complete the proof of Proposition \ref{U-prop} in this section. 
\begin{lemma} \label{um-est-lemma}
The function $\mathtt u_m$ defined in \eqref{def-um} admits the following estimate:
\begin{equation} \label{um-est}
|\mathtt u_m(\xi)| \leq C_0 \min \left[1, \frac{\mathtt s_{m-1}^{\mathtt b_{m-1}}}{|\xi|}, \frac{|\xi|}{\mathtt s_m^{\mathtt a_m}}\right], \quad \xi \in \mathbb Z. 
\end{equation} 
\end{lemma} 
\begin{proof} 
This is a re-statement of \eqref{u-pointwise-estimate} in Lemma \ref{u-lemma}, with $\mathtt t = \mathtt s_{m-1}^{\mathtt b_{m-1}}$, $\mathtt s = \mathtt s_m$ and $\mathtt a = \mathtt a_m$.   
\end{proof} 
\noindent {\em{Remark: }} The bound \eqref{um-est} of $|\mathtt u_m(\xi)|$ depends only on $|\xi|$, and therefore cannot account for finer estimates that may result from special arithmetic properties of $\xi$. 
This will be an important point of distinction from $\mathtt v_m$.  \label{uniform-bound-remark}  
\subsection{Proof of Proposition \ref{U-prop}}\label{V-prop-proof}  
The inequality \eqref{Um-sum} claimed in Proposition \ref{U-prop} involves a triple sum $\mathcal U_m$. The inner sum $\mathscr{U}_m$ given by \eqref{def-um} depends in turn on the summands $\mathtt u_m$, which we have studied in Section \ref{u-pointwise-estimate-section} and Lemma \ref{um-est-lemma}.   
The estimate \eqref{um-est} obtained in Lemma \ref{um-est-lemma} provides two non-trivial upper bounds for $\mathtt u_m(\xi)$. 
This motivates our treatment of $\mathscr{U}_m$; we decompose the defining sum of $\mathscr{U}_m$ into two sub-sums over $(u, v)$ in complementary regimes, where one of these bounds is smaller than the other. This is formalized as follows; set     
\begin{align} 
c_m &:= \frac{\mathtt a_m}{2} \log_b(\mathtt s_m) \; \text{ and decompose } \; \mathbb Z_N^2 = \mathbb U_1 \sqcup \mathbb U_2, \text{ where }  \label{def-cm} \\ 
\mathbb U_1 &:= \bigl\{(u, v) \in \mathbb Z_{N}^2 : b^{u+v} \leq \mathtt s_m^{\frac{\mathtt a_m}{2}}\bigr\} = \bigl\{(u, v) \in \mathbb Z_N^2 : u+v \leq c_m \bigr\}, \nonumber \\  
\mathbb U_2 &:= \bigl\{(u, v) \in \mathbb Z_N^2 : b^{u+v} > \mathtt s_m^{\frac{\mathtt a_m}{2}}\bigr\} = \bigl\{(u,v) \in \mathbb Z_N^2 : u + v > c_m \bigr\}. \nonumber
\end{align} 
According to the definition \eqref{def-um}, the set $\mathbb Z_{N}^2$ is the domain of summation of $\mathscr{U}_m$. The splitting \eqref{def-cm} of $\mathbb Z_{N}^2$ therefore leads to a corresponding decomposition of the sum $\mathscr{U}_m$:
\begin{equation} 
\mathscr{U}_m = \mathscr{U}_{m1} + \mathscr{U}_{m2}, 
\end{equation}  
where $\mathscr{U}_{mj}(N; h, b)$ represents the sum of $\bigl|\mathtt u_m(hb^u (b^v-1)) \bigr|$ over all indices $(u, v) \in \mathbb U_j$, $j = 1,2$. For $\mathscr{U}_{m1}$, where the argument of $\mathtt u_m(\cdot)$ is relatively small, we invoke the bound $|\mathtt u_m(\xi)| \leq C_0 |\xi| \mathtt s_m^{-\mathtt a_m}$ from Lemma \ref{um-est-lemma}. This gives 
\begin{align}
\mathscr{U}_{m1} &\leq \sum_{(u,v) \in \mathbb U_1} \bigl| \mathtt u_m \bigl(h b^u (b^v -1) \bigr)\bigr| \leq C_0 \sum_{(u,v) \in \mathbb U_1}\frac{|h b^u(b^v-1)|}{\mathtt s_m^{\mathtt a_m}} \nonumber \\ 
&\leq C_0 |h| \sum_{(u, v) \in \mathbb U_1} { b^{u+v}}{\mathtt s_m^{-\mathtt a_m}} \leq C(h) \mathtt s_m^{-\mathtt a_m} \sum_{u=0}^{N-1} \sum_{r=0}^{\lfloor c_m \rfloor} b^{r} \nonumber \\
&\leq C(h, b) N \mathtt s_m^{-\mathtt a_m} b^{c_m} \leq C(h, b) N \mathtt s_m^{-\frac{\mathtt a_m}{2}}. \label{Um1-est}
\end{align}
The fourth inequality above results from a change of variable in the double sum, replacing the summation index $(u,v)$ by $(u, r)$, with $r = u+v$. The corresponding inequality then follows from the inclusion $\mathbb U_1 \subseteq \{(u, v): u \in \mathbb Z_N, u+v \leq c_m\}$. The last step uses the definition \eqref{def-cm} of $c_m$, which says that $b^{c_m} = \mathtt s_m^{\mathtt a_m/2}$. 
\vskip0.1in
\noindent The estimation for $\mathscr{U}_{m2}$ is similar, except this time the bound $|\mathtt u_m(\xi)| \leq C_0 \mathtt s_{m-1}^{\mathtt b_{m-1}}/|\xi|$ from Lemma \ref{um-est-lemma} is better. A sequence of analogous computations leads to 
\begin{align}
\mathscr{U}_{m2} &\leq \sum_{(u,v) \in \mathbb U_2} \bigl| \mathtt u_m \bigl(h b^u (b^v -1) \bigr)\bigr| \leq C_0 \sum_{(u,v) \in \mathbb U_2} \frac{\mathtt s_{m-1}^{\mathtt b_{m-1}}}{|h b^u (b^v-1)|} \nonumber \\
&\leq \frac{C(b)}{|h|} \mathtt s_{m-1}^{\mathtt b_{m-1}} \sum_{(u,v) \in \mathbb U_2}  b^{-(u+v)} \leq C(h, b) \mathtt s_{m-1}^{\mathtt b_{m-1}} \sum_{u=1}^{N} \sum_{r= \lfloor c_m \rfloor}^N b^{-r} \nonumber \\ &\leq C(h,b) \mathtt s_{m-1}^{\mathtt b_{m-1}} N b^{-c_m}  
\leq  C(h, b) N \mathtt s_{m-1}^{\mathtt b_{m-1}} \mathtt s_{m}^{-\frac{\mathtt a_m}{2}}. \label{Um2-est}
\end{align} 
Inserting \eqref{Um1-est} and \eqref{Um2-est} into the expression for $\mathcal U_{m2}$, we obtain 
\[ \mathcal U_m(h, b) = \sum_{N=1}^{\infty} N^{-3} \mathscr{U}_{m}(N; h, b) \leq C(h, b) \mathtt s_{m-1}^{\mathtt b_{m-1}} \mathtt s_{m}^{-\frac{\mathtt a_m}{2}} \sum_{N=1}^{\infty} N^{-2} 
\leq C(h, b)   \mathtt s_{m-1}^{\mathtt b_{m-1}} \mathtt s_{m}^{-\frac{\mathtt a_m}{2}}. \] 
This is the estimate claimed in \eqref{Um-sum}.
\qed
\vskip0.1in
\noindent As noted in the remark following Proposition \ref{U-prop}, its proof works for any choice of $b \in \mathbb N \setminus \{1\}$, independent of $\mathtt s_m$. 

\section{A pointwise estimate for $\mathtt v$: Take 1} \label{v-pointwise-section} 
We now turn our attention to Proposition \ref{V-prop}, which involves the estimation of the sum $\mathcal V_m$ defined in \eqref{Vm-sum}. The procedure shares some similarities with the treatment of $\mathcal U_m$. Like $\mathtt u_m$, the summands $\mathtt v_m(\xi)$ of $\mathscr{V}_m$ obey pointwise estimates depending only on $|\xi|$. Recorded in Section \ref{w-pointwise-section} below, these estimates will turn out to be useful in controlling parts of the sum $\mathcal{V}_m$. 
\vskip0.1in 
\noindent The feature that distinguishes $\mathcal V_m$ from $\mathcal U_m$ is this.  Since the pointwise estimates for $\mathtt v_m$ originating in this section do not rely on any special number-theoretic property of $\mathtt s_m$ and $b$ such as multiplicative independence, they cannot entirely account for the proof of Proposition \ref{V-prop}. However, by dispatching certain sub-sums of $\mathcal{V}_m$, these pointwise estimates reduce the problem of estimating $\mathcal V_m$ to considering a finite sub-sum of $\mathscr{V}_m$ localized on intermediate frequencies. Different tools are needed for studying the latter sum. 
\subsection{An identity} 
\noindent As in the case for $\mathtt u$, we first simplify the expression \eqref{uv-def} for $\mathtt v$. Combining \eqref{Phi1-hat}, \eqref{factorization} and \eqref{intermediate-Fourier-coefficient-2} yields 
\begin{align}
\mathtt v(\xi) & = \widehat{\Phi}_1(\xi) - \widehat{\Psi}_1(\xi) = \Omega_1(\xi) \times \mathfrak A(\xi) \times \bigl[ \widehat{1}(\xi \mathtt s^{-\mathtt b}) \mathfrak B(\xi) - \widehat{1}(\xi \mathtt s^{-\mathtt a}) \bigr],\; \xi \in \mathbb{Z}, \label{identity}  \\  &\text{ so that } \bigl| \mathtt v(\xi) \bigr| \leq \bigl|\mathfrak A(\xi) \bigr| \times \bigl| \mathfrak w(\xi) \bigr|, \; \; \text{ with } \;
\mathfrak w(\xi) :=  \widehat{1}(\xi \mathtt s^{-\mathtt b}) \mathfrak B(\xi) - \widehat{1}(\xi \mathtt s^{-\mathtt a}). \label{def-wm}
\end{align}  
The functions $\mathfrak A, \mathfrak B$ in the display above are the geometric sums of complex exponentials defined in \eqref{def-A}, \eqref{def-B} respectively. We observe that the factor $\mathfrak A$ is tied to the first part of the elementary operation $\mathscr{O}$, whereas $\mathfrak w$ corresponds to the second part. Each factor will play an important role in the proof of normality. 
\vskip0.1in 
\noindent The relation \eqref{def-wm} says that in order to estimate $\mathtt v$, we need
estimates for both $\mathfrak A$ and $\mathfrak w$. Let us first record the latter function in a simplified form. 
\begin{lemma} 
The function $\mathfrak w(\xi)$ defined in \eqref{def-wm} admits the following form:
\begin{align} \label{wm-alternate}
\mathfrak w(\xi) &= \varepsilon \widehat{\mathtt 1}(\xi \mathtt s^{-\mathtt b}) \Bigl[ \mathfrak B^{\ast}(\xi) - \mathfrak C(\xi) \Bigr] = \underline{\mathfrak w}_1(\xi) + \underline{\mathfrak w}_2(\xi),  \text{ with }\; \xi \in \mathbb{Z},\\
\underline{\mathfrak w}_1(\xi) &:=  \varepsilon \widehat{\mathtt 1}(\xi \mathtt s^{-\mathtt b}) \Bigl[ \mathfrak B^{\ast}(\xi) - 1 \Bigr], \quad \underline{\mathfrak w}_2(\xi) := \varepsilon \widehat{\mathtt 1}(\xi \mathtt s^{-\mathtt b}) \Bigl[  1 - \mathfrak C(\xi) \Bigr], \label{def-w1w2}
\end{align} 
where $\mathfrak B^{\ast}$ and $\mathfrak C$ are the exponential sums defined in \eqref{B*C} and $\varepsilon \in (0,1)$ is the ``bias" parameter as in Section \ref{param-elem-op-section}. 
\end{lemma}
\begin{proof} 
The second equality in \eqref{wm-alternate} is a direct consequence of the first, with the given definitions of $\underline{\mathfrak w}_1, \underline{\mathfrak w}_2$ in \eqref{def-w1w2}. It suffices therefore to just prove the first identity. Combining the definition \eqref{def-wm} of $\mathfrak w$ with the expression \eqref{def-B} for $\mathfrak B$, we find that 
\[ \mathfrak w(\xi) = \varepsilon \widehat{\mathtt 1}(\xi \mathtt s^{-\mathtt b}) \Bigl[ \mathfrak B^{\ast}(\xi) - \mathfrak C(\xi) \Bigr] + \Bigl[ \mathfrak C(\xi) \widehat{\mathtt 1}(\xi \mathtt s^{-\mathtt b}) - \widehat{\mathtt 1}(\xi \mathtt s^{-\mathtt a})\Bigr], \]
where $\mathfrak B^{\ast}$ and $\mathfrak C$ are defined in \eqref{B*C}. The right hand side of the display above is a sum of two terms, the first of which coincides with the claimed expression for $\mathfrak w$ in \eqref{wm-alternate}. The desired formula \eqref{wm-alternate} for $\mathfrak w$ thus follows from the claim that the second term is zero, namely
\begin{equation} 
\mathfrak C(\xi) \widehat{\mathtt 1}(\xi \mathtt s^{-\mathtt b}) - \widehat{\mathtt 1}(\xi \mathtt s^{-\mathtt a}) \equiv 0 \quad \text{ for all } \xi \in \mathbb Z.  \label{second-term-zero-claim}
\end{equation}
\noindent Lemma \ref{BC-lemma} provides a closed-form expression for $\mathfrak C(\xi)$ in three mutually exclusive and exhaustive scenarios depending on the divisibility properties of $\xi$ in terms of powers of $\mathtt s$.  
We will verify the claim \eqref{second-term-zero-claim} in each of these three cases. 
\begin{itemize} 
\item First suppose that $\mathtt s^{\mathtt b} \mid \xi$. It follows from the expression \eqref{1-hat} for $\widehat{\mathtt 1}$ that $\widehat{1}(\xi \mathtt s^{-\mathtt b}) =  \widehat{1}(\xi \mathtt s^{-\mathtt a}) = 0$ since $\mathtt a < \mathtt b$. Therefore each term in the left hand side of \eqref{second-term-zero-claim} vanishes.
\item Next suppose that  $\mathtt s^{\mathtt a} \mid \xi$, but $\mathtt s^{\mathtt b} \nmid \xi$. Then $\widehat{1}(\xi \mathtt s^{-\mathtt a}) = 0$ from \eqref{1-hat}, whereas $\mathfrak C(\xi) = 0$ from \eqref{C} in Lemma \ref{BC-lemma}. Once again, each term on the left side of \eqref{second-term-zero-claim} vanishes. 
\item Finally, let $\mathtt s^{\mathtt a} \nmid \xi$. Then we obtain from \eqref{C} in Lemma \ref{BC-lemma} that
\begin{align*} 
\mathfrak C(\xi) \times \widehat{\mathtt 1}(\xi \mathtt s^{-\mathtt b}) &= \Bigl[\mathtt s^{\mathtt a - \mathtt b} \frac{1 - e(\xi \mathtt s^{-\mathtt a})}{1 - e(\xi \mathtt s^{- \mathtt b})} \Bigr] \times \Bigl[\frac{1 - e \bigl(\xi \mathtt s^{-\mathtt b}\bigr)}{2 \pi i \xi \mathtt s^{-\mathtt b}} \Bigr] \\
&= \frac{1 - e(\xi \mathtt s^{-\mathtt a})}{2 \pi i \xi \mathtt s^{-\mathtt a}} = \widehat{\mathtt 1}(\xi \mathtt s^{-\mathtt a}).
\end{align*}
\end{itemize}
This concludes the proof of the claim \eqref{second-term-zero-claim}, and therefore the proof of the lemma.  
\end{proof}
\subsection{The function $\underline{\mathfrak w}_1$} \label{w1-section}
The expression \eqref{def-w1w2} identifies the function $\underline{\mathfrak w}_1$ as dependent on $\mathfrak B^{\ast}$ defined in \eqref{B*C}. In Lemma \ref{BC-lemma} of Section \ref{C-section}, we presented a way of expressing a certain sum of exponentials as a product. An argument similar to Lemma \ref{BC-lemma} yields a product formula for $\mathfrak B^{\ast}$, recorded here. Let us recall, from \eqref{special-digits} and \eqref{L*} respectively, that $\mathtt r = \# \mathscr{D} = \# [\mathscr{D}(\mathtt s)]< \mathtt s$ and elements $\ell \in \mathbb{L}^\ast$ are identified via the representation $\ell = \sum_{j = 0}^{\mathtt b-\mathtt a - 1} \mathtt d_j \mathtt s^j$ with $\mathtt d_j \in \mathscr{D}$. Using these relations, we obtain 
\begin{align} 
\mathfrak B^{\ast}(\xi) &= \mathtt r^{\mathtt a - \mathtt b} \sum_{\mathbf d \in \mathscr{D}^{\mathtt b -\mathtt a}} e \Bigl[ \frac{\xi}{\mathtt s^{\mathtt b}} \sum_{j=0}^{\mathtt b- \mathtt a -1} \mathtt d_j \mathtt s^j \Bigr] \nonumber \\ &= \mathtt r^{\mathtt a - \mathtt b}\prod_{j=0}^{\mathtt b -\mathtt a -1} \Bigl[ \sum_{\mathtt d \in \mathscr{D}} e\Bigl(\xi \mathtt d \mathtt s^{j - \mathtt b} \Bigr) \Bigr]  = \prod_{j= \mathtt a +1}^{\mathtt b} \mathfrak g(\xi, j; \mathtt s, \mathscr{D}), \; \text{ where } \label{B*-rep} \\  \mathfrak g(\xi, j) &= \mathfrak g(\xi, j ; \mathtt s,  \mathscr{A} ) := \frac{1}{\#( \mathscr{A} )}\sum_{\mathtt d \in  \mathscr{A} } e\bigl(\xi \mathtt d \mathtt s^{-j} \bigr) \; \text{ for }  \mathscr{A}  \subseteq  \mathbb Z_{\mathtt s}, \; j \in \mathbb N. \label{exp-sum-g-def} 
\end{align}
\begin{lemma} \label{(B*-1)-lemma}
There exists an absolute constant $C_0 > 0$ for which the function $\underline{\mathfrak w}_1$ given by \eqref{def-w1w2}
obeys the following estimate uniformly in $\varepsilon\in (0,1)$:  
 \begin{equation} \label{(B*-1)-estimate}
\bigl| \underline{\mathfrak w}_1(\xi) \bigr| \leq C_0 \min \Bigl[1, \frac{\mathtt s^{\mathtt b}}{|\xi|}\Bigr] \times \min \Bigl[1, \frac{|\xi|}{\mathtt s^{\mathtt a}}\Bigr],\; \xi \in \mathbb{Z}. 
 \end{equation} 
\end{lemma}  
\begin{proof} 
The defining relation \eqref{def-w1w2} of $\underline{\mathfrak w}_1$ implies 
\begin{equation} \label{w1-preliminary-estimate}
\bigl| \underline{\mathfrak w}_1(\xi)\bigr| \leq  \bigl| \widehat{1}(\xi \mathtt s^{-\mathtt b}) \bigr| \times |\mathfrak B^{\ast}(\xi) - 1|, \quad \text{ uniformly in } \varepsilon\in(0,1). 
\end{equation} 
The estimate \eqref{1-hat-bound} on $\widehat{1}$ from Lemma \ref{1-hat-lemma} leads to the first factor on the right hand side of \eqref{(B*-1)-estimate}. It therefore suffices to show that the second factor generates an upper bound for $|\mathfrak B^{\ast} -1|$.
\vskip0.1in 
\noindent The representation \eqref{B*-rep} permits the expansion of $\mathfrak B^{\ast}$ as a telescoping sum, using progressively fewer factors of the form $\mathfrak g$. Each factor $\mathfrak g$ is an average of uni-modular terms, and therefore at most one in absolute value. This results in the following estimate: 
\begin{align} 
\bigl|\mathfrak B^{\ast}(\xi) - 1 \bigr|  &= \Bigl| \Bigl[\mathfrak g(\xi, \mathtt a+1) - 1\Bigr] + \sum_{k = \mathtt a+1}^{\mathtt b -1} \Bigl[ \prod_{j=\mathtt a+1}^{k+1} \mathfrak g(\xi, j) - \prod_{j=\mathtt a+1}^{k} \mathfrak g(\xi, j) \Bigr] \Bigr| \nonumber \\  
&\leq  \bigl|\mathfrak g(\xi, \mathtt a+1) - 1\bigr| +   \sum_{k = \mathtt a+1}^{\mathtt b -1} \bigl|\mathfrak g(\xi, k+1) - 1\bigr| = \sum_{k=\mathtt a}^{\mathtt b} \bigl| \mathfrak g(\xi, k) -1 \bigr| \nonumber \\
&\leq |\xi| \sum_{k = \mathtt a+1}^{\mathtt b} \mathtt s^{-(k-1)} \leq C_0 |\xi| \mathtt s^{-\mathtt a}. \label{B*-1-final}
\end{align}  
The penultimate inequality follows from 
\[ \bigl|\mathfrak g(\xi, k) - 1 \bigr| \leq \frac{1}{\mathtt r} \sum_{\mathtt d \in \mathscr{D}} \bigl| e \bigl(\xi \mathtt d \mathtt s^{-k} \bigr) - 1\bigr| \leq \frac{1}{\mathtt r} \sum_{\mathtt d \in \mathscr{D}} |\xi| \mathtt d \mathtt s^{-k} \leq |\xi| \mathtt s^{-(k-1)}\;  \text{with}\; \mathtt d < \mathtt s,\]
applied for $k = \mathtt a+1, \ldots, \mathtt b$. Combining \eqref{B*-1-final} with the trivial bound \eqref{B*C-trivial-bound} on $\mathfrak B^{\ast}$, we obtain
\[ \bigl| \mathfrak B^{\ast}(\xi) -1 \bigr| \leq C_0 \min \Bigl[1, \frac{|\xi|}{\mathtt s^{\mathtt a}} \Bigr]. \] 
Inserting this into \eqref{w1-preliminary-estimate} leads to the estimate \eqref{(B*-1)-estimate}. 
 \end{proof}  
\subsection{Estimates for the function $\underline{\mathfrak w}_2$} \label{w2-section}
Like $\underline{\mathfrak w}_1$, the function $\underline{\mathfrak w}_2$ defined in \eqref{def-w1w2} depends on one of the elementary functions introduced in \eqref{B*C}, namely $\mathfrak C$. Section \ref{C-section} provided an explicit formula for the function $\mathfrak C$. The following estimate for $\underline{\mathfrak w}_2$ is a direct consequence of Lemma \ref{BC-lemma} in that section. 
\begin{lemma} \label{1-hat-C-corollary}
There exists an absolute constant $C_0 > 0$ such that 
the following estimate holds uniformly in $\varepsilon\in (0,1)$:
\begin{equation} \label{1-hat-C} 
\bigl| \underline{\mathfrak w}_2(\xi) \bigr| \leq \bigl|\widehat{\mathtt 1}(\xi \mathtt s^{-\mathtt b}) \bigr| \times \bigl|\mathfrak C(\xi) - 1\bigr| \leq C_0 \min \left[1, \frac{|\xi|}{\mathtt s^{\mathtt a}}, \frac{\mathtt s^{\mathtt b}}{|\xi|} \right],\; \xi \in \mathbb{Z}.
\end{equation} 
\end{lemma} 
\begin{proof} 
Motivated by \eqref{C}, we consider the expression of the left side of \eqref{1-hat-C} in the three cases that lead to different values of $\mathfrak C$.
\begin{itemize} 
\item If $\mathtt s^{\mathtt b} \mid \xi$, then \eqref{C} dictates that $\mathfrak C(\xi) = 1$, in which case the relation \eqref{1-hat-C} holds trivially, since the expression on its left side is zero. 
\item  If $\mathtt s^{\mathtt a} \mid \xi, \mathtt s^{\mathtt b} \nmid \xi$, then $\mathfrak C(\xi) = 0$ by \eqref{C}. Combining this with Lemma \ref{1-hat-lemma} we obtain 
\[ \bigl|\widehat{\mathtt 1}(\xi \mathtt s^{-\mathtt b}) \bigr| \times \bigl|\mathfrak C(\xi) - 1\bigr| =  \bigl|\widehat{\mathtt 1}(\xi \mathtt s^{-\mathtt b}) \bigr| \leq  \min \Bigl[1, \frac{\mathtt s^{\mathtt b}}{\pi |\xi|} \Bigr] \leq C_0   \min \Bigl[1, \frac{\mathtt s^{\mathtt b}}{|\xi|}, \frac{|\xi|}{\mathtt s^{\mathtt a}} \Bigr]. \] 
The final inequality uses the assumption $\mathtt s^{\mathtt a} \mid \xi$, which means $|\xi| \mathtt s^{-\mathtt a} \geq 1$. 
\item Finally, suppose that $\mathtt s^{\mathtt a} \nmid \xi$. Then it follows from the expression \eqref{1-hat} of $\widehat{\mathtt 1}$ and \eqref{C} that 
\begin{align*} 
\bigl|\widehat{\mathtt 1}(\xi \mathtt s^{-\mathtt b}) \bigr| \times \bigl|\mathfrak C(\xi) - 1\bigr| &= \left| \frac{1 - e(\xi \mathtt s^{-\mathtt b})}{2 \pi i \xi \mathtt s^{-\mathtt b}}\right| \times \left| \mathtt s^{\mathtt a - \mathtt b} \frac{1 - e(\xi \mathtt s^{-\mathtt a})}{1 - e(\xi \mathtt s^{-\mathtt b})} - 1 \right| \\ 
&= \Bigl| \widehat{\mathtt 1}(\xi \mathtt s^{-\mathtt a}) - \widehat{\mathtt 1}(\xi \mathtt s^{-\mathtt b})\Bigr|  \leq C_0 \min \left[1, \frac{|\xi|}{\mathtt s^{\mathtt a}}\right] \times \min \Bigl[1, \frac{\mathtt s^{\mathtt b}}{|\xi|} \Bigr] \\
&\leq C_0 \min \Bigl[1, \frac{|\xi|}{\mathtt s^{\mathtt a}}, \frac{\mathtt s^{\mathtt b}}{|\xi|} \Bigr].
\end{align*} 
The penultimate inequality is a consequence of \eqref{1-hat-difference-estimate}, with $\theta_1 = \xi \mathtt s^{-\mathtt b}$, $\theta_2 = \xi \mathtt s^{-\mathtt a}$.  
\end{itemize} 
The estimates obtained in all three cases are consistent with \eqref{1-hat-C}, completing the proof. 
\end{proof} 

\subsection{A pointwise estimate for $\mathtt v$ and $\mathfrak w$} \label{w-pointwise-section}  
The results from Sections \ref{w1-section} and \ref{w2-section} can be combined to infer a pointwise estimate for $\mathtt v$ that will used shortly in Section \ref{estimating-vm-section} to control the sum $\mathcal V_m$. 
\begin{lemma} \label{v-lemma} 
There exists an absolute constant $C_0 > 0$ such that the function $\mathtt v$ given by \eqref{uv-def} (or \eqref{identity}) obeys the estimate:
\begin{equation} \label{v-pointwise-estimate}
\bigl| \mathtt v(\xi) \bigr| \leq \bigl|\mathfrak w(\xi) \bigr| 
\leq C_0 \min \Bigl[1, \frac{|\xi|}{\mathtt s^{\mathtt a}}, \frac{\mathtt s^{\mathtt b}}{|\xi|}\Bigr],\; \xi\in\mathbb{Z}.
\end{equation}  
\end{lemma} 
\begin{proof}
In view of the trivial estimate $|\mathfrak A| \leq 1$ from \eqref{B*C-trivial-bound}, the relation \eqref{def-wm} implies that $|\mathtt v(\xi)| \leq |\mathfrak w(\xi)|$.  
Combining this with \eqref{wm-alternate}, and inserting the estimates for $\underline{\mathfrak w}_1, \underline{\mathfrak w}_2$ from \eqref{(B*-1)-estimate}, \eqref{1-hat-C} (Lemmas \ref{(B*-1)-lemma}, \ref{1-hat-C-corollary}) results in  
\begin{align*} 
 \bigl|\mathfrak w(\xi)  \bigr| & \leq \bigl| \underline{\mathfrak w}_1(\xi) \bigr| + \bigl|\underline{\mathfrak w}_2(\xi)  \bigr| \\ 
&\leq C_0  \min \Bigl[1, \frac{\mathtt s^{\mathtt b}}{|\xi|}\Bigr] \times \min \Bigl[1, \frac{|\xi|}{\mathtt s^{\mathtt a}}\Bigr] + C_0 \min \Bigl[1, \frac{|\xi|}{\mathtt s^{\mathtt a}}, \frac{\mathtt s^{\mathtt b}}{|\xi|} \Bigr] \\ 
&\leq C_0 \min \Bigl[1, \frac{|\xi|}{\mathtt s^{\mathtt a}}, \frac{\mathtt s^{\mathtt b}}{|\xi|} \Bigr],
\end{align*}
which is the claimed bound \eqref{v-pointwise-estimate}. 
\end{proof}

\section{Decomposition of $\mathcal{V}_m$ into multiple frequency ranges} \label{estimating-vm-section}
In Lemma \ref{u-lemma}, we obtained a pointwise estimate of $\mathtt u$, which translated to a pointwise estimate for $\mathtt u_m$ in Lemma \ref{um-est-lemma}. Exactly in the same way, Lemma \ref{v-lemma} gives rise to a pointwise estimate for $\mathtt v_m$ below, by substituting $\mathtt s= \mathtt s_m$, $\mathtt a = \mathtt a_m$, $\mathtt b = \mathtt b_m$.  
\begin{lemma} \label{vm-est-lemma}
There is an absolute constant $C_0 > 0$ such that for any choice of parameters $\pmb{\Pi}$ obeying \ref{param-seq}, 
\eqref{abs} and \eqref{rmsm-assumption}, 
\begin{equation} \label{vm-pointwise-estimate} 
\bigl| \mathtt v_m(\xi) \bigr| \leq C_0 \min \left[1, \frac{|\xi|}{\mathtt s_m^{\mathtt a_m}}, \frac{\mathtt s_m^{\mathtt b_m}}{|\xi|}\right], \quad m \geq 1, \; \xi \in \mathbb Z.   
\end{equation}  
\end{lemma} 
\noindent The bound on $\mathtt v_m$ recorded in \eqref{vm-pointwise-estimate} is free of the arithmetic properties of $\xi$, by the same reasoning as the remark on page \pageref{uniform-bound-remark} following Lemma \ref{um-est-lemma}. This section will isolate the sub-sums of $\mathcal{V}_{m}$ that are amenable to the pointwise bound \eqref{vm-pointwise-estimate}, thereby identifying the portion of the sum that rely on special number-theoretic properties of the relevant bases. The latter sum is estimated in the subsequent sections using different tools from number theory. With this goal in mind, we partition the domain of summation of $\mathcal V_m = \mathcal V_m(h, b)$ as follows: 
\begin{equation} \label{Vm-decomposition} 
\mathbb V := \bigl\{(u, v, N) : (u, v) \in \mathbb Z_N^2, \; N \in \mathbb N \bigr\} = \mathbb V_1 \sqcup \mathbb V_2, \text{ with } \mathcal V_m = \mathcal V_{m1} + \mathcal V_{m2} 
\end{equation}  
The sets $\mathbb V_1, \mathbb V_2$ of the partition will be described momentarily. The quantity $\mathcal V_{mj} = \mathcal V_{mj}(h, b)$ denotes the sum of $N^{-3} |\mathtt v_m(h b^u(b^v-1))|$ over $\mathbb V_j$ for $j = 1, 2$. The sum $\mathcal V_{m1}$ can be handled using \eqref{vm-pointwise-estimate} alone, but not $\mathcal V_{m2}$. 
\vskip0.1in 
\noindent Let us describe the decomposition \eqref{Vm-decomposition} in a bit more detail. The set $\mathbb V_1$ covers very low or very high values either of the frequency $hb^{u}(b^v-1)$ or of $N$. It is described as a four-fold union: 
\begin{align} 
\mathbb V_1 := \bigcup_{j=1}^{4} \mathbb V_{1j}, \; &\text{ so that the sum $\mathcal V_{m1}$ can be estimated as } \;  \mathcal V_{m1} \leq \sum_{j=1}^{4} \mathcal W_j, \text{ where } \label{sum-of-W} \\ \mathcal W_{j} &:= \sum \Bigl\{\frac{1}{N^3}|\mathtt v_m(h b^u(b^v-1))| : (u, v, N) \in \mathbb V_{1j} \Bigr\}.
\label{Wj} \end{align} 
The subsets $\mathbb V_{1j}$ are determined depending on the size of $b^{u+v}$ or $N$:
\begin{align} 
\mathbb V_{11} &:= \bigl\{(u, v, N) \in \mathbb V: b^{u+v} \leq \mathtt s_m^{\frac{\mathtt a_m}{2}}\bigr\}, \; 
\mathbb V_{12} := \bigl\{(u, v, N) \in \mathbb V : b^{u+v} \geq m^2 \mathtt s_m^{\mathtt b_m}\bigr\}, \nonumber\\
\mathbb V_{13} &:=  \bigl\{(u, v, N) \in \mathbb V : N \leq \frac{\mathtt a_m}{4} \log_{b} (\mathtt s_m) \bigr\}, \nonumber\\
\mathbb V_{14} &:= \bigl\{ (u, v, N) \in \mathbb V : N \geq m \mathtt b_m \log_{b} (\mathtt s_m) \bigr\}.
\end{align}
The collections $\mathbb V_{1j}$ are distinct in the sense that none of them is contained in another, but they are not necessarily disjoint. The intermediate values of $b^{u+v}$ and $N$ are collected in $\mathcal V_{m2}$:
\begin{align} 
\mathcal V_{m2} &= \mathcal V_{m2}(h, b) := \sum \Bigl\{\frac{1}{N^{3}} \bigl| \mathtt v_m \bigl(h b^u(b^v-1)\bigr) \bigr| : (u, v, N) \in \mathbb V_2 \Bigr\}, \text{ where } \label{Vm2-sum-definition} \\  
\mathbb V_2 &:= \left\{(u, v, N) \in \mathbb V \; \Biggl| \; \begin{aligned} &\mathtt s_m^{\frac{\mathtt a_m}{2}} < b^{u+v} < m^2 \mathtt s_m^{\mathtt b_m}, \text{ and } \\  
& \frac{\mathtt a_m}{4} \log_{b} (\mathtt s_m) < N < m \mathtt b_m \log_{b} (\mathtt s_m)
\end{aligned}  \right\}. \label{def-collection-V2}
\end{align} 
The main results concerning $\mathcal V_{mj}$ appear in Section \ref{Vprop-proof-section-part1}, and are similar to $\mathcal U_m$. Namely for $j=1,2$, there is a summable sequence $\{\mathtt c_m(h, b, j) : m \geq 1\}$ that dominates $\mathcal V_{mj}$ element-wise: 
\begin{itemize} 
\item For $\mathcal V_{m1}$, this is true for {\em{any}} two bases $b$ and $\mathtt s_m$. 
\item For $\mathcal V_{m2}$, the base $b$ has to be multiplicatively independent of $\mathtt s_m$. 
\end{itemize} 
\subsection{Conditional proof of Proposition \ref{V-prop}}  \label{Vprop-proof-section-part1} 
In light of the decomposition \eqref{Vm-decomposition} of $\mathcal V_m(h,b)$, it is clear that Proposition \ref{V-prop} follows from the two propositions below. Both propositions are similar in their statements and jointly lead to Proposition \ref{V-prop}, but they rely on quite different reasoning. 
\begin{proposition} \label{W-sum-prop}
Let $\pmb{\Pi}$ be any choice of parameters obeying \eqref{param-seq}, \eqref{abs} and \eqref{rmsm-assumption}. Then for any $h \in \mathbb Z \setminus \{0\}$, $b \in \mathbb N \setminus \{1\}$, there exists a sequence $\{\mathtt c_m(h, b) : m \geq 1\} \subseteq (0, \infty)$ such that   
 \begin{equation} \label{W-sum-estimate}
 \mathcal V_{m1} \leq \sum_{j=1}^{4} \mathcal W_j \leq \mathtt c_m(h, b) \; \; \text{ for all } m \geq 1, \quad\text{with}\; \;  \sum_{m=1}^{\infty} \mathtt c_m(h, b) < \infty. 
 \end{equation} 
\end{proposition} 
\begin{proposition} \label{V2-sum-prop}
Suppose that exactly one of the following assumptions is true. 
\begin{enumerate}[(a)] 
\item The parameter $\pmb{\Pi}$ obeys the hypotheses of Proposition \ref{normality-special-prop}, and  
the base $b$ lies in $\mathscr{B}_1$, the collection of bases given by \eqref{bases-B1-def}. \label{V2-sum-prop-parta}
\item  The parameter $\pmb{\Pi}$ obeys the hypotheses of Proposition \ref{normal-prop}, 
and the base $b$ lies in $\mathscr{B}$, the collection of bases given by \eqref{bases-BC-bar}. \label{V2-sum-prop-partb}
\end{enumerate} 
For each pair $(\pmb{\Pi}, b)$ obeying either \eqref{V2-sum-prop-parta} or \eqref{V2-sum-prop-partb} above, and any $h \in \mathbb Z \setminus \{0\}$, there exists a sequence $\{\mathtt c_m(h, b) : m \geq 1\} \subseteq (0, \infty)$ for which the sum $\mathcal V_{m2}$ given by \eqref{Vm2-sum-definition} satisfies the conclusion
\begin{equation} \label{V2-sum-estimate}
\mathcal V_{m2}(h, b) \leq \mathtt c_m(h, b) \; \; \text{ for all } m \geq 1, \; \; \text{ with } \; \; \sum_{m=1}^{\infty} \mathtt c_m(h, b) < \infty.  
\end{equation}   
\end{proposition} 
\noindent We will prove Proposition \ref{W-sum-prop} in Section \ref{W-sum-prop-proof-section} and Proposition \ref{V2-sum-prop} in Sections \ref{estimating-vm-section-Part1} and \ref{estimating-vm-section-Part2}. Parts \eqref{V2-sum-prop-parta} and \eqref{V2-sum-prop-partb} of Proposition \ref{V2-sum-prop} are covered in Corollaries \ref{V2-sum-corollary-parta} and \ref{V2-sum-corollary-partb} respectively.   

\section{Estimating $\mathcal{V}_{m1}$} \label{estimating-Vm1-section} 
We will prove Proposition \ref{W-sum-prop} in this section, using the pointwise estimate on $\mathtt v_m$ obtained in Lemma \ref{vm-est-lemma}. We have noted in \eqref{sum-of-W} that $\mathcal V_{m1}$ is bounded above by four sums $\{\mathcal W_j: 1 \leq j \leq 4\}$. Of these, $\mathcal W_1$ and $\mathcal W_3$ correspond respectively to a small value of $b^{u+v}$ or of $N$, whereas for $\mathcal W_2, \mathcal W_4$, at least one of these two quantities is large. We treat them separately.  
 \subsection{Proof of Proposition \ref{W-sum-prop}} \label{W-sum-prop-proof-section} 
\subsubsection{Estimation of $\mathcal{V}_{m1}$ in low frequencies} 
\begin{lemma}
For $\pmb{\Pi}, h, b$ as in Proposition \ref{W-sum-prop}, there is a constant $C(\pmb{\Pi}, h, b) > 0$ for which the sums $\mathcal W_1$, $\mathcal W_3$ defined in \eqref{Wj} obey the estimate
\begin{equation} \label{W1W3-bound} 
\mathcal W_1 + \mathcal W_3 \leq C(\pmb{\Pi}, h, b) \mathtt s_m^{-\frac{\mathtt a_m}{2}} \leq C(\pmb{\Pi}, h, b) 2^{-\frac{m}{2}}.
\end{equation} 
The last inequality of \eqref{W1W3-bound} follows from the assumption $2\mathtt s_m^{\mathtt a_m} < 2\mathtt s_m^{\mathtt b_m} \leq \mathtt s_{m+1}^{\mathtt a_{m+1}}$ ensured by \eqref{abs}.
\end{lemma} 
\begin{proof} 
The procedure for estimating $\mathcal W_1 + \mathcal W_3$ is similar to the estimation of $\mathscr{U}_{m1}$ in Proposition \ref{U-prop}. We use the pointwise upper bound $|\xi| \mathtt s_m^{-\mathtt a_m}$ on $\mathtt v_m(\xi)$ provided by \eqref{vm-pointwise-estimate}:
\begin{align} 
\mathcal W_1 + \mathcal W_3 &= \sum \Bigl\{ \frac{1}{N^{3}} \bigl| \mathtt v_m \bigl(hb^u(b^v-1) \bigr) \bigr| : (u, v) \in \mathbb V_{11} \cup \mathbb V_{13} \Bigr\} \nonumber \\  &\leq C_0 \sum \left\{\frac{1}{N^{3}} \frac{|h b^u(b^v-1)|}{\mathtt s_m^{\mathtt a_m}} : \mathbb V_{11} \cup \mathbb V_{13} \right\}. \label{w1+w3-step1}
\end{align} 
We observe the following property of the indices $(u, v, N) \in \mathbb V_{11} \cup \mathbb V_{13}$:
\[b^{u+v} \leq \begin{cases} \mathtt s_m^{\frac{\mathtt a_m}{2}} \; \; \text{ by definition }  &\text{ on } \mathbb V_{11} \\ b^{2N} \leq b^{\frac{\mathtt a_m}{2}  \log_{\mathtt b} (\mathtt s_m) } = \mathtt s_m^{\frac{\mathtt a_m}{2}} &\text{ on } \mathbb V_{13}. \end{cases} \] 
Substituting this into the bound \eqref{w1+w3-step1} for the sum above, we find that
\begin{align*} 
\mathcal W_1 + \mathcal W_3 &\leq 2 \mathtt s_m^{-\mathtt a_m} \sum \Bigl\{\frac{1}{N^3}|h| b^{u+v} : b^{u+v} \leq \mathtt s_m^{\frac{\mathtt a_m}{2}}, (u, v) \in \mathbb Z_N^2, N \in \mathbb N \Bigr\} \\ &\leq C(h)  \mathtt s_m^{-\mathtt a_m} \sum_{N=1}^{\infty} N^{-3} \sum_{u=1}^{N} \sum_r \left\{b^r : 1 \leq b^r \leq \mathtt s_m^{\frac{\mathtt a_m}{2}} \right\} \\ &\leq C(h, b) \mathtt s_m^{-\mathtt a_m} \mathtt s_m^{\frac{\mathtt a_m}{2}} \sum_{N=1}^{\infty} N^{-2} \leq C(h, b)  \mathtt s_m^{-\frac{\mathtt a_m}{2}}. 
\end{align*} 
The second inequality is a consequence of a change of variables $(u, v) \mapsto (u, r)$, with $r = u+v$. This is the claimed inequality \eqref{W1W3-bound}
\end{proof} 
\subsubsection{Estimation of $\mathcal{V}_{m1}$ in high frequencies} 
\begin{lemma} 
For $\pmb{\Pi}, h, b$ as in Proposition \ref{W-sum-prop}, there is a constant $C(\pmb{\Pi}, h, b) > 0$ for which the sums $\mathcal W_2$ and $\mathcal W_4$ defined in \eqref{Wj} obey the estimate
\begin{equation} 
\mathcal W_2 + \mathcal W_4 \leq C(h, b) m^{-2} \quad \text{for all}\;m \geq 4. \label{W2W4-bound}
\end{equation} 
Thus $\mathcal W_2 + \mathcal W_4$ is dominated by a term that is summable in $m$. 
\end{lemma} 
\begin{proof} 
The quantity $\mathcal W_2$ is handled in a way simiilar to $\mathcal W_1$ and $\mathcal W_3$, except here we use the bound $|\mathtt v_m(\xi)| \leq C_0 \mathtt s_m^{\mathtt b_m}/|\xi|$ from \eqref{vm-pointwise-estimate}:  
\begin{align} 
\mathcal W_2 &= \sum \Bigl\{ \frac{1}{N^3}\bigl| \mathtt v_m \bigl(hb^u(b^v-1) \bigr) \bigr| : (u, v) \in \mathbb V_{12} \Bigr\} \nonumber  \\ 
&\leq C_0  \sum \Bigl\{ \frac{1}{N^{3}} \frac{\mathtt s_m^{\mathtt b_m}}{|h b^u(b^v-1)|}  : (u, v) \in \mathbb V_{12} \Bigr\} \nonumber \\
&\leq C_0 \mathtt s_m^{\mathtt b_m}  \sum \Bigl\{ \frac{1}{N^{3} |h| b^{u+v}} : b^{u+v} \geq m^2 \mathtt s_m^{{\mathtt b_m}}, (u, v) \in \mathbb Z_N^2, N\in\mathbb{N} \Bigr\} \nonumber \\ 
&\leq C(h) \mathtt s_m^{\mathtt b_m} \sum_{N=1}^{\infty} N^{-3} \sum_{u=1}^{N} \sum_r \left\{b^{-r} : b^r \geq m^2 \mathtt s_m^{\mathtt b_m} \right\} \nonumber \\ 
&\leq C(h, b) \mathtt s_m^{\mathtt b_m} m^{-2} \mathtt s_m^{-\mathtt b_m} \sum_{N=1}^{\infty} N^{-2} \leq C(h, b) m^{-2}. \label{W2} 
\end{align}
The quantity $\mathcal W_4$ is analyzed in a different way. In view of the already obtained bounds on $\mathcal W_1$ and $\mathcal W_2$,  it suffices to only bound the quantity $\mathcal W_4^{\ast}$, the sum of $|\mathtt v_m(h b^u (b^v-1))|/N^3$ on
\[ \mathbb V_{14}^{\ast} := \bigl\{(u, v, N) \in \mathbb V_{14} : \mathtt s_m^{\frac{\mathtt a_m}{2}} \leq b^{u+v} \leq m^2 \mathtt s_m^{\mathtt b_m}\bigr\}, \]
since $|\mathcal W_4 - \mathcal W_4^{\ast}|$ is bounded above by at most a constant multiple of $\mathcal W_1 + \mathcal W_2$. 
On $\mathbb V_{14}^{\ast}$, the indices $(u, v)$ satisfy the relation
\begin{equation}  \max(u, v) \leq u + v \leq \log_b\bigl(m^2 \mathtt s_m^{\mathtt b_m}\bigr) \leq 2 \mathtt b_m \log_b (\mathtt s_m). \label{uv-max} \end{equation} 
The last inequality follows from \eqref{abs}, which implies \[ m^2 \leq 2^{m} \leq \mathtt s_m^{\mathtt b_m} \quad \text{ for all } m \geq 4,  \] since $2\mathtt s_m^{\mathtt b_m} \leq \mathtt s_{m+1}^{\mathtt a_{m+1}} <\mathtt s_{m+1}^{\mathtt b_{m+1} }$ and $\mathtt s_{0}^{\mathtt b_{0} } = 1$. Combining the bound \eqref{uv-max} on $u, v$ with the trivial bound $|\mathtt v_m| \leq 1$ from \eqref{vm-pointwise-estimate}, we obtain 
\begin{align} 
\mathcal W_4^{\ast} &\leq \sum \Bigl\{\frac{1}{N^3}|\mathtt v_m \bigl(hb^u(b^v-1) \bigr)| : (u, v, N) \in \mathbb V_{14}^{\ast}  \Bigr\} \nonumber \\
&\leq \sum \Bigl\{\frac{1}{N^3} : u, v \leq 2\mathtt b_m \log_b (\mathtt s_m), \;  N \geq m \mathtt b_m \log_b (\mathtt s_m) \Bigr\} \nonumber \\ 
&\leq \sum \Bigl\{ \frac{1}{N^3}  \bigl(2\mathtt b_m \log_b (\mathtt s_m) \bigr)^2 : N \geq m \mathtt b_m \log_b (\mathtt s_m) \Bigr\} \nonumber \\
&\leq C_0  \bigl( \mathtt b_m \log_b (\mathtt s_m) \bigr)^2 \bigl(  m \mathtt b_m \log_b (\mathtt s_m) \bigr)^{-2} \leq C_0 m^{-2}.  \label{W4}
\end{align}    
Combining \eqref{W2} and \eqref{W4} leads to \eqref{W2W4-bound}. 
\end{proof} 
\noindent As a direct consequence of \eqref{sum-of-W}, \eqref{W1W3-bound} and \eqref{W2W4-bound}, we find that the conclusion \eqref{W-sum-estimate} of Proposition \ref{W-sum-prop} holds, with the summable sequence $\mathtt c_m(h,b) = C(h,b) m^{-2}$.
 \qed

\section{Pointwise estimates for $\mathtt v$: Take 2} \label{v-pointwise-section-take-2} 
So far, we have reduced the proof of Proposition \ref{normal-prop} to that of Proposition \ref{V2-sum-prop}, i.e., the estimation of $\mathcal V_{m2}$ given by \eqref{Vm2-sum-definition}. As indicated earlier, the proof of this proposition is heavily dependent on the multiplicative independence of the bases $b$ and $\mathtt s_m$. In this section, we record a few analytic estimates for $\mathtt v$ that will be needed for the proof of Proposition \ref{V2-sum-prop}. These estimates offer potentially sharper bounds than Lemma \ref{vm-est-lemma} in the intermediate frequency ranges of $\xi$ where \eqref{vm-pointwise-estimate} is not useful, but these bounds are significantly strong only for frequencies $\xi$ with special arithmetic properties. In our applications, frequencies with such properties turn out to be numerous only due to the multiplicative independence of $b$ and $\mathtt s_m$. The number-theoretic facts necessary for the proof are gathered in Section \ref{number-theoretic-tools-section}. The proof of Proposition \ref{V2-sum-prop} combining all this information follows thereafter.  

\subsection{The function $\mathfrak w$ revisited}
The inequality \eqref{def-wm} provides an upper bound for the function $\mathtt v$ in terms of $\mathfrak A$ and $\mathfrak w$. In Lemma \ref{v-lemma}, we derived a pointwise estimate on $\mathfrak w(\xi)$ depending only on $|\xi|$. Finding a finer one is the goal of this section. 
\vskip0.1in 
\noindent Let us recall the expression of $\mathfrak w$ from \eqref{wm-alternate}, which involves $\mathfrak B^{\ast}$ and $\mathfrak C$. The product formulae for $\mathfrak B^{\ast}$ and $\mathfrak C$ from \eqref{B*-rep} and \eqref{C-factor} give rise to 
\begin{equation} \label{B*C-product}
\mathfrak B^{\ast}(\xi) = \prod_{j = \mathtt a + 1}^{\mathtt b} \mathfrak g(\xi, j; \mathtt s, \mathscr{D}), \qquad \mathfrak C(\xi) = \prod_{j=\mathtt a+1}^{\mathtt b} \mathfrak g(\xi, j; \mathtt s, \mathbb Z_{\mathtt s}), 
\end{equation} 
where each factor $\mathfrak g(\xi, j; \cdot, \cdot)$ is defined as in \eqref{exp-sum-g-def}. It is clear from \eqref{B*C-product} that the sizes of $\mathfrak B^{\ast}$ and $\mathfrak C$ are dictated by the number of indices $j$ for which the factor $\mathfrak g(\xi, j; \cdot, \cdot)$ is quantifiably smaller than the trivial bound 1. The following lemma makes this precise.  
\begin{lemma} \label{restricted-lemma}
There exists an absolute constant $c_0 > 0$ with the following property. Suppose that $\mathtt s \in \mathbb N \setminus \{1\}$, $\xi \in \mathbb Z \setminus \{0\}$, and that 
\begin{equation} \label{xi-expansion} |\xi| = \sum_{j=0}^{\infty} \mathtt d_j(|\xi|) \mathtt s^j, \qquad \mathtt d_j(|\xi|) \in \{0, 1, \ldots, \mathtt s-1\},  \end{equation} 
denotes the digit expansion of $|\xi|$ in base $\mathtt s$. 
\begin{enumerate}[(a)]
\item \label{restricted-lemma-part-a} If there exists an index $j \geq 1$ such that 
\begin{equation} \label{not-extreme}
1 \leq \mathtt d_{j-1}(|\xi|) + \mathtt d_j(|\xi|) \mathtt s \leq \mathtt s^2-2,   
\end{equation}  
then for any choice of a digit set $\mathscr{A} \subseteq \mathbb Z_{\mathtt s}$ containing $0$ and $1$, with $\#\mathscr{A}=:\mathtt r$,
\begin{equation} \label{less-than-one}
\bigl|\mathfrak g(\xi, j+1; \mathtt s, \mathscr{A})| \leq 1 - \frac{c_0}{\mathtt r \mathtt s^4} < 1. 
\end{equation}   
\item \label{restricted-lemma-part-b} Let
\begin{equation} \label{def-cardinality-J}
\overline{\mathtt J} = \overline{\mathtt J}(\xi; \mathtt s) :=  \# \bigl\{ \mathtt a \leq j \leq \mathtt b-1 : 1 \leq \mathtt d_{j-1}(|\xi|) + \mathtt d_j(|\xi|) \mathtt s \leq \mathtt s^2-2 \bigr\}
\end{equation} 
denote the number of indices $j \in \{\mathtt a, \ldots, \mathtt b-1\}$ obeying \eqref{not-extreme}. Then for all $\mathscr{D} \subseteq \mathbb Z_{\mathtt s}$ containing $\{0, 1\}$, with $\#\mathscr{D}=:\mathtt r$, one has the estimates
\begin{equation} \label{B*C-special-estimates}
\bigl| \mathfrak B^{\ast}(\xi) \bigr| \leq \left(1 - \frac{c_0}{\mathtt r \mathtt s^4} \right)^{\overline{\mathtt J}} \quad \text{ and } \quad \bigl| \mathfrak C(\xi) \bigr| \leq \left(1 - \frac{c_0}{\mathtt s^5}\right)^{\overline{\mathtt J}}.
\end{equation} 
As a result, the functions $\mathtt v$ and $\mathfrak w$ given by \eqref{identity} and \eqref{wm-alternate} obey 
\begin{equation} 
|\mathtt v(\xi)| \leq |\mathfrak w(\xi)| \leq |\mathfrak B^{\ast}(\xi)| + |\mathfrak C(\xi)| \leq 2 \left(1 - \frac{c_0}{\mathtt s^5}\right)^{\overline{\mathtt J}}. \label{v-pointwise-2}
\end{equation}  
\end{enumerate} 
\end{lemma} 
\vskip0.1in 
\noindent {\em{Remarks: }} \begin{enumerate}[1.]
\item The improvement in \eqref{less-than-one} over the trivial bound is absent if $\mathscr{A}$ is a singleton. In the latter case $|\mathfrak g(\xi, j; \mathtt s, \mathscr{A}| \equiv 1$, as can be seen from \eqref{exp-sum-g-def}. 
\item The estimates in \eqref{B*C-special-estimates} present a geometric (in $\overline{\mathtt J}$) improvement over the trivial estimates for $\mathfrak B^{\ast}$ and $\mathfrak C$. This is significantly better than the trivial bounds \eqref{B*C-trivial-bound} especially when $\overline{\mathtt J}(\xi; \mathtt s)$ is large.    
\end{enumerate} 
\begin{proof} 
The relations in \eqref{B*C-product} identify $\mathfrak B^{\ast}$ and $\mathfrak C$ as products of factors of the form $\mathfrak g(\xi, j; \mathtt s, \cdot)$. Applying the bound \eqref{less-than-one} to each factor, with $\mathscr{A} = \mathscr{D}$ and $\mathbb Z_{\mathtt s}$ for $\mathfrak{B}^{\ast}$ and $\mathfrak C$ respectively, we see that part \eqref{restricted-lemma-part-b} follows from part \eqref{restricted-lemma-part-a}. We therefore focus only on proving part \eqref{restricted-lemma-part-a}. We first do this in the special situation $\mathscr{A} = \{0, 1\}$. In this case 
\begin{align} 
&\mathfrak g(\xi, j+1; \mathtt s, \mathscr{A}) = \frac{1}{2} \bigl( 1 + e(\xi \mathtt s^{-j-1})\bigr), \text{ so that } \nonumber  \\ 
&\bigl|\mathfrak g(\xi, j+1; \mathtt s, \mathscr{A}) \bigr| = \bigl| \cos \bigl({\pi \xi}{\mathtt s^{-j-1}}\bigr) \bigr| =   \bigl|\cos \bigl(\pi\bigl\{{ |\xi|}{\mathtt s^{-j-1}} \bigr\}\bigr) \bigr|,  \quad \{x\} = x - \lfloor x \rfloor.  \label{D01} 
\end{align} 
From the digit expansion \eqref{xi-expansion}, we see that 
\[ \{|\xi| \mathtt s^{-j-1}\} = \mathtt s^{-(j+1)} \sum_{k=0}^j \mathtt  d_k(|\xi|) \mathtt s^k.  \]   
If the index $j$ obeys the assumption \eqref{not-extreme}, the relation above yields two estimates:
 \begin{align}  
 \left\{\frac{|\xi|}{\mathtt s^{j+1}}\right\}  &=  \mathtt s^{-j-1} \sum_{k=0}^{j} \mathtt d_k(|\xi|) \mathtt s^k \geq \mathtt s^{-j-1} \bigl( \mathtt d_{j-1}(|\xi|) \mathtt s^{j-1} + \mathtt d_j(|\xi|) \mathtt s^j \bigr) \geq \mathtt s^{-2}; \text{ similarly } \label{D01+}  \\ 
\left\{\frac{|\xi|}{\mathtt s^{j+1}}\right\} & 
\leq  \mathtt s^{-j-1} \Bigl[ \sum_{k=0}^{j-2}  (\mathtt s-1) \mathtt s^k + \mathtt s^{j-1}(\mathtt s^2-2)  \Bigr]  = \mathtt s^{-j-1} \left({\mathtt s^{j-1}-1}\right) + \mathtt s^{-2}(\mathtt s^2-2)  \nonumber \\ &\leq \mathtt s^{-2} + \mathtt s^{-2}(\mathtt s^2-2)  \leq 1 - \mathtt s^{-2}. \label{D01-}
 \end{align}
Substituting the two bounds \eqref{D01+} and \eqref{D01-} into \eqref{D01}, and using the fact that the function $\theta \mapsto |\cos(\pi \theta)|$ is symmetric about the origin and decreasing on $[0, 1/2]$ we obtain 
\begin{equation} \label{g-nontrivial-bound}
|\mathfrak g(\xi, j+1; \mathtt s, \mathscr{A})| \leq \bigl|\cos \bigl({\pi \bigl\{|\xi|}{\mathtt s^{-j-1}} \bigr\}\bigr) \bigr| \leq \cos \bigl(\pi \mathtt s^{-2} \bigr) \leq 1 - \frac{c_0}{\mathtt s^2}
\end{equation}    
for some absolute constant $c_0 > 0$. This proves \eqref{less-than-one}, in fact a stronger bound than \eqref{less-than-one}, in the special case $\mathscr{A} = \{0,1\}$.  
\vskip0.1in
\noindent Let us continue to the proof of \eqref{less-than-one} for a general set of restricted digits $\mathscr{A}$ containing $\{0,1\}$. Here we express the exponential sum $ \mathfrak g(\xi, j+1; \mathtt s, \mathscr{A})$ in two parts: one corresponding to the digits $\{0, 1\}$ on which we apply our previously found estimate, and the remainder which is treated using trivial bounds. This leads to the following estimate: 
\begin{align*} \mathfrak g(\xi, j+1; \mathtt s, \mathscr{A}) &= \frac{1}{\mathtt r} \Bigl[ 2 \mathfrak g(\xi, j+1; \mathtt s, \{0, 1\}) + \sum_{\mathtt d \in \mathscr{A} \setminus \{0, 1\}} e \bigl(\xi \mathtt d \mathtt s^{-j-1}\bigr) \Bigr]  \text{ so that } \\ 
\bigl| \mathfrak g(\xi, j+1; \mathtt s, \mathscr{A}) \bigr| &\leq \frac{1}{\mathtt r} \bigl| 2 \cos \bigl( \pi \mathtt s^{-2} \bigr) + \mathtt r - 2 \bigr| \leq \frac{1}{\mathtt r} \Bigl[\mathtt r - 4 \sin^2 \left(\frac{\pi}{2\mathtt s^2}\right)\Bigr] \\ &\leq \frac{1}{\mathtt r} \left(\mathtt r - \frac{c_0}{\mathtt s^4}\right) \leq 1 - \frac{c_0}{\mathtt r \mathtt s^4} \; \text{ for some absolute } c_0 > 0. 
\end{align*} 
This completes the proof of Lemma \ref{restricted-lemma}.  
\end{proof}

\subsection{The function $\mathfrak A$ revisited}
As can be seen from \eqref{def-wm}, the function $\mathfrak A$ defined as in \eqref{def-A} also plays a role in estimating $\mathtt v$, along with $\mathfrak w$. In Section \ref{A-and-B-section}, we  obtained an explicit formula for $\mathfrak A$ and derived certain estimates that were useful in establishing the Rajchman property; see for example Lemma \ref{A2-lemma}. This lemma applies only for frequencies $\xi$ that are multiples of $\mathtt s^{\mathtt b}$ in a certain range. Here we will establish a different estimate for $\mathfrak A$ that shows that under specified conditions, $\mathfrak A$ may be small for a larger collection of $\xi$. Such a statement is helpful in proving convergence of $\mathcal V_{m2}$. 
\begin{lemma} 
Let $\mathfrak A(\xi; \mathtt N, \mathtt s, \mathtt a)$ be the function defined in \eqref{def-A} with $\mathtt N, \mathtt s, \mathtt a$ as in Section \ref{param-elem-op-section}. Then for every $\mathtt a' \in \mathbb N$ with $\mathtt a' < \mathtt a$, the following inequality holds: 
\begin{equation} \label{A-alt-est} 
\bigl| \mathfrak A(\xi; \mathtt N, \mathtt s, \mathtt a) \bigr| \leq \frac{\mathtt s^{\mathtt a- \mathtt a'}}{\mathtt N} + \prod_{k = \mathtt a'+1}^{\mathtt a} |\mathfrak g(\xi, k; \mathtt s, \mathbb Z_{\mathtt s})|,\; \xi \in \mathbb{Z}, 
\end{equation} 
with $\mathfrak g(\xi, j; \mathtt s, \mathbb Z_{\mathtt s})$ as in \eqref{exp-sum-g-def}. 
\end{lemma} 
\begin{proof} 
Given an integer $1 \leq \mathtt a' < \mathtt a$, let us define the unique non-negative integer $\mathtt Q$ as follows: 
\begin{equation} \label{division} 
\mathtt N = \mathtt Q \mathtt s^{\mathtt a-\mathtt a'} + \mathtt N', \quad \text{ where } \quad \mathtt N' \in \bigl\{0, 1, \ldots, \mathtt s^{\mathtt a - \mathtt a'} - 1 \bigr\}.  
\end{equation} 
Using this, we decompose the sum defining $\mathfrak A$ into $\mathtt Q+1$ blocks:
\begin{align} 
\mathtt N \mathfrak A(\xi) &= \sum_{\ell=0}^{\mathtt N-1} e \left(\xi \ell \mathtt s^{-\mathtt a} \right) = \sum_{j=0}^{\mathtt Q}  \sum_{\ell} \Bigl\{ e \left(\xi \ell \mathtt s^{-\mathtt a} \right) : j\mathtt s^{\mathtt a-\mathtt a'} \leq \ell < \min \bigl[(j+1) \mathtt s^{\mathtt a-\mathtt a'}, \mathtt N \bigr] \Bigr\} \nonumber  \\ 
&= \sum_{j, \rho} \Biggl\{ e \left[\xi \mathtt s^{-\mathtt a}\bigl( j \mathtt s^{\mathtt a-\mathtt a'} + \rho\bigr)\right] \Bigl| \; \begin{aligned} &0 \leq j \leq \mathtt Q-1 \\ &0 \leq \rho < \mathtt s^{\mathtt a-\mathtt a'} \end{aligned} \Biggr\} + \mathfrak R(\xi),  \; \text{ where } \label{NA-formula} \\ 
\mathfrak R(\xi) &:=  \sum  \Bigl\{ e \left(\xi \ell \mathtt s^{-\mathtt a} \right) : \mathtt Q \mathtt s^{\mathtt a-\mathtt a'} \leq \ell < \mathtt N \Bigr\}, \qquad \bigl| \mathfrak R(\xi) \bigr| \leq \mathtt N' + 1 < \mathtt s^{\mathtt a-\mathtt a'}.  \label{remainder} 
\end{align} 
The representation \eqref{NA-formula} can be further simplified. Expressing $\rho \in \{0, 1, \ldots, \mathtt s^{\mathtt a-\mathtt a'}-1\}$ via its digit expansion in base $\mathtt s$, we find that  
\begin{align} 
\mathtt N \mathfrak A(\xi) - \mathfrak R(\xi) &= \sum_{j=0}^{\mathtt Q-1} e \bigl(\xi j \mathtt s^{-\mathtt a'}\bigr) \sum_{\rho} \Bigl\{ e \bigl( \xi \rho \mathtt s^{-\mathtt a} \bigr) \; \Bigl| \; 0 \leq \rho < \mathtt s^{\mathtt a-\mathtt a'} \Bigr\} \nonumber \\ 
&= \sum_{j=0}^{\mathtt Q-1} e \bigl(\xi j \mathtt s^{-\mathtt a'} \bigr) \sum_{\pmb{\rho}} \Bigl\{ e \Bigl[\xi \mathtt s^{-\mathtt a} \sum_{r=0}^{\mathtt a-\mathtt a'-1} \rho_r \mathtt s^r \Bigr] \; \Bigl| \; \pmb{\rho} = (\rho_0, \rho_1, \ldots, \rho_{\mathtt a-\mathtt a'-1})\in \mathbb Z_{\mathtt s}^{\mathtt a-\mathtt a'} \Bigr\} \nonumber \\ 
&= \sum_{j=0}^{\mathtt Q-1} e \bigl(\xi j \mathtt s^{-\mathtt a'}\bigr) \prod_{k=\mathtt a'+1}^{\mathtt a} \Bigl[ 1 + e \bigl(\xi \mathtt s^{ -k} \bigr) + e\bigl( 2\xi \mathtt s^{- k} \bigr) + \ldots + e \bigl( \xi(\mathtt s-1) \mathtt s^{-k} \bigr) \Bigr] \nonumber \\ 
&= \sum_{j=0}^{\mathtt Q-1} e \bigl( \xi j \mathtt s^{-\mathtt a'}\bigr) \mathtt s^{\mathtt a-\mathtt a'} \prod_{k=\mathtt a'+1}^{\mathtt a} \mathfrak g(\xi, k; \mathtt s, \mathbb Z_{\mathtt s}). \label{NA-R}
\end{align} 
The last step follows from the definition \eqref{exp-sum-g-def} of $\mathfrak g$. Combining \eqref{NA-R}, \eqref{NA-formula}, \eqref{remainder} and \eqref{division} yields  
\[ \mathtt N |\mathfrak A(\xi)| \leq \mathtt s^{\mathtt a-\mathtt a'}\mathtt Q \prod_{k=\mathtt a'+1}^{\mathtt a} \bigl| \mathfrak g(\xi, k; \mathtt s, \mathbb Z_{\mathtt s}) \bigr| + \bigl|\mathfrak R(\xi) \bigr| \leq \mathtt N \prod_{k=\mathtt a'+1}^{\mathtt a}\bigl| \mathfrak g(\xi, k; \mathtt s, \mathbb Z_{\mathtt s}) \bigr|  +  \mathtt s^{\mathtt a-\mathtt a'}. \] 
This is a re-statement of the intended conclusion \eqref{A-alt-est}. 
\end{proof} 

\subsection{New estimates for $\mathtt v$}
\noindent The inequality \eqref{A-alt-est} for $\mathfrak A$ leads to pointwise estimates for $\mathtt v$ that are different from \eqref{v-pointwise-estimate} and \eqref{v-pointwise-2}. The two estimates derived below will be used in Sections \ref{X3-estimation-section} and \ref{Y4-estimation-section} respectively. 
\begin{corollary} \label{v-alt-cor}
For any $\mathtt a' \in \mathbb N$ with $\mathtt a' < \mathtt a$, and $\xi \in \mathbb Z$ given by \eqref{xi-expansion}, let
\begin{equation} \label{def-cardinality-J-bar}
\underline{\mathtt J} = \underline{\mathtt J}(\xi; \mathtt s, \mathtt a', \mathtt a) :=  \# \bigl\{ \mathtt a' \leq j \leq \mathtt a-1 : 1 \leq \mathtt d_{j-1}(|\xi|) + \mathtt d_j(|\xi|) \mathtt s \leq \mathtt s^2-2 \bigr\}
\end{equation} 
denote the number of indices $j \in \{\mathtt a', \ldots, \mathtt a-1\}$ obeying \eqref{not-extreme}. Then the function $\mathtt v$ given by \eqref{uv-def} (or equivalently \eqref{identity}) obeys the estimate: 
\begin{equation} \label{v-alt-estimate} 
\bigl| \mathtt v(\xi) \bigr| \leq 2\bigl| \mathfrak A(\xi; \mathtt N, \mathtt s, \mathtt a)\bigr| \leq 2 \Bigl[\frac{\mathtt s^{\mathtt a-\mathtt a'}}{\mathtt N} + \left(1 - \frac{c_0}{\mathtt s^{5}} \right)^{\underline{\mathtt J}} \Bigr].
\end{equation} 
\end{corollary} 
\begin{proof} 
The first inequality in \eqref{v-alt-estimate} is a direct consequence of the identity \eqref{identity}, combined with the trivial estimates for $\Omega_1$ and $\mathfrak w$. The second inequality follows from \eqref{A-alt-est} and \eqref{less-than-one}, since exactly $\underline{\mathtt J}$ of the factors occurring in the product of \eqref{A-alt-est}, namely those corresponding to the indices $j + 1 \in \{\mathtt a'+1, \ldots, \mathtt a\}$, obey the estimate \eqref{less-than-one}.  
\end{proof} 
\begin{corollary} \label{corollary-composite}
Suppose that $\xi \in \mathbb Z \setminus \{0\}$ is of the form $\xi = \mathtt s^{\mathtt k_0} \zeta$ for some $\zeta \in \mathbb Z$ and a non-negative integer $\mathtt k_0$. For positive integers $\mathtt a'< \mathtt a < \mathtt b$,  let  
\begin{equation} \label{Jab} 
\mathtt J = \mathtt J(\xi; \mathtt s, \mathtt a', \mathtt b) :=  \# \bigl\{ \mathtt a' -\mathtt k_0 \leq j \leq \mathtt b - \mathtt k_0-1 : 1 \leq \mathtt d_{j-1}(|\zeta|) + \mathtt d_j(|\zeta|) \mathtt s \leq \mathtt s^2-2 \bigr\}, 
\end{equation} 
where $\{\mathtt d_j(|\zeta|): j \geq 0\}$ denote the digits of $|\zeta| \in \mathbb N$ in base $\mathtt s$. For $\mathtt a' -\mathtt k_0 <0 $, we declare $\mathtt d_j(|\zeta|) = 0$ if $j < 0$. Then the function $\mathtt v$ given by \eqref{uv-def} (or \eqref{identity}) obeys the estimate 
\begin{equation} \label{corollary-composite-estimate}
\bigl| \mathtt v(\xi) \bigr| \leq C_0 \Bigl[ \frac{\mathtt s^{\mathtt a-\mathtt a'}}{\mathtt N} + \Bigl(1 - \frac{c_0}{\mathtt s^5}\Bigr)^{\mathtt J} \Bigr].  
\end{equation} 
\end{corollary} 
\begin{proof} 
We appeal to the bound \eqref{def-wm} for $\mathtt v$, using the inequality \eqref{A-alt-est} to bound $\mathfrak A$ and the identity \eqref{wm-alternate} for $\mathfrak w$ in terms of $\mathfrak B^{\ast}$ and $\mathfrak C$.  This leads to  
\begin{align} 
\bigl| \mathtt v(\xi) \bigr| &\leq \bigl| \mathfrak A(\xi) \bigr| \times \bigl| \mathfrak w(\xi) \bigr| \nonumber \\ 
&\leq \Bigl[ \frac{\mathtt s^{\mathtt a - \mathtt a'}}{\mathtt N} + \Bigl| \prod_{k = \mathtt a'+1}^{\mathtt a} \mathfrak g(\xi, k; \mathtt s, \mathbb Z_{\mathtt s}) \Bigr|\Bigr] \times \bigl| \mathfrak w(\xi) \bigr| \nonumber \\ 
&\leq  2 \Bigl[ \frac{\mathtt s^{\mathtt a - \mathtt a'}}{\mathtt N} +   \Bigl| \prod_{k = \mathtt a'+1}^{\mathtt a} \mathfrak g(\xi, k; \mathtt s, \mathbb Z_{\mathtt s}) \Bigr| \times \bigl| \mathfrak w(\xi) \bigr| \Bigr] \nonumber \\ 
&\leq 2 \Bigl[ \frac{\mathtt s^{\mathtt a - \mathtt a'}}{\mathtt N} + \bigl| {\mathfrak v}_1(\xi) \bigr|  + \bigl| {\mathfrak v}_2(\xi) \bigr|  \Bigr]. \label{v-v1-v2}
\end{align} 
At the penultimate step, we have used the trivial bound $|\mathfrak w(\xi)| \leq 2$ for the first term. The last step is based on the factorizations \eqref{B*C-product} for $\mathfrak B^{\ast}$ and $\mathfrak C$, with the following definitions for $\mathfrak v_1, \mathfrak v_2$:  
\begin{align} 
\mathfrak v_1(\xi) &:=  \prod_{k = \mathtt a'+1}^{\mathtt a} \mathfrak g(\xi, k; \mathtt s,  \mathbb Z_{\mathtt s}) \times \mathfrak B^{\ast}(\xi) = \prod_{k = \mathtt a'+1}^{\mathtt a} \mathfrak g(\xi, k; \mathtt s, \mathbb Z_{\mathtt s}) \times \prod_{k=\mathtt a+1}^{\mathtt b} \mathfrak g(\xi, k; \mathtt s, \mathscr{D}), \label{v1-formula}  \\ 
\mathfrak v_2(\xi) &:=  \prod_{k = \mathtt a'+1}^{\mathtt a} \mathfrak g(\xi, k; \mathtt s, \mathbb Z_{\mathtt s}) \times \mathfrak C(\xi) =   \prod_{k = \mathtt a'+1}^{\mathtt b} \mathfrak g(\xi, k; \mathtt s, \mathbb Z_{\mathtt s}).  \label{v2-formula} 
\end{align} 
The definition \eqref{exp-sum-g-def} of $\mathfrak g$ implies that for $\xi = \mathtt s^{\mathtt k_0} \zeta$,
\begin{equation*}  \mathfrak g(\xi, k; \mathtt s, \cdot) = 
	\mathfrak g(\zeta, k - \mathtt k_0; \mathtt s, \cdot).
\end{equation*} 
Substituting this into \eqref{v1-formula}, \eqref{v2-formula} and invoking \eqref{less-than-one}, we obtain  
\begin{align} 
\bigl|\mathfrak v_1(\xi) \bigr| &= \prod_{j = \mathtt a' - \mathtt k_0 +1}^{\mathtt a - \mathtt k_0} \bigl|\mathfrak g(\zeta, j; \mathtt s, \mathbb Z_{\mathtt s}) \bigr| \times \prod_{j=\mathtt a - \mathtt k_0+1}^{\mathtt b - \mathtt k_0} \bigl|\mathfrak g(\zeta, j; \mathtt s, \mathscr{D})\bigr| \nonumber \\ 
&=  \prod_{j = \mathtt a' - \mathtt k_0}^{\mathtt a - \mathtt k_0-1} \bigl|\mathfrak g(\zeta, j+1; \mathtt s, \mathbb Z_{\mathtt s}) \bigr| \times \prod_{j=\mathtt a - \mathtt k_0}^{\mathtt b - \mathtt k_0-1} \bigl|\mathfrak g(\zeta, j+1; \mathtt s, \mathscr{D})\bigr| \leq \left(1 - \frac{c_0}{\mathtt s^5}\right)^{\mathtt J}, \label{v1-bound-final} \\ 
\bigl|\mathfrak v_2(\xi) \bigr| &= \prod_{j = \mathtt a' - \mathtt k_0+1}^{\mathtt b - \mathtt k_0} \bigl| \mathfrak g(\zeta, j; \mathtt s, \mathbb Z_{\mathtt s}) \bigr| = \prod_{j = \mathtt a' - \mathtt k_0}^{\mathtt b - \mathtt k_0-1} \bigl| \mathfrak g(\zeta, j+1; \mathtt s, \mathbb Z_{\mathtt s}) \bigr| \leq \left(1 - \frac{c_0}{\mathtt s^5}\right)^{\mathtt J}.  \label{v2-bound-final} 
\end{align} 
The last two inequalities above follow from the definition \eqref{Jab} of $\mathtt J$. For both $\mathfrak v_1$ and $\mathfrak v_2$, the quantity $\mathtt J$ represents the number of factors, indexed by $j \in \{\mathtt a'-\mathtt k_0, \ldots, \mathtt b - \mathtt k_0-1\}$, for which \eqref{less-than-one} applies.  Substituting \eqref{v1-bound-final} and \eqref{v2-bound-final} into \eqref{v-v1-v2} yields the conclusion \eqref{corollary-composite-estimate}. 
\end{proof}

\section{Number-theoretic tools} \label{number-theoretic-tools-section} 
 This section is given over to a collection of number-theoretic facts that will be needed for the proof of Proposition \ref{V2-sum-prop} (Sections \ref{estimating-vm-section-Part1} and \ref{estimating-vm-section-Part2}). Some of these facts are well-known in number theory, and we state them with appropriate references. Others follow from the work of Schmidt \cite{s60}, and are deduced here. 
\subsection{Values generated by consecutive digits} \label{consecutive-digits-section}
Given any integer $\xi \in \mathbb N$ and a base $\mathtt s \in \mathbb N \setminus \{1\}$, let 
\begin{equation*} 
\xi = \sum_{j=0}^{\infty} \mathtt d_j(\xi) {\mathtt s}^j, \qquad \mathtt d_j(\xi) \in \{0, 1, \ldots, \mathtt s-1\} 
\end{equation*} 
denote the digit expansion of $\xi$ in base $\mathtt s$. In view of the role that \eqref{not-extreme} plays in the estimation of $\mathtt v(\xi)$ (see Lemma \ref{restricted-lemma}, Corollaries \ref{v-alt-cor} and \ref{corollary-composite}), it is natural to ask about the frequency of integers $\xi$ where this condition is met for many indices $j$, resulting in a large value of the parameters $\overline{\mathtt J}, \underline{\mathtt J}$ or $\mathtt J$. 
In principle, the number $\mathtt d_{j-1} + \mathtt d_j \mathtt s$ can lie anywhere in $\{0, 1, \ldots, \mathtt s^2-1\}$, but intuitively for most $\xi$ and most indices $j$, it does not assume the extremal values $0$ or $\mathtt s^2-1$. In other words, there cannot be too many integers $\xi \in \{0, 1, \ldots, \mathtt s^{\ell}-1\}$ whose digit sequence contains numerous consecutive digit pairs of the form $(0, 0)$ or $(\mathtt s-1, \mathtt s-1)$. Equivalently stated, quantities like $\overline{\mathtt J}(\xi; \mathtt s)$ defined in \eqref{def-cardinality-J} (or their variants $\underline{\mathtt J}$ and $\mathtt J$) are large for most $\xi$. A lemma of Schmidt \cite{s60} makes this intuition precise. For $\mathbf d = (\mathtt d_0, \ldots, \mathtt d_{\ell-1}) \in \mathbb Z_{\mathtt s}^{\ell}$, set 
\begin{equation} \label{def-nl}
\mathfrak n_{\ell}(\mathbf d, \mathtt s) := \# \{ 1 \leq j \leq \ell-1 : 1 \leq \mathtt d_{j-1} + \mathtt d_j \mathtt s \leq \mathtt s^2-2 \}. 
\end{equation} 
\begin{lemma}[{\cite[Lemma 3]{s60}}] \label{Schmidt-lemma-extremal-digits}
For $\mathtt s \in \mathbb N \setminus \{1\}$, let $\kappa = \kappa(\mathtt s) \in (0, \frac{1}{4})$ be a constant such that 
\begin{equation} \label{def-kappa}
\frac{(\mathtt s^2-2)^{\kappa} 2^{\frac{1}{2} - \kappa}}{(2 \kappa)^{\kappa} (1 - 2\kappa)^{\frac{1}{2} - \kappa}} < 2^{\frac{3}{4}}.
\end{equation} 
Then there exists a positive integer $\mathtt K_0(\mathtt s) $ such that for all $\ell > \mathtt K_0(\mathtt s)$,
\begin{equation} \label{extremal-digit-bounds} 
\#\left(\mathbb D(\ell, \mathtt s) \right)  < 2^{\frac{3}{4}\ell}, \; \; \text{ where } \; \; \mathbb D ( \ell, \mathtt s ) :=  \left\{\mathbf d = (\mathtt d_0, \ldots, \mathtt d_{\ell-1}) \in \mathbb Z_{\mathtt s}^{\ell} : \mathfrak n_{\ell}(\mathbf d, \mathtt s) \leq \kappa \ell \right\}. 
\end{equation} \label{extremal-digit-bounds-modified} 
Equivalently, there exists $\kappa = \kappa(\mathtt s) > 0$ obeying \eqref{def-kappa} for which we have the following bound
\begin{equation} \label{choice-of-kappa-2}
\#\bigl(\mathbb D(\ell, \mathtt s) \bigr)  < \kappa^{-1} 2^{\frac{3}{4}\ell} \quad \text{ for all } \ell \geq 1.  
\end{equation} 
\end{lemma}  
\noindent This lemma will be used in Section \ref{X1-estimation-section}. 
\subsection{Order of an integer} \label{integer-order-section} 
The main result in this sub-section is Lemma \ref{I1-lemma}. The following two facts will be useful in its proof. For two co-prime positive integers $\mathtt a$ and $\mathtt m$, the {\em{order}} of $\mathtt a$ (mod $\mathtt m$), denoted $\text{ord}_{\mathtt m}(\mathtt a)$ is defined as 
\begin{equation} 
\text{ord}_{\mathtt m}(\mathtt a) := \min\bigl\{n \in \mathbb N: {\mathtt a}^{\mathtt n} \equiv 1 \; (\text{mod } \mathtt m) \bigr\}. 
\end{equation}   
\begin{lemma}[{\cite[Theorem 88]{hw08}}] \label{hw-lemma} 
Suppose that $\mathtt a, \mathtt m \in \mathbb N$, with {\rm gcd}$(\mathtt a, \mathtt m) = 1$. Then for any $\mathtt n \in \mathbb N$, 
\begin{equation*} 
\mathtt a^{\mathtt n} \equiv 1 \pmod {\mathtt m} 
 \quad \text{ if and only if } \quad \text{ \rm ord}_{\mathtt m}(\mathtt a) \text{ divides } \mathtt n.  
\end{equation*} 
\end{lemma} 
\begin{lemma}[{\cite[Lemma 4]{s60}}] \label{Schmidt-lemma-1}
Assume $b \in \mathbb N \setminus \{1\}$ and $\mathtt p$ is a prime such that $\mathtt p \nmid b$. Then  
\begin{equation} \label{ord-lower-bound} 
\text{\rm ord}_{\mathtt p^k}(b) \geq \frac{\mathtt p^{k}}{2b^{\mathtt p}} \quad \text{ for all } k \in \mathbb N. 
\end{equation} 
\end{lemma} 
\begin{proof} 
Lemma 4 in \cite{s60} proves the statement \eqref{ord-lower-bound} in a slightly weaker form: 
\begin{equation} \label{schmidt-ord-bound} 
\text{ord}_{\mathtt p^k}(b) \geq c(b, \mathtt p) \mathtt p^{-k} \text{ without explicitly specifying the constant $c(b, \mathtt p)$},
\end{equation}  
but one can derive a lower bound on $c(b, \mathtt p)$ by carefully following the proof. We sketch the details. The proof given in \cite{s60} shows that 
\begin{equation} \label{schmidt-c} 
c(b, \mathtt p) \geq \frac{g}{\mathtt p^{\alpha}},
\end{equation}  
where $\alpha = \alpha(b, \mathtt p) > 0$ is an integer constant such that 
\begin{equation}  
b^{g} \equiv 1 + \mathtt q \mathtt p^{\alpha-1} (\text{mod } \mathtt p^{\alpha}) \;  \text{ with } \mathtt q \not\equiv 0 \text{ mod $\mathtt p$,} \; \;  g = g(\mathtt p) = \begin{cases} \mathtt p-1 &\text{ if $\mathtt p$ is odd}, \\ 2 &\text{ if } \mathtt p=2.  \end{cases} 
\end{equation} 
Fermat's theorem \cite[Theorem 71]{hw08} ensures that $b^{g} \equiv 1$ (mod $\mathtt p$); one can therefore choose $\alpha$ to be the largest integer such that $\mathtt p^{\alpha-1}$ divides $b^g-1$. This implies in particular that
\begin{equation} \label{towards-c}
b^{g} > b^g-1 \geq \mathtt p^{\alpha-1} \quad \text{ or } \quad \mathtt p^{\alpha} \leq \mathtt p b^{g} \leq \mathtt p b^{\mathtt p}.
\end{equation}     
Substituting \eqref{towards-c} into \eqref{schmidt-c} gives $c(b, \mathtt p) \geq (\mathtt p-1)/(\mathtt p b^{\mathtt p}) \geq 1/(2 b^{\mathtt p})$. Combining this with \eqref{schmidt-ord-bound} yields \eqref{ord-lower-bound}. 
\end{proof} 
\begin{lemma} \label{I1-lemma}
Let $b$ and $\mathtt p$ be as in Lemma \ref{Schmidt-lemma-1}. Then the following relation holds for all integers $k \geq 1$ and $N \geq 2$. 
\begin{equation} \label{W1-cardinality}
\# \bigl(\mathbb O \bigr)  \leq 2 b^{\mathtt p} \mathtt p^{-k} N \quad \text{ where } \quad \mathbb O = \mathbb O(N, k; \mathtt p, b) := \bigl\{v \in \mathbb Z_N: \mathtt p^{k} \mid (b^v-1) \bigr\}. 
\end{equation} 
\end{lemma}  
\begin{proof} 
The set $\mathbb O$ is by definition the collection of non-negative integers $v \in \mathbb Z_{N}$ for which $b^{v} \equiv 1$ (mod $\mathtt p^k$). By Lemma \ref{hw-lemma}, 
$v$ is a multiple of $\text{ord}_{\mathtt p^k}(b)$. This means that 
\[ \#(\mathbb O) \leq \frac{N}{\text{ord}_{\mathtt p^k}(b)} \leq {2 b^{\mathtt p} N}{\mathtt p^{-k}}, \]
where the last step follows from Lemma \ref{Schmidt-lemma-1}.   
\end{proof}  
\noindent Lemma \ref{I1-lemma} will be used in Section \ref{X1-estimation-section}, to control parts of $\mathcal V_{m2}$ where the summand $\mathtt v_m$ could be large. 
 \subsection{Digit distribution in multiplicatively independent bases}   
Let $b$ and $\mathtt s$ be multiplicatively independent bases. Then there exists a unique integer $n \geq 2$ and a unique sequence of distinct primes $\{\mathtt p_1, \ldots, \mathtt p_n\}$ such that 
 \begin{equation} \label{prime-factorization} 
  b = \prod_{j=1}^{n} \mathtt p_j^{\beta_j}, \qquad \qquad \mathtt s = \prod_{j=1}^{n} \mathtt p_j^{\sigma_j}, 
 \end{equation} 
 where $\beta_j, \sigma_j$ are non-negative integers with the following properties: {\footnote{This complements a similar statement for multiplicatively dependent bases: $ b \sim \mathtt s$ if and only if they have identical prime factors and  \eqref{prime-factorization} holds with $\beta_j, \sigma_j > 0$ for all $j$ and  $\frac{\beta_1}{\sigma_1} = \cdots = \frac{\beta_n}{\sigma_n}$.}} 
  \begin{align}
&\beta_j \text{ and $\sigma_j$ cannot vanish simultaneously for any $j \in \{1, 2, \ldots, n\}$}, \label{p0} \\ 
&\frac{\sigma_1}{\beta_1} \geq \frac{\sigma_2}{\beta_2} \geq \cdots \geq \frac{\sigma_n}{\beta_n}, \text{ with at least one strict inequality; } \label {p1} \\ &\text{in particular this means } \sigma_1 \geq 1 \; \text{ or } \; \mathtt p_1 \mid \mathtt s, \; \text{ and } \; \mathtt p_1^{\sigma_1} \leq \mathtt s. \label{p2}
 \end{align}
 In \cite{s60}, Schmidt recorded a number of observations quantifying the number-theoretic ``independence'' of bases $b$ and $\mathtt s$ with $b \nsim \mathtt s$, loosely based around a meta-principle: if a large collection of integers happens to be highly structured in base $b$, in the sense that their residues modulo $b^k$ is concentrated on a small set, then their residues modulo $\mathtt s^k$ are relatively well-distributed. The following precise formulation of this phenomenon will be important in our analysis.  
\vskip0.1in 
\noindent Given $\mathtt k, \mathtt j \in \mathbb N$ with $\mathtt j \geq 2$, let us define positive integers $(\mathtt k)_{\mathtt j}$ and $(\mathtt k)_{\mathtt j}'$ by 
\begin{equation} \label{factor-power-power'} 
\mathtt k = (\mathtt k)_{\mathtt j} \times (\mathtt k)_{\mathtt j}' \; \; \text{ where } \;\;  (\mathtt k)_{\mathtt j} \text{ is a non-negative integer power of } \mathtt j \; \; \text{ and } \; \; \mathtt j \nmid (\mathtt k)_{\mathtt j}'. 
\end{equation} 
\begin{lemma}[{\cite[Lemma 5A]{s60}}] \label{Schmidt-mult-indep-lemma}
 Let $b, \mathtt s \in \mathbb N \setminus \{1\}$ be two multiplicatively independent bases admitting  the factorizations \eqref{prime-factorization} with the ordering of exponents \eqref{p1}. Then there exists a constant $\widetilde{c}(b, \mathtt s) > 0$ with the following property. For any $\xi, k \in \mathbb N$, $\varrho \in \{0, 1, \ldots, \mathtt s^k-1\}$ and any choice $\mathcal K(\mathtt s^k)$ of a complete system of non-negative residues mod $\mathtt s^k$, one has the estimate 
 \begin{align} 
& \#\bigl( \mathbb X(\xi, \varrho; b, \mathtt s, k) \bigr)  \leq \widetilde{c}(b, \mathtt s) (\xi)_{\mathtt p_1} \Bigl(\frac{\mathtt s}{\mathtt p_1}\Bigr)^k, \; \text{ where } \label{X-cardinality} \\ 
&\mathbb X(\xi, \varrho; b, \mathtt s, k)  := \bigl\{u \in \mathcal K(\mathtt s^k)   : \xi (b^u)_{\mathtt s}' = \varrho \text{ mod } \mathtt s^k \bigr\}. \label{def-X-set}
 \end{align} 
 Here $( \cdot)_{\mathtt s}'$ is defined as in \eqref{factor-power-power'}. 
 \vskip0.1in
\noindent Stated differently, if $\mathbf d(u; \xi, \mathtt s, k)$ denotes the ordered sequence of the first $k$ digits of the integer $\xi (b^u)_{\mathtt s}'$ in base $\mathtt s$, then the number of identical digit sequences  $\mathbf d(u; \xi, \mathtt s, k)$ as $u$ ranges over $\mathcal K(\mathtt s^k)$ is at most $\widetilde{c}(b, \mathtt s) (\xi)_{\mathtt p_1} \left({\mathtt s}/{\mathtt p_1}\right)^k$. 
 \end{lemma}
 \noindent {\em{Remarks: }} 
 \begin{enumerate}[1.]
 \item The version of Lemma \ref{Schmidt-mult-indep-lemma} stated in \cite[Lemma 5A]{s60} gives $\widetilde{c}(b, \mathtt s) (\xi)_{\mathtt p_1} (\mathtt s/2)^{k}$ as the right hand side of \eqref{X-cardinality}, but a perusal of the proof yields the stronger bound presented here. 
 \item The proof of \cite[Lemma 5A]{s60} also prescribes an explicit value for the constant $\widetilde{c}(b, \mathtt s)$ in \eqref{X-cardinality}, using the upper bound in \eqref{W1-cardinality} in the special case $N = \mathtt p^k$ and $\mathtt p = \mathtt p_1$: 
\begin{align} 
\widetilde{c}(b, \mathtt s)
 &\leq 2\sigma_1 b^{\mathtt p_1} \bigl(1 + b + b^2 + \ldots + b^{\sigma_1-1}\bigr) \nonumber \\ & = 2 \sigma_1 b^{\mathtt p_1} \frac{b^{\sigma_1}-1}{b-1} \leq  2 \sigma_1 b^{\mathtt p_1 + \sigma_1} \leq 2 \log_2(\mathtt s) b^{2 \mathtt s}. \label{quantitative-cbs}
\end{align} 
The last displayed inequality uses the relations $2^{\sigma_1} \leq \mathtt p_1^{\sigma_1} \leq \mathtt s$ and $\mathtt p_1 \leq \mathtt s$ from \eqref{p2}.
\end{enumerate}  
\subsection{Exponents connecting multiplicatively independent bases} 
Let $b, \mathtt s \in \mathbb N \setminus \{1\}$, $b \nsim \mathtt s$. For any $\xi \in \mathbb N$ and any non-negative integer $u$, let us write 
\begin{equation} \label{def-dl} 
\xi (b^u)'_{\mathtt s} = \sum_{j=0}^{\infty} \mathtt d_j \mathtt s^j = \sum_{j=0}^{\infty} \mathtt d_j(u; \xi, b, \mathtt s) \mathtt s^j , \qquad \mathtt d_j \in \mathbb Z_{\mathtt s} \text{ for all $j$.}
\end{equation}  
Given $\ell \in \mathbb N$, an integer $0 \leq \ell' < \ell$, and $u \in \mathbb Z_{\mathtt s^{\ell}}$, let 
\begin{equation} \label{dell} 
 \mathbf d_{\ell}= \mathbf d_{\ell}(u; \xi, b, \mathtt s) := (\mathtt d_0, \ldots, \mathtt d_{\ell-1}) 
 \end{equation} denote the block of first $\ell$ digits of the number $\xi (b^u)'_{\mathtt s}$ appearing in \eqref{def-dl}. We express $\mathbf d_{\ell}$ as 
\begin{equation} 
\mathbf d_{\ell} = \bigl(\underline{\mathbf d}, \overline{\mathbf d} \bigr), \quad \text{ where } \quad \underline{\mathbf d} := {\mathbf d}_{\ell'} = (\mathtt d_0, \mathtt d_1, \ldots, \mathtt d_{\ell'-1}), \quad \overline{\mathbf d} = \overline{\mathbf d}_{\ell \ell'} := (\mathtt d_{\ell'}, \ldots, \mathtt d_{\ell-1}),
\end{equation}
with the understanding that $\underline{\mathbf d}$ is empty and $\overline{\mathbf d} = \mathbf d_{\ell}$ if $\ell' = 0$. 
\vskip0.1in
\noindent The goal in this section is to estimate how many values of $u$ give rise to non-generic digit sequences for $\xi (b^u)_{\mathtt s}'$ in base $\mathtt s$, in the sense of Section \ref{consecutive-digits-section}; namely, how many values of $u$ generate blocks of digits $\overline{\mathbf d}$ where $\mathfrak n_{\ell-\ell'}(\overline{\mathbf d}, \mathtt s)$ is small. With this in mind, we define   
\begin{equation} \label{def-Ell'}
\mathbb E(\xi, \ell, \ell') = \mathbb E(\xi, \ell, \ell'; \mathtt s)  := \bigl\{u \in \mathbb Z_{\mathtt s^{\ell}} : \mathfrak n_{\ell - \ell'}\bigl(\overline{\mathbf d}(u; \xi,  b, \mathtt s), \mathtt s \bigr) \leq \kappa (\ell - \ell') \bigr\}.
\end{equation} 
Here the quantity $\mathfrak n_{\ell - \ell'}(\overline{\mathbf d}, \mathtt s)$, defined as in \eqref{def-nl}, denotes the cardinality of adjacent digit pairs $(\mathtt d_{j-1}, \mathtt d_j)$ in $\overline{\mathbf d}$ that avoid the edge values 0 and $\mathtt s^2-1$ when expressed in base $\mathtt s$. The small constant $\kappa = \kappa(\mathtt s) >0$ is chosen to obey the inequality \eqref{choice-of-kappa-2} guaranteed by Lemma \ref{Schmidt-lemma-extremal-digits}. Since $\xi (b^u)_{\mathtt s}'$ is roughly equally distributed among the residue classes mod $\mathtt s^{\ell}$ as $u$ ranges over $\mathbb Z_{\mathtt s^{\ell}}$, one expects the size of  $\mathbb E(\xi, \ell, \ell')$ to be smaller than the trivial bound. The following lemma quantifies this intuition. 
\begin{lemma} \label{E-cardinality-lemma}
There is an absolute constant $C_0 > 0$ for which the collection of indices $\mathbb E(\xi, \ell, \ell')$ given by \eqref{def-Ell'} with $b,\mathtt s,\xi$ as in \eqref{dell} obeys the estimate 
\begin{equation} \label{E-cardinality}
\# \bigl[\mathbb E(\xi, \ell, \ell') \bigr] \leq C_0 \kappa^{-1} 2^{-\frac{\ell}{4} - \frac{3 \ell'}{4}} \log_2(\mathtt s) \mathtt s^{\ell + \ell'} b^{2\mathtt s} \bigl(\xi \bigr)_{\mathtt p_1} \quad \text{ for all } 0 \leq \ell' < \ell.
\end{equation}
\end{lemma} 
\begin{proof} 
We decompose $\mathbb E(\xi, \ell, \ell')$ according to the value of $\mathbf d_{\ell}$:
\begin{align}
\mathbb E(\xi, \ell, \ell') &= \bigsqcup_{\pmb{\mathfrak d}} \Bigl\{u \in \mathbb Z_{\mathtt s^{\ell}} : \mathfrak n_{\ell - \ell'}(\overline{\pmb{\mathfrak d}}, \mathtt s) \leq \kappa (\ell - \ell'), \; \;  \mathbf d_{\ell}(u; \xi, b, \mathtt s) = \pmb{\mathfrak d} = (\underline{\pmb{\mathfrak d}}, \overline{\pmb{\mathfrak d}}) \in \mathbb Z_{\mathtt s}^{\ell'} \times \mathbb Z_{\mathtt s}^{\ell - \ell'} \Bigr\} \nonumber \\
&= \bigsqcup_{\pmb{\mathfrak d}} \Bigl\{u \in \mathbb Z_{\mathtt s^{\ell}} : {\mathbf d}_{\ell}(u; \xi, b, \mathtt s) = \pmb{\mathfrak d} =  \bigl(\underline{\pmb{\mathfrak d}}, \overline{\pmb{\mathfrak d}} \bigr)   \in  \mathbb Z_{\mathtt s}^{\ell'} \times \mathbb D(\ell - \ell', s) \Bigr\} \nonumber \\ 
&= \bigsqcup_{\pmb{\mathfrak d}} \Bigl\{ \mathbb X(\xi, \varrho(\pmb{\mathfrak d}); b, \mathtt s, \ell) : \pmb{\mathfrak d} =  \bigl(\underline{\pmb{\mathfrak d}}, \overline{\pmb{\mathfrak d}} \bigr)   \in  \mathbb Z_{\mathtt s}^{\ell'} \times \mathbb D(\ell - \ell', s) \Bigr\}. \label{Exi-decomp} 
\end{align}
Here $\varrho(\mathbf d_{\ell})$ denotes the unique integer in $\mathbb Z_{\mathtt s^{\ell}}$ whose digits are $\mathbf d_{\ell} = (\mathtt d_0, \mathtt d_1, \ldots, \mathtt d_{\ell-1})$. The sets $\mathbb D(\ell, \mathtt s)$ and $\mathbb X(\xi, \varrho; b, \mathtt s, \ell)$ are defined as in \eqref{extremal-digit-bounds} and \eqref{def-X-set} respectively. The cardinality of $\mathbb E(\xi, \ell, \ell')$ is now computed using the size of each component of the decomposition \eqref{Exi-decomp}:
\begin{align*}
\# \bigl[ \mathbb E(\xi, \ell, \ell')\bigr] &= \sum \Bigl\{\# \bigl[ \mathbb X(\xi, \varrho(\pmb{\mathfrak d}); b, \mathtt s, \ell)\bigr] : \pmb{\mathfrak d} \in \mathbb Z_{\mathtt s}^{\ell'} \times \mathbb D(\ell - \ell', \mathtt s) \Bigr\} \\  
&\leq \# \bigl[ \mathbb Z_{\mathtt s}^{\ell'} \times \mathbb D(\ell - \ell', \mathtt s) \bigr] \times \Bigl[ \max_{\varrho}\bigl\{\#\bigl(\mathbb X(\xi, \varrho; b, \mathtt s, \ell)\bigr) : \varrho \in \mathbb Z_{\mathtt s^{\ell}} \bigr\} \Bigr] \\ &\leq  \Bigl[ \mathtt s^{\ell'} \kappa^{-1} \times 2^{\frac{3}{4} (\ell - \ell')} \bigr] \times \Bigl[ \widetilde{c}(b, \mathtt s) (\xi)_{\mathtt p_1} \Bigl(\frac{\mathtt s}{\mathtt p_1}\Bigr)^{\ell} \Bigr] \\ 
&\leq \bigl[ \kappa^{-1} \mathtt s^{\ell'} \times 2^{\frac{3}{4}(\ell - \ell')} \bigr] \times \bigl[ 2 \log_2(\mathtt s) b^{2 \mathtt s} \left(\frac{\mathtt s}{2}\right)^{\ell} \bigl(\xi\bigr)_{\mathtt p_1} \bigr] \\ &\leq \kappa^{-1} 2^{1 - \frac{\ell}{4} - \frac{3 \ell'}{4}} \log_2(\mathtt s) b^{2 \mathtt s} \mathtt s^{\ell + \ell'} \bigl( \xi \bigr)_{\mathtt p_1}. 
\end{align*}
At the third step above, we have used the bounds on $\mathbb D$ and $\mathbb X$ obtained in Lemmas \ref{Schmidt-lemma-extremal-digits} and \ref{Schmidt-mult-indep-lemma} respectively. The fourth step uses the bound on $\widetilde{c}(b, \mathtt s)$ recorded in \eqref{quantitative-cbs}. The bound obtained in the last step is the same as the one claimed in \eqref{E-cardinality}, with $C_0 =2$. 
\end{proof}

\section{Estimating $\mathcal V_{m2}$: Proof of Proposition \ref{V2-sum-prop} \eqref{V2-sum-prop-parta}} \label{estimating-vm-section-Part1}
With the analytical and number-theoretic tools set up in Sections \ref{v-pointwise-section-take-2} and \ref{number-theoretic-tools-section} respectively, we are ready to prove Proposition \ref{V2-sum-prop}. This concerns the estimation of $\mathcal V_{m2}(h, b)$ for $h \in \mathbb Z \setminus \{0\}$. Part \eqref{V2-sum-prop-parta} of Proposition \ref{V2-sum-prop} considers bases $b$ in the collection $\mathscr{B}_1$, with $\pmb{\Pi}$ obeying the hypotheses of Proposition \ref{normality-special-prop}. The defining condition \eqref{bases-B1-def} of $\mathscr{B}_1$ says that every $\mathtt s_m \in \mathcal S$ has a prime divisor not dividing $b$. We complete the proof of  Proposition \ref{V2-sum-prop} \eqref{V2-sum-prop-parta} in this section.  
\vskip0.1in 
\noindent Let us recall from \eqref{prime-factorization} the factorizations of $b$ and $\mathtt s_m$ obeying \eqref{p0} and \eqref{p1} with $\mathtt s = \mathtt s_m$. 
The numbers $\mathtt s, n, \mathtt p_j, \beta_j, \sigma_j$ appearing in these factorizations depend on the index $m$ (and also $b$), but we will suppress this dependence to ease the notational burden, except in places where this becomes important (e.g. when we need to verify if a quantity is summable in $m$).  

\vskip0.1in 

\subsection{Partitioning of $\mathcal V_{m2}$ for $b \in \mathscr{B}_1$} 
We start by exploring the summands $\mathtt v$ of $\mathcal V_{m2}$ in \eqref{Vm2-sum-definition}, as given by the identity \eqref{identity}. The assumption $b \in \mathscr{B}_1$ is equivalent to the statement that $\beta_1= 0$ in \eqref{prime-factorization}, which means that $\mathtt p_1 \nmid b$. In this case, \eqref{p0} implies that $\sigma_1 \ne 0$, i.e., $\mathtt p_1 \mid \mathtt s$. Therefore 
\begin{equation}  \label{sb-case1}
\mathtt s = \mathtt s_m \nmid b^u \text{ for any $u \in \mathbb N$, i.e.,} \quad  b^u = (b^u)'_{\mathtt s} \text{ in the notation of \eqref{factor-power-power'}}. \end{equation}   However, $\mathtt s$ may share prime factors with $(b^v-1)$, and the resulting cancellation may cause $|\mathtt v(hb^u(b^v-1))|$ to be large. This will be the case, for example, if $\mathtt s = \mathtt p_1$ and $\mathtt p_1^{\mathtt a} \mid h(b^v-1)$, but $\mathtt p_1^{\mathtt b} \nmid h(b^v-1)$ in view of \eqref{identity}. Similarly, in view of the estimates \eqref{B*C-special-estimates}, \eqref{def-wm} and \eqref{wm-alternate}, the quantity $|\mathtt v(\cdot)|$ can also be large if the number of indices $j$ obeying \eqref{not-extreme} is small. Our first job is to show that while such situations are possible, they are infrequent. 
\vskip0.1in 
\noindent As in the analysis of $\mathbb V_1$,  we will decompose the set $\mathbb V_2$ of indices given by \eqref{def-collection-V2} into several sub-domains. Each of these sub-domains will contribute to the smallness of $\mathcal V_{m2}$ due to a different reason. Let us recall the definitions of 
\begin{itemize} 
\item the integer sequences $\{\mathtt K_m : m \geq 1\}$ from \eqref{Km-choice} and $\{\mathtt a_m, \mathtt b_m: m \geq 1\}$ from \eqref{normality-hypotheses}; 
\item the sequence of numerical scalars $\{\kappa_m = \kappa(\mathtt s_m): m \geq 1\} \subseteq (0, \infty)$ depending only on $\mathcal S$, given by \eqref{choice-of-kappa-2} from Lemma \ref{Schmidt-lemma-extremal-digits}. \label{kappa-crossref}
\end{itemize} 
\noindent Given any $N$, let $R = R(N, m)\in \mathbb N$ denote the unique index such that 
\begin{equation} \mathtt s_m^{R-1} \leq N < \mathtt s_m^R.  \label{N-and-R} \end{equation}
The range of $N$ given by the definition \eqref{def-collection-V2} of $\mathbb V_2$, when combined with \eqref{Km-choice},  \eqref{normality-hypotheses} and \eqref{N-and-R}, posits that for all sufficiently large $m$ depending on $b$, 
\begin{align}
\mathtt s_m^{\mathtt K_m^2/2} \leq \frac{\mathtt a_m}{4} \log_{b} (\mathtt s_m) &< N < \mathtt s_m^R \; \text{ and } \; \mathtt s_m^{R-1} \leq N < m\mathtt b_m \log_b(\mathtt s_m) < \mathtt s_m^{2\mathtt K_m^2-1}, \nonumber \\ &\text{so that } \; \frac{\mathtt K_m^2}{2} < R < 2\mathtt K_m^2. \label{R-and-K}
\end{align} 
Since $m$ is fixed, we will write $\mathtt K = \mathtt K_m$,  $\kappa = \kappa_m$, $\mathtt a = \mathtt a_m$, $\mathtt b = \mathtt b_m$ and $\mathtt s = \mathtt s_m$. Further, let us recall from \eqref{def-dl} the definition of $\mathbf d_{R} = \mathbf d_R(u; {|\xi|,} b, \mathtt s) = (\mathtt d_0, \ldots, \mathtt d_{R-1})$, with $\xi = \xi(h, v) = { h } (b^v-1)$; in view of \eqref{sb-case1}, this means that $\mathbf d_{R}$ is the string of first $R$ digits of $|\xi| b^{u} = |h|b^u (b^v-1)$ written in base $\mathtt s$:   
\begin{align} 
{|\xi|} \bigl(b^u\bigr)_{\mathtt s}' &=  |h|(b^v-1) (b^u)_{\mathtt s}' \nonumber \\ &= |h|b^u(b^v-1) = \sum_{j=0}^{\infty} \mathtt d_j(u; { |\xi|}, b, \mathtt s) \mathtt s^j, \quad  \xi = \xi(h, v) =  { h} (b^v-1). \label{what-is-xi}
\end{align} 
Let $\overline{\mathbf d}$ denote the truncation of the string $\mathbf d_{R}$ to the last $(R-\mathtt K)$ digits, namely \begin{equation} 
\overline{\mathbf d} = \overline{\mathbf d}(u; {|\xi|}, b, \mathtt s) :=  (\mathtt d_{\mathtt K}, \mathtt d_{\mathtt K+1}, \ldots, \mathtt d_{R-1}). \label{what-is-d-bar}
\end{equation} 
This is possible for all sufficiently large $m$ in light of \eqref{R-and-K}, which ensures that $\mathtt K < R$.  
\vskip0.1in
\noindent Equipped with these quantities, we partition the index set $\mathbb V_2$ given in \eqref{def-collection-V2} as follows:   
\begin{align} 
\mathbb V_2 &= \bigsqcup_{j=1}^{3} \mathbb V_{2j} \; \text{ where } \; \mathbb V_{21} := \Bigl\{(u, v, N) \in \mathbb V_2\; \big| \; \mathtt p_{1}^{\mathtt K} \mid (b^v-1) \Bigr\}, \label{V2-decomp}  \\ 
\mathbb V_{22} &:= \Bigl\{(u, v, N) \in \mathbb V_2 \; | \; \mathtt p_{1}^{\mathtt K} \nmid (b^v-1), \; {\mathfrak n_{R-\mathtt K} (\overline{\mathbf d}, \mathtt s )} \leq \kappa (R-\mathtt K)  \Bigr\}, \label{def-V22} \\
\mathbb V_{23} &:= \Bigl\{(u, v, N) \in \mathbb V_2 \; |  \; \mathtt p_{1}^{\mathtt K} \nmid (b^v-1), \; {\mathfrak n_{R-\mathtt K} (\overline{\mathbf d}, \mathtt s )} > \kappa (R - \mathtt K)  \Bigr\}. \label{def-V23}
\end{align}   
Here $\mathfrak n_{\ell}(\mathbf d, \mathtt s)$ denotes the index set associated to non-extremal values of a digit sequence $\mathbf d$, as given by \eqref{def-nl}. The partitioning \eqref{V2-decomp} of $\mathbb V_{2}$ leads to a corresponding decomposition for $\mathcal V_{m2}$:
\begin{equation} \label{Xj-def}
\mathcal V_{m2} = \sum_{j=1}^{3} \mathcal X_j, \; \text{ where } \; \mathcal X_j := \sum \Bigl\{\frac{1}{N^3}\bigl| \mathtt v_m \bigl(hb^u(b^v-1) \bigr) \bigr| : (u, v, N) \in \mathbb V_{2j}\Bigr\}.  
\end{equation} 
\begin{lemma} \label{X1-X2-X3-lemma}
Let $\pmb{\Pi}$ be a choice of parameters as in Proposition \ref{V2-sum-prop}(\ref{V2-sum-prop-parta})  (which is the same as in Proposition \ref{normality-special-prop}), and let $\mathscr{B}_1$ be the choice of bases given by \eqref{bases-B1-def}. Then for every base $b \in \mathscr{B}_1$ and every $h \in \mathbb Z \setminus \{0\}$,  there exist positive constants $c, C$ depending only on $h, b, \pmb{\Pi}$ such that the following inequalities hold for all $m \geq 1$:
\begin{equation} \label{X1+X2+X3}
\mathcal X_1 + \mathcal X_3  \leq C(h, b, \pmb{\Pi}) 2^{-c \mathtt K_m}, \qquad \mathcal X_2 \leq C(h, b, \pmb{\Pi}) 2^{-c\mathtt K_m^2}.
\end{equation} 
 Since $\{\mathtt K_m : m \geq 1\}$ is a strictly increasing sequence of positive integers according to \eqref{Km-choice}, the $m$-dependent terms on the right hand side of \eqref{X1+X2+X3} are summable in $m$. 
\end{lemma} 
\begin{corollary} \label{V2-sum-corollary-parta} 
The conclusion \eqref{V2-sum-estimate} of Proposition \ref{V2-sum-prop} holds under the hypotheses of part \eqref{V2-sum-prop-parta} of that proposition. 
\end{corollary}
\subsection{Proof of Lemma \ref{X1-X2-X3-lemma}}  
In the discussion below, values of the constants $C, C_0$ and $c, c_0$ may change from one occurrence to the next, but subject to the same dependencies as stated in the lemma; $c_0, C_0$ will denote absolute constants, and $c, C$ may depend on $h, b, \pmb{\Pi}$.
\subsubsection{Estimation of $\mathcal X_1$} \label{X1-estimation-section}
We claim that $\mathcal X_1$ is small because the cardinality of the index set $\mathbb V_{21}$ is small. Lemma \ref{I1-lemma} allows us to estimate this cardinality. Indeed, $\mathbb V_{21}$ given by \eqref{V2-decomp} and $\mathbb O$ given by \eqref{W1-cardinality} have the same defining conditions, making $\mathbb O$ a fibre of $\mathbb V_{21}$ for fixed $(u, N)$. For any subset $\mathbb S \subseteq \mathbb V$, let us denote by $\pi_N(\mathbb S)$ the set of indices $(u, v, N) \in \mathbb S$ for a fixed integer $N$. Then the inequality in \eqref{W1-cardinality} yields
\begin{equation} \label{O-estimate-1}
\# \bigl[ \pi_N \bigl(\mathbb V_{21} \bigr) \bigr] \leq  N \# \bigl(\mathbb O(N, \mathtt K_m; \mathtt p_{1m}, b) \bigr) \leq 2N^2 b^{\mathtt p_{1m}} \mathtt p_{1m}^{-\mathtt K_m} \leq 2N^2 b^{\mathtt s_m} 2^{-\mathtt K_m}.
\end{equation} 
At the last step above, we have used the fact that $\mathtt p_{1m} \mid \mathtt s_m$ and therefore $2 \leq \mathtt p_{1m} \leq \mathtt s_m$. The divisibility of $\mathtt s_m$ by $\mathtt p_{1m}$ follows from the assumption $b\in \mathscr{B}_1$; this argument appears in the line preceding \eqref{sb-case1}. Combining \eqref{O-estimate-1} with the trivial estimate $|\mathtt v_m| \leq C_0$ from \eqref{vm-pointwise-estimate}, we obtain 
\begin{align} 
\mathcal X_1 &\leq C_0 \sum \Bigl\{\frac{1}{N^3} \#\bigl( \pi_N(\mathbb V_{21}) \bigr) : \frac{\mathtt a_m}{4} \log_b (\mathtt s_m) < N < m \mathtt b_m  \log_b (\mathtt s_m)\Bigr\} \nonumber \\ 
&\leq C_0 b^{\mathtt s_m} 2^{-\mathtt K_m} \sum \Bigl\{\frac{1}{N} : \frac{\mathtt a_m}{4} \log_b (\mathtt s_m) < N < m \mathtt b_m \log_b (\mathtt s_m)\Bigr\} \nonumber \\
&\leq C_0 b^{\mathtt s_m} 2^{-\mathtt K_m} \log_b \Bigl(\frac{4m \mathtt b_m}{\mathtt a_m}\Bigr) \leq C_0  b^{\mathtt s_m} 2^{-\mathtt K_m} \log_b (m) \leq C(b) 2^{-\mathtt K_m/2}. \label{X1-bound} 
\end{align} 
The second inequality in \eqref{X1-bound} uses the relation $\mathtt b_m = 2m \mathtt a_m$ from \eqref{normality-hypotheses}, and also the inequality 
\[ b^{\mathtt s_m} \log_b(m) < 2^{\mathtt K_m/2}, \; \; \text{ which is the same as } \; \; \mathtt s_m \log_2 b + \log_2 \log_b m < \frac{\mathtt K_m}{2}.  \]  The latter inequality holds for all sufficiently large $m$ depending on $b$, by virtue of the growth assumption \eqref{Km-choice} on $\{\mathtt K_m\}$.
\qed 
\subsubsection{Estimation of $\mathcal X_2$} \label{X2-estimation-section} 
Like $\mathcal X_1$, the quantity $\mathcal X_2$ is also small due to the small size of the underlying index set $\mathbb V_{22}$, given by \eqref{def-V22}. The defining condition in $\mathbb V_{22}$ involving $\mathfrak n_{R-\mathtt K}$ is the same as the one for $\mathbb E(\xi, {R, \mathtt K})$ given by \eqref{def-Ell'}. Thus $\mathbb V_{22}$ permits the following description:  
\[ \mathbb V_{22} = \Bigl\{(u, v, N) \in \mathbb V_2: v \notin \mathbb O(N, \mathtt K; \mathtt p_1, b), \; u \in \mathbb E({|\xi|, R, \mathtt K}) \Bigr\}. \] 
Here $\xi = \xi(h, v)$ is given by \eqref{what-is-xi}. In order to estimate $\pi_N \bigl(\mathbb V_{22}\bigr)$, we appeal to Lemma \ref{E-cardinality-lemma} which yields the following bound:
\begin{align} 
\# \bigl[ &\pi_N \bigl( \mathbb V_{22} \bigr) \bigr]  
\leq \sum_{v} \Bigl\{ \# \bigl[ \mathbb E\bigl(|\xi(h,v)|, R, \mathtt K\bigr)\bigr] : v \notin \mathbb O(N, \mathtt K; \mathtt p_1, b) \Bigr\} \nonumber \\
&\leq N \max_{v} \Bigl\{\# \bigl[ \mathbb E\bigl(|\xi(h,v)|, R, \mathtt K \bigr)\bigr] :  v \notin \mathbb O(N, \mathtt K; \mathtt p_1, b) \Bigr\} \nonumber \\ 
&\leq C_0 N \max_v \Bigl\{\kappa^{-1} 2^{-\frac{R}{4} - \frac{3\mathtt K}{4}} \log_2(\mathtt s) \mathtt s^{R + \mathtt K} b^{2\mathtt s} \bigl(|\xi(h,v)|\bigr)_{\mathtt p_1} :  v \notin \mathbb O(N, \mathtt K; \mathtt p_1, b)  \Bigr\}. \label{V22-card-step0}
\end{align}
Let us pause for a moment to estimate the size of $(|\xi(h,v)|)_{\mathtt p_1}$, which according to \eqref{factor-power-power'}, is the highest power of $\mathtt p_1$ in $|\xi(h,v)|$. From the defining property of $v \notin \mathbb O(N, \mathtt K; \mathtt p_1, b)$, we find that $\mathtt p_1^{\mathtt K} \nmid (b^v-1)$, and therefore 
\[ \bigl( |\xi(h,v)| \bigr)_{\mathtt p_1} = \bigl({|h|} (b^v-1) \bigr)_{\mathtt p_1} \leq {|h|} (b^v-1)_{\mathtt p_1} \leq {|h|} \mathtt p_1^{\mathtt K}.  \]
Using this, the bound in \eqref{V22-card-step0} reduces to
\begin{align}
\# \bigl[ \pi_N \bigl( \mathbb V_{22} \bigr) \bigr]  &\leq C_0 {|h|} N \kappa^{-1} 2^{-\frac{R}{4} - \frac{3\mathtt K}{4}} \log_2(\mathtt s) \mathtt s^{R + \mathtt K} b^{2\mathtt s} \mathtt p_1^{\mathtt K} \nonumber \\  &\leq C(h) N \kappa^{-1} 2^{-\frac{R}{4} - \frac{3\mathtt K}{4}} \log_2(\mathtt s) \mathtt s^{R + 2\mathtt K} b^{2\mathtt s} \nonumber \\ 
&\leq C(h) N^2 \kappa_m^{-1} \mathtt s_m^{1 + 2\mathtt K_m} b^{2 \mathtt s_m} \log_2(\mathtt s_m) 2^{-\frac{R}{4} - \frac{3\mathtt K_m}{4}}.  \label{piV22-estimate}
\end{align}   
The second inequality in the display above follows from $\mathtt p_1 \leq \mathtt s$, a consequence of $\mathtt p_1 \mid \mathtt s$. The last inequality uses the bound $\mathtt s_{m}^{R-1} \leq N$ from \eqref{N-and-R}. Substituting \eqref{piV22-estimate} into the sum \eqref{Xj-def} representing $\mathcal X_2$, and applying the trivial bound $|\mathtt v_m| \leq C_0$ from \eqref{vm-pointwise-estimate}, we arrive at
\begin{align}
\mathcal X_2 &\leq \sum \Bigl\{\frac{1}{N^3} \# \bigl[ \pi_N(\mathbb V_{22}) \bigr] : \frac{\mathtt a_m}{4} \log_b(\mathtt s_m) < N < m \mathtt b_m \log_b(\mathtt s_m) \Bigr\} \nonumber \\
&\leq C(h) \kappa_m^{-1} \mathtt s_m^{1 + 2\mathtt K_m} b^{2 \mathtt s_m} 2^{-\frac{3\mathtt K_m}{4}}\log_2(\mathtt s_m) \sum \Bigl\{N^{-1} 2^{-\frac{R}{4}} : \frac{\mathtt a_m}{4} \log_b(\mathtt s_m) < N < m \mathtt b_m \log_b(\mathtt s_m)  \Bigr\} \nonumber \\
&\leq C(h) \kappa_m^{-1} \mathtt s_m^{2\mathtt K_m+1} b^{2 \mathtt s_m} 2^{-\frac{3\mathtt K_m}{4}}\log_2(\mathtt s_m)  \sum  \Bigl\{N^{-1} 2^{-\frac{R}{4}}  : \mathtt s_m^{R-1} \leq N < \mathtt s_m^R, \; R > \frac{\mathtt K_m^2}{2} \Bigr\} \nonumber \\
&\leq C(h) \kappa_m^{-1} \mathtt s_m^{2\mathtt K_m+1}  b^{2 \mathtt s_m} 2^{-\frac{3\mathtt K_m}{4}} \bigl[ \log_2(\mathtt s_m) \bigr]^2 2^{- \mathtt K_m^2/8} \leq C(h, b) 2^{- c\mathtt K_m^2}. \label{X2-bound} 
\end{align}  
{Here $c$} denotes a small positive constant depending only on $b$ and $\pmb{\Pi}$. At the third step of the display above, we have decomposed the sum in $N$ from the second step into two iterated sums, the inner sum involving $N \in [\mathtt s_m^{R-1}, \mathtt s_m^R)$ for a fixed integer $R$, and the outer sum ranging over the geometric scales $R$. The range of $R$ follows from \eqref{R-and-K}.  
The final inequality in \eqref{X2-bound} uses the rapidly increasing property \eqref{Km-choice} of $\{\mathtt K_m: m \geq 1\}$. A more detailed verification of this step is as follows, 
\begin{align*} 
&\kappa_m^{-1} \mathtt s_m^{2\mathtt K_m+1} b^{2 \mathtt s_m} 2^{-\frac{3\mathtt K_m}{4}} \bigl[ \log_2 (\mathtt s_m) \bigr]^2 2^{- \mathtt K_m^2/8} \leq \kappa_m^{-1} \mathtt s_m^{3\mathtt K_m} b^{3\mathtt s_m} 2^{-\frac{3\mathtt K_m}{4}}  2^{- \mathtt K_m^2/8}
= 2^{\mathtt L_m - \mathtt K_m^2/8}, \; \text{ with } \\
&\mathtt L_m := 3 \mathtt K_m \log_2(\mathtt s_m) - \frac{3\mathtt K_m}{4} + 3 \mathtt s_m \log_2(b) + \log_2(\kappa_m^{-1}), \text{ so that } \mathtt L_m \leq 10\mathtt K_m \log_2(\mathtt K_m) < \eta \mathtt K_m^2  
\end{align*}  
for any $\eta > 0$ and for all sufficiently large $m$ (depending on $b$), by virtue of \eqref{Km-choice}. Combining \eqref{X1-bound} and \eqref{X2-bound} produces the desired conclusion \eqref{X1+X2+X3}.  
\qed
\subsubsection{Estimation of $\mathcal X_3$}  \label{X3-estimation-section} 
Finally, we turn to $\mathcal X_3$. Let us recall the definitions of $\mathbb V_{23}$ and $\underline{\mathtt J}$ from \eqref{def-V23} and \eqref{def-cardinality-J-bar} respectively. It follows that for $(u,v, N) \in \mathbb V_{23}$ and $\zeta = h b^u(b^v-1)$, 
\begin{align} 
\underline{\mathtt J}(\zeta; \mathtt s, \mathtt K, \mathtt a) &= \# \bigl\{ \mathtt K \leq j \leq \mathtt a -1: 1 \leq \mathtt d_{j-1}({|\zeta|}) + \mathtt d_j({|\zeta|}) \mathtt s \leq \mathtt s^2-2  \bigr\} \nonumber \\
&\geq \# \bigl\{ \mathtt K + 1 \leq j \leq   R-1: 1 \leq \mathtt d_{j-1}({|\zeta|}) + \mathtt d_j({|\zeta|}) \mathtt s \leq \mathtt s^2-2  \bigr\} \nonumber \\ 
& = \mathfrak n_{R - \mathtt K}(\overline{\mathbf d}, \mathtt s) >  \kappa (R - \mathtt K).  \label{J-lower-bound} 
\end{align} 
The inequality at the second step of the preceding display uses the inclusion $\{\mathtt K, \mathtt K+1, \ldots, \mathtt a-1\} \subseteq \{\mathtt K+1, \mathtt K+2, \ldots, \mathtt R-1 \}$. This in turn follows from the assumptions \eqref{Km-choice}, \eqref{normality-hypotheses} on $\mathtt a_m$ and the relation \eqref{R-and-K} between $R$ and $\mathtt K$. Specifically, for $\mathtt K \geq 3$, \eqref{normality-hypotheses} and \eqref{R-and-K} give 
\[ \mathtt a -1 = \mathtt s^{\mathtt K^2} -1 > 2 \mathtt K^2 -1 > R -1 > \frac{\mathtt K^2}{2} -1> \mathtt K. \] 
The last inequality in \eqref{J-lower-bound} follows from the definition \eqref{def-V23} of $\mathbb V_{23}$. 
Substituting \eqref{J-lower-bound} into \eqref{v-alt-estimate} with $\mathtt a' = \mathtt K$, we obtain for all large $m$ that 
\begin{align} 
\bigl| \mathtt v_m(\zeta) \bigr| &\leq 2 \Bigl[ \frac{\mathtt s_m^{\mathtt a_m - \mathtt K_m}}{\mathtt N_m} + \Bigl(1 - \frac{c_0}{\mathtt s_m^5}\Bigr)^{\kappa_m(R - \mathtt K_m)} \Bigr] \nonumber \\
&\leq 4 \Bigl[ \frac{\mathtt s_m^{\mathtt a_m-\mathtt K_m} \mathtt s_{m-1}^{\mathtt b_{m-1}}}{\mathtt s_m^{\mathtt a_m}} + \Bigl(1 - \frac{c_0}{\mathtt s_m^5}\Bigr)^{\kappa_m R/2} \Bigr]  \nonumber \\ 
&\leq 4 \Bigl[\mathtt s_{m-1}^{\mathtt b_{m-1}} \mathtt s_m^{-\mathtt K_m} +  \Bigl(1 - \frac{c_0}{\mathtt s_m^5}\Bigr)^{\kappa_m R/2}  \Bigr]. \label{vm-bound-X3}
\end{align} 
The second inequality uses the relations $\mathtt N_m \geq \frac{1}{2}\mathtt s_m^{\mathtt a_m} \mathtt s_{m-1}^{-\mathtt b_{m-1}}$ and $R - \mathtt K_m \geq R/2$, which follow from \eqref{def-J} and \eqref{R-and-K} respectively. The latter inequality, which is equivalent to $R > 2\mathtt K_m$ follows from $R > \mathtt K_m^2/2$ in \eqref{R-and-K}, provided $\mathtt K_m \geq 4$. Inserting the bound \eqref{vm-bound-X3} for $\bigl|\mathtt v_m(\zeta) \bigr|$ into the expression \eqref{Xj-def} for $\mathcal X_3$ leads to 
\begin{align} 
\mathcal X_3 &= \sum \Bigl\{\frac{1}{N^3} \bigl|\mathtt v_m \bigl( hb^u(b^v-1)\bigr) \bigr| \; \Bigl| \; (u, v, N) \in \mathbb V_{23} \Bigr\} \nonumber \\ &\leq C_0 \sum \Biggl\{ \frac{1}{N^3} \Biggl[ \mathtt s_{m-1}^{\mathtt b_{m-1}} \mathtt s_m^{-\mathtt K_m} +  \Bigl(1 - \frac{c_0}{\mathtt s_m^5}\Bigr)^{\kappa_m R/2} \Biggr] \; \Bigl| \; (u, v, N) \in \mathbb V_{23} \Biggr\} \nonumber \\ 
&\leq C_0 \sum \Biggl\{ \frac{1}{N} \Biggl[ \frac{\mathtt s_{m-1}^{\mathtt b_{m-1}}}{\mathtt s_m^{\mathtt K_m}}  + \Bigl(1 - \frac{c_0}{\mathtt s_m^5}\Bigr)^{\kappa_m R/2} \Biggr] \; \Bigl| \; \frac{\mathtt a_m}{4} \log_b(\mathtt s_m) < N < m \mathtt b_m \log_b(\mathtt s_m)   \Biggr\}. \label{X3-step-1}
\end{align}
The range of $N$ in the sum above is the same as it was for $\mathcal X_2$, since for every $j = 1,2,3$, the collection $\mathbb V_{2j}$ inherits this property from $\mathbb V_2$ given in \eqref{def-collection-V2}. As in the proof for $\mathcal X_2$, we decompose the sum in $N$ into a double sum, with the inner sum over $N \in [\mathtt s_{m}^{R-1}, \mathtt s_m^R)$ and the outer sum over $R$ in the range \eqref{R-and-K}. Inserting these into \eqref{X3-step-1} yields a geometric sum in $R$:
\begin{align} 
\mathcal X_3 &\leq C_0 \sum \Biggl\{ \frac{1}{N} \Biggl[ \frac{\mathtt s_{m-1}^{\mathtt b_{m-1}}}{\mathtt s_m^{\mathtt K_m}}   + \Bigl(1 - \frac{c_0}{\mathtt s_m^5}\Bigr)^{\kappa_m R/2} \Biggr] \; \Biggl| \; \mathtt s_m^{R-1} \leq N < \mathtt s_m^R, \; \frac{\mathtt K_m^2}{2} <  R < 2 \mathtt K_m^2   \Biggr\} \nonumber \\ 
&\leq C_0 {\log_2 }(\mathtt s_m) \sum \Bigl\{ \frac{\mathtt s_{m-1}^{\mathtt b_{m-1}}}{\mathtt s_m^{\mathtt K_m}}   + \Bigl(1 - \frac{c_0}{\mathtt s_m^5}\Bigr)^{\kappa_m R/2} \; \Bigl| \; \frac{\mathtt K_m^2}{2} < R < 2\mathtt K_m^2   \Bigr\} \nonumber \\ 
&\leq C_0 {\log_2 }(\mathtt s_m) \Biggl[\mathtt K_m^2  \mathtt s_{m-1}^{\mathtt b_{m-1}}\mathtt s_m^{-\mathtt K_m} + \kappa_m^{-1} \mathtt s_m^5 \Bigl(1 - \frac{c_0}{\mathtt s_m^5}\Bigr)^{\kappa_m \mathtt K_m^2/4} \Biggr] \leq C_0 2^{-{c_0} \mathtt K_m}. \label{X3-final}
\end{align} 
The steps in the display above require some justification. The penultimate inequality involves the evaluation of a geometric sum  of the form 
\begin{equation} \label{theta-sum} \sum_{R} \Bigl\{ \vartheta^{R}: \frac{\mathtt K^2}{2} < R < \mathtt K^2\Bigr\} \leq \frac{\vartheta^{\mathtt K^2/2}}{1 - \vartheta}, \quad \text{ with } \quad \vartheta =  \Bigl(1 - \frac{c_0}{\mathtt s_m^5}\Bigr)^{\kappa_m /2}. \end{equation}  Bernoulli's inequality says that $(1-x)^r\le 1-rx$ for all real numbers $0\le r\le 1$ and $x\leq 1$. Since $\vartheta$ is of the form $(1-x)^r$ with $r = \kappa_m/2 \in (0,1)$ and $x = c_0/\mathtt s_m^5 \leq 1$, we use this inequality to estimate $(1 - \vartheta)$ from below,  
\begin{equation}  \label{1-theta-estimate}
1 - \vartheta = 1 - \Big(1 - \frac{c_0}{\mathtt s_m^5}\Big)^{\kappa_m/2}
\geq \frac{\kappa_m c_0}{2 \mathtt s_m^5}.
\end{equation}
Combining \eqref{theta-sum} and \eqref{1-theta-estimate} leads to the first expression in \eqref{X3-final}. The last inequality in \eqref{X3-final} is ensured by the choice of $\mathtt a_m, \mathtt b_m$ {in \eqref{normality-hypotheses} } and the growth properties \eqref{Km-choice} of $\mathtt K_m$. Specifically, the assumption $\log \mathtt K_m \geq (m \mathtt K_{m-1} \mathtt s_{m-1})^{C_0}$ implies that
\[ \mathtt b_{m-1} \log_2 (\mathtt s_{m-1}) = 2(m-1) \mathtt s_{m-1}^{\mathtt K_{m-1}^2} \log_2 (\mathtt s_{m-1})  < c_0 \mathtt K_m\quad\text{for some}\ 0<{c_0}<1, \]
as a result of which we obtain the inequality 
\begin{equation} \label{X3-verification-1} 
\log_2(\mathtt s_m) \mathtt K_m^2 \mathtt s_{m-1}^{\mathtt b_{m-1}} \mathtt s_m^{-\mathtt K_m} \leq \mathtt K_m^3  \mathtt s_{m}^{\mathtt b_{m-1} \log_2 (\mathtt s_{m-1}) } \mathtt s_m^{-\mathtt K_m} \leq 2^{-c_0 \mathtt K_m},
\end{equation} 
which bounds the first summand in \eqref{X3-final}. Similarly the relation $\kappa_m \mathtt K_m \geq 2\mathtt s_m^5$, which is a consequence of $\kappa_m \mathtt K_m \mathtt s_m^{-C_0} \rightarrow \infty$ from \eqref{Km-choice} implies 
\begin{equation} 
\kappa_m^{-1} \mathtt s_m^5 \Bigl(1 - \frac{c_0}{\mathtt s_m^5}\Bigr)^{\kappa_m \mathtt K_m^2/2} \leq \mathtt K_m \Bigl(1 - \frac{c_0}{\mathtt s_m^5}\Bigr)^{\mathtt s_m^5 \mathtt K_m} \leq {C_0}2^{-c_0 \mathtt K_m}. 
\label{X3-verification-2}
\end{equation}
This justifies the bound on the second summand in \eqref{X3-final}. The establishment of the bound \eqref{X3-final} completes the proof of \eqref{X1+X2+X3}. 
\qed

\section{Estimating $\mathcal V_{m2}$: Proof of Proposition \ref{V2-sum-prop} \eqref{V2-sum-prop-partb}} \label{estimating-vm-section-Part2}
This section assumes the hypotheses \eqref{normality-hypotheses} and \eqref{normality-digit-01} on $\pmb{\Pi}$, as required by Proposition \ref{normal-prop}. 
Under these assumptions, we aim to establish the relation \eqref{V2-sum-estimate} for a base $b$ in the collection $\mathscr{B}$ defined in Proposition \ref{normal-prop}. The case where $b \in \mathscr{B}_1 \subsetneq \mathscr{B}$ corresponds to $\beta_1= \beta_{1m} = 0$ in the factorization \eqref{prime-factorization}. This case has already been addressed in Section \ref{estimating-vm-section-Part1} as part of the proof of Proposition \ref{normality-special-prop}, under the restriction $0\in \mathscr{D}_m$ which is weaker than \eqref{normality-digit-01}. This section is devoted to the proof of \eqref{V2-sum-estimate} in the complementary case $b \in \mathscr{B} \setminus \mathscr{B}_1$. This corresponds to the assumption that $\beta_{1} = \beta_{1m}> 0$ in \eqref{prime-factorization}, which in turn means by \eqref{p0} and \eqref{p1} that
\begin{equation} \label{beta1-nonzero} 
\beta_{j} = \beta_{jm}> 0 \text{ for every $1 \leq j \leq n = n(m)$ in the factorization \eqref{prime-factorization}.} 
\end{equation} 
In other words, each prime divisor of $\mathtt s = \mathtt s_m$ divides $b$. Obviously $\mathtt s \nmid (b^v-1)$ for any $v\in\mathbb{N}$. 
\subsection{Partitioning of $\mathcal V_{m2}$ for $b \in \mathscr{B} \setminus \mathscr{B}_1$} 
\noindent An important quantity in the analysis of this case is
\begin{equation} \label{what-is-taum}
\tau = \tau_m := \frac{\beta_{1m}}{\sigma_{1m}} > 0; \quad  \text{ so that } \quad \tau = \tau_m = \min \left\{\frac{ \beta_{jm}}{\sigma_{jm}} : 1 \leq j \leq n \right\}.
\end{equation}  
This leads to the following identity for all non-negative integers $u$:
\begin{equation} \label{ku}
\bigl(b^{u} \bigr)_{\mathtt s} = \mathtt s^{\mathtt k(u)}, \quad \mathtt k(u) := \lfloor u \tau \rfloor. 
\end{equation}
For $h \in \mathbb Z \setminus \{0\}$, let us factorize $h$ as in \eqref{factor-power-power'}, 
\begin{equation} \label{h-factorization} 
h = (h)_{\mathtt s} (h)'_{\mathtt s} = \mathtt s^{\mathtt h} (h)'_{\mathtt s}, \quad  \text{ where } (h)_{\mathtt s} = \mathtt s^{\mathtt h}, \; \mathtt h \in \mathbb N \cup \{0\}, \; \mathtt s \nmid (h)'_{\mathtt s}. 
\end{equation}  
Combining \eqref{ku} and \eqref{h-factorization}, one can express a frequency $hb^u(b^v-1) \in \mathbb Z \setminus \{0\}$ in the form  
\begin{equation} \label{new-xi} 
\left\{
\begin{aligned} 
&{h}b^u(b^v-1) = \xi (b^u)_{\mathtt s}' \mathtt s^{\mathtt k(u) + \mathtt h} = \zeta \mathtt s^{\mathtt k(u) + \mathtt h}, \; \text{ with } \\ &\xi  = \xi(h, v) := (h)_{\mathtt s}'(b^v-1), \; \zeta := \xi (b^u)_{\mathtt s}'.  
\end{aligned} \right. 
\end{equation} 
Let $\mathtt K = \mathtt K_m$ and $R = R(N, m)$ be as in Section \ref{estimating-vm-section-Part1}, given by \eqref{Km-choice} and \eqref{N-and-R} respectively. As before, we decompose the index set $\mathbb V_2$ given by \eqref{def-collection-V2} into several sub-domains,   
\begin{align} 
\mathbb V_2 &= \bigsqcup_{j=1}^{4} \overline{\mathbb V}_{2j} \; \; \text{ where } \; \; \overline{\mathbb V}_{21} := \bigl\{(u, v, N) \in \mathbb V_2 : 0 \leq \mathtt k(u) + \mathtt h \leq \mathtt K \bigr\},  \label{V2-decomp-2} \\ 
\overline{\mathbb V}_{22} &:= \Bigl\{(u, v, N) \in \mathbb V_2 : \mathtt k(u) + \mathtt h \geq \mathtt b - R \Bigr\}, \\  
\overline{\mathbb V}_{23} &:= \Bigl\{(u, v, N) \in \mathbb V_2 : \mathtt K < \mathtt k(u) + \mathtt h < \mathtt b - R, \;  \mathfrak n_{R}\bigl({\mathbf d}_R(u; {|\xi|}, b, \mathtt s), \mathtt s \bigr) \leq \kappa R  \Bigr\}, \label{V23-bar} \\ 
\overline{\mathbb V}_{24} &:= \Bigl\{(u, v, N) \in \mathbb V_2 : \mathtt K < \mathtt k(u) + \mathtt h < \mathtt b - R, \;  \mathfrak n_{R}\bigl({\mathbf d}_R(u; {|\xi|}, b, \mathtt s), \mathtt s \bigr) > \kappa R  \Bigr\}. \label{V2-decomp-2-4}
\end{align} 
Here $\kappa$ is the $\mathtt s$-dependent positive constant specified by \eqref{choice-of-kappa-2}. The vector $\mathbf d
_{R}$, defined as in \eqref{def-dl} and \eqref{dell}, denotes the sequence of first $R$ digits of $|\xi| (b^u)_{\mathtt s}'$, with $\xi$ given by \eqref{new-xi}. The quantity $\mathfrak n_{R}$ has been defined in \eqref{def-nl} and has also appeared in \eqref{def-Ell'}. The relation 
\[ \mathtt K + R + 1 < \mathtt K + 2 \mathtt K^2 + 1< \mathtt a = \mathtt s^{\mathtt K^2} < \mathtt b, \] obtained from \eqref{R-and-K} and \eqref{normality-hypotheses}, justifies the validity of the definitions \eqref{V2-decomp-2}--\eqref{V2-decomp-2-4}. The decomposition \eqref{V2-decomp-2} of $\mathbb V_2$ splits the sum $\mathcal V_{m2}$ into four parts:   
\begin{equation} \label{Yj-def}
\mathcal V_{m2} = \sum_{j=1}^{4} \mathcal Y_j, \; \text{ where } \; \mathcal Y_j := \sum \Bigl\{\frac{1}{N^3}\bigl| \mathtt v_m \bigl(hb^u(b^v-1) \bigr) \bigr| : (u, v, N) \in \overline{\mathbb V}_{2j}\Bigr\}.  
\end{equation}
The analogues of Lemma \ref{X1-X2-X3-lemma} and Corollary \ref{V2-sum-corollary-parta} in this setting are the following: 
\begin{lemma} \label{Yj-lemma}
Let $\pmb{\Pi}$ be a choice of parameters {as in Proposition \ref{V2-sum-prop} \eqref{V2-sum-prop-partb}}, and let $\mathscr{B}$ be the choice of bases given in Proposition \ref{normal-prop}. Then for every base $b \in \mathscr{B}$ and every $h \in \mathbb Z \setminus \{0\}$,  there exist positive constants $c, C$ depending only on $h, b, \pmb{\Pi}$ such that the following inequalities hold.
\begin{equation} \label{Y1+Y2+Y3+Y4}
\mathcal Y_1 + \mathcal Y_2 + \mathcal Y_3  \leq C(h,b, \pmb{\Pi}) 2^{-c \mathtt K_m^2}, \qquad \mathcal Y_4 \leq C(h,b, \pmb{\Pi}) 2^{-c\mathtt K_m}.
\end{equation} 
 \end{lemma} 
\begin{corollary} \label{V2-sum-corollary-partb}
The conclusion \eqref{V2-sum-estimate} of Proposition \ref{V2-sum-prop} holds under the hypotheses of part \eqref{V2-sum-prop-partb} of that proposition. 
\end{corollary} 
\subsection{Estimation of $\mathcal Y_1$}
We control $\mathcal Y_1$ using the bound $|\mathtt v_m| \leq C_0$ from \eqref{vm-pointwise-estimate}, the defining properties \eqref{def-collection-V2} of $\mathbb V_2$, and the small size of $\overline{\mathbb V}_{21}$. Summing consecutively in $u, v, N$, we obtain
\begin{align} 
\mathcal Y_1 &\leq \sum_{u, v, N} \Biggl\{\frac{1}{N^3}\bigl| \mathtt v_m \bigl(hb^u(b^v-1) \bigr) \bigr| \; \; \Biggl| \; \; \begin{aligned} &\mathtt k(\mathtt u) + \mathtt h = \lfloor u\tau_m \rfloor + \mathtt h \leq  \mathtt K_m, \; v\in \mathbb Z_N, \\  &\frac{\mathtt a_m}{4} \log_b (\mathtt s_m) < N < m\mathtt b_m \log_b (\mathtt s_m) \end{aligned} \Biggr\} \nonumber \\
&\leq C_0 \sum_{N} \Bigl\{\frac{\mathtt K_m}{\tau_m N^2} :  \frac{\mathtt a_m}{4} \log_b (\mathtt s_m) < N < m\mathtt b_m \log_b (\mathtt s_m) \Bigr\} \label{Y1-initial} \\ 
&\leq C_0 \frac{\mathtt K_m}{\tau_m \mathtt a_m \log_b (\mathtt s_m)}   = C_0 \mathtt K_m \mathtt s_m^{-\mathtt K_m^2} \frac{\sigma_{1m}}{\beta_{1m}} \frac{1}{\log_b (\mathtt s_m)} \; \; (\text{using \eqref{what-is-taum} and}\ \mathtt a_m=\mathtt s_m^{\mathtt K_m^2})\nonumber \\
&\leq C_0 \mathtt K_m \mathtt s_m^{-\mathtt K_m^2} \frac{\log_2 (\mathtt s_m)}{\log_b (\mathtt s_m)} \leq C(b) \mathtt K_m \mathtt s_m^{-\mathtt K_m^2} \leq C(b, \pmb{\Pi}) 2^{-c_0 \mathtt K_m^2} \text{ for some } c_0 > 0. \label{Y1-final}
\end{align}
The estimate \eqref{Y1-initial} in the summation above uses the fact that the number of summands $(u, v)$ is at most $2N \mathtt K_m \tau_m^{-1}$, which in turn follows from 
\[ u \tau_m - 1 + \mathtt h < \lfloor u \tau_m \rfloor + \mathtt h \leq \mathtt K_m, \quad \text{i.e.} \quad 0 \leq u \leq u + \frac{\mathtt h}{\tau_m} \leq \frac{\mathtt K_m+1}{\tau_m} \leq \frac{2\mathtt K_m}{\tau_m}. \] In the last line \eqref{Y1-final} of the display, we have used the inequalities 
\begin{equation} \label{beta-sigma-inequality} \beta_{1m} \geq 1 \quad \text{ and } \quad \sigma_{1m} \leq \sum_{j=1}^{n} \sigma_{jm} \log_2 (\mathtt p_j) \leq \log_2 (\mathtt s_m). \end{equation}
The first inequality in \eqref{beta-sigma-inequality} is the main assumption \eqref{beta1-nonzero}, the second follows from \eqref{prime-factorization}. 
\subsection{Estimation of $\mathcal Y_2$}
The sum $\mathcal Y_2$ is also controlled using \eqref{vm-pointwise-estimate}, although we use different parts of that inequality here.  
Let us write  
\begin{equation} \mathcal Y_2 = \mathcal Y_{2}' + \mathcal Y_{2}''; \label{Y2-decomp} \end{equation} 
the sum $\mathcal Y_2'$ ranges over the subset of $\overline{\mathbb V}_{22}$ where $\mathtt b_m - R \leq \mathtt k(u) + \mathtt h \le \mathtt b_m$. The sum $\mathcal Y_{2}''$ picks up the remaining portion. The trivial bound $|\mathtt v_m| \leq C_0$ in \eqref{vm-pointwise-estimate} and the range of $R$ in \eqref{R-and-K} yield  
\begin{align} 
\mathcal Y_{2}' & \leq \sum_{u, v, N} \Biggl\{\frac{1}{N^3}\bigl| \mathtt v_m \bigl(hb^u(b^v-1) \bigr) \bigr| \; \; \Biggl| \; \; \begin{aligned} & \mathtt b_m - R \leq \mathtt k(u) + \mathtt h \leq  \mathtt b_m, \; v \in \mathbb Z_N, \\  &\frac{\mathtt a_m}{4} \log_b (\mathtt s_m) < N  \end{aligned} \Biggr\} \nonumber \\ 
&\leq \sum_{u, v, N} \Biggl\{\frac{1}{N^3}\bigl| \mathtt v_m \bigl(hb^u(b^v-1) \bigr) \bigr| \; \; \Biggl| \; \; \begin{aligned} & \mathtt b_m - R \leq u\tau_m + \mathtt h \leq  \mathtt b_m + 1, \; v \in \mathbb Z_N, \\  &\frac{\mathtt a_m}{4} \log_b (\mathtt s_m) < N  \end{aligned} \Biggr\} \nonumber \\
&\leq C_0 \sum_N \Bigl\{\frac{R}{\tau_m N^2} \; \Bigl| \; \frac{\mathtt a_m}{4} \log_b (\mathtt s_m) < N  \Bigr\} \leq   \frac{C_0 \mathtt K_m^2}{\tau_m \mathtt a_m \log_b(\mathtt s_m)} \nonumber \\ &\leq C_0 \frac{\sigma_{1m}}{\beta_{1m}} \frac{\mathtt K_m^2}{\log_b (\mathtt s_m)} \mathtt s_m^{-\mathtt K_m^2} \leq C(b) \mathtt K_m^2 \mathtt s_m^{-\mathtt K_m^2} \leq C(b, \pmb{\Pi}) 2^{-c_0 \mathtt K_m^2}. \label{Y2-final-1} 
\end{align} 
The last step uses the fact $\mathtt a_m = \mathtt s_m^{\mathtt K_m^2}$ from \eqref{normality-hypotheses} and the inequality \eqref{beta-sigma-inequality} in the same way it was used to derive \eqref{Y1-final}. 
\vskip0.1in
\noindent For $\mathcal Y_2''$, we use the pointwise inequality $|\mathtt v_m(\xi)| \leq C_0 \mathtt s_m^{\mathtt b_m}/|\xi|$ from \eqref{vm-pointwise-estimate}. Combining this with the fact that $\sum_{v \geq 0} b^{-v} = b/(b-1)$, we obtain  for all large $m$, 
\begin{align} 
\mathcal Y_{2}'' &\leq \sum_{u, v, N} \Biggl\{\frac{1}{N^3}\bigl| \mathtt v_m \bigl(hb^u(b^v-1) \bigr) \bigr| \; \; \Biggl| \; \; \begin{aligned} &  \mathtt k(u) + \mathtt h >  \mathtt b_m, \; v \in \mathbb Z_N, \\  &N > \frac{\mathtt a_m}{4} \log_b (\mathtt s_m)\end{aligned} \Biggr\} \nonumber \\ 
&\leq C_0 \sum_{u, v, N}  \Biggl\{\frac{1}{N^3} \frac{\mathtt s_m^{\mathtt b_m}}{|h|b^{u+v}} \; \; \Biggl| \; \; \begin{aligned} &  u\tau_m >  \mathtt b_m - \mathtt h, \; v \in \mathbb Z_N, \\  &N > \frac{\mathtt a_m}{4} \log_b (\mathtt s_m)  \end{aligned} \Biggr\} \nonumber \\ 
&\leq C(h, b) \sum_{u, N} \Bigl\{ \frac{\mathtt s_m^{\mathtt b_m}}{N^3} b^{-u} \; \Bigl| \; u\tau_m >  \mathtt b_m - \mathtt h, \; N > \frac{\mathtt a_m}{4} \log_b (\mathtt s_m)  \Bigr\} \label{Y2-intermediate} \\ 
&\leq C(h, b) \mathtt s_m^{\mathtt b_m} b^{- \lfloor(\mathtt b_m - \mathtt h)/\tau_m \rfloor} \sum_N \Bigl\{N^{-3} | N > \frac{\mathtt a_m}{4}  \log_b (\mathtt s_m) \Bigr\} \nonumber \\  
& \leq C(h, b) \bigl[\mathtt s_m^{\mathtt b_m} b^{-\mathtt b_m /\tau_m} \bigr] \times b^{\mathtt h/\tau_m} \times \sum_N \Bigl\{N^{-3} | N > \frac{\mathtt a_m}{4}  \log_b (\mathtt s_m) \Bigr\} \nonumber  \\ 
&\leq C(h, b) \mathtt s_m^{\mathtt K_m} \bigl[ \mathtt a_m \log_b (\mathtt s_m) \bigr]^{-2}  \leq C(h, b) \mathtt s_m^{- c_0 \mathtt K_m^2} \text{ for some } c_0 > 0. \label{Y2-final-2}
\end{align} 
In \eqref{Y2-intermediate}, we observe from \eqref{h-factorization} that $2^{\mathtt h} \leq \mathtt s^{\mathtt h} \leq h$, i.e., $\mathtt h \leq \log_2(h)$, whereas $\mathtt b_m \rightarrow \infty$. As a result $\mathtt b_m - \mathtt h >0$ for all sufficiently large $m$ depending on $h$, justifying the summation in $u$ for such $m$.
In the last line \eqref{Y2-final-2} above, we have invoked the relation \eqref{p1} to bound the factors $\mathtt s^{\mathtt b} b^{-\mathtt b/\tau}$ and $b^{\mathtt h/\tau}$ from the preceding step:
\[
b^{\mathtt b/\tau} = \bigl(b^{\sigma_1/\beta_1} \bigr)^{\mathtt b} = \Bigl[ \prod_{j=1}^{n} \mathtt p_j^{\beta_j \sigma_1/\beta_1}\Bigr]^{\mathtt b}  \geq  \Bigl[ \prod_{j=1}^{n} \mathtt p_j^{\sigma_j} \Bigr]^{\mathtt b} = \mathtt s^{\mathtt b}, \text{ so that } \mathtt s^{\mathtt b} b^{-\mathtt b/\tau} \leq 1. \]
On the other hand, $b^{\mathtt h/\tau} \leq \mathtt s^{\mathtt K}$, or equivalently $\mathtt h \leq \tau \mathtt K \log_{b}(\mathtt s) = \mathtt K \frac{\beta_1}{\sigma_1} \log_b(\mathtt s)$ for all sufficiently large $m$ depending on $b$ and $h$. This follows from the observation that $\mathtt h \leq \log_2(h)$ is bounded above by a constant independent of $m$, whereas the relation $\sigma_{1m} \leq \log_2(\mathtt s_m)$ from \eqref{beta-sigma-inequality} dictates that 
\[ \mathtt K_m \frac{\beta_{1m}}{\sigma_{1m}} \log_{b}(\mathtt s_m) \geq \mathtt K_m \sigma_{1m}^{-1} \log_{2}(\mathtt s_m)/\log_{2}(b) \geq \mathtt K_m/\log_2(b) \rightarrow \infty. \] 
In view of \eqref{Y2-decomp}, \eqref{Y2-final-1} and \eqref{Y2-final-2}, we arrive at the estimate 
\begin{equation} \label{Y2-final} 
\mathcal Y_2 \leq C(h, b, \pmb{\Pi}) 2^{-c_0 \mathtt K_m^2}.
\end{equation} 
\subsection{Estimation of $\mathcal Y_3$}
The estimation of $\mathcal Y_3$ is similar to that of $\mathcal X_2$ in Section \ref{X2-estimation-section}, with small adjustments. Recalling the relation \eqref{N-and-R} between $N$ and $R$, and comparing the defining condition \eqref{V23-bar} of $\overline{\mathbb V}_{23}$ with the definition \eqref{def-Ell'} of $\mathbb E(\xi, \ell, \ell')$, we find that  
\begin{align} 
&\overline{\mathbb{\mathbb V}}_{23} \subseteq \bigl\{(u, v, N) \in \mathbb V_2 : u \in \mathbb E(|\xi(h,v)|, R, 0) \bigr\} \text{ with $\xi$ as in \eqref{new-xi}; as a result } \nonumber \\ 
&\# \bigl[ \pi_N \bigl( \overline{\mathbb V}_{23}\bigr) \bigr] \leq  N \sup_{v} \#\bigl[ \mathbb E(|\xi(h,v)|, R, 0) \bigr] \nonumber \\ &\hskip0.75in \leq C_0  {N}\kappa_m^{-1} 2^{-\frac{R}{4}} \log_2 (\mathtt s_m) \mathtt s_m^R b^{2 \mathtt s_m} \bigl(|\xi(h,v)|\bigr)_{\mathtt p_1} \nonumber \\
&\hskip0.75in \leq C(h) {N} \kappa_m^{-1} 2^{-\frac{R}{4}} \log_2 (\mathtt s_m) \mathtt s_m^R b^{2 \mathtt s_m}, \label{cardinality-V23-bar} 
\end{align}  
where $\pi_N(\cdot)$ is defined as in Section \ref{X1-estimation-section}.The second inequality in the display above uses the cardinality estimate \eqref{E-cardinality} for $\mathbb E$ derived in Lemma \ref{E-cardinality-lemma} with $\ell' = 0$, $\ell=R$. The inequality \eqref{cardinality-V23-bar} uses the assumption \eqref{beta1-nonzero} that $\beta_{1m} > 0$ or  equivalently $\mathtt p_1 \mid b$, as a result of which we deduce  
\begin{align*}
& \mathtt p_1 \mid b^v \text{ for all } v \in \mathbb N, \text{ which means } \mathtt p_1 \nmid (b^v-1), \text{ i.e. } (b^v-1)_{\mathtt p_1} = 1, \; \text{ and therefore } \\
&\bigl( |\xi(h,v)|\bigr)_{\mathtt p_1} = \bigl({|(h)'_{\mathtt s}|}(b^v-1)\bigr)_{\mathtt p_1} = ({|(h)'_{\mathtt s}|})_{\mathtt p_1}  (b^v-1)_{\mathtt p_1} \leq |(h)'_{\mathtt s}| \leq |h|. 
\end{align*} 
Substituting \eqref{cardinality-V23-bar} into the expression for $\mathcal Y_3$ {and using $|\mathtt v_m(\xi)| \leq C_0$ from \eqref{vm-pointwise-estimate} lead} to
\begin{align} 
\mathcal Y_3 &= \sum_{u,v, N} \Bigl\{\frac{1}{N^3}\bigl| \mathtt v_m \bigl(hb^u(b^v-1) \bigr) \bigr| \; \Bigl| \; (u, v, N) \in \overline{\mathbb V}_{23} \Bigr\} \leq C_0 \sum_{N} N^{-3} \# \bigl[\pi_N \bigl(\overline{\mathbb V}_{23} \bigr) \bigr] \nonumber \\ 
&\leq C_0 \sum_N \Bigl\{\frac{1}{N^2} \kappa_m^{-1} 2^{-\frac{R}{4}} \log_2 (\mathtt s_m) \mathtt s_m^R b^{2 \mathtt s_m} \; \Bigr| \; N > \frac{\mathtt a_m}{4} \log_2(\mathtt s_m) \Bigr\} \nonumber \\
&\leq C(h) \bigl[ \mathtt a_m \log_b(\mathtt s_m) \bigr]^{-1} \kappa_m^{-1} 2^{-c_0 \mathtt K_m^2} \log_2(\mathtt s_m) b^{2 \mathtt s_m} \leq C(h, b) 2^{-{c} \mathtt K_m^2}, \label{Y3-final}
\end{align} 
where the positive constant $c$ depends on $b$. The final inequality again follows from the rapid growth \eqref{Km-choice} of $\mathtt K_m$:
\[ \mathtt a_m^{-1} \kappa_m^{-1} 2^{-c_0 \mathtt K_m^2} b^{2 \mathtt s_m} = \kappa_m^{-1}  \mathtt s_m^{-\mathtt K_m^2} 2^{-c_0 \mathtt K_m^2 + 2 \mathtt s_m \log_2 b} \leq 2^{-{c} \mathtt K_m^2}. \] 
\subsection{Estimation of $\mathcal Y_4$} \label{Y4-estimation-section}
The procedure for estimating $\mathcal Y_4$ is similar to that for $\mathcal X_3$ in Section \ref{X3-estimation-section}. One of the defining conditions for $\overline{\mathbb V}_{24}$ dictates that $\mathtt K_m < \mathtt k(u) + \mathtt h < \mathtt b_m - R$ for $(u, v, N) \in \overline{\mathbb V}_{24}$. This means that
\begin{equation} \label{R-ku-h-inclusion} 
\{1, 2, \ldots, R\} \subseteq \bigl\{\mathtt K_m - \mathtt k(u) - \mathtt h, \ldots, \mathtt b_m - \mathtt k(u) - \mathtt h \bigr\}.  
\end{equation}  
The inclusion \eqref{R-ku-h-inclusion} results in a lower bound on the quantity $\mathtt J(h b^u(b^v-1); \mathtt s_m, \mathtt K_m, \mathtt b_m)$ given by \eqref{Jab}. In view of the factorization $hb^u(b^v-1) = {\mathtt s}^{\mathtt k(u) + \mathtt h} \xi (b^u)_{\mathtt s}' = \mathtt s^{\mathtt k(u) + \mathtt h} \zeta$ from \eqref{new-xi},
we find that for $(u, v) \in \overline{\mathbb V}_{24}$,
\begin{align*} 
\mathtt J &\bigl(h b^u(b^v-1); \, \mathtt s_m, \mathtt K_m, \mathtt b_m \bigr) \\  &:= \#\bigl\{\mathtt K_m - \mathtt k(u) - \mathtt h \leq j \leq \mathtt b_m -\mathtt k(u) - \mathtt h -1 : 1 \leq \mathtt d_{j-1}(|\zeta|) + \mathtt s \mathtt d_j(|\zeta|) \leq \mathtt s_m^2-2 \bigr\}  \\
&\geq \#\bigl\{1 \leq j \leq R-1: 1 \leq \mathtt d_{j-1}(|\zeta|) + \mathtt s_m \mathtt d_j(|\zeta|) \leq \mathtt s_m^2-2 \bigr\}\\ 
& = \mathfrak n_{R} \bigl(\mathbf d_R(u; {|\xi|}, b, \mathtt s_m), \mathtt s_m\bigr) > \kappa_m R, 
\end{align*} 
where $\mathbf d_R(u; {|\xi|}, b, \mathtt s_m)$ is defined as in \eqref{def-dl} and \eqref{dell}.
Substituting this into the estimate for $\mathtt v_m$ given by \eqref{corollary-composite-estimate} in Corollary \ref{corollary-composite} leads to the following bound for $\mathcal Y_4$:
\begin{align}
\mathcal Y_4 &= \sum_{u, v, N} \Bigl\{\frac{1}{N^3}\bigl| \mathtt v_m \bigl( hb^u(b^v-1) \bigr) \bigr| \; \Bigl| \; (u, v, N) \in \overline{\mathbb V}_{24} \Bigr\} \nonumber \\ 
&\leq C_0 \sum_{u, v, N} \Biggl\{ \frac{1}{N^3} \Bigl[ \frac{\mathtt s_m^{\mathtt a_m - \mathtt K_m}}{\mathtt N_m} + \Bigl( 1 - \frac{c_0}{\mathtt s_m^5}\Bigr)^{\kappa_m R} \Bigr] \; \Biggl| \; \begin{aligned} &(u, v) \in \mathbb Z_N^2, \\ &\frac{\mathtt a_m}{4}  {\log_b (\mathtt s_m) } < N < m \mathtt b_m \log_b(\mathtt s_m) \end{aligned} \Biggr\} \nonumber \\
&\leq C_0 \sum_{N, R} \Biggl\{\frac{1}{N} \Bigl[ \frac{\mathtt s_{m-1}^{\mathtt b_{m-1}}}{\mathtt s_m^{\mathtt K_m}} + \Bigl( 1 - \frac{c_0}{\mathtt s_m^5}\Bigr)^{\kappa_m R}\Bigr] \; \Biggl| \; \begin{aligned} &\mathtt s_m^{R-1} \leq N < \mathtt s_m^R \\  &\frac{\mathtt K_m^2}{2} < R < 2\mathtt K_m^2 \end{aligned} \Biggr\} \nonumber \\
&\leq C_0 {\log(\mathtt s_m)} \sum_{R} \Bigl\{\frac{\mathtt s_{m-1}^{\mathtt b_{m-1}}}{\mathtt s_m^{\mathtt K_m}} + \Bigl( 1 - \frac{c_0}{\mathtt s_m^5}\Bigr)^{\kappa_m R} \; \Bigl| \; \frac{\mathtt K_m^2}{2} < R < 2\mathtt K_m^2 \Bigr\} \nonumber \\
&\leq C_0 {\log(\mathtt s_m)} \Bigl[ \mathtt K_m^2 \frac{\mathtt s_{m-1}^{\mathtt b_{m-1}}}{\mathtt s_m^{\mathtt K_m}} + \kappa_m^{-1} \mathtt s_m^5 \Bigl( 1 - \frac{c_0}{\mathtt s_m^5}\Bigr)^{\kappa_m \mathtt K_m^2/4} \Bigr] \leq C_0 2^{-c_0 \mathtt K_m}.  \label{Y4-final}
\end{align} 
The calculations leading to the last inequality are identical to \eqref{X3-final} and have been described in \eqref{theta-sum}--\eqref{X3-verification-2}. 
\vskip0.1in
\noindent Combining the inequalities \eqref{Y1-final}, \eqref{Y2-final}, \eqref{Y3-final} and \eqref{Y4-final} establishes the claim \eqref{Y1+Y2+Y3+Y4} in the statement of Lemma \ref{Yj-lemma}.

\Addresses
\end{document}